\documentclass [leqno, spanish, 10pt]{amsart}

\usepackage[headings, myheadings]{fullpage}

\usepackage{lmodern}
\usepackage{microtype}

\usepackage[T1]{fontenc}
\usepackage[utf8]{inputenc}
\usepackage{amssymb, amsmath, amsfonts, amsbsy, amscd, amsxtra, stmaryrd}
\usepackage{amsthm}
\usepackage[bbgreekl]{mathbbol} 
\usepackage[mathscr]{euscript}
\usepackage{upgreek}
\usepackage{mathtools}
\usepackage{yhmath} 
\usepackage[inline]{enumitem}

\usepackage{url}
\usepackage{hyperref}
\hypersetup{
	colorlinks=true, linktocpage=true, hidelinks
}

\newtheorem{theorem}{Theorem}[section]

\newtheorem{corollary}[theorem]{Corollary}
\newtheorem{lemma}[theorem]{Lemma}
\newtheorem{definition}[theorem]{Definition}

\theoremstyle{definition}

\usepackage{thmtools}

\declaretheoremstyle[
  headfont=\normalfont\bfseries,
  sharenumber = theorem,
  bodyfont=\normalfont,
    qed = {\hbox{$\triangleleft$}}
]{examplestyle2}

\declaretheorem[
  style=examplestyle2,
  title=Example,
  refname={example,examples},
  Refname={Example,Examples}
]{example}

\declaretheorem[
  style=examplestyle2,
  title=Remark,
  refname={remark, remarks},
  Refname={Remark, Remarks}
]{remark}
\numberwithin{equation}{section}

\makeatletter
\def\l@subsection{\@tocline{2}{0pt}{1pc}{4.6em}{}}
\renewcommand{\tocsubsection}[3]{%
  \indentlabel{\@ifnotempty{#2}{\hspace*{2.3em}\makebox[2.3em][l]{%
    \ignorespaces#1 #2.\hfill}}}#3}
\makeatother

\newcommand{\chnabla}{\smash{\overline{\widehat{\nabla}}}}
\newcommand{\cone}{\mathbb{K}}

\newcommand{\lr}{\mathsf{r}}
\newcommand{\lric}{\mathsf{ric}}
\newcommand{\sn}{\mathsf{S}}
\newcommand{\cn}{\mathsf{C}}
\newcommand{\cliff}{\operatorname{Cl}} 
\newcommand{\jrad}{\operatorname{rad}}
\newcommand{\brad}{\operatorname{rad}}
\newcommand{\ctr}{\mathfrak{Z}}

\newcommand{\torus}{\mathbb{T}}
\newcommand{\sphere}{\mathbb{S}}
\newcommand{\jsphere}{\mathbb{J}}

\newcommand{\Euc}{E}
\newcommand{\SEuc}{SE}
\newcommand{\euc}{\mathfrak{e}}
\newcommand{\su}{\mathfrak{su}}
\renewcommand{\sp}{\mathfrak{sp}}
\newcommand{\so}{\mathfrak{so}}
\newcommand{\mat}[2]{\mathscr{M}_{#1}(#2)}
\newcommand{\heisen}{\mathfrak{heis}}
\newcommand{\Heisen}{\operatorname{Heis}}

\newcommand{\onabla}{\op{\nabla}}
\newcommand{\imt}{\iota}
\newcommand{\std}{\ste^{\ast}}
\newcommand{\er}{\rho}
\newcommand{\divn}{\div_{\nabla}}

\newcommand{\zero}{\mathsf{Zero}}
\newcommand{\mc}{\omega_{G}}

\newcommand{\hopf}{\mathscr{H}}
\newcommand{\btau}{\bar{\tau}}

\newcommand{\fie}{\mathbb{k}}
\newcommand{\fiet}{\mathbb{k}^{\ast}}

\newcommand{\conftor}{\mathcal{K}}
\newcommand{\bt}{\mathcal{L}}

\newcommand{\sT}{\mathscr{T}}

\renewcommand{\r}{\mathsf{r}}

\newcommand{\chr}{\operatorname{char}}

\newcommand{\alg}{\mathbb{A}}

\newcommand{\h}{\mathfrak{h}}

\newcommand{\prin}{\mathsf{Prin}}
\newcommand{\hol}{\mathsf{hol}}
\newcommand{\dev}{\mathsf{dev}}

\renewcommand{\div}{\operatorname{div}}

\newcommand{\om}{\omega}

\newcommand{\mlt}{\circ}

\newcommand{\Om}{\Omega}

\newcommand{\vol}{\mathsf{vol}}

\newcommand{\tD}{\tilde{D}}

\newcommand{\ka}{\kappa}

\newcommand{\sR}{\mathscr{R}}

\newcommand{\T}{\mathcal{T}}

\renewcommand{\part}{\vdash}

\newcommand{\rad}{\mathbb{E}}
\newcommand{\vdens}{\mathcal{V}}
\newcommand{\qa}[1][]{\mathbb{A}\ifx#1\empty\else{\left(\frac{#1}{\fie}\right)}\fi}
\newcommand{\qaf}[2][]{\ifx#1\mathbb{A}\else{\left(\frac{#1}{#2}\right)}\fi}
\newcommand{\dens}[1][]{ \mathcal{D}\ifx#1\empty\else{^{#1}}\fi}

\newcommand{\Id}{\operatorname{Id}}

\newcommand{\dum}{\,\cdot\,\,}
\newcommand{\Ga}{\Gamma}

\newcommand{\ctn}{T^{\ast}N}

\newcommand{\prm}{\prime}
\newcommand{\ric}{\mathsf{Ric}}

\newcommand{\Q}{\mathcal{Q}}

\renewcommand{\sc}{\mathit{s}}

\renewcommand{\j}{\mathsf{i}}

\newcommand{\taut}{\mathbb{O}}

\newcommand{\la}{\lambda}
\newcommand{\ep}{\epsilon}

\newcommand{\zmodtwo}{\mathbb{Z}/2\mathbb{Z}}
\newcommand{\reat}{\mathbb{R}^{\ast}}

\newcommand{\reap}{\mathbb{R}^{+}}
\makeatletter
\newcommand{\ext}{\@ifnextchar^\@extp{\@extp^{\,}}}
\def\@extp^#1{\mathop{\bigwedge\nolimits^{\!#1}}}
\makeatother
\newcommand{\cinf}{C^{\infty}}

\newcommand{\Det}{\operatorname{Det}}

\newcommand{\ben}{[\Bar{\nabla}]}

\newcommand{\hD}{\Hat{D}}
\def\brs#1{\ensuremath{\mkern+0mu{\mkern+1mu ^{s}\mkern-2mu#1}}}

\newcommand{\hDs}{\hD}
\newcommand{\Ds}{\brs{D}}

\newcommand{\eno}{\text{End}}

\newcommand{\lin}{\mathsf{L}}
\newcommand{\linft}{\mathscr{L}}

\newcommand{\si}{\sigma}

\newcommand{\pr}{\partial}

\newcommand{\ctm}{T^{\ast}M}

\newcommand{\hR}{\hat{R}}

\newcommand{\pxi}{\tfrac{\partial}{\partial x^{i}}}
\newcommand{\pxj}{\tfrac{\partial}{\partial x^{j}}}

\newcommand{\pzi}{\tfrac{\partial}{\partial z^{i}}}

\newcommand{\gr}{\operatorname{Gr}}

\newcommand{\sign}{\operatorname{sgn}}
\newcommand{\bnabla}{\bar{\nabla}}
\newcommand{\anabla}{\wideparen{\nabla}}
\newcommand{\integer}{\mathbb{Z}}
\newcommand{\enb}{\{\nabla\}}

\def\op#1{\ensuremath{\,\!^{\mathsf{op}}#1}}

\newcommand{\en}{[\nabla]}

\newcommand{\lie}{\mathfrak{L}}

\newcommand{\nablap}{\nabla^{\prime}}
\newcommand{\re}{\operatorname{Re}}
\newcommand{\im}{\operatorname{Im}}

\newcommand{\tpc}{\tau}
\newcommand{\spc}{\chi}
\newcommand{\wpc}{\nu}

\newcommand{\epc}{\eta}
\newcommand{\aff}{\mathcal{A}}

\newcommand{\F}{\mathcal{F}}

\newcommand{\lb}{\langle}
\newcommand{\ra}{\rangle}

\newcommand{\stz}{\mathbb{V}\setminus\{0\}}
\newcommand{\ste}{\mathbb{V}}

\newcommand{\stez}{\mathbb{V} \setminus \{0\}}

\newcommand{\projp}{\mathbb{P}^{+}}

\newcommand{\al}{\alpha}
\newcommand{\be}{\beta}
\newcommand{\ga}{\gamma}
\newcommand{\frameb}[1][]{
	\mathcal{F}\ifx#1\empty\else{\left(#1\right)}\fi
}
\newcommand{\orientb}[1][]{
	\mathcal{O}\ifx#1\empty\else{\left(#1\right)}\fi
}

\newcommand{\zsec}{\mathsf{0}}
\newcommand{\zs}{\mathsf{0}}

\newcommand{\G}{\mathscr{G}}

\newcommand{\spn}{\text{Span}\,}

\newcommand{\emf}{\mathcal{E}}
\newcommand{\lmf}{\mathcal{L}}

\newcommand{\hnabla}{\widehat{\nabla}}
\newcommand{\tnabla}{\tilde{\nabla}}

\newcommand{\sll}{\mathfrak{sl}}

\newcommand{\eul}{\mathbb{X}}
\newcommand{\veul}{\mathbb{X}^{\vdens}}

\newcommand{\proj}{\mathbb{P}}

\DeclareMathOperator{\diff}{Diff}
\DeclareMathOperator{\vect}{Vec}

\DeclareMathOperator{\Aut}{Aut}

\newcommand{\m}{\mathfrak{m}}

\newcommand{\g}{\mathfrak{g}}

\newcommand{\n}{\mathfrak{n}}

\newcommand{\ad}{\operatorname{ad}}
\newcommand{\Ad}{\operatorname{Ad}}

\newcommand{\tensor}{\otimes}
\newcommand{\stensor}{\widehat{\tensor}}

\newcommand{\rea}{\mathbb R}
\newcommand{\com}{\mathbb C}

\newcommand{\tr}{\mathsf{tr} \,}
\newcommand{\rank}{\text{rank\,}}

\newcommand{\ann}{\text{Ann}\,}

\renewcommand{\P}{\mathcal{P}}

\def\crn#1#2{{\vcenter{\vbox{
        \hbox{\kern#2pt \vrule width.#2pt height#1pt
           }
          \hrule height.#2pt}}}}

\begin{document}
\title{Conelike radiant structures}

\author{Daniel J. F. Fox} 
\address{Departamento de Matemática Aplicada\\ Escuela Técnica Superior de Arquitectura de Madrid\\ Universidad Politécnica de Madrid\\Av. Juan de Herrera 4 \\ 28040 Madrid España}
\email{daniel.fox@upm.es}


\begin{abstract}
Analogues of the classical affine-projective correspondence are developed in the context of statistical manifolds compatible with a radiant vector field. These utilize a formulation of Einstein equations for special statistical structures that generalizes the usual Einstein equations for pseudo-Riemannian metrics and is of independent interest.

A conelike radiant structure is a not necessarily flat affine connection equipped with a family of surfaces that behave like the intersections of the planes through the origin with a convex cone in a real vector space. A radiant structure is a torsion-free affine connection and a vector field whose covariant derivative is the identity endomorphism. A radiant structure is conelike if for every point and every two-dimensional subspace containing the radiant vector field there is a totally geodesic surface passing through the point and tangent to the subspace. Such structures exist on the total space of any principal bundle with one-dimensional fiber and on any Lie group with a quadratic structure on its Lie algebra. 

The affine connection of a conelike radiant structure can be normalized in a canonical way to have antisymmetric Ricci tensor. Applied to a conelike radiant structure on the total space of a principal bundle with one-dimensional fiber this yields a generalization of the classical Thomas connection of a projective structure. The compatibility of radiant and conelike structures with metrics is investigated and yields a construction of connections for which the symmetrized Ricci curvature is a constant multiple of a compatible metric that generalizes well-known constructions of Riemannian and Lorentzian Einstein-Weyl structures over Kähler-Einstein manifolds having nonzero scalar curvature. A formulation of Einstein equations for special statistical manifolds is given that generalizes the Einstein-Weyl equations and encompasses these more general examples.

There are constructed left-invariant conelike radiant structures on a Lie group endowed with a left-invariant nondegenerate bilinear form, and the case of three-dimensional unimodular Lie groups is described in detail.
\end{abstract}

\maketitle

\setcounter{tocdepth}{1}  
\tableofcontents

\section{Introduction}
A \emph{statistical structure} $(\nabla, h)$ comprises a pseudo-Riemannian metric $h$ and a torsion-free affine connection $\nabla$ satisfying $\nabla_{[i}h_{j]k} = 0$. The compatibility condition is equivalent to the complete symmetry of the \emph{cubic form} $\bt_{ijk} = \nabla_{i}h_{jk}$. The torsion-free connection $\bnabla = \nabla + h^{kp}\bt_{ijp}$ constitutes with $h$ the \emph{conjugate} statistical structure having cubic form $\bnabla_{i}h_{jk} = -\bt_{ijk}$. Conjugacy of statistical structures is an involution that generalizes duality of hypersurfaces in flat affine space. The self-conjugate statistical structures are pairs comprising a metric and its Levi-Civita connection. 

The notion of statistical structure was introduced by Amari, see \cite{Amari-Nagaoka}, and Chentsov, see \cite{Chentsov}, in the course of applying differential geometric methods to study the Fisher-Rao Riemannian metric on the parameter space of a parametric family of probability spaces \cite{Ay-Jost-Le-Schwachhofer}. Background on this approach to parametric statistics, now usually called information geometry, can be found in the books of Amari  \cite{Amari, Amari-informationgeometry, Amari-Nagaoka} and in \cite{Ay-Jost-Le-Schwachhofer} which also surveys the relevant literature.

A statistical structure $(\nabla, h)$ is \emph{special} if $\nabla$ preserves $\det h$. That a cooriented nondegenerate hypersurface in flat affine space acquires a pair of conjugate special statistical structures (see Example \ref{affinehypersurfaceexample}) was emphasized by H. Matsuzoe \cite{Matsuzoe, Matsuzoe-conformallyprojectively} and implicitly plays a role in earlier work of T. Kurose \cite{Kurose-dual, Kurose} and H. Amari and collaborators \cite{Amari, Amari-Nagaoka}, among others. In this context, statistical structures have been called by other names such as \emph{Codazzi structures} \cite{Marugame-codazzi, Shima} or the author's terminology \emph{exact AH structures} \cite{Fox-2dahs, Fox-crm}. A statistical structure $(\nabla, h)$ is \emph{flat} if $\nabla$ is flat. There is an extensive literature treating flat statistical structures which are also called \emph{Hessian structures} \cite{Cardoso-Mohaupt, Duistermaat, Shima-compacthessian, Shima-homogeneoushessian, Shima-Yagi, Totaro} or \emph{Kähler affine structures} \cite{Cheng-Yau-realmongeampere, Fox-schwarz, Klartag-Kolesnikov}. Possibly singular flat statistical structures are related to the Frobenius manifolds and WDVV equations studied in mathematical physics \cite{Chen-Kontsevich-Schwarz, Dubrovin, Dubrovin-integrable,Hertling-Manin, Hitchin-frobenius}. 

Even without such motivations, statistical structures are of interest simply as a direct generalization of pseudo-Riemannian metrics and in recent years have received more attention in this sense, for example in \cite{Blaga-Chen, Fujitani, Le, Marugame-bonnet, Opozda-bounded} (no attempt is made here at completeness).

A structure arising in a diversity of initially apparently little related contexts is likely to be rich, while the study of its properties from the perspective appropriate to one context may yield unexpected insights in the other contexts.

Here the differential geometric point of view on statistical structures is pursued. This differential geometric perspective leads to focusing on a class of examples different from those that form the center of attention in the main literature on statistical structures. The specific context considered (described in detail later in the introduction) generalizes the classical relation between geometric structures on a domain in projective space and geometric structures on the cone over this domain. It leads to the following notions:
\begin{itemize}
\item A notion of Einstein equations for special statistical structures that generalizes the usual Einstein equations for pseudo-Riemannian metrics and is preserved by conjugacy. For example, the Cheng-Yau metric of a properly convex flat projective structure constitutes with a distinguished affine connection representing the given projective structure an Einstein special statistical structure.
\item A notion of locally statistical structure, called an \emph{AH (affine hypersurface)} structure that extends the usual notion in much the same way that Weyl structures extend pseudo-Riemannian metrics. Einstein equations are defined also for these more general AH structures.
\end{itemize}
The author studied these notions previously \cite{Fox-2dahs, Fox-crm, Fox-schwarz} in the more general context of AH structures and without reference to statistical structures as such. It is hoped that the exposition here will be more accessible. 

The definition of Einstein equations for special statistical structures is less obvious than might be expected naively. It is not enough to require the vanishing of the trace-free Ricci tensor, nor that this hold for the statistical structure and its conjugate, as this need not imply constancy of some understandable quantity interpretable as a scalar curvature. It appears that such an additional condition must be imposed as part of the definition. Its formulation and motivation, which are both a bit subtle, are treated in Section \ref{einsteinstatisticalsection}. They are such that a special statistical structure solves the Einstein equations for statistical structures if and only if its underlying metric solves the Einstein field equations with a stress-energy tensor built from its cubic form (see Corollary \ref{conservativestatisticalcorollary} for a precise statement and \eqref{stdefined} for the definition of the stress-energy tensor).

Statistical structures interact with other geometric structures, such as almost complex or symplectic structures or generalized geometries. See \cite{Blaga-Nannicini, Haba-Matsuzoe} for examples.
Here a key role is played by the notion of a \emph{radiant statistical structure}, which is a triple $(\nabla, h, \rad)$ comprising a torsion-free affine connection $\nabla$, a pseudo-Riemannian metric $h$, and a vector field $\rad$ such that $(\nabla, h)$ is a statistical structure; $\rad$ is radiant, meaning $\nabla \rad = \Id_{TM}$; and $h$ is self-similar, meaning $\lie_{\rad}h = 2h$. Such structures occur naturally in a variety of contexts: on pointed convex cones as a consequence of deep work of Cheng-Yau \cite{Cheng-Yau-realmongeampere,Cheng-Yau-affinehyperspheresI, Fox-schwarz,Hildebrand-canonical, Loftin-affinekahler}, total spaces of one-dimensional principal bundles over manifolds carrying particular geometric structures (see sections \ref{codazzisection}-\ref{ewsection}), and in the construction of the Fefferman-Graham ambient metric \cite{Fefferman-Graham-ambient, Rodnianski-Shlapentokh-Rothman}. From the point of view of the study of statistical structures the main results of the paper can be framed as extensions of these ideas:
\begin{itemize}
\item Theorem \ref{liftedahtheorem} that shows that if the curvature of a principal connection on a one-dimensional principal bundle over a manifold equipped with an Einstein special statistical structure is compatible in a certain way with the statistical structure then on the total space of the principal bundle there is an Einstein AH structure built from the downstairs special statistical structure and the curvature of the principal connection. As is explained in more detail later in the introduction, this construction generalizes vastly the classical correspondence between projective space and affine space. It generalizes to statistical structures constructions associating Einstein-Weyl structures with positive scalar curvature Kähler-Einstein metrics \cite{Calderbank-Pedersen, Pedersen-Swann-submersions}.
\item Theorem \ref{convexprojectivetheorem} applies Theorem \ref{liftedahtheorem} to show that if a principal circle bundle over a compact orientable surface of genus at least two has Euler number of absolute value less than minus the Euler characteristic of the surface then on the total space of the circle bundle there is a Lorentzian signature Einstein AH structure that, in a certain precise sense, induces on the base surface the Cheng-Yau Einstein special statistical structure of a properly convex projective structure.
\end{itemize}
The idea of a correspondence between geometric structures on a base manifold and other geometric structures on the total space of a fiber bundle over the base goes back in different forms and specific instances to Kaluza-Klein \cite{Wang-einstein}, T.~Y. Thomas \cite{Thomas-book}, and many others. In the context of statistical structures such constructions have been explored in various forms. In \cite{Matsuzoe-Inoguchi}, Matsuzoe and Inoguchi showed that the tangent bundle over a flat statistical manifold admits a natural complex structure compatible with a statistical structure whose underlying metric is the Norden metric. In \cite{Balan-Peyghan-Sharahi} there were studied statistical structures with underlying metric a Sasaki metric and there was explored the case where the metric on the base manifold is Kähler. The results in Section \ref{ewsection} are analogous, with a one-dimensional principal bundle in place of the tangent bundle, and an Einstein statistical structure on the base.

The formulations and proofs of Theorems \ref{liftedahtheorem} and \ref{convexprojectivetheorem} are other related results utilize the notion of conelike radiant structure which gives the title of the article. The remainder of the introduction is devoted to motivating this notion, which has independent interest. The reader mainly interested in statistical structures can begin reading in Section \ref{codazzisection}, although the results of the earlier sections are necessary for the proofs of the results in Section \ref{ewsection}. 

The notion of \emph{conelike radiant structure} formalizes the idea of a not necessarily flat affine connection equipped with a family of surfaces that behave like the intersections with a cone of the planes through the origin in a real vector space. A \emph{radiant structure} on a smooth manifold $M$ is a pair $(\nabla, \rad)$ comprising a torsion-free affine connection $\nabla$ and a vector field $\rad$ such that $\nabla \rad$ is the identity endomorphism of the tangent bundle $TM$. In a manifold with a radiant structure, a smoothly immersed surface is \emph{planelike} if it is totally geodesic and everywhere tangent to $\rad$. A radiant structure is \emph{conelike} if it \emph{admits a complete set of planelike surfaces} in the sense that for every point $p \in M$ and every two-dimensional subspace $L \subset T_{p}M$ that contains $\rad_{p}$, there is a planelike surface $\Sigma \subset M$ containing $p$ and such that $T_{p}\Sigma = L$. The curvature of a radiant structure $(\nabla, \rad)$ means the curvature of $\nabla$. The flat radiant structure on a real vector space is that given by the flat affine structure and the Euler vector field generating dilations. The interior of a cone carries the flat conelike radiant structure generated by this radiant structure and whose planelike surfaces are its intersections with two-dimensional subspaces. The planelike surfaces of a conelike radiant structure generalize the two-dimensional subspaces through the origin in the flat model. Their dilation invariance is replaced by tangency to the radiant vector field, and their flatness is replaced by being totally geodesic.

Conelike radiant structures exist on manifolds with nontrivial topology. For example, by Theorem \ref{extendedthomastheorem}, the total space of any principal $G$-bundle with $\dim G = 1$ admits conelike radiant structures and, by Theorem \ref{killingnonnulltheorem}, on a semisimple Lie group there is a conelike radiant structure associated with each adjoint orbit in its Lie algebra. More generally, by Theorem \ref{gnonnulltheorem}, any Lie group whose Lie algebra is equipped with a structure of a quadratic Lie algebra admits a left-invariant conelike radiant structure. In general these conelike radiant structures are not flat, although their Ricci tensors are antisymmetric.

The projection $\stez \to \proj(\ste)$ of the complement of the origin in an $(n+1)$-dimensional real vector space $\ste$ to the $n$-dimensional projectivization $\proj(\ste)$ is the model for the affine-projective/homogeneous-inhomogeneous paradigm that pervades much of geometry. Together with some additional functional data this setup yields projective and hyperbolic geometry. Work due in its essential points to Cheng-Yau \cite{Cheng-Yau-affinehyperspheresI, Cheng-Yau-mongeampere}, establishing a picture conjectured by Calabi \cite{Calabi-bernsteinproblems, Calabi-completeaffine} and developed also partly in \cite{Loewner-Nirenberg, Sasaki}, extends this basic picture to the context of affine spheres and the cones over them. This context includes the universal covers of properly convex flat real projective structures. It shows that the interior of any convex cone that is pointed (contains no half space) is foliated in a unique way by complete hyperbolic affine spheres asymptotic to the cone and having center at the vertex of the cone (see also \cite{Loftin-affinekahler, Loftin-survey}). The cone itself supports a related Lorentzian Monge-Ampere metric \cite{Fox-schwarz} (see also \cite{Calabi-nonhomogeneouseinstein, Loftin-Yau-Zaslow} for closely related constructions). For the Minkowski cone the foliating affine spheres are hyperboloids with Blaschke metrics homothetic to the hyperbolic metric and the associated Lorentzian Monge-Ampere metric is the Minkowski metric. The notion of conelike radiant structure introduced here makes it possible to generalize this picture in several ways - most obviously by dropping the flatness of the affine connection (equivalently the projective flatness of the downstairs projective structure) but also by abandoning convexity and replacing the Euler field generating the group of dilations on a vector space by a more general notion of Euler vector field. 

Example \ref{projflatexample} indicates in more detail the basic structural features whose generalization to a curved setting, culminating in the results of Section \ref{ewsection}, occupies much of the present article. Example \ref{projflatexample} is written using notations that make it straightforward to see how it is a special case of results given later, such as Lemma \ref{liftedmetriclemma} and Theorem \ref{liftedahtheorem}, that can be understood as extending it to a curved setting. 

\begin{example}\label{projflatexample}
Let $\rho:\F = \ste \setminus\{0\} \to M = \projp(\ste)$ be the principal $\reap$-bundle defining the oriented projectivization of the $(n+1)$-dimensional vector space $\ste$, where $\reap = GL^{+}(1, \rea)$ is the multiplicative group of positive real numbers, regarded as a topological group. 
Let $z^{I}$ be coordinates on $\ste$ so that $\tfrac{\pr}{\pr z^{I}}$ is a parallel frame for the flat affine connection $\hnabla$ on $\ste$. Where $z^{n+1} \neq 0$, $t = z^{n+1}$ is a fiber coordinate corresponding with the flat principal $\reap$-connection $\al = d\log t= d\log z^{n+1}$ and determining local coordinates $x^{i}$ on the base by $z^{i} = tx^{i}$, so that $(x^{i}, t)$ are coordinates adapted to the bundle structure in which $\eul = t\tfrac{\pr}{\pr t} = z^{I}\tfrac{\pr}{\pr z^{I}}$ is the usual Euler field. From $\tfrac{\pr}{\pr z^{i}} = t^{-1}\tfrac{\pr}{\pr x^{i}}$ and $\tfrac{\pr}{\pr z^{n+1}} = \tfrac{\pr}{\pr t} - t^{-1}x^{p}\tfrac{\pr}{\pr x^{p}}$, it follows that the $\al$-horizontal lift of $\tfrac{\pr}{\pr x^{i}}$ is $z^{n+1}\tfrac{\pr}{\pr z^{i}}$. The connection $D$ induced on the base by $\hnabla$ and $\al$ is the flat connection with respect to which $\tfrac{\pr}{\pr x^{i}}$ is a parallel local frame. For a function $u$ on the base $M$ there holds
\begin{align}\label{fhddu}
\hnabla \rho^{\ast}(du) = \rho^{\ast}(Ddu) - du \tensor d\log t - d\log t \tensor du.
\end{align}
Define $F = t^{2}\rho^{\ast}(u)$. The function $u$ is related to $F$ in that $u^{2}|dx^{1}\wedge \dots \wedge dx^{n}|^{-2/(n+1)}$ is the expression in local coordinates of the density on $M$, viewed as a section of a line bundle associated with the principal bundle $\rho:\F \to M$, corresponding with the positive homogeneity $2$ function $F$, viewed as an equivariant function on the total space of $\F$. 
Let $\ga = -(1/2)d\log u$ and define 
\begin{align}
\be = \al - \rho^{\ast}(\ga) = d\log t + \tfrac{1}{2}\rho^{\ast}(d\log u) = \tfrac{1}{2}F^{-1}d F. 
\end{align}
When $\Sigma = \projp(F^{-1}(0)) = u^{-1}(0)$ is a submanifold, $\be$ is a principal $\reap$-connection over the connected component $M^{+} = \{[z] \in M: F(z) > 0\}$ of $M \setminus \Sigma$.
The connection $\nabla$ induced on $M^{+}$ by $\hnabla$ and $\be$ equals $D + 2\ga_{(i}\delta_{j)}\,^{k}$ and the $\be$-horizontal lift of $\tfrac{\pr}{\pr x^{i}}$ is 
\begin{align}
z^{n+1}\tfrac{\pr}{\pr z^{i}} + \rho^{\ast}\left(\ga(\tfrac{\pr}{\pr x^{i}})\right)\eul = z^{n+1}\tfrac{\pr}{\pr z^{i}} - \tfrac{1}{2} \rho^{\ast}(u^{-1}\tfrac{\pr u}{\pr x^{i}})\eul.
\end{align}
Because $D$ is flat, by \eqref{projvary}, the projective Schouten tensor $P$ of $\nabla$ is
\begin{align}\label{fpu}
P = D\ga - \ga \tensor \ga = - \tfrac{1}{2}u^{-1}Ddu + \tfrac{1}{4}d\log u \tensor d\log u = -|u|^{-1/2}Dd|u|^{1/2}.
\end{align} 
The tensors $H = \tfrac{1}{2}\hnabla dF$ and $G = \hnabla \be + 2\be \tensor \be$ are related by $H =  FG$. Straightforward computation using \eqref{fhddu} and \eqref{fpu} shows
\begin{align}\label{fpu2}
\begin{split}
G & = F^{-1}H  = \al \tensor \al - \rho^{\ast}(\ga)\tensor \al - \al \tensor \rho^{\ast}(\al)  + \tfrac{1}{2}\rho^{\ast}(u^{-1}Ddu)\\
& =\be \tensor \be - \rho^{\ast}(\ga)\tensor \rho^{\ast}(\ga) + \tfrac{1}{2}\rho^{\ast}(u^{-1}Ddu)= \be \tensor \be - \rho^{\ast}(P).
\end{split}
\end{align}
Consider the volume forms $\mu = dx^{1}\wedge \dots \wedge dx^{n}$ on $M^{+}$ and $\Psi = d\log t\wedge \rho^{\ast}(\mu) = \be \wedge \rho^{\ast}(\mu)$ on $\F$.  From \eqref{fpu} and \eqref{fpu2} there follows
\begin{align}
\begin{split}
\tfrac{\det \hnabla dF}{\Psi^{\tensor 2}} &= 2^{n+1}F^{n+1}\tfrac{\det G}{\Psi^{\tensor 2}} = (-1)^{n}2^{n+1}t^{2(n+1)}\rho^{\ast}\left(u^{n+1}\tfrac{\det P}{\mu^{\tensor 2}}\right)
\\&
 = 2^{n+1}t^{2(n+1)}\rho^{\ast}\left((|u|^{1/2})^{n+2}\tfrac{\det Dd|u|^{1/2}}{\mu^{\tensor 2}}\right).
\end{split}
\end{align}
In particular, there is a nonzero constant $\ka$ such that 
\begin{align}\label{affinespheres}
&\det \hnabla dF = \ka \Psi^{\tensor 2},& &\text{if and only if} &&\det Dd|u|^{1/2} = \ka (|u|^{1/2})^{-(n+2)}\mu^{\tensor 2}. 
\end{align}
In this case the level sets of $F$ are affine spheres \cite{Fox-schwarz, Loftin-affinekahler, Loftin-survey}. The solution of these equations on properly convex cones or domains under appropriate boundary/growth conditions is due to Cheng and Yau \cite{Cheng-Yau-mongeampere, Cheng-Yau-affinehyperspheresI, Cheng-Yau-realmongeampere}. Klartag \cite{Klartag-elliptic} has shown how to extend this picture to obtain incomplete affine spheres that are elliptic rather than hyperbolic. (The absolute values can be removed from \eqref{affinespheres} at the cost of distracting fussing about signs.) There seems to be no systematic study of \eqref{affinespheres} in the absence of convexity, though model solutions can be constructed from relative invariants of irreducible prehomogeneous vector spaces \cite{Fox-autoiso}.

For the simplest concrete example, let $\ep \in \{\pm 1\}$ and let $u(x) = 1 +\ep x^{i}x^{j}g_{ij}$, where $g_{ij}$ is a $D$-parallel metric on the, so that $F(z) = (z^{n+1})^{2} + \ep z^{i}z^{j}\rho^{\ast}(g)_{ij} = t^{2}(1 +\ep x^{i}x^{j}g_{ij})$. Since $F$ is a quadratic polynomial, $\hnabla H  = 0$. If $\ep = -1$, the level set $F^{-1}(0)$ is the null cone of the Minkowski metric, and the submanifold $\Sigma = \projp(F^{-1}(0))$ is an $(n-1)$ sphere identified in the coordinates $x$ on the base as the unit Euclidean sphere in $\rea^{n}$, whereas if $\ep = 1$ then $F^{-1}(0)$ is empty. Over the connected component $M^{+}$ the one-form $\be = \tfrac{1}{2}F^{-1}dF = d\log{t} +\ep \tfrac{x^{a}g_{ab}dx^{b}}{1 + \ep |x|^{2}}$ is a principal $\reap$-connection. The $\be$-horizontal lifts of the standard vector fields $\tfrac{\pr}{\pr x^{i}}$ are
\begin{align}
\widehat{\tfrac{\pr}{\pr x^{i}}} = \tfrac{\pr}{\pr x^{i}} -\ep \tfrac{x^{a}g_{ia}}{1 +\ep |x|^{2}}\eul = z^{n+1}\tfrac{\pr}{\pr z^{i}} -\ep \tfrac{z^{n+1}z^{a}g_{ia}}{F}\eul.
\end{align}
Let $\nabla$ be the unique symmetric representative of the standard flat projective structure on $\projp(\ste)$ inducing $\be$ and let $P$ be its projective Schouten tensor. Straightforward computation using the formulas just given shows
\begin{align}
&\nabla_{\pxi}\pxj = -\ep \tfrac{x^{a}g_{ia}}{1 + \ep |x|^{2}}\pxj -\ep \tfrac{x^{b}g_{jb}}{1 +\ep |x|^{2}}\pxi,&
&P(\pxi, \pxj) =  -\tfrac{\ep}{1 +\ep |x|^{2}}\left(g_{ij} -\ep \tfrac{x^{a}x^{b}g_{ia}g_{jb}}{1 +\ep |x|^{2}}\right).
\end{align}
Further computation shows that (whatever is $\ep$),  $H = F\left(\beta \tensor \beta - \rho^{\ast}(P)\right) = F(\hnabla \be + 2\be \tensor \be)$, and that $\nabla P = 0$. Hence $\nabla$ is the Levi-Civita connection of the projectively flat constant sectional curvature $\ep$ metric $-\ep P$. Note that $\nabla$ is projectively equivalent to the flat Euclidean connection $D$ and that $\ep P(\pxi, \pxj) =  u^{-1}Ddu(\pxi, \pxj)$.

Other choices of $u$ or $F$ solving \eqref{affinespheres} yield more interesting geometries. A more interesting example, for which $g_{ij}$ has indefinite signature is given by $u = v^{2}$ where
\begin{align}
v(x_{1}, x_{2}, x_{3}) = \left(x_{2}^{2}x_{3}^{2} + 18x_{1}x_{2}x_{3} - 4x_{1}x_{3}^{3} - 4x_{2}^{3} - 27x_{1}^{2} \right)^{1/4}
\end{align}
which solves $v^{-5}\det Dv = 81/16$. Its radial graph is a level set of the discriminant of a binary cubic form and is a homogeneous indefinite signature proper affine sphere. 
See \cite{Fox-autoiso, Fox-schwarz, Loftin-affinekahler, Loftin-survey, Loftin-Yau-Zaslow} for further examples and references. 
\end{example}

The picture of Example \ref{projflatexample} can be generalized by relaxing the conditions on the upstairs affine/homogeneous model comprising $\hnabla$, $\rad$, and the two-dimensional planes through the origin or on the downstairs projective/inhomogeneous model comprising $\en$ and its geodesics. Although related, the two approaches do not necessarily lead to exactly the same constructions or point of view. The author took the latter approach in \cite{Fox-ahs, Fox-2dahs, Fox-crm}. Here the first approach is taken, one point of departure being that flat radiant structures have already been studied extensively \cite{Choi-radiant, Fried-Goldman-Hirsch, Goldman-Hirsch}. Simply allowing the connection to have curvature is the most naive generalization. Another is to view a cone as a principal $\reat$ bundle and to consider instead principal $S^{1}$-bundles. 

To begin, Section \ref{radiantstructuresection} explores more general notions of compatibility between a connection $\nabla$ and a vector field $X$. Namely, $X$ is said to be \emph{dilatative} with respect to $\nabla$ if $\nabla_{i}X^{j} = f\delta_{i}\,^{j}$ for some smooth function $f$ (this generalizes the notion of a vector field concircular with respect to a Levi-Civita connection). It is shown that off the zero set of $f$, $\nabla$ and $X$ can be modified to yield a radiant structure. A more general notion of a projectively dilatative vector field is also considered. The conclusion is that if $X$ is projectively dilatative with respect to $\nabla$, then off the zero set of $X$ it is dilatative with respect to a connection projectively equivalent to $\nabla$. These results motivate an initial focus on radiant structures and this is what is done in the rest of the paper. Further exploration of the more general notions remains as a potentially interesting project.

Section \ref{radiantstructuresection} generalizes to general radiant structures basic results about flat radiant structures due to Fried, Goldman, and Hirsch \cite{Fried-Goldman-Hirsch, Goldman-Hirsch}.  For example, Lemmas \ref{radiantisolatedlemma} and \ref{paralleloneformlemma} show that a compact manifold supporting a radiant structure must have nonnegative Euler characteristic and admits no parallel volume form. A compact flat radiant manifold must have positive first Betti number. Since Theorem \ref{qaconetheorem} shows that the three sphere admits a radiant structure, this fails for non-flat radiant structures, but Lemma \ref{nosymmetriclemma} shows that a compact manifold with vanishing first Betti number admits no radiant structure with symmetric Ricci tensor. Benzecri \cite{Benzecri} showed that a compact surface admits a flat affine structure if and only if it has vanishing Euler characteristic. Both the torus and Klein bottle admit flat radiant affine structures \cite{Arrowsmith-Furness-locallysymmetric, Arrowsmith-Furness, Nagano-Yagi, Benoist-affine-tori}. Theorem \ref{2dtorustheorem} extends these results to non-flat radiant structures, showing that a compact surface admits a radiant structure if and only if it has Euler characteristic zero. This is deduced as follows. Because by Lemma \ref{radiantisolatedlemma} the Euler characteristic must be nonnegative, it suffices to rule out positive Euler characteristic. Lemma \ref{2dsymmetriclemma} shows that a radiant structure on any surface has symmetric Ricci tensor and Lemma \ref{riccisymmetriclemma} shows that on a manifold with vanishing first Betti number a torsion-free affine connection with symmetric Ricci tensor admits a parallel volume density. As Lemma \ref{paralleloneformlemma} precludes a parallel volume form for a radiant structure, this suffices. Note that a \emph{volume density} means a nonvanishing $1$-density, which is a different object from a volume form and makes sense on a nonorientable manifold. 
The rest of Section \ref{radiantstructuresection} records calculations regarding the curvature of radiant structures that are used throughout the remainder of the paper. 

As remarked, the situation in dimension three and higher is different. There are manifolds that admit no flat radiant structure that admit radiant structures with purely antisymmetric Ricci tensor.

The volume form $\Psi$ on $\ste$ in Example \ref{projflatexample} is an additional bit of structure that plays a role in particular with respect to the formulation of the Monge-Ampere equations \eqref{affinespheres} solved by $F$ and $u$. It is compatible with the radiant structure in that it is preserved by the affine connection $\hnabla$ and has homogeneity $\dim \ste$ with respect to the radiant vector field $\rad$. Both of these conditions generalize in an obvious way to the context of not necessarily flat radiant structures to yield a notion of equiaffine radiant structure. The generalization of the compatibility of $\rad$ and $\Psi$ leads to the notion of Euler structure, a vector field $\rad$ and volume form $\Psi$ on an $n$-manifold $M$ such that $\lie_{\rad}\Psi = n \Psi$. Section \ref{eulersection} explores this notion and its relation to the notion of an Euler-like vector field studied in \cite{Bursztyn-Lima-Meinrenken, HajSaeediSadegh-Higson, Meinrenken-eulerlike}. Although the two notions are independent, a radiant vector field is Euler-like. 

Basic examples of Euler manifolds, needed later in Section \ref{thomassection}, are the total spaces of real line bundles whose local sections transform like densities of weight $-1/(n+1)$. A pseudo-hyperplane line bundle is a real line bundle $\emf$ equipped with an isomorphism $\emf^{2n+2} \simeq (\Det TM)^{\tensor 2}$; its dual is called a pseudo-tautological line bundle. Up to isomorphism there is a unique such bundle with a given first Stiefel-Whitney class. The total space of any such line bundle carries in a canonical way a pseudo-Euler structure (\emph{pseudo} means the volume form is replaced by a volume density; see section \ref{eulersection} for the precise definition). Such a pseudo-tautological line bundle is a real analogue of a square-root of the canonical bundle of a complex manifold, and generalizing the construction of the Blaschke metric of a properly convex flat real projective structure from the Cheng-Yau construction of affine spheres requires considering projectively invariant Monge-Ampère equations on sections of such line bundles; this is discussed in the author's previous \cite{Fox-ahs, Fox-2dahs, Fox-crm} and the author intends to discuss it further in future work formulating a projective analogue of the Calabi conjecture. See also \cite[Section $4$]{Cap-Mettler} where these notions are discussed in the more general framework of parabolic geometries associated with a $|1|$-grading.

A radiant Euler structure is a radiant structure equipped with volume form $\Psi$ that with $\rad$ constitutes an Euler structure. It is \emph{equiaffine} if $\Psi$ is parallel. Section \ref{eulersection} gives the basic properties of such structures. 

Section \ref{conelikesection} contains the main technical results about conelike radiant structures. The main conclusion of Theorem \ref{conenormalizationtheorem} is that given a conelike radiant structure $(\nabla, \rad)$ with nonsingular radiant vector field $\rad$ and for which the Ricci curvature $\ric$ satisfies that the one-form $\er = \ric(\rad, \dum)$ vanishes, there is a unique $\rad$-invariant connection that with $\rad$ constitutes a conelike radiant connection having the same planelike surfaces as $(\nabla, \rad)$, inducing on $|\Det \ctm|$ the same connection as that induced by $\nabla$, and having antisymmetric Ricci tensor. (Whether the vanishing of $\er$ is merely a technical limitation of the proof or is really necessary is not resolved.)

From the point of view of differential geometry, the flat projective structure on $\proj(\ste)$ is determined by its projective geodesics, which are the images in $\proj(\ste)$ of planes through the origin in $\ste$. The classical construction of T.~Y. Thomas \cite{Thomas-affine, Thomas-book, Veblen-Thomas-paths} extends this picture to an $n$-manifold $M$ equipped with a projective structure $\en$. It associates with $(M, \en)$ a principal $\reat$-bundle $\rho:\F \to M$ equipped with a Ricci-flat affine connection $\hnabla$ and a $\hnabla$-parallel volume form $\Psi$ satisfying certain certain conditions expressing compatibility between $\hnabla$ and the principal bundle structure and the projective structure on the base. 
The Thomas construction is essentially local in nature and is usually built following a bottom up perspective, associating with a representative $\nabla \in \en$ a lifted connection $\hnabla$ satisfying certain conditions, the most of important of which is Ricci-flatness, that a posteriori is seen to be independent of the choice of $\nabla$. The affine connection $\hnabla$ forms with (a multiple of) the vector field $\rad$ generating the dilations in the fibers of $\F$ an example of a conelike radiant structure. 

Specializing the results of Section \ref{conelikesection} to the setting of the total space of a principal $G$-bundle over $M$ with $\dim G = 1$ yields an extension of the Thomas construction that is a generalization in several respects. First, the principal bundle fibering over $M$ need not be topologically trivial. An additional bit of data is the choice of a principal connection on this bundle, which no longer need be linked a priori to the projective structure $\en$. Rather the basic data is an \emph{extended projective structure} $[\nabla, \be]$ comprising an orbit of pairs $(\nabla, \be)$ where $\nabla$ is a connection and $\be$ is a principal connection under the action $\ga\cdot(\nabla, \be) =( \nabla + 2\ga_{(i}\delta_{j)}\,^{k},\be - \rho^{\ast}(\ga))$. This is explained and motivated in Section \ref{extendedprojectivesection}. It would be interesting to explore further couplings of a projective structure to a principal connection with higher-dimensional or nonabelian structure group, or to some further data on an associated vector bundle. In the context of flat bundles on affine manifolds, something along these lines has been studied in \cite{Biswas-Loftin-Stemmler}.

The Thomas connection is closely related to the tractor connection associated with the projective structure in the parabolic geometry formalism \cite{Cap-Slovak-book, Cap-Gover} and from either of these essentially equivalent structures there can be constructed the regular normal Cartan connection associated with the the underlying projective structure \cite{Kobayashi-Nagano}. 
Projective structures are a particular kind of parabolic geometry. General theorems associate with a parabolic geometry satisfying certain conditions a regular, normal Cartan connection. The adjectives regular and normal refer to certain normalizations on the curvature of the associated Cartan connection that in the context of projective connections amount in essence to the requirement that the classical Thomas connection be Ricci-flat. The general theory interprets these normalizations in terms of the harmonic Hodge theory due to Kostant, and these normalizations are satisfying from the point of view of representation theory. However, their geometric significance is in general somewhat opaque. In section \ref{extendedprojectivesection} it is shown that the geometric significance of the Ricci-flat normalization for projective structures is that it links parametrizations of projective geodesics with parametrizations of their lifts. Precisely, without this condition an unnormalized Thomas-like connection determines a projective structure on each projective geodesic, and the Ricci-flat normalization forces this induced projective structure to be the standard one.

Section \ref{thomassection} shows how to recover the classical Thomas connection from the formalism described here. It extends the classical construction to associate with a projective structure on $M$ a Thomas connection on the total space of any pseudo-tautological line bundle over $M$. This clarifies the global and functorial properties of the classical construction.

To fully develop the picture corresponding to Example \ref{projflatexample}, there remains to consider the relation of metric geometry with a conelike radiant structure. Section \ref{codazzisection} treats metrics compatible with a radiant structure and what it means for such structures to be conelike. Given a radiant structure $(\nabla, \rad)$, a metric $h$ can be compatible with each of $\nabla$ and $\rad$. Following \cite{Rodnianski-Shlapentokh-Rothman}, the metric is \emph{self-similar} if $\lie_{\rad}h = 2h$. A \emph{radiant Hessian structure} is a radiant structure $(\nabla, \rad)$ and a self-similar metric $h$ such that $\nabla h(\rad, \dum) = h$. In this case $h$ is the Hessian of $v = \tfrac{1}{2}h(\rad, \rad)$ and Lemma \ref{radcodazzilemma} shows that around any point where $v$ is not zero there is a neighborhood on which $h$ is isometric to a warped product. A triple $(\nabla, \rad, h)$ such that $(\nabla, \rad)$ is radiant and $\nabla_{[i}h_{j]k} = 0$ is a \emph{radiant statistical structure}. The metric of a radiant statistical structure is necessarily self-similar, so such a structure is a radiant Hessian structure. The usual involutive notion of conjugacy (also called duality) of statistical structures extends to radiant statistical structures. 
The conjugate radiant statistical structure $(\bnabla, \rad, h)$ has the same underlying self-similar metric structure $(h, \rad)$. A radiant statistical structure with nonsingular radiant vector field and its conjugate are both conelike if and only if the connection of each is $\rad$-invariant. 
A radiant statistical structure is \emph{special} if $\nabla |\det h| = 0$. Theorem \ref{conjugatethomastheorem} shows that there is a metric $g$ that constitutes with the Thomas connection of a projective structure a special radiant statistical structure if and only if the $g$-conjugate connection of the Thomas connection is also the Thomas connection of a projective structure. For properly convex flat projective structures this recovers the duality of affine spheres, and so it shows how to extend this duality to the nonconvex and nonflat settings.

It is useful to generalize the notion of radiant statistical structure by replacing the statistical condition $\nabla_{[i}g_{j]k} = 0$ with the condition $\nabla_{[i}g_{j]k} =\chi_{[i}h_{j]k}$ for some one-form $\chi_{i}$. Although this condition behaves well with respect to conformal rescaling of the metric $g$, it does not link $\nabla$ with the density $\det g$ in any way. A pair $(\nabla, [g])$ comprising a conformal structure $[g]$ and a torsion-free affine connection $\nabla$ is an \emph{AH structure} if for each $g \in [g]$ there is a one-form $\chi_{i}$ such that $\nabla_{[i}g_{j]k} =\chi_{[i}h_{j]k}$ and there holds the \emph{alignment} condition that $g^{pq}\nabla_{i}g_{pq} = ng^{pq}\nabla_{p}g_{iq}$ (this does not depend on the choice of $g \in [g]$). The basic properties of AH structures generalize those of statistical structures, with suitable modifications to accommodate $\chi$. Although these are all given in the author's \cite{Fox-ahs, Fox-2dahs, Fox-crm, Fox-schwarz}, they are developed from scratch in Section \ref{einsteinstatisticalsection} in a manner more streamlined and hopefully more readable than in \cite{Fox-ahs, Fox-2dahs, Fox-crm}.  AH structures are \emph{locally statistical structures} in a sense made precise by Lemma \ref{locallycodazzilemma}. 

The conclusion of Lemma \ref{liftedmetriclemma} illustrates the need for the notion of AH structure. It shows that associated with a special statistical structure and a connection $\be$ on a principal bundle with one-dimensional structure group compatible in an appropriate sense with the metric of the statistical structure there is on the total space of the principal bundle a pair of connections $\hD$ and a metric $G_{IJ}$ satisfying $\hD_{[I}G_{J]K} = \chi_{[I}G_{J]k}$, where $\chi$ is a constant multiple of $\be$, and such that $\hD$ has Ricci tensor equal to a multiple of the metric. More detailed background and motivation for Lemma \ref{liftedmetriclemma} and the other results of Section \ref{ewsection} is given at the beginning of that section.

More fundamentally, the usual Einstein-Weyl equations can be extended to AH structures as well and Theorem \ref{liftedahtheorem} shows that in the setting of Lemma \ref{liftedmetriclemma}, with suitable additional geometrica hypotheses, there result Einstein AH structures on $N$. Section \ref{einsteinstatisticalsection} gives a self-contained introduction to the Einstein equations for special statistical structures and AH structures. The precise notion is Definition \ref{einsteinahdefinition} which is a reformulation (with slightly different terminology) of the Einstein equations for AH structures defined and studied previously in \cite{Fox-ahs, Fox-2dahs, Fox-crm, Fox-ricweyl, Fox-cubicpoly, Fox-schwarz}. A special case of the definition yields a notion of Einstein equations for special statistical structures, given in Definition \ref{einsteinspecialstatisticaldefinition}, that is of independent interest. The motivating observation is that the statistical structure induced on a cooriented nondegenerate hypersurface in flat affine space is Einstein if and only if the hypersurface is an affine sphere. For a flat ambient connection the generalization of Example \ref{projflatexample} for a convex cone is given by the Blaschke metric on the associated affine sphere. The special statistical structure induced on an affine hypersurface is Einstein in the sense defined here if and only if the hypersurface is an affine sphere, so the Einstein equations for special statistical structures, or more generally AH structures, can be regarded as a generalization of the equations for affine spheres and correspondingly as the appropriate equations for generalizing Example \ref{projflatexample}.

As explained in Example \ref{ewexample}, a special case of Theorem \ref{liftedahtheorem} recovers a construction associating with a Kähler-Einstein metric of positive scalar curvature an Einstein-Weyl structure due to \cite{Pedersen-Swann-submersions}. It also yields the following, detailed in Example \ref{convexprojectiveexample} and summarized in Theorem \ref{convexprojectivetheorem}. Given a properly convex flat real projective structure $\en$ on an oriented surface $M$ of genus $g \geq 2$, there is a distinguished representative $\nabla \in \en$ with symmetric Ricci tensor $R_{ij}$ such that $g_{ij} = -R_{ij}$ is positive definite and satisfies $\nabla_{[i}g_{j]k} = 0$ and $\nabla_{i}|\det g| = 0$. Namely, $g_{ij}$ and $\nabla$ are induced by the Blaschke metric and induced affine connection on the hyperbolic sphere asymptotic to the cone over the universal cover of $M$. See \cite{Fox-2dahs} for details. For the complex structure $J$ on $M$ determined by $g_{ij}$ and the given orientation, the two-form $\om = \tfrac{2\pi e}{\vol_{g}(M)^{2}}g(J(\dum), \dum)$ represents the first Chern class of a principal circle bundle $\rho:N \to M$ equipped with a principal connection $\be$ having curvature $d\be = \rho^{\ast}(\om)$ provided $|e| \leq |\chi(M)|$. 
Example \ref{convexprojectiveexample} shows that in this case there are on $N$ a Lorentzian metric $G_{IJ}$ and a one-parameter family of connections $\hDs$ satisfying $\hDs_{[I}G_{J]K} = -2s\be_{[I}G_{J]K}$, $\hDs_{I}G_{JK} = - 2s\be_{I}G_{JK} + \rho^{\ast}(\nabla g)_{IJK}$, $G^{PQ}\hDs_{I}G_{PQ} - 3G^{PQ}\hDs_{P}G_{QI} = 0$, and $\ric(\hDs)_{IJ} =  - \tfrac{3}{2}s\rho^{\ast}(\om)_{IJ} - \tfrac{2\pi^{2}e^{2}}{\vol_{g}(M)^{2}}G_{IJ}$. Moreover, $(\hDs, [G])$ is an Einstein AH structure in the sense of Definition \ref{einsteinahdefinition}. In the special case that $\nabla$ is the Levi-Civita connection of the hyperbolic metric $g$, then $(\hDs, [G])$ is a Lorentzian Einstein-Weyl structure. 

Section \ref{leftinvariantsection} constructs left-invariant conelike radiant structures on Lie groups. By Theorem \ref{killingnonnulltheorem}, with the adjoint orbit of a Killing anisotropic element of the Lie algebra of a Lie group there is a conelike radiant structure determined uniquely up to automorphism. This has interest independent of the main thrust of the article, as it associates a canonical geometric structure with every Killing anisotropic adjoint orbit. It should be contrasted with the theorem of Helmstetter \cite{Helmstetter} showing that a semisimple Lie group carries no flat affine structure (so in particular no flat radiant structure).

In section \ref{3dsection} the three-dimensional unimodular Lie groups are given a uniform description in terms of real Clifford algebras and the conelike radiant structures on them obtained by applying the results of the preceding section are described in detail. The resulting radiant affine connections have purely antisymmetric Ricci tensor, so are in some sense as close to flat as is possible.
As an example, it is explained how to construct from them the well-known Einstein-Weyl structures on the Berger spheres and their Lorentzian analogues on $SL(2, \rea)$. This gives a novel approach to constructing these Einstein-Weyl structures that grounds them in the classical affine/projective dichotomy and situates them in the landscape of statistical manifolds.

\subsection*{Notational conventions}
Smooth means infinitely differentiable and all manifolds, bundles, sections, etc. are smooth unless otherwise indicated. Throughout $M$ is a connected, smooth $n$-manifold. The notation $Tf(u)(X)$ indicates the image of a vector $X \in T_{u}M$ tangent to $M$ at $u \in M$ under the differential of a smooth map $f$ with domain containing $u$. 

Given a free $\reat$ action $R:E \times \reat \to E$ by diffeomorphisms on a smooth manifold $E$, a tensor $S$ on $E$ is \emph{homogeneous} (resp. \emph{positively homogeneous}) of degree $\la \in \rea$ if $R_{r}^{\ast}(S) = r^{\la}S$ for all $r \in \reat$ (resp. for all $r \in \reap$). 

Where convenient, tensors are indicated using the abstract index notation, so that, for instance, $a_{ij}$ indicates a covariant two tensor. The indices are lables indicating tensor valencies, types, and symmetries and do not indicate any choice of frame, although were a frame fixed they could be interpreted as indicating components with respect to it. Enclosure of indices in square brackets or parentheses indicates complete antisymmetrization or symmetrization over the enclosed indices. For example, $a^{ij}= a^{(ij)} + a^{[ij]}$ indicates the decomposition of a contravariant two-tensor into its symmetric and antisymmetric parts. 
The exterior product is defined consistently with the convention that $X \wedge Y = X \tensor Y - Y \tensor X$ for vector fields $X$ and $Y$, so $(X \wedge Y)^{ij} = 2X^{[i}Y^{j]}$.
Inclusion of an index between vertical bars $|\,|$ indicates its omission from an indicated symmetrization. For example $2a_{[i|jk|l]} = a_{ijkl} - a_{ljki}$. The summation convention is always in effect in the following form: indices are in either \emph{up} position or \emph{down} and the trace pairing of the corresponding tensor factors is indicated by repeating a label appearing as both an up index and a down index. For example $a_{pij}\,^{p}$ indicates the pairing of the contravariant part with the first covariant factor. Since polynomials on the vector space $\ste$ are tautologically identified with symmetric tensors on the dual vector space $\ste^{\ast}$, the index $i$ in $\tfrac{\pr}{\pr y_{i}}$ has to be regarded as an \emph{up} index. 
Although in common use by relativists and conformal geometers, these conventions sometimes bother those who dislike coordinate expressions, so it bears repeating that abstract indices \emph{do not} refer to coordinates and are simply a notationally compact way of indicating invariantly defined objects. In the same spirit, those who use abstract indices sometimes insist that every tensor be labeled with indices. Here indices are omitted when they communicate little. So a metric is sometimes written $g_{ij}$, sometimes simply $g$ (this is no different than writing sometimes $g(\dum, \dum)$ and sometimes simply $g$) and $g(X, Y)$ and $g_{ij}X^{i}Y^{j}$ are notational \emph{synonyms}.

A line-bundle valued tensor is said to be \emph{weighted}. A nondegenerate, possibly weighted, covariant two-tensor $h_{ij}$ determines a contravariant two-tensor $h^{ij}$ of complementary weight (valued in the dual line bundle) defined uniquely by $h^{ip}h_{jp} = \delta_{j}\,^{i}$, in which here, as always, $\delta_{i}\,^{j}$ is the tautological $\binom{1}{1}$-tensor determined by the pairing of vectors with covectors. 
The horizontal position of an index is maintained when it is raised or lowered using $h^{ij}$ or $h_{ij}$.

\section{Background on projective structures}
The basic facts about projective structures go back to H. Weyl, E. Cartan, and T.~Y. Thomas and are well assimilated into an extensive literature although it seems no one source collects all the relevant information. What follows summarizes what is needed here. More detail from a similar perspective is available in \cite{Gover-Neusser-Willse}.
Terminology is abused in the traditional way (e.g. \cite[section III.$3$]{Kobayashi-Nomizu}) and \emph{affine connection} is used to mean a linear connection, or the covariant derivative it induces on the tangent bundle or any other tensor bundle. 

A \emph{path geometry} on a smooth manifold $M$ is a one-dimensional foliation $W \subset T\proj(TM)$ of the projectivized tangent bundle $\rho:\proj(TM) \to M$ that is transverse to the vertical subbundle $\ker T \rho$ and tangent to the Cartan subbundle $E \subset T\proj(TM)$ defined by $E_{\ell} = \{X \in T_{\ell}\proj(TM): T\rho(\ell)X \in \ell\}$. Its \emph{paths} are the images in $M$ of the maximal integral submanifolds of $W$. A path geometry is a \emph{projective structure} if there is on $M$ a torsion-free affine connection $\nabla$ such that the images of the geodesics of $\nabla$ are contained in its paths.
Two affine connections are \emph{projectively equivalent} if the image of every geodesic of one is contained in the image of a geodesic of the other. Two torsion-free affine connections $\bnabla$ and $\nabla$ are projectively equivalent if and only if there is a one-form $\si_{i}$ such that $\bnabla - \nabla = 2\si_{(i}\delta_{j)}\,^{k}$. A projective structure is identified with the equivalence class $\en$ of projectively equivalent connections determining its paths. Here, a representative $\nabla \in \en$ is always assumed to be torsion-free. 

Let $\ste$ be a vector space with projectivization $\proj(\ste)$ and let $\gr(2, \ste)$ be the Grassmannian of two-dimensional subspaces. The flat model path geometry is given by the foliation of the flag manifold $\{([\ell, K] \in \proj(\ste)\times \gr(2, \ste): \ell \subset K\}$ by the fibers of the projection onto the second factor. The paths in $\proj(\ste)$ are the images in $\proj(\ste)$ of the two-dimensional subspaces of $\ste$. This path geometry is a projective structure, and a projective structure is \emph{flat} if it is locally equivalent to this model. In this case it can be described in terms of an atlas of charts modeled on open domains in $\proj(\ste)$ with transition functions in $PGL(\ste)$, and a projective structure is flat if and only if it can be covered by coordinate charts in which the images of its geodesics are contained in straight lines. See \cite{Goldman-notes} for background on flat projective structures.

A projective structure determines and is determined by a regular normal Cartan connection that can be constructed via the Thomas connection \cite{Thomas-book, Bailey-Eastwood-Gover}, via the tractor formalism \cite{Cap-Gover}, working with jet bundles \cite{Kobayashi-Nagano}, or via the method of equivalence \cite{Grossman} and the claims about local flatness that follow can all be proved using this correspondence; see \cite{Cap-Slovak-book} for a modern approach in a more general setting.

The curvature $R_{ijk}\,^{l}$ and torsion $\tau_{ij}\,^{k}$ of an affine connection $\nabla$ are defined by $2\nabla_{[i}\nabla_{j]}X^{k} = R_{ijp}\,^{k}X^{p} - \tau_{ij}\,^{p}\nabla_{p}X^{k}$ for a vector field $X \in \Ga(TM)$. The \emph{Ricci curvature} of $\nabla$ is $R_{ij} = R_{pij}\,^{p}$. The curvature of the covariant derivative induced on $\Det \ctm$ by a torsion-free affine connection $\nabla$ is $-R_{ijp}\,^{p} = 2R_{[ij]}$, and the \emph{projective Weyl} and \emph{projective Cotton} tensors $B_{ijk}\,^{l}$ and $C_{ijk}$ of $\nabla$ are defined by 
\begin{align}\label{bijkl}
&B_{ijk}\,^{l} = R_{ijk}\,^{l} + 2\delta_{[i}\,^{l}P_{j]k}  - 2\delta_{k}\,^{l}P_{[ij]},& &C_{ijk} = 2\nabla_{[i}P_{j]k},
\end{align}
in which the \emph{projective Schouten tensor} is defined by
\begin{align}\label{projectiveschouten}
P_{ij}= \tfrac{1}{1-n}R_{(ij)} - \tfrac{1}{n+1}R_{[ij]} = \tfrac{1}{1-n}\left(R_{ij} - \tfrac{2}{n+1}R_{[ij]} \right).
\end{align}
Tracing the differential Bianchi identity yields $\nabla_{p}R_{ijk}\,^{p} = 2\nabla_{[i}R_{j]k}$ and so the algebraic Bianchi identity yields $0 = \nabla_{p}R_{[ijk]}\,^{p} = 2\nabla_{[i}R_{jk]} = 2(n+1)\nabla_{[i}P_{jk]} = (n+1)C_{[ijk]}$, showing that $C_{[ijk]} = 0$. In a similar fashion the trace-free parts of the Bianchi identities for $\nabla$ yield:
\begin{align}\label{projectivebianchi}
&B_{[ijk]}\,^{l} = 0,& & C_{[ijk]} = 0,&
&\nabla_{[i}B_{jk]l}\,^{p} = -\delta_{[i}\,^{p}C_{jk]l},& &\nabla_{p}B_{ijk}\,^{p} = (2-n)C_{ijk}.
\end{align}
The Ricci identity and the algebraic Bianchi identity yield $\nabla_{[i}C_{jk]l} = 2\nabla_{[i}\nabla_{j}P_{k]l} = - R_{[ij|l|}\,^{p}P_{k]p} =  - B_{[ij|l|}\,^{p}P_{k]p}$, or $\nabla_{[i}C_{jk]l} + P_{[i|p|}B_{jk]l}\,^{p} = 0$. 
The corresponding tensors associated with $\tnabla = \nabla + 2\si_{(i}\delta_{j)}\,^{k}$ are indicated $\tilde{R}_{ijk}\,^{l}$, $\tilde{P}_{ij}$, etc., and satisfy
\begin{align}\label{projvary}
\begin{split}
\tilde{R}_{ijk}\,^{l} & = R_{ijk}\,^{l} + (\nabla_{i}\si_{k} - \si_{i}\si_{k})\delta_{j}\,^{l} - (\nabla_{j}\si_{k} - \si_{j}\si_{k})\delta_{i}\,^{l} + d\si_{ij}\delta_{k}\,^{l},\\
\tilde{R}_{ij} & = R_{ij} + (1-n)(\nabla_{i}\si_{j} - \si_{i}\si_{j}) - d\si_{ij},\\
\tilde{P}_{ij} & = P_{ij} + \nabla_{i}\si_{j} - \si_{i}\si_{j},\qquad
\tilde{C}_{ijk}  = C_{ijk} - B_{ijk}\,^{p}\si_{p}.
\end{split}
\end{align}
By \eqref{projvary}, the projective Weyl tensor satisfies $\tilde{B}_{ijk}\,^{l} = B_{ijk}\,^{l}$, so does not depend on the choice of $\nabla \in \en$. When $n = 2$ the projective Weyl tensor is identically zero and, by \eqref{projvary}, the projective Cotton tensor does not depend on the choice of $\nabla \in \en$. 
A projective structure is flat if and only if its projective Cotton and Weyl tensors vanish. By \eqref{projectivebianchi}, when $n > 2$ this is the case if and only if the projective Weyl tensor vanishes. 

The remainder of this section details the well known fact that a projective geodesic of a projective structure acquires in a canonical manner a projective structure \cite{HarEl, Kobayashi-projectivelyinvariantdistances, Wu-projectivehyperbolicity}. Although this material is well known, it is hard to find it written in the form needed here, so precise statements are given (although proofs are only indicated). The key point, that is needed in the proof of Lemma \ref{projectiveparametrizationlemma}, is the definition of the projective parametrization of a projective geodesic. 
This uses some elementary facts relating initial value problems for second order linear ordinary differential equations, Ricatti differential equations, and the Schwarzian derivative. 

Let $M$ be equipped with a projective structure $\en$. If $\tnabla, \nabla \in \en$ and $f \neq 0$, then $fX \wedge \tnabla_{fX}(fX) = f^{3}(X \wedge \nabla_{X}X)$ for any local vector field $X$, so the condition $X \wedge \nabla_{X}X = 0$ is preserved if $\nabla$ is varied within $\en$ or $X$ is rescaled. A $C^{1}$ immersed one-dimensional submanifold $C \subset M$ is a \emph{projective geodesic} of $\en$ if $X \wedge \nabla_{X}X = 0$ for any local vector field tangent to $C$ along $C$ and any $\nabla \in \en$. Equivalently, $\dot{\ga}\wedge \nabla_{d/dt}\dot{\ga} = 0$ for any local $C^{2}$ parametrization $\ga:(-\ep, \ep) \to C$ and any $\nabla \in \en$. Locally any projective geodesic can be parametrized as a geodesic of a given $\nabla \in \en$, but in general such a parametrization is not possible globally. The image of a geodesic of any $\nabla \in \en$ is called a \emph{path} or \emph{projective geodesic} of $\en$. 

The Schwarzian derivative $S(f)$ of a $C^{3}$ diffeomorphism $f:I \to J$ between subintervals $I$ and $J$ of $\rea$ is defined by $S(f) = \dddot{f}/\dot{f} - \tfrac{3}{2}\left(\ddot{f}/\dot{f}\right)^{2}$, where dots denote derivatives.
If $\phi$ is a $C^{3}$ diffeomorphism, then, where the composition is defined, there holds the cocycle identity
\begin{align}\label{cocycle}
S(f \circ \phi) = \dot{\phi}^{2}S(f)\circ \phi + S(\phi),
\end{align}
which in part reflects that $S(f)(t)$ is the coordinate expression of the quadratic differential $S(f)(t)dt^{2}$.
	

Lemma \ref{schwarzianivplemma} follows straightforwardly from the standard existence and regularity theorem for a first order linear system of ordinary differential equations \cite[Theorem 5.1 of Chapter $1$]{Coddington-Levinson}, so its proof is omitted.

\begin{lemma}\label{schwarzianivplemma}
Let $I \subset \rea$ be an open interval containing $0$ and let $r \in C^{k}(I)$ for some $k \geq 0$. For $a, c \in \rea$ and $b \in \reat$, there are an open interval $J \subset I$ containing $0$ and a unique $f \in C^{k+3}(J)$ solving the initial value problem
\begin{align}\label{schwarzian}
&S(f) = 2r,&  &f(0) = a,&&\dot{f}(0) = b \neq 0, &&\ddot{f}(0) = c,
\end{align}
and such that $\dot{f}$ does not vanish on $J$. Moreover:
\begin{enumerate}
\item\label{ffromx} The solution $f$ has the form $f = x_{1}/x_{2}$ where, for any $\la \in \reat$, $x_{1}, x_{2} \in C^{k+2}(J)$ are the uniquely determined solutions of the initial value problem
\begin{align}\label{ddotxr}
&\ddot{x} + rx = 0, & &x(0) = x_{0},& &\dot{x}(0) = y_{0},
\end{align}
satisfying $x_{1}(0)  = 2ab\la$, $\dot{x}_{1}(0) = (2b^{2} - ac)\la$, $x_{2}(0) = 2b\la$, and $\dot{x}_{2}(0) = -c\la$.
\item\label{xfromf} Given $f \in C^{k+3}(J)$ solving \eqref{schwarzian} and with $\dot{f}$ not vanishing on $J$, the functions $x_{1} = f|\dot{f}|^{-1/2}$ and $x_{2} = |\dot{f}|^{-1/2}$ solve \eqref{ddotxr} on $J$ with initial conditions  $x_{1}(0)  = 2ab\la$, $\dot{x}_{1}(0) = (2b^{2} - ac)\la$ and $x_{2}(0) = 2b\la$, $\dot{x}_{2}(0) = -c\la$, where $\la = |b|^{-1/2}b^{-1}$.
\item\label{uvfromxf} Given $f$ solving \eqref{schwarzian} and $x_{1} = f|\dot{f}|^{-1/2}$ and $x_{2} = |\dot{f}|^{-1/2}$, the functions $u = \dot{x}_{2}/x_{2} = -(1/2)(\ddot{f}/\dot{f})$ and $v = \dot{x}_{1}/x_{1} = u + \dot{f}/f$ are the unique solutions in $C^{k+1}(J)$ of the initial value problem 
\begin{align}\label{ricatti}
&\dot{w} + w^{2} + r = 0,& & w(0) = w_{0}
\end{align}
with initial conditions $u(0) = -c/(2b)$ and $v(0) = (2b^{2} - 2ac)/2ab =b/a -c/(2b)$.
\end{enumerate}
\end{lemma}

\begin{example}\label{schwarzianflatexample}
For constants $a, c \in \rea$ and $b \in \reat$, the unique solution of the initial value problem
\begin{align}\label{schwarzianflat}
&S(f) =0,&  &f(0) = a,&&\dot{f}(0) = b \neq 0, &&\ddot{f}(0) = c.
\end{align}
is the linear fractional transformation $f(t) = \tfrac{(ac - 2b^{2})t - 2ab}{ct - 2b}$.
\end{example}

\begin{corollary}\label{schwarzlftcorollary}
Let $I \subset \rea$ be an open interval. Functions $f, g \in C^{3}(I)$ satisfy $S(f) = S(g)$ on $I$ if and only if there is a linear fractional transformation $\phi$ mapping $g(I)$ to $f(I)$ such that $f = \phi \circ g$.
\end{corollary}
\begin{proof}
This follows from the cocycle identity \eqref{cocycle}, Lemma \ref{schwarzianivplemma}, and Example \ref{schwarzianflatexample}
\end{proof}

\begin{lemma}\label{1dprojstructurelemma}
Let $\en$ be a projective structure on a smooth manifold $M$. Let $C \subset M$ be a projective geodesic. For any $p \in C$ there are an open interval $I \subset \rea$ containing $0$ and a $C^{2}$ immersion $\ga:I \to M$ such that $\ga(0) = p$, $\ga(I) \subset C$, and such that the function $q \in C^{1}(I)$ defined by $\nabla_{d/dt}\dot{\ga} = q\dot{\ga}$ satisfies
\begin{align}
\dot{q} - \tfrac{1}{2}q^{2} - 2P(\dot{\ga}, \dot{\ga}) = 0.
\end{align}
The immersion $\ga$ is determined uniquely up to precomposition with a linear fractional transformation. As a consequence $C$ acquires in a canonical way a projective structure.
\end{lemma}
\begin{proof}
Let $C \subset M$ be a projective geodesic. Let $I$ and $\bar{I}$ be open subintervals of $\rea$ and let $\phi:I \to \bar{I}$ be a $C^{3}$ diffeomorphism. Let $\ga:I \to M$ and $\bar{\ga}:\bar{I} \to M$ be $C^{2}$ immersions such that $\ga(I), \bar{\ga}(\bar{I}) \subset C$ and $\ga = \bar{\ga} \circ \phi$. Consider $I$ and $\bar{I}$ with the flat affine structures induced from that on $\rea$ and let $t$ and $\bar{t}$ be affine coordinates in $I$ and $\bar{I}$. Write derivatives with respect to $t$ using $\dot{\,}$ and derivatives with respect to $\bar{t}$ using $\prime$. For example, $\dot{\ga} = \dot{\phi}(\bar{\ga}^{\prime} \circ \phi)$. Let $\nabla$ and $\tnabla = \nabla + 2\si_{(i}\delta_{j)}\,^{k}$ be torsion-free representatives of $\en$. By \eqref{projvary}, their projective Schouten tensors satisfy $\tilde{P}_{(ij)} = P_{(ij)} + \nabla_{(i}\si_{j)} - \si_{i}\si_{j}$. Since $C$ is a projective geodesic there are functions $q, \tilde{q} \in C^{1}(I)$ and $\bar{Q} \in C^{1}(\bar{I})$ such that $\nabla_{d/dt}\dot{\ga} = q\dot{\ga}$, $\nabla_{\frac{d}{d\bar{t}}}\bar{\ga}^{\prime} = \bar{q}\bar{\ga}^{\prime}$, and $\tnabla_{d/dt}\dot{\ga} = \tilde{q}\dot{\ga}$. From
\begin{align}
\begin{split}
 \dot{\phi}q (\bar{\ga}^{\prime} \circ \phi)  & = q\dot{\ga} = \nabla_{d/dt}\dot{\ga} = \dot{\phi}^{2}\left(\nabla_{d/d\bar{t}}\bar{\ga}^{\prime} \right)\circ \phi + \ddot{\phi}(\bar{\ga}^{\prime}) \circ \phi = \dot{\phi}^{2}\left( \bar{q} \circ \phi + \ddot{\phi}/\dot{\phi}\right)\bar{\ga}^{\prime} \circ \phi,
\end{split}
\end{align} 
there follows $q = \dot{\phi}\bar{q}\circ \phi + \ddot{\phi}/\dot{\phi}$, so that
\begin{align}
& \dot{q} = \dot{\phi}^{2}(\bar{q}^{\prime}\circ \phi) + \ddot{\phi}(\bar{q}\circ \phi) + \left(\ddot{\phi}/\dot{\phi}\right)^{2},&& q^{2} = (\bar{q}\circ \phi)^{2} + 2\ddot{\phi}(\bar{q}\circ \phi) + \left(\ddot{\phi}/\dot{\phi}\right)^{2},
\end{align}
from which there follows
\begin{align}\label{dotqqtransform}
\dot{q} - \tfrac{1}{2}q^{2} - 2P(\dot{\ga}, \dot{\ga}) = \dot{\phi}^{2}\left(\bar{q}^{\prime} - \tfrac{1}{2}\bar{q}^{2} - 2P(\bar{\ga}^{\prime}, \bar{\ga}^{\prime})\right)\circ \phi + S(\phi).
\end{align}
From $\tilde{q} \dot{\ga} = \tnabla_{\dot{\ga}}\dot{\ga} = \nabla_{d/dt}\dot{\ga} + 2\si(\dot{\ga})\dot{\ga} = (q + 2\si(\dot{\ga}))\dot{\ga}$, it follows that $\tilde{q} = q + 2\si(\dot{\ga})$, so that
\begin{align}
\begin{split}
\dot{\tilde{q}} - \tfrac{1}{2}\tilde{q}^{2}  - 2\tilde{P}(\dot{\ga}, \dot{\ga})&= \dot{q} - \tfrac{1}{2}q^{2} + 2\tfrac{d}{dt}\si(\dot{\ga})  - 2P(\dot{\ga}, \dot{\ga}) -  2(\nabla_{d/dt}\si)(\dot{\ga}) - 2\si(\dot{\ga}) = \dot{q} - \tfrac{1}{2}q^{2}  - 2P(\dot{\ga}, \dot{\ga}),
\end{split}
\end{align}
so that the function $\ka = \dot{q} - \tfrac{1}{2}q^{2} - 2P(\dot{\ga}, \dot{\ga}) \in C^{1}(I)$ does not depend on the choice of $\nabla \in \en$. Likewise the function $\bar{\ka} = \bar{q}^{\prime} - \tfrac{1}{2}\bar{q}^{2} - 2P(\bar{\ga}^{\prime}, \bar{\ga}^{\prime}) \in C^{1}(\bar{I})$ does not depend on the choice of $\nabla \in \en$. By \eqref{dotqqtransform},
\begin{align}
\ka = \dot{\phi}^{2}\bar{\ka} \circ \phi + S(\phi).
\end{align}
By \eqref{cocycle}, $\bar{f} \in C^{3}(I)$ solves $S(\bar{f}) = \bar{\ka}$ if and only if $f = \bar{f}\circ \phi$ solves $S(f) = \ka$. 

Given $\ga$, and so $q$ and $\ka$, replacing $I$ by a subinterval if necessary, by Lemma \ref{schwarzianivplemma} there can be found a $C^{4}$ diffeomorphism $\phi$ such that $S(\phi) = \ka$ on $I$. In this case $\bar{\ka} = 0$. Regard $ U = \ga(I) = \bar{\ga}(\bar{I})$ as an open set in $C$ and $\bar{\ga}^{-1}:U \to \rea$ as a chart. In this chart, $\bar{\ka} = 0$. Such a chart can be constructed around any point of $C$, and on the overlap the transition functions must have vanishing Schwarzian derivative, so are restrictions of elements of $PGL(2, \rea)$, hence these charts determine a projective structure on $C$.
\end{proof}

Lemma \ref{1dprojstructurelemma} motivates the following definition. A smooth parametrization $\ga:I \to M$ of a projective geodesic is \emph{projective parametrization} if the associated function $\ka = \dot{q} - \tfrac{1}{2}q^{2} - 2P(\dot{\ga}, \dot{\ga})$ vanishes on $I$.

A more geometric description of the projective structure induced on a projective geodesic is given by Lemmas \ref{maxgeodesiclemma} and \ref{projectiveparametrizationlemma}.

The infinitesimal variation of a projective geodesic, $C \subset M$, through projective geodesics is described by a section of the normal bundle $\nu(C) = TM_{|C}/TC$ of $C$. Precisely, let $\bar{\ga}(t, s):(a, b) \times (-\epsilon, \epsilon) \to M$ be a $C^{2}$ immersion such that for each $s$ the image of $\bar{\ga}_{s} =\bar{\ga}(\dum, s)$ is contained in a projective geodesic, and so that $\bar{\ga}(t, 0) = \ga(t)$ is a parametrization of $\bar{\ga}((a, b) \times \{0\}) \subset C$. Let $\nabla \in \en$, $\bar{v} = \frac{\partial\bar{\ga}}{\partial t}$, $\bar{Y} = \frac{\partial \bar{\ga}}{\partial s}$ and let $Y(t) = \bar{Y}(t, 0)$. The assumption $C$ is a projective geodesic means $\bar{v}\wedge\nabla_{\bar{v}}\bar{v} = 0$ for each $s \in (-\ep, \ep)$. The vector field $Y$ depends on the chosen parametrization of $C$, but changing this parametrization modifies $Y$ only by the addition of a vector field tangent to $C$.

From $\bar{v}\wedge\nabla_{\bar{v}}\bar{v} = 0$ it follows that $\nabla_{\bar{v}}\bar{v}\wedge \nabla_{\bar{v}}\bar{Y} = \bar{v}\wedge \nabla_{\nabla_{\bar{v}}\bar{v}}\bar{Y}$. Differentiating $\bar{v}\wedge\nabla_{\bar{v}}\bar{v} = 0$ and using the preceding observation and also $\nabla_{\bar{Y}}\bar{v} = \nabla_{\bar{v}}\bar{Y}$ to simplify the result gives
\begin{align}\label{projjac1}
\begin{split}
0 &= \nabla_{\bar{Y}}\left( \bar{v}\wedge \nabla_{\bar{v}}\bar{v}\right) = - \nabla_{\bar{v}}\bar{v}\wedge \nabla_{\bar{v}}\bar{Y} + \bar{v}\wedge \nabla_{\bar{Y}}\nabla_{\bar{v}}\bar{v}
= - \bar{v}\wedge \nabla_{\nabla_{\bar{v}}\bar{v}}\bar{Y} + \bar{v}\wedge \left(\nabla_{\bar{v}}\nabla_{\bar{Y}}\bar{v} + R(\bar{Y}, \bar{v})\bar{v}\right) \\
& = \bar{v}\wedge \left(\nabla_{\bar{v}}\nabla_{\bar{v}}\bar{Y} - \nabla_{\nabla_{\bar{v}}\bar{v}}\bar{Y}+ R(\bar{Y}, \bar{v})\bar{v} \right) = \bar{v}\wedge (\lie_{\bar{Y}}\nabla)(\bar{v}, \bar{v}),
\end{split}
\end{align}
where the Lie derivative $\lie_{X}\nabla$ of an affine connection $\nabla$ along a vector field $X$ is defined by differentiating the difference tensor of the pullback of $\nabla$ via the flow of $X$ with $\nabla$. Define the difference tensor $\ben - \en$ of the projective structures $\ben$ and $\en$ to be the trace-free part of the difference tensor $\ben - \en$ of any representatives $\bnabla \in \ben$ and $\nabla \in \en$, which does not depend on the choices of representatives. The Lie derivative $\lie_{X}\en$ is defined to be the derivative at $t = 0$ of the difference tensor $\phi_{t}^{\ast}\en - \en$ where $\phi_{t}$ is the local flow generated by $X \in \Ga(TM)$. For any $\nabla \in \en$, $(\lie_{X}\en)_{ij}\,^{k}$ equals the trace-free part $(\lie_{X}\nabla)_{ij}\,^{k} - \tfrac{2}{n+1}\delta_{(i}\,^{k}\nabla_{j)}\div_{\nabla}(X)$ of $(\lie_{X}\nabla)_{ij}\,^{k} = \nabla_{i}\nabla_{j}X^{k} + X^{p}R_{pij}\,^{k}$. 
If $\nabla$ is replaced in \eqref{projjac1} by $\tnabla \in\en$ having curvature $\tilde{R}$ then $\bar{Y}\wedge \bar{v}\wedge (\tilde{R}(\bar{Y}, \bar{v})\bar{v} - R(\bar{Y}, \bar{v})\bar{v}) = 0$. Since also $\bar{Y}\wedge \bar{v}\wedge (\tnabla_{\bar{v}}\bar{Y} - \nabla_{\bar{v}}\bar{Y}) = 0$, it follows from \eqref{projjac1} that a variation of $C$ through projective geodesics is represented by a section $[Y]$ of $\nu(C)$ such that for any parametrization of $C$ and any $\nabla \in \en$, any vector field $Y(t)$ representing $[Y]$ satisfies 
\begin{align}\label{projjacobi}
\begin{split}
Y \wedge \dot{\ga}\wedge (\lie_{Y}\en)(\dot{\ga}, \dot{\ga})& = Y \wedge \dot{\ga} \wedge (\nabla_{d/dt}\nabla_{d/dt}Y - \nabla_{\nabla_{d/dt}\dot{\ga}}Y + B(Y, \dot{\ga})\dot{\ga}) \\
&= Y \wedge \dot{\ga}\wedge (\nabla_{d/dt}\nabla_{d/dt}Y - \nabla_{\nabla_{d/dt}\dot{\ga}}Y + R(Y, \dot{\ga})\dot{\ga}) = 0.
\end{split}
\end{align}
A section $[Y]$ of the normal bundle of a projective geodesic solving \eqref{projjacobi} is a \emph{projective Jacobi field}. This notion is needed in the proof of Lemma \ref{coneconditionlemma}.

\section{Radiant structures}\label{radiantstructuresection}
This section defines the radiant structures that are the focus of this paper.
It is convenient to introduce them in a more general context for which it is helpful to describe some additional geometric notions.

A section of the normal bundle $\nu(C)$ of a projective geodesic $C$ can be viewed also as a rank two subbundle $K$ of the restriction $TM_{|C}$ to $C$ of $TM$ such that $K$ contains $TC$. A pair $(C, K)$ such that $C$ is an immersed one-dimensional submanifold and $K$ is a rank two subbundle of $TM_{|C}$ containing $TC$ is a \emph{flag} with \emph{base path} $C$. A flag $(C, K)$ is a \emph{geodesic flag} if the base path $C$ is a projective geodesic. A flag $(C, K)$ is a \emph{Jacobi flag} if it is a geodesic flag and $K/TC$ is a projective Jacobi field.

Given a projective structure $\en$ on $M$, a flag $(C, K)$ is \emph{parallel} if for any $\nabla \in \en$, any parametrization $\ga:(a, b) \to C$, and any vector field $X$ defined along $C$ such that $\dot{\ga}$ and $X$ span $K$, there holds
\begin{align}\label{flageq}
&X \wedge \dot{\ga} \wedge  \nabla_{d/dt}X = 0.
\end{align}
It is straightforward to check that the equations \eqref{flageq} do not depend on the choice of $\nabla$, the choice of parametrization of $C$, or the choice of $X$. A parallel flag $(C, K)$ such that the base path $C$ is a projective geodesic is a \emph{parallel geodesic flag}. This amounts to adding to \eqref{flageq} the equation $\dot{\ga} \wedge \nabla_{d/dt}\dot{\ga} = 0$.

\begin{lemma}\label{affdillemma}
Let $M$ be a smooth manifold of dimension $n \geq 2$. For a vector field $X \in \Ga(TM)$ and a torsion-free affine connection $\nabla$ the following two conditions are equivalent:
\begin{enumerate}
\item\label{affdil1} $Y \wedge \nabla_{Y}X = 0$ for all $Y \in \Ga(TM)$.
\item\label{affdil2} There is $f \in \cinf(M)$ such that $\nabla_{i}X^{j} = f\delta_{i}\,^{j}$.
\end{enumerate}
In this case $nf = \div_{\nabla}(X)$ where $\div_{\nabla}(X) = \nabla_{p}X^{p}$.
\end{lemma}
\begin{proof}
That \eqref{affdil2} implies \eqref{affdil1} is immediate. If $X$ and $\nabla$ satisfy \eqref{affdil1}, then $\delta_{(i}\,^{[k}\nabla_{j)}X^{l]} = 0$. Tracing this in $j$ and $l$ shows $n\nabla_{i}X^{k} = \nabla_{p}X^{p}\delta_{i}\,^{k}$, so $X$ satisfies \eqref{affdil2} with $f = n^{-1}\nabla_{p}X^{p}$.
\end{proof}

A vector field $X$ is \emph{dilatative} with respect to a torsion-free affine connection $\nabla$ if it satisfies either of the equivalent conditions of Lemma \ref{affdillemma}. 

\begin{remark}
A vector field $X$ dilatative with respect to the Levi-Civita connection $D$ of a pseudo-Riemannian metric $g_{ij}$ is \emph{concircular} \cite[Section $1.10$]{Chen-warped}. In the equivalent form $D_{i}X_{j} = fg_{ij}$, where $X_{i} = X^{p}g_{pi}$, this condition was studied in \cite{Fialkow}. A Riemannian manifold that admits a nowhere vanishing concircular vector field is locally a warped product \cite[Theorem $3.1$]{Chen-concircular}. Lemma \ref{radcodazzilemma} extends this statement to the context of radiant Hessian structures.
\end{remark}

A vector field $X \in \Ga(TM)$ is \emph{projectively dilatative} with respect to the torsion-free affine connection $\nabla$ if $X \wedge Y \wedge \nabla_{Y}X = 0$ for all $Y \in \Ga(TM)$. 
If $\dim M = 2$, then any vector field is projectively dilatative, so the notion is vacuous, but in higher dimensions it is not automatic. 
The terminology \emph{projectively dilatative} is justified by Lemma \ref{projdiltransformlemma}, that can be proved by straightforward computation.

\begin{lemma}\label{projdiltransformlemma}
Let $X \in \Ga(TM)$ be projectively dilatative with respect to the torsion-free affine connection $\nabla$. 
\begin{enumerate}
\item For $\ga_{i} \in \Ga(\ctm)$ and $Q_{ij} \in \Ga(S^{2}\ctm)$, $X$ is projectively dilatative with respect to the torsion-free affine connection $\tnabla = \nabla + 2\ga_{(i}\delta_{j)}\,^{k} + Q_{ij}X^{k}$. 
\item For $g \in \cinf(M)$, the vector field $\tilde{X} = gX$ is projectively dilatative with respect to $\nabla$.
\end{enumerate}
\end{lemma}

A \emph{line field} on $M$ is a one-dimensional subbundle of $TM$. A compact manifold admits a line field if and only if its Euler characteristic $\chi(M)$ is zero, but an open manifold always admits a smooth line field.

By Lemma \ref{projdiltransformlemma}, it makes sense to define $X \in \Ga(TM)$ to be \emph{projectively dilatative} with respect to the projective structure $\en$ if it is projectively dilatative with respect to any $\nabla \in \en$, and it make sense to say that a line field $L \subset TM$ is projectively dilatative with respect to a torsion-free connection $\nabla$ (or projective structure $\en$) if any section $X \in \Ga(L)$ is projectively dilatative with respect to $\nabla$ (or any $\nabla \in \en$). 

\begin{lemma}\label{projdillemma}
Suppose $n = \dim M > 2$ and let $X \in \Ga(TM)$ be such that $M_{\ast} = M_{\ast}(X) = \{p \in M: X_{p} \neq 0\}$ is nonempty. For a torsion-free affine connection $\nabla$ the following two conditions are equivalent:
\begin{enumerate}
\item\label{projdil1} $X$ is projectively dilatative with respect to $\nabla$.
\item\label{projdil2} There is $\si_{i} \in \Ga(T^{\ast}M_{\ast})$ such that $\nabla_{i}X^{j} - \tfrac{1}{n}\divn(X)\delta_{i}\,^{j} = \si_{i}X^{j} - \tfrac{1}{n}\si_{p}X^{p}\delta_{i}\,^{j}$ on $M_{\ast}$.
\end{enumerate}
If there hold \eqref{projdil1}-\eqref{projdil2} then
\begin{enumerate}
\setcounter{enumi}{2}
\item\label{projdil4}  $X \wedge \nabla_{X}X = 0$ on $M$.
\item\label{projdilcurvature} On $M_{\ast}$ there hold
\begin{align}\label{projdilweyl}
X^{p}B_{ijp}\,^{k} & = (d\si_{ij} + \tfrac{2}{n+1}R_{[ij]})X^{k} + \tfrac{1}{n-1}\left(\delta_{j}\,^{k}(d\si_{ip} + \tfrac{2}{n+1}R_{[ip]}) - \delta_{i}\,^{k}(d\si_{jp} + \tfrac{2}{n+1}R_{[jp]}) \right)X^{p},\\
\label{projdilcotton}
\tfrac{1-n}{2}X^{p}C_{ijp} &= f(R_{[ij]} + \tfrac{n+1}{2}d\si_{ij}) + \tfrac{1}{n+1}X^{p}\nabla_{p}(R_{[ij]} + \tfrac{2}{n+1}d\si_{ij}).
\end{align}
where $f = \tfrac{1}{n}(\div_{\nabla}(X) - \si(X))$.
\end{enumerate}
Moreover:
\begin{enumerate}
\setcounter{enumi}{4}
\item\label{pdcon1} If $\ga_{i} \in \Ga(\ctm)$, $Q_{ij} \in \Ga(S^{2}\ctm)$, and $\tnabla = \nabla + 2\ga_{(i}\delta_{j)}\,^{k} + Q_{ij}X^{k}$, then $\tilde{\si}_{i} = \si_{i} + \ga_{i} + Q_{ip}X^{p}$ in place of $\si_{i}$ in \eqref{projdil2} and $d\tilde{\si}_{ij} + \tfrac{2}{n+1}\tilde{R}_{[ij]}  = d\si_{ij} + \tfrac{2}{n+1}R_{[ij]} + \tfrac{n}{n+1}dq_{ij}$, where $q_{i} = Q_{ip}X^{p}$ and $\tilde{R}_{ij}$ is the Ricci curvature of $\tnabla$.
\item\label{pdcon2} If $c \in \cinf(M)$, then $\tilde{X} = cX$ is projectively dilatative with respect to $\nabla$, and on $M_{\ast}(cX) \subset M_{\ast}(X)$, $\tilde{\si}_{i} = \si_{i} + d\log{c}_{i}$ in place of $\si_{i}$ in \eqref{projdil2}. 
\end{enumerate}
\end{lemma}

\begin{proof}
Condition \eqref{projdil1} can be rewritten as the tensorial identity 
\begin{align}\label{projdilre}
X^{[a}\delta_{(i}\,^{b}\nabla_{j)}X^{c]} = 0.
\end{align}
It is immediate that \eqref{projdil2} implies \eqref{projdilre} on $M_{\ast}$. Because $\nabla_{i}X^{j}$ is smooth on $M$, that \eqref{projdilre} holds on $M_{\ast}$ implies it holds on $M$ as well. Thus \eqref{projdil2} implies \eqref{projdil1}. Suppose $X$ and $\nabla$ satisfy \eqref{projdil1}. Tracing \eqref{projdilre} in $j$ and $c$ gives
\begin{align}\label{pd1}
0 = (1-n)X^{[a}\nabla_{i}X^{b]} + X^{[a}\delta_{i}\,^{b]}\nabla_{p}X^{p} - X^{p}\nabla_{p}X^{[a}\delta_{i}\,^{b]}.
\end{align}
Contracting \eqref{pd1} with an arbitrary vector field $Y^{i}$ gives 
\begin{align}\label{pd2b}
0 = ((n-1)\nabla_{Y}X - \divn(X)Y)\wedge X + Y\wedge \nabla_{X}X.
\end{align}
Taking $Y = X$ in \eqref{pd2b} shows $(n-2)X \wedge \nabla_{X}X= 0$, so $X \wedge \nabla_{X}X = 0$, since $n > 2$. This shows \eqref{projdil1} implies \eqref{projdil4}. It also implies that there is $g \in \cinf(M_{\ast})$ such that $\nabla_{X}X = gX$ on $M_{\ast}$. Substituting this into \eqref{pd2b} shows that on $M_{\ast}$ there holds $0 = ((n-1)\nabla_{Y}X + (g -\divn(X))Y) \wedge X$ for all $Y \in \Ga(TM)$. Consequently, for each $Y \in \Ga(TM)$ there is $b(Y) \in \cinf(M_{\ast})$ such that $\nabla_{Y}X + \tfrac{1}{n-1}(g -\divn(X))Y) = b(Y)X$ on $M_{\ast}$. From this last expression it is evident that the assignment $Y \to b(Y)$ is a $\cinf(M_{\ast})$-module map, so that there is $\si_{i} \in \Ga(T^{\ast}M_{\ast})$ such that $b(Y) = \si_{p}Y^{p}$ on $M_{\ast}$. This shows $\nabla_{i}X^{j} - \tfrac{1}{n}\divn(X)\delta_{i}\,^{j} = \si_{i}X^{j} - \tfrac{1}{n}\si_{p}X^{p}\delta_{i}\,^{j}$ on $M_{\ast}$, so shows \eqref{projdil2}. 

Suppose $X$ is projectively dilatative with respect to $\nabla$ so that $\nabla_{i}X^{j} = \si_{i}X^{j} + f\delta_{i}\,^{j}$ where $f = \tfrac{1}{n}(\div_{\nabla}(X) - \si(X))$. Antisymmetrizing $\nabla_{i}\nabla_{j}X^{k} = df_{i}\delta_{j}\,^{k} + X^{k}\nabla_{i}\si_{j} + \si_{j}\nabla_{i}X^{k} = df_{i}\delta_{j}\,^{k} + f\si_{j}\delta_{i}\,^{k} + \left(\nabla_{i}\si_{j} + \si_{i}\si_{j}\right)X^{k}$ gives
\begin{align}\label{radskew}
X^{p}R_{ijp}\,^{k} & = 2\nabla_{[i}\nabla_{j]}X^{k} = 2df_{[i}\delta_{j]}\,^{k} - 2f\si_{[i}\delta_{j]}\,^{k} + d\si_{ij}X^{k}.
\end{align}
Tracing \eqref{radskew} gives
\begin{align}\label{ricrad}
R_{ip}X^{p} & = (1-n)(df_{i} - f\si_{i}) + d\si_{pi}X^{p}.
\end{align}
Combining \eqref{radskew} and \eqref{ricrad} gives \eqref{projdilweyl}.
Differentiating \eqref{ricrad} and using $\nabla_{i}X^{j} = f\delta_{i}\,^{j} + \si_{i}X^{j}$ to simplify the result gives
\begin{align}\label{fric}
\begin{split}
f(R_{ij} + d\si_{ij} &+ (1-n)\nabla_{j}\si_{j}) 
  = -X^{p}(\nabla_{j}R_{ip} + \nabla_{j}d\si_{ip}) + (n-1)\left(2\si_{(i}df_{j)} - f\si_{i}\si_{j} - \nabla_{j}df_{i}\right).
\end{split}
\end{align}
Antisymmetrizing \eqref{fric} and using the traced differential Bianchi identity to simplify the result yields
\begin{align}\label{fricskew}
\begin{split}
f(R_{[ij]} &+ \tfrac{n+1}{2}d\si_{ij}) = X^{p}(\nabla_{[i}R_{j]p} + \nabla_{[i}d\si_{j]p}) = \tfrac{1}{2}X^{p}(\nabla_{q}R_{ijp}\,^{q} - \nabla_{p}d\si_{ij}).
\end{split}
\end{align}
Since $C_{ijk} = \tfrac{2}{1-n}\nabla_{[i}R_{j]k} - \tfrac{2}{n^{2} - 1}\nabla_{k}R_{[ij]}$, the identity \eqref{fricskew} can be rewritten as \eqref{projdilcotton}.

The claims \eqref{pdcon1} and \eqref{pdcon2} about how $\si_{i}$ transforms when $\nabla$ and $X$ are replaced by $\tnabla$ and $\tilde{X}$ follow from straightforward computations, with the exception of the final claim of \eqref{pdcon1}, which is a bit more involved. Write $\Pi_{ij}\,^{k} = 2\ga_{(i}\delta_{j)}\,^{k} + Q_{ij}X^{k}$, so that $\tnabla = \nabla + \Pi_{ij}\,^{k}$. Then $\Pi_{ip}\,^{p} = (n+1)\ga_{i} + q_{i}$, where $q_{i} = Q_{ip}X^{p}$. From $2\tilde{R}_{[ij]} - 2R_{[ij]} = -2\nabla_{[i}\Pi_{j]p}\,^{p} = - (n+1)d\ga_{ij} - dq_{ij}$ and $d\tilde{\si}_{ij} = d\si_{ij} + d\ga_{ij} + dq_{ij}$ there results \eqref{pdcon1}. 
\end{proof}

\begin{corollary}\label{projdilnormalcorollary}
On a manifold, $M$, of dimension $n > 2$, if $X \in \Ga(TM)$ is projectively dilatative with respect to the projective structure $\en$, there is a unique $\nabla \in \en$ such that $X$ is dilatative with respect to $\nabla$ on $M_{\ast} = \{p \in M: X_{p} \neq 0\}$.
\end{corollary}
\begin{proof}
By Lemma \ref{projdillemma}, for any $\nabla \in \en$ there are $f \in \cinf(M_{\ast})$ and  $\si_{i} \in \Ga(T^{\ast} M_{\ast})$ such that $\nabla_{i}X^{j} = f\delta_{i}\,^{j} + \si_{i}X^{j}$ on $M_{\ast}$. The unique $\tnabla\in \en$ such that $X$ is dilatative with respect to $\tnabla$ on $M_{\ast}$ is $\tnabla = \nabla - 2\si_{(i}\delta_{j)}\,^{k}$.
\end{proof}

\begin{example}\label{projdilexample}
This example shows that the $\si_{i}$ in \eqref{projdil2} need not extend smoothly to all of $M$. Let $D$ be the flat connection determined by the affine structure on the $n$-dimensional vector space $\ste$ and let $h_{ij}$ be a Euclidean metric on $\ste$, meaning that it is $D$-parallel and positive definite. The Euler vector field $\eul^{i}$ generating the flow by dilations around the origin satisfies $D_{i}\eul^{j} = \delta_{i}\,^{j}$. Let $X^{i} = |\eul|^{2\al}_{h}\eul^{i}$.
For $\al > -1/2$, the vector field $X$ has a unique zero at the origin, so that $\ste_{\ast} = \ste \setminus\{0\}$. When $\al \leq - 1/2$, $X$ is defined and nonvanishing on $\ste\setminus\{0\}$. On $\ste\setminus\{0\}$, $D_{i}X^{j} = |\eul|^{2\al}_{h}\delta_{i}\,^{j} + 2\al|\eul|^{-2}_{h}\eul^{p}h_{pi}X^{j}$, so $D_{p}X^{p} = (n+2\al)|\eul|^{2\al}_{h}$ and $\si_{i} = 2\al|\eul|^{-2}_{h}\eul^{p}h_{pi}$.  
In the case $\al > -1/2$, $\si_{i}$ does not extend smoothly to the origin, but in both cases $\si_{p}X^{p} = 2\al |\eul|^{2\al}_{h}$ is smooth on the entire domain of definition of $X$. There holds $X \wedge Y \wedge D_{Y}X = 0$ for all vector fields $Y$ on the domain of definition of $X$. Note also that $d\si_{ij} = 0$ where $\si_{i}$ is defined. When $\al = -n/2$, $X$ is an example of a divergence free projectively dilatative vector field.
\end{example}


By \eqref{projdil4} of Lemma \ref{projdillemma} a vector field $X$ projectively dilatative with respect to $\en$ is projectively geodesic in the sense that $X \wedge \nabla_{X}X = 0$ for all $\nabla \in \en$, so every integral manifold of the line field $L$ spanned by $X$ is a projective geodesic of $\en$. Consequently, by Lemma \ref{projdillemma}, that a vector field $X$ be projectively dilatative with respect to a projective structure $\en$ means that every integral manifold of the line field $L$ spanned by $X$ is a projective geodesic of $\en$, and that every flag $(C, K)$, the base curve of which is an integral manifold of $L$, is a parallel geodesic flag for $\en$. 

If a line field $L$ is dilatative with respect to connections $\nabla$ and $\bnabla$ then it makes sense to declare $\nabla$ and $\bnabla$ equivalent if every parallel geodesic flag for $\nabla$ with base curve an integral manifold of $L$ is a parallel geodesic flag for $\bnabla$, and similarly with the roles of $\bnabla$ and $\nabla$ interchanged. Define a \emph{projective dilatation structure} to be a pair $(\enb, L)$ comprising a line field $L \subset TM$ and an equivalence class $\enb$ of torsion-free connections such that $L$ is projectively dilatative with respect to each $\nabla \in \enb$ and any two representatives of $\enb$ are equivalent in the preceding sense.

If $\nabla \in \enb$ then $\enb$ contains the projective structure $\en$ generated by $\nabla$, so it makes sense to say that a projective structure $\en$ is \emph{subordinate} to a projective dilatation structure $(\enb, L)$ if $\en \subset \enb$. It makes sense to ask for conditions distinguishing a projective structure subordinate to a projective dilatation structure. Theorem \ref{conenormalizationtheorem} can be viewed as accomplishing this in a particular setting. 

If $(\enb, L)$ is a projective dilatation structure, then the integral manifolds of $L$ are projective geodesics for any $\nabla \in\enb$. As the parallel flags the base paths of which are the integral manifolds of $L$ do not depend on the choice of $\nabla \in \enb$, it makes sense to call these parallel flags the \emph{radial flags} of the projective dilatation structure $(\enb, L)$.

A \emph{radiant structure} on a manifold $M$ is a pair $(\nabla, \rad)$ comprising a torsion-free affine connection $\nabla$ and a vector field $\rad \in \Ga(TM)$ such that $\nabla_{i}\rad^{j} = \delta_{i}\,^{j}$. The vector field $\rad$ is said to be \emph{radiant} with respect to $\nabla$. The curvature of a radiant structure means the curvature of $\nabla$, and a radiant structure is flat if $\nabla$ is flat. A flat affine manifold admits a radiant vector field if and only if its affine holonomy has a fixed point. For back ground on flat radiant manifolds see \cite{Choi-radiant, Fried-Goldman-Hirsch, Goldman-Hirsch}. The simplest examples of radiant flat affine manifolds are cones in a vector space and their cocompact quotients by subgroups of dilations, such as affine Hopf manifolds. 

\begin{example}\label{hopfexample}
Let $\ste$ be an $(n+1)$-dimensional vector space and let $\rad$ be the radial vector field generating dilations by $e^{t}$ and let $\hnabla$ be the flat affine connection determined by the vector space structure.
For $\la > 1$ let $\lb \la \ra$ denote the cyclic subgroup of $\reat$ generated by the powers of $\la$. As $\rad$ and $\hnabla$ are invariant under $\lb \la \ra$, they descend to the quotient $\hopf^{n}(\la)$ of $\ste \setminus \{0\}$ by $\lb \la \ra$, which is therefore a radiant flat affine manifold. The manifold $\hopf^{n}(\la)$ is called an \emph{affine Hopf manifold}. It is diffeomorphic to $S^{n-1} \times S^{1}$. 

More generally, the following construction due to Benzecri \cite[Section $2$]{Benzecri} shows that if an $n$-manifold $M$ carries a flat projective structure, then $M \times S^{1}$ carries an affine structure. The presentation here follows  Y. Benoist \cite[Section $1.3$]{Benoist-affine-tori}. Suppose $n > 1$ so that $\ste \setminus\{0\}$ is simply-connected. The idea comes from viewing the affine structure on $\hopf(\la)$ as a lifting of the projective structure on the space of rays in $\ste$, $\projp(\ste)$ via the projection $\bar{\rho}:\hopf(\la) \to \projp(\ste)$ induced by the canonical projection $\rho:\ste \setminus \{0\} \to \projp(\ste)$. The group of projective transformations of $\projp(\ste)$ is isomorphic to the group $SL^{\pm}(\ste)$ defined as the subgroup of $GL(\ste)$ acting unimodularly on $\ext^{n+1}\ste$. Let $\pi:\tilde{M} \to M$ be a universal cover of $M$, let $\dev:\tilde{M} \to \projp(\ste)$ and $\hol:\pi(M) \to SL^{\pm}(\ste)$ be the corresponding developing and holonomy maps of the projective structure. Let $\hat{M}= \{(x, v) \in \tilde{M} \times \hopf^{n}(\la): \dev(x) = \bar{\rho}(v)\}$.  Define an action of $\ga \in \pi(M)$ on $\hat{M}$ by $\ga(x, v) = (\ga x, \hol(\ga)v)$ and let $\bar{M}$ be the quotient of $\hat{M}$ by this action of $\pi_{1}(M)$. A universal cover of $\bar{M}$ is given by $N = \{(x, v) \in \tilde{M} \times \ste \setminus \{0\}: \dev(x) = \si(v)\}$ and the developing map $\bar{\dev}:N \to \ste$ of the resulting affine structure on $\bar{M}$ is given by $\bar{\dev}(x, v) = \si(v)$. Since $\bar{\dev}(N) \subset \ste \setminus\{0\}$, this affine structure is radiant.

In this example, $\bar{M}$ is diffeomorphic to $M \times S^{1}$. Example \ref{thomashopfexample} generalizes this example by allowing the projective structure to be nonflat (in this case the resulting radiant structure is also nonflat).
\end{example}

\begin{example}
If $\nabla$ is the Levi-Civita connection of a pseudo-Riemannian metric and $X$ is radiant with respect to $\nabla$, then $X$ is said to be \emph{concurrent}. See \cite[Section $1.10$]{Chen-warped}.
\end{example}

\begin{example}
A dilatative vector field $X$ satisfies that $\div_{\nabla}(X)$ is a nonzero constant if and only if $\nabla_{i}\nabla_{j}X^{k} = 0$. In this case, a constant multiple of $X$ is radiant.
\end{example}

\begin{lemma}\label{ricshiftlemma}
Let $M$ be a smooth manifold of dimension $n \geq 2$. If $X \in \Ga(TM)$ is dilatative with respect to $\nabla$ and $\nabla_{i}X^{j} = f\delta_{i}\,^{j}$ with $f\in \cinf(M)$ nowhere vanishing, then $\tnabla = \nabla + \tfrac{1}{(1-n)}f^{-1}R_{ij}X^{k}$ and $\tilde{X}^{i} = f^{-1}X^{i}$ satisfy $\tnabla_{i}\tilde{X}^{j} = \delta_{i}\,^{j}$, so constitute a radiant structure.
\end{lemma}

\begin{proof}
By \eqref{ricrad}, $R_{ip}X^{p} = (1-n)df_{i}$, so $\tnabla_{i}X^{j} = f\delta_{i}\,^{j}$, so that
\begin{align}
\tnabla_{i}\tilde{X}^{j} & = -f^{-2}df_{i}X^{j} + f^{-1}(\nabla_{i}X^{j} + \tfrac{1}{1-n}f^{-1}R_{ip}X^{p}X^{j}) = f^{-1}\nabla_{i}X^{j} = \delta_{i}\,^{j}.  \qedhere
\end{align}
\end{proof}

\begin{corollary}\label{almostradiantcorollary}
On a manifold, $M$, of dimension $n > 2$, suppose $X \in \Ga(TM)$ is projectively dilatative with respect to torsion-free affine connection $\nabla$ and $\nabla_{i}X^{j} = f\delta_{i}\,^{j} + \si_{i}X^{j}$ for $f \in \cinf(M)$ and $\si_{i} \in \Ga(T^{\ast}M_{\ast})$ where $M_{\ast} = \{p \in M: X_{p} \neq 0\}$. On the subset $M_{\ast\ast}$ of $M_{\ast}$ where $f - \si_{p}X^{p} \neq 0$, the connection $\tnabla = \nabla  - 2\si_{(i}\delta_{j)}\,^{k} + \tfrac{1}{1-n}(f - \si_{p}X^{p})^{-1}(R_{ij} + (n-1)(\nabla_{i}\si_{j} + \si_{i}\si_{j}) + d\si_{ij})X^{k}$ and $\tilde{X} = (f - \si_{p}X^{p})^{-1}X^{i}$ satisfy $\tnabla_{i}\tilde{X}^{j} = \delta_{i}\,^{j}$ so constitute a radiant structure generating on $M_{\ast\ast}$ the same projective dilatation structure as that generated by $(\nabla, X)$.
\end{corollary}

\begin{proof}
Combine successively Corollary \ref{projdilnormalcorollary} and Lemma \ref{ricshiftlemma} making use of \eqref{projvary}.
\end{proof}

Corollary \ref{almostradiantcorollary} shows that if $(\enb, L)$ is a projective dilatation structure on $M$ that admits a globally defined projectively dilatative representative pair $(\nabla, X)$, then there is an open subset of $M$ on which $(\enb, L)$ is represented by a radiant structure $(\nabla, X)$. In general there is no reason to think a projective dilatation structure will admit a globally defined representative $(\nabla, X)$, and less still to expect that when it does that $f - \si_{p}X^{p}$ will be everywhere nonzero, but nonetheless the corollary provides motivation for studying radiant structures as local avatars of projective dilatation structures. 

Although the remainder of the paper focuses on radiant structures, projective dilatation structures are a priori more general objects whose existence is in principle less limited by topological restrictions than is the existence of radiant structures (as will be shown next). In some imprecise sense radiant structures are to projective dilatation structures as affine structures are to projective structures or tori are to hyperbolic surfaces. 

Let $(\nabla, \rad)$ be a radiant structure on $M$. If $M$ is compact then $\rad$ is complete, but $\nabla$ need not be complete. In the case $\nabla$ is flat, if $\nabla$ is complete then its affine holonomy cannot have a fixed point unless it is trivial, and so $\nabla$ cannot admit a radiant vector field unless it is the standard flat connection on a vector space. Radiant flat affine structures on compact manifolds are the simplest examples of incomplete flat affine connections on compact manifolds. 

Lemma \ref{maxgeodesiclemma} relates the images of geodesics of $\nabla$ with the integral curves of $\rad$. In considering Lemma \ref{maxgeodesiclemma} it is helpful to consider the following example.
If $\si(t)$ is an integral curve of $\rad$, then the reparametrization $\ga(t) = \si(\log|1 - t|)$ is a geodesic of $\nabla$ passing through $\si(0)$. Whether $\ga(t)$ blows up or not when $t$ approaches $1$ depends on the behavior of $\si(s)$ as $s \to -\infty$. If $\si(s)$ approaches a zero of $\rad$ as $s \to -\infty$ then in general the geodesic $\ga(t)$ continues through the apparent singularity. For example this happens on $\rea^{n}$. 
Consider the radial vector field $\eul = x^{i}\tfrac{\pr}{\pr x^{i}}$ on $\rea^{n}$ generated by dilations by $e^{s}$. If $0 \neq v \in \rea^{n}$ then $\eul_{v} \neq 0$. The maximal integral curve of $\eul$ passing through $v$ is $\si(s) = e^{s}v$, defined for $s \in \rea$. The curve $\ga(t) = \si(\log|1 - t|)$ is the geodesic of the standard flat affine connection such that $\ga(0) = v$ and $\tfrac{d}{dt}\ga(0) = -\eul_{v}$. It is defined for all $t \in (-\infty, 1)$, but in fact, as $\si(\log|1-t|) = (1-t)v$ for $t \in (-\infty, 1)$, it is not a maximal geodesic with these initial conditions, such being given by $\tau(t) = (1-t)v$ for $t \in \rea$. The image of this maximal geodesic is the union of three images of integral curves of $\eul$, namely the origin and the two rays leaving the origin in the $v$ and $-v$ directions. Lemma \ref{maxgeodesiclemma} shows that a similar relation between the images of the integral curves of $\rad$ and the images of geodesics of $\nabla$ holds for any radiant structure $(\nabla, \rad)$.

\begin{lemma}\label{maxgeodesiclemma}
For a radiant structure $(\nabla, \rad)$ there hold
\begin{enumerate}
\item The image of the maximal integral curve of $\rad$ passing through a nonsingular point $p$ of $\rad$ is contained in the image of the maximal geodesic of $\nabla$ passing through $p$ and having at $p$ tangent $\rad_{p}$.
\item The image of a maximal geodesic of $\nabla$ passing through a singular point of $\rad$ contains the union of the images of the maximal integral curves of $\rad$ the closures of which contain the singular point. 
\item If $\rad$ is tangent to a maximal totally geodesic submanifold $\Sigma \subset M$ at a point $p \in M \cap \Sigma$ then $\Sigma$ contains every image of an integral curve of $\rad$ the closure of which contains $p$. Moreover, if $p \in \Sigma$ is a singular point of $\rad$ then $\rad$ is tangent to $\Sigma$ along some open neighborhood of $p$ in $\Sigma$. 
\end{enumerate}
\end{lemma}

\begin{proof}
First it is shown that if $\rad_{p} \neq 0$ then the image of the maximal integral curve $\si:J \to M$ of $\rad$ passing through $p$ is contained in the image of the maximal geodesic through $p$ tangent to $\rad_{p}$ at $p$. Here $J$ is an open interval in $\rea$ containing $0$, $\si(0) = p$, and $\tfrac{d}{ds}\si(s) = \rad_{\si(s)}$ for $s \in J$. It is convenient to write $J = (a, b)$, allowing $a = -\infty$ and $b = \infty$. From $\rad \wedge \nabla_{\rad}\rad  = 0$ it follows that the curve $\ga(t) = \si(\log|1 - t|)$ is the unique $\nabla$ geodesic such that $\ga(0) = p$ and $\tfrac{d}{dt}\ga(0) = -\rad_{p}$. Since $a < 0$, $(1 - e^{b}, 1-e^{a}) \subset (-\infty, 1)$, where $e^{a}$ has to be interpreted as $0$ in the case $a = -\infty$, and $1 - e^{b}$ has to be interpreted as $-\infty$ in the case $b = \infty$. It follows that the image of $\ga$ contains the image of $\si$. The point is that $\ga$ is defined on $(1 - e^{a}, 1 - e^{b})$, but could extend to a larger interval. 

Let $x^{i}$ be normal coordinates in a neighborhood of a singular point $p \in M$ such that the origin corresponds to $p$ and write $\rad = A^{i}(x)\tfrac{\pr}{\pr x^{i}}$ and $X_{i} = \tfrac{\pr}{\pr x^{i}}$. By assumption $X_{i} = \nabla_{X_{i}}\rad = \left(dA^{j}(X_{i}) + A^{p}\Ga_{ip}\,^{j}\right)X_{j}$ in which $\Ga_{ij}\,^{k}$ are the Christoffel symbols of $\nabla$ with respect to the frame $X_{i}$. Because $\rad_{p} = 0$ is assumed, there holds at the origin $\tfrac{\pr A^{j}}{\pr x^{i}}(0) = \delta_{i}\,^{j}$ and so $A^{i} = x^{i} + c^{i}$ for some constants $c^{i}$. Moreover, since $A^{i}(0) = 0$, these constants are $0$, and so $A^{i} = x^{i}$ and $\rad$ has the form $x^{i}X_{i}$ of the usual Euler vector field on $\rea^{n}$, the integral curves of which are the rays leaving the origin and the origin itself. Because the coordinates are normal the image of the geodesic of $\nabla$ passing through the origin with velocity $v$ at the origin is the line passing through the origin in the direction $v$, which is contained in a union of images of integral curves of $\rad$, namely the origin and the rays with directions $\pm v$. This shows that there is an open neighborhood $U$ of $p$ such that the intersection with $U$ of the image of a maximal geodesic passing through $p$ is contained in the intersection with $U$ of a union of images of integral curves of $\rad$. By the preceding paragraph and the fact that the singular points of $\rad$ are isolated, these images are themselves contained in the maximal geodesic. If the singular point $p$ is contained in the closure of the image $L$ of a maximal integral curve of $\rad$ then in the normal coordinates considered above, either $L$ is $p$ itself, or the intersection $L \cap U$ corresponds in coordinates to part of a ray limiting to the origin. In either case, this shows $L$ is contained in a maximal geodesic passing through $p$.

Suppose $\Sigma\subset M$ is a maximal totally geodesic submanifold. If $p\in \Sigma$ is not a singular point of $\rad$ and $\rad_{p} \in T_{p}\Sigma$ then the image of the maximal integral curve of $\rad$ passing through $p$ is contained in the image of a maximal geodesic of $\nabla$, so is contained in $\Sigma$. If $p$ is a singular point then $\rad_{p} = 0$ and so trivially $\rad_{p} \in T_{p}\Sigma$. In this case, by the preceding paragraph, the image of any geodesic passing through $p$ contains a union of images of integral curves of $\rad$ the closure of each of which contains $p$. Since the image of the geodesic is contained in $\Sigma$, it contains these images too. In particular there is an open neighborhood $U$ of $p$ in $\Sigma$ such that for $q \in U$ different from $p$ the intersection with $U$ of the image of the maximal integral curve of $\rad$ passing through $q$ contains $p$ in its closure, and so is contained in $\Sigma$; hence $\rad_{q} \in T_{q}\Sigma$. 
\end{proof}

\begin{remark}
A consequence of \cite[Theorem $3.3$]{Fried-Goldman-Hirsch} is that the radiant vector field of a flat radiant structure on a compact manifold is nonsingular. More readable proofs are given in \cite[Corollary $6.5.3$]{Goldman-notes} and \cite[Theorem $3$]{Daly}. The idea of the proofs is to show that the flow of $\rad$ retracts the manifold onto the singular set of $\rad$. All the proofs make essential use of the fact that the developing map is defined globally on the universal cover and it is not clear to the author how to modify them when flatness is not assumed.
\end{remark}

If a compact orientable manifold, $M$, admits a nonsingular radiant structure, then $\chi(M) = 0$. 
It is proved in \cite[p. $502$]{Fried-Goldman-Hirsch} that a radiant vector field on a compact flat affine manifold $M$ is nonsingular, and so $\chi(M) = 0$ for such a manifold. In general it is not clear whether a radiant vector field on a compact manifold $M$ can have zeroes. Lemma \ref{radiantisolatedlemma} shows that if it does, then $M$ must be even-dimensional.

\begin{lemma}\label{radiantisolatedlemma}
If $(\nabla, \rad)$ is a radiant structure, a zero of $\rad$ is nondegenerate, so isolated, and its index is $+1$.
\begin{enumerate}
\item If a compact manifold $M$ admits a radiant structure $(\nabla, \rad)$, then $\chi(M) \geq 0$, with equality if and only if $\rad$ is nonsingular. 
\item If a compact and odd-dimensional manifold admits a radiant structure $(\nabla, \rad)$, then $\rad$ is nonsingular. 
\end{enumerate}
\end{lemma}
\begin{proof}
Let $p \in \zero(\rad)$. Let $E_{1}, \dots, E_{n} \in \Ga(TU)$ be a local frame over an open neighborhood $U \subset M$ of $p$, let $\nabla_{E_{i}}E_{j} = \sum_{k = 1}^{n}\Ga_{ij}\,^{k}E_{k}$, and write $\rad = \sum_{i = 1}^{n}f^{i}E_{i}$. Then $E_{i} = \nabla_{E_{i}}\rad = \left(df^{j}(E_{i}) + f^{q}\Ga_{iq}\,^{j}\right)E_{j}$, so on $U\cap \zero(\rad)$ there holds $df^{j}(E_{i}) = \delta_{j}\,^{i}$. Hence $p$ is nondegenerate, and so is isolated. Since $\det \nabla \rad$ is positive at $p$, the index of $\rad$ at $p$ is $1$ \cite[Lemma $6.4$]{Milnor-topology}. If $M$ is compact and orientable, then by the Poincarè-Hopf Index Theorem, $\chi(M)$ equals the sum of the indices of any vector field with isolated zeroes, in particular the sum of the indices of $\rad$, which is positive unless $\rad$ is nonsingular, in which case it is $0$. If $M$ is compact and nonorientable, then the radiant structure lifts to its connected oriented double cover and the same conclusion follows. If $M$ is compact and odd-dimensional then $\chi(M) = 0$, and so $\rad$ is nonsingular.
\end{proof}

\begin{remark}
Lemma \ref{radiantisolatedlemma} implies that neither a compact surface of genus at least two, nor the product of any such surface with an even-dimensional sphere, admits a radiant structure.
In higher dimensions it is not clear how limiting is the conclusion of Lemma \ref{radiantisolatedlemma}. It rules out radiant structures on sphere bundles over higher genus surfaces. On the other hand, the affirmative resolution of question attributed to Thurston \cite[Problem $4.10$]{Kirby} asserts that an aspherical compact four manifold has nonnegative Euler characteristic. The simplest example of a simply-connected, compact manifold which by Lemma \ref{radiantisolatedlemma} admits no radiant structure seems to be the connected sum of $S^{3} \times S^{3}$ with itself.
\end{remark}

\begin{remark}
In this paper there is not given any example of a manifold with positive Euler characteristic admitting a radiant structure. It would be interesting to know if such examples exist. The natural place to look is compact homogeneous spaces, which have nonnegative Euler characteristic \cite{Mostow-structure}.
By a conjecture attributed to Chern, a compact manifold admitting a flat affine structure has Euler characteristic zero, and this would imply no compact manifold with positive Euler characteristic can admit a flat radiant structure. Can such a manifold admit a radiant structure? In addressing such questions it may be interesting to ask them for radiant structures with symmetric Ricci tensor. For example, Lemma \ref{nosymmetriclemma} shows that compact manifold with vanishing first Betti number admits no Ricci symmetric radiant structure, while Theorem \ref{qaconetheorem} shows that the three sphere does admit a radiant structure (with purely antisymmetric Ricci tensor).
\end{remark}

\begin{lemma}\label{liesslemma}
Let $(\nabla, \rad)$ be a radiant structure on an $n$-manifold. For $S_{i_{1}\dots i_{r}}^{j_{1}\dots j_{s}} \in \Ga(\tensor^{r}\ctm \tensor \tensor^{s}TM \tensor |\Det \ctm|^{\la})$ there holds $\nabla_{\rad}S = \lie_{\rad}S + (s - r - n\la)S$. 
\end{lemma}
\begin{proof}
For a torsion-free connection $\nabla$ and $X\in \Ga(TM)$ there holds
\begin{align}\label{liederivative}
\begin{split}
(\lie_{X}S)_{i_{1}\dots i_{r}}^{j_{1}\dots j_{s}} &- X^{p}\nabla_{p}S_{i_{1}\dots i_{r}}^{j_{1}\dots j_{s}}  =  \sum_{a = 1}^{r}S_{i_{1}\dots p\dots i_{r}}^{j_{1}\dots\,\,\dots j_{s}}\nabla_{i_{a}}X^{p} - \sum_{a = 1}^{s}S_{i_{1}\dots\,\,\dots i_{r}}^{j_{1}\dots p\dots j_{s}}\nabla_{p}X^{j_{a}} + \la (\div_{\nabla}(X))S_{i_{1}\dots i_{r}}^{j_{1}\dots j_{s}},
\end{split}
\end{align}
in which $\div_{\nabla}(X) =  \nabla_{p}X^{p}$. 
Taking $X = \rad$ in \eqref{liederivative} gives the claim. 
\end{proof}

Lemma \ref{paralleloneformlemma} generalizes \cite[Theorems $3.1$ and $3.2$]{Fried-Goldman-Hirsch} to the nonflat setting.
\begin{lemma}\label{paralleloneformlemma}
If $(\nabla, \rad)$ is a radiant structure on a compact manifold then there is no $\nabla$-parallel volume form and a $\nabla$-parallel one-form is identically zero.
\end{lemma}
\begin{proof}
If $\Psi$ is an $n$-form on a radiant $n$-manifold then $d(i(\rad)\Psi) = \lie_{\rad}\Psi = \nabla_{\rad}\Psi  + n\Psi$ by Lemma \ref{liesslemma}. If $\Psi$ is a parallel volume form, then $\Psi$ is exact and $M$ is orientable, and so by Stokes's theorem $M$ must be noncompact. 
A one-form parallel with respect to a torsion-free connection is closed, so if $\be$ is a parallel one-form then, by Lemma \ref{liesslemma} there holds $\be = \lie_{\rad}\be - \nabla_{\rad}\be = d(\be(\rad)) = dg$, where $g = \be(\rad)$. Since $dg$ is parallel, if it vanishes at one point it vanishes identically. If $M$ is compact then $g$ must have a minimum, at which $dg$ vanishes, and so $\be = dg$ is identically zero.
\end{proof}

\begin{example}
Because the affine Hopf manifold $\hopf^{n}(\la)$ of example \ref{hopfexample} is compact it admits no parallel (nor any exact) volume form. 
\end{example}

\begin{lemma}\label{riccisymmetriclemma}
On a manifold with vanishing first Betti number, a torsion-free affine connection with symmetric Ricci tensor admits a parallel volume density.
\end{lemma}

\begin{proof}
If a torsion-free affine connection $\nabla$ has symmetric Ricci tensor, then in a neighborhood of every point it admit a parallel volume density. Let $\mu$ be a global nonvanishing volume density and define a one-form $\ga$ by $\nabla \mu = \ga \tensor \mu$. Then $d\ga_{ij}\mu = 2\nabla_{[i}\nabla_{j]}\mu = -2R_{ijp}\,^{p}\mu = 2R_{[ij]}\mu = 0$, so $\ga$ is closed and hence locally exact. If $f$ is local primitive of $\ga$ then $\nabla(e^{-f}\mu) = 0$. Let $\{U_{\al}\}$ be an open cover such that on $U_{\al}$ there is a $\nabla$-parallel nonvanishing volume density $\mu_{\al}$. There is $f_{\al\be} \in \cinf(U_{\al}\cap U_{\be})$ such that $\mu_{\al} = e^{f_{\al\be}}\mu_{\be}$ on $U_{\al}\cap U_{\be}$. Because $0 = \nabla \mu_{\al} = df_{\al\be}\tensor \mu_{\be}$, refining the open cover if necessary, $f_{\al\be}$ is constant on $U_{\al}\cap U_{\be}$. Because the first Betti number is zero, the Cech cocycle $\{f_{\al\be}\}$ is exact, so there is a $0$-chain $\{g_{\al}\}$ with values in the constant sheaf $\rea$ (again, replacing the given cover with a refinement if necessary) such that $f_{\al\be} = g_{\al} - g_{\be}$ on $U_{\al}\cap U_{\be}$, and $e^{-g_{\al}}\mu_{\al}$ patch together to give a global parallel volume density.
\end{proof}

For flat radiant structures Lemma \ref{nosymmetriclemma} recovers a special case of \cite[Proposition $2.7$]{Goldman-Hirsch}. 

\begin{lemma}\label{nosymmetriclemma}
A compact manifold with vanishing first Betti number admits no radiant structure with symmetric Ricci tensor.
\end{lemma}

\begin{proof}
If there were a radiant structure with symmetric Ricci tensor, by Lemma \ref{riccisymmetriclemma} it would admit a parallel volume form, contrary to Lemma \ref{paralleloneformlemma}.
\end{proof}

In constrast with Lemma \ref{nosymmetriclemma}, Theorem \ref{qaconetheorem} shows that the three sphere admits a radiant structure with skew-symmetric Ricci tensor. 

Some issues arise when the curvature of a radiant structure $(\nabla, \rad)$ is not identically zero that are absent in the case that $\nabla$ is flat. The condition $\nabla_{i}\rad^{j} = \delta_{i}\,^{j}$ implies $(\lie_{\rad}\nabla)_{ij}\,^{k} = \rad^{p}R_{pij}\,^{k}$, so that a radiant vector field need not be an infinitesimal affine automorphism. Geometrically this corresponds to the nonexistence, even locally, of totally geodesic surfaces to which $\rad$ is tangent.
On the other hand, it is an important feature of a radiant vector field on a flat affine manifold that it is necessarily an infinitesimal affine automorphism.

For a torsion-free affine connection, $\nabla$, and a vector field, $X$, differentiating $(\lie_{X}\nabla)_{ij}\,^{k} = \nabla_{i}\nabla_{j}X^{k} + X^{p}R_{pij}\,^{k}$, tracing the result, and using the Ricci identity, the differential Bianchi identity, and \eqref{liederivative} yields
\begin{align}\label{lienabladiv}
\begin{split}
\nabla_{p}(\lie_{X}\nabla)_{ij}\,^{p} 
& = \nabla_{i}\nabla_{j}\nabla_{p}X^{p} + 2\nabla_{(i}X^{p}R_{j)p} + X^{p}\nabla_{p}R_{ij} + 2X^{p}\nabla_{i}R_{[jp]}\\
& =   \nabla_{i}\nabla_{j}\nabla_{p}X^{p} +  (\lie_{X}\ric)_{ij}  + 2R_{[jp]}\nabla_{i}X^{p} + 2 X^{p}\nabla_{i}R_{[jp]},
\end{split}
\end{align}
which is needed in the proof of Lemma \ref{lieradlemma}.

\begin{lemma}\label{lieradlemma}
Let $(\nabla, \rad)$ be a radiant structure and define a one-form $\er_{i}$ by $\er_{i} = \rad^{p}R_{pi}$. 
There hold
\begin{align}\label{lieradric}
&(\lie_{\rad}\ric)_{(ij)}  = \nabla_{p}(\lie_{\rad}\nabla)_{ij}\,^{p} + \nabla_{(i}\er_{j)} = \rad^{q}\nabla_{p}R_{qij}\,^{p} + R_{ij} + \nabla_{(i}\er_{j)},&
&(\lie_{\rad}\ric)_{[ij]}  
= \nabla_{[i}\er_{j]} = \tfrac{1}{2}d\er_{ij},\\
\label{radlierad}
&\rad^{p}(\lie_{\rad}\ric)_{pi}  = \er_{i} + \rad^{p}\nabla_{p}\er_{i} = (\lie_{\rad}\er)_{i},&
&\rad^{p}(\lie_{\rad}\ric)_{ip}  = 0.
\end{align}
In particular, if $R_{ij}$ is $\la$-positively homogeneous, then $\er_{i}$ is $\la$-positively homogeneous.
\end{lemma}
\begin{proof}
Antisymmetrizing, tracing, and differentiating the identity $\nabla_{i}\nabla_{j}\rad^{k} = 0$ yields 
\begin{align}\label{radid}
\begin{aligned}
&R_{ijp}\,^{k}\rad^{p} = 0,& & R_{ip}\rad^{p} = 0, & \\& R_{ji} = -\rad^{p}\nabla_{i}R_{jp},&
&R_{ijk}\,^{l} = - \rad^{p}\nabla_{k}R_{ijp}\,^{l},&& 2R_{[ij]} = 2\rad^{p}\nabla_{[i}R_{j]p} = \rad^{p}\nabla_{q}R_{ijp}\,^{q} .
\end{aligned}
\end{align}
In particular, $\er_{i}$ vanishes if $\nabla$ has symmetric Ricci tensor. 
That $\nabla_{i}\nabla_{j}\rad^{k} = 0$ implies immediately the first of 
\begin{align}\label{lieradnabla}
&(\lie_{\rad}\nabla)_{ij}\,^{k} = \rad^{p}R_{pij}\,^{k}, && (\lie_{\rad}\nabla)_{ip}\,^{p} = -\er_{j},
\end{align}
while the second follows from the first and \eqref{radid}. From \eqref{radid} it follows that $\nabla_{i}\er_{j} = 2R_{[ij]} + 2\rad^{p}\nabla_{i}R_{[pj]}$.
In \eqref{lienabladiv} this yields $(\lie_{\rad}\ric)_{ij}  = \nabla_{p}(\lie_{\rad}\nabla)_{ij}\,^{p} + \nabla_{i}\er_{j}$, and decomposing this into its symmetric and antisymmetric parts and using again \eqref{radid} yields \eqref{lieradric}. Contracting $\rad^{i}$ with \eqref{lieradric} and simplifying yields \eqref{radlierad}.
\end{proof}

\begin{remark}
Example \ref{nonuniquenessexample} exhibits a radiant structure for which $\er$ is not identically zero. Example \ref{leftinvarianterexample} exhibits a radiant structure for which $\er$ is nowhere vanishing and not closed. In neither of these examples is the underlying manifold compact, and it would be interesting to know if there is a radiant structure on a compact manifold having nonzero $\er$.
\end{remark}

\begin{remark}
The identity $R_{ijp}\,^{k}\rad^{p} = 0$ means that the affine connection $\nabla$ of a radiant structure has \emph{Weyl nullity} in the sense of \cite{Gover-Matveev}. See in particular \cite[Proposition $2.5$]{Gover-Matveev}.
\end{remark}

\begin{lemma}\label{2dsymmetriclemma}
A radiant structure on a surface has symmetric Ricci tensor.
\end{lemma}

\begin{proof}
By \eqref{bijkl}, in dimension $2$, $R_{ijk}\,^{l} = -2R_{k[i}\delta_{j]}\,^{l}$, so by \eqref{lieradnabla}, $(\lie_{\rad}\nabla)_{ij}\,^{k} = \rad^{p}R_{pij}\,^{k} = R_{ji}\rad^{k}$. Hence $0 = (\lie_{\rad}\nabla)_{[ij]}\,^{k} = -R_{[ij]}\rad^{k}$. Because the zeros of $\rad$ are nondegenerate, so isolated, this implies $R_{[ij]} = 0$.
\end{proof}

By \cite{Benzecri} a compact surface admits a flat affine structure if and only if it is a torus or a Klein bottle. These are classified in \cite{Arrowsmith-Furness-locallysymmetric, Arrowsmith-Furness, Nagano-Yagi} and some of them are radiant. See also \cite[Section $4$]{Baues-torus} and \cite{Benoist-affine-tori}. For radiant structures, Theorem \ref{2dtorustheorem} extends this result to non-flat structures.

\begin{theorem}\label{2dtorustheorem}
A compact surface admits a radiant structure if and only if it is a torus or a Klein bottle.
\end{theorem}

\begin{proof}
By Lemma \ref{radiantisolatedlemma}, a compact surface that admits a radiant structure has nonnegative Euler characteristic, while by Lemmas \ref{nosymmetriclemma} and \ref{2dsymmetriclemma} the two-sphere admits no radiant structure. Were the projective plane to admit a radiant structure, it could be lifted to the two-sphere. A two-dimensional affine Hopf manifold is a torus, so tori admit radiant structures, as do Klein bottles \cite{Arrowsmith-Furness} and \cite[Section $6.5.2.2$]{Goldman-notes}.
\end{proof}

\begin{corollary}
A radiant structure on a compact surface is nonsingular.
\end{corollary}
\begin{proof}
By Theorem \ref{2dtorustheorem}, the surface has Euler characteristic zero, so this follows from Lemma \ref{radiantisolatedlemma}.
\end{proof}

\section{Euler manifolds and equiaffine radiant structures}\label{eulersection}
The differential of a vector field, $X$, is well defined at a point $p \in \zero(X)$ as the value at $p$ of $\nabla_{i}X^{j}$ for any torsion-free affine connection $\nabla$, for if $\bnabla = \nabla + \Pi_{ij}\,^{k}$ is any other torsion-free affine connection, then, at $p$ there holds $\bnabla_{i}X^{k} = \nabla_{i}X^{j} + \Pi_{ip}\,^{j}X^{p} = \nabla_{i}X^{j}$. Let $N \subset M$ be a connected smoothly embedded submanifold of $M$. The case where $N$ is $0$-dimensional, so a point, is allowed, and is the case of principal interest here. Following \cite{Bursztyn-Lima-Meinrenken, HajSaeediSadegh-Higson, Meinrenken-eulerlike}, a vector field $X \in \Ga(TM)$ is \emph{Euler-like along $N$} if, for every $f \in \cinf(M)$ that vanishes to first order along $N$, $df(X) - f$ vanishes to second order along $N$. In \cite{Bursztyn-Lima-Meinrenken} it is also required that $X$ be complete, but here this condition is omitted, as also in \cite{HajSaeediSadegh-Higson}, although in fact in the examples of interest considered here, the Euler-like vector field is complete. That $df(X) - f$ vanish to second order along $N$ for any $f$ vanishing to first order along $N$ implies that $X$ vanishes along $N$ as well.

A vector field $X$ is Euler-like along a submanifold $N$ of codimension $n - r$ if and only if for every $p \in N$ and every choice of local coordinates $x^{1}, \dots, x^{r}, y^{1}, \dots, y^{n-r}$ centered on $p$ such that $\{x^{1} = 0, \dots, x^{r} = 0\}$ defines the part of $N$ contained in the coordinate neighborhood of $p$, the vector field $X$ has the form $\sum_{i = 1}^{n-r}a^{i}(x, y)\tfrac{\pr}{\pr y^{i}} + \sum_{i = 1}^{r}(x^{i} + b^{i}(x, y))\tfrac{\pr}{\pr x^{i}}$ with $a^{i}(x, y)$ vanishing when $x^{1} = 0, \dots, x^{r} = 0$, and $b^{i}(x, y)$ vanishing to second order when $x^{1} = 0, \dots, x^{r} = 0$. In particular, along $N$ the differential of $X$ is a projection of rank $r$ onto the normal bundle of $N$, and the linear approximation to $X$ along $N$ is the usual Euler field on the normal bundle of $N$.

A vector field $X \in \Ga(TM)$ is \emph{Euler-like} if each connected component of its zero set $\zero(X)$ is a smoothly embedded submanifold of $M$ and $X$ is Euler-like along each connected component of $\zero(X)$.

\begin{lemma}\label{radianteulerlikelemma}
If $(\nabla, \rad)$ is a radiant structure on a smooth manifold $M$, then $\rad$ is Euler-like.
\end{lemma}

\begin{proof}
By Lemma \ref{radiantisolatedlemma}, $\zero(\rad)$ is a discrete set of points. That $\rad$ be radiant means that its differential at $p \in \zero(\rad)$ equals $\nabla_{i}\rad^{j} = \delta_{i}\,^{j}$. It follows that in geodesic normal coordinates $x^{1}, \dots, x^{n}$ centered at $p$, $\rad$ has the form $\sum_{i = 1}^{n}x^{i}\tfrac{\pr}{\pr x^{i}} + G$ where $G$ is a smooth vector field vanishing to second order at $p$. 
\end{proof}

On an $n$-manifold $M$, a  \emph{volume form} means a nowhere-vanishing $n$-form $\Psi \in \Ga(\Det \ctm)$. 
\begin{definition}
Let $M$ be an $n$-manifold.
\begin{enumerate}
\item A vector field $\rad$ on $(M, \Psi)$ is \emph{Euler} if $\lie_{\rad}\Psi = n\Psi$. 
\item An \emph{Euler structure} on $M$ is a pair $(\Psi, \rad)$ comprising a volume form $\Psi \in \Ga(\Det \ctm\setminus\zsec(M))$ and an Euler-like Euler vector field $\rad \in \Ga(TM)$. 
\end{enumerate}
\end{definition}
By Stokes' Theorem, a manifold admitting an Euler vector field is noncompact, because a volume form admitting an Euler vector field is exact, for $\Psi = d \mu$ where $\mu= \tfrac{1}{n}i(\rad)\Psi$. Because $\Psi$ is nowhere-vanishing, prescribing $\mu$ determines $\rad$, so the form $\mu$ and the Euler field are equivalent data. (The form $\mu$ of an Euler structure satisfies some nontrivial condition corresponding with the Euler-like condition on the Euler vector field. This condition is not elucidated because it is not used.)

\begin{remark}
Note that \emph{Euler} and \emph{Euler-like} are \emph{not} synonyms. The Euler-like condition refers to the vector field alone, while the Euler condition depends on a choice of volume form. Even when both make sense, neither condition implies the other. With respect to the volume form $\Psi = dx \wedge dy$ on $\rea^{2}$, the vector field $X = (x+y)\pr_{x} + (y-x) \pr_{y}$ is Euler but is not Euler-like (for $f(x, y) = x$, $df(X) -f = y$ does not vanish to second order at the origin) while the vector field $Y = x\pr_{x}$ is Euler-like but not Euler. 
\end{remark}

\begin{remark}
In some contexts it might be desirable to impose further conditions as part of the definition of an Euler structure. First, the Euler vector field could be required to be complete; this is generally the case in the examples considered later. Second, if $M$ were allowed to have boundary $\pr M$ and $M$ were oriented, then there could be required additionally that, along $\pr M$, $\rad$ point to the exterior of $M$, or that $\mu$ restrict to a volume form on $\pr M$ consistent with the induced orientation.
\end{remark}

\begin{lemma}\label{eelemma}
For a volume form $\Psi$ and a radiant structure $(\nabla, \rad)$ on an $n$-manifold $M$ the following are equivalent:
\begin{enumerate*}
\item\label{ee1} $\nabla_{\rad}\Psi = 0$;
\item\label{ee2} $\rad$ is an Euler vector field with respect to $\nabla$; 
\item\label{ee3} $(\rad, \Psi)$ is an Euler structure.
\end{enumerate*}
\end{lemma}
\begin{proof}
Any volume form $\Psi$ and any radiant vector field $\rad$ satisfy $\nabla_{\rad}\Psi = \lie_{\rad}\Psi - n\Psi$, so $\nabla_{\rad}\Psi = 0$ if and only if $\lie_{\rad}\Psi = n \Psi$. Since, by Lemma \ref{radianteulerlikelemma}, a radiant vector field is Euler-like, the equivalent conditions \eqref{ee1}-\eqref{ee2} imply $(\rad, \Psi)$ is an Euler structure. That \eqref{ee3} implies \eqref{ee2} is immediate.
\end{proof}

\begin{definition}\label{radianteulerdefined}
\noindent
\begin{enumerate}
\item A \emph{radiant Euler structure} is a triple $(\nabla, \rad, \Psi)$ such that $(\nabla, \rad)$ is a radiant structure and $(\rad, \Psi)$ is an Euler structure. 
\item An \emph{equiaffine radiant structure} is a radiant Euler structure $(\nabla, \rad, \Psi)$ such that $\nabla \Psi = 0$.
\end{enumerate}
\end{definition}

By Lemma \ref{eelemma}, if $(\nabla, \rad)$ is a radiant structure and $\nabla \Psi = 0$, then $(\nabla, \rad, \Psi)$ is an equiaffine radiant structure.

For any equiaffine radiant structure there holds $\nabla i(\rad)\Psi = \Psi$.

\begin{remark}
If $\rad$ is Euler with respect to $\Psi$ then it is Euler with respect to any nonzero constant multiple of $\Psi$. If $(\nabla, \rad, \Psi)$ is a radiant Euler structure, so is $(\nabla, \rad, c\Psi)$ for any $c \in \reat$, so it is the homothety class of $\Psi$ that matters in Definition \ref{radianteulerdefined}, rather than $\Psi$.

If $(\Psi, \rad)$ is a pair comprising a volume form $\Psi$ and a vector field $\rad$ such that $\lie_{\rad}\Psi = \la \Psi$ for some $\la \in \reat$, then $\tfrac{n}{\la}\rad$ is an Euler vector field. However, if $(\nabla, \rad)$ is radiant, rescaling $\rad$ destroys the property that $(\nabla, \rad)$ be radiant. Requiring that $\la$ be $n$ in the definition of Euler vector field is motivated by the condition of compatibility with a radiant connection in Definition \ref{radianteulerdefined} and by Example \ref{tautologicalbundleexample}. 
\end{remark}

\begin{example}
Products of each of radiant structures, Euler structures, and equiaffine radiant structures are again structures of the same kind.
Identify a vector field $X$ on $M$ with the vector field $X \oplus \zs_{M^{\prm}}$ on $M \times M^{\prm}$ generated by the flow given by the product of the flow of $X$ with the identity map on the $M^{\prm}$ factor, and do similarly for a vector field on $M^{\prm}$. If $(\nabla, \rad)$ and $(\nabla^{\prm}, \rad^{\prm})$ are radiant structures on $M$ and $M^{\prm}$, respectively, then $\rad \oplus \rad^{\prm}$ means the vector field on $M \times M^{\prm}$ generated by the flow $\phi^{t} \times \phi^{\prm\,t}$ where $\phi^{t}$ and $\phi^{\prm\,t}$ are the flows of $\rad$ and $\rad^{\prm}$, respectively. If $\nabla$ and $\nablap$ are connections on $M$ and $M^{\prime}$, the product connection $D$ is defined by $D_{X\oplus Y}A\oplus B = \nabla_{X}A \oplus \nablap_{Y}B$. 

It follows from the definition that if $(\nabla, \rad)$ and $(\nablap, \rad^{\prm})$ are radiant structures on $M$ and $M^{\prm}$ then $(D, \rad \oplus \rad^{\prm})$ is a radiant structure on $M \times M^{\prm}$. Similarly, it is straightforward to check that if $(\Psi, \rad)$ and $(\Psi^{\prm}, \rad^{\prm})$ are Euler structures on $M$ and $M^{\prm}$ respectively, then $(\pi^{\ast}(\Psi)\wedge (\pi^{\prime})^{\ast}(\Psi^{\prm}), \rad \oplus \rad^{\prm})$ is an Euler structure on $M \times M^{\prm}$, where $\pi$ and $\pi^{\prm}$ are the projections from $M \times M^{\prm}$ onto its factors. Likewise, if $(\nabla, \rad, \Psi)$ and $(\nablap, \rad^{\prm}, \Psi^{\prm})$ are equiaffine radiant structures on $M$ and $M^{\prm}$, then $(D, \rad \oplus \rad^{\prm}, \pi^{\ast}(\Psi)\wedge (\pi^{\prime})^{\ast}(\Psi^{\prm}))$ is an equiaffine radiant structure on $M \times M^{\prm}$.

Since the curvature of the product connection is the sum of the curvatures of the component connections, these constructions preserve the subclasses of the radiant, Euler, or equiaffine radiant structures for which the connection is flat or Ricci-flat.
\end{example}

\begin{example}\label{standardradiantexample}
Let $\ste$ be an $n$-dimensional real vector space.   
The affine structure determined by the vector space structure on $\ste$ determines a torsion-free affine connection $\nabla$ on $\ste$ such that if $\{e_{1}, \dots, e_{n}\}$ is a basis of $\ste$ and $x = \sum_{i = 1}^{n}x^{i}e_{i}$ are the corresponding coordinates on $\ste$, then $dx^{1}, \dots, dx^{n}$ is a $\nabla$-parallel coframe. Coordinates $x^{1}, \dots, x^{n}$ with this property are \emph{affine coordinates} on $\ste$. In this case the volume form $\Psi = dx^{1}\wedge \dots \wedge dx^{n}$ is $\nabla$-parallel and any $\nabla$-parallel volume form is a constant multiple of $\Psi$. In any choice of affine coordinates the vector field $\rad$ generated by the flow $\tau_{t}(p) = e^{t}p$ by dilations around the origin has the form $\rad = \sum_{i = 1}^{n}x^{i}\pzi$, so is Euler-like. Since $\lie_{\rad}\Psi = n\Psi$, $\rad$ is an Euler vector field. With $\Psi$ it constitutes the \emph{standard} Euler structure on $\ste$, and with $\nabla$ constitutes the \emph{standard} radiant structure on $\ste$. All these structures together constitute the \emph{standard} equiaffine radiant structure $(\ste, \nabla, \rad, \Psi)$. The automorphisms of the radiant structure $(\ste, \nabla, \rad)$ are $GL(\ste) = GL(n, \rea)$. The following stronger result is true, because a diffeomorphism of $\ste$ that preserves $\rad$ necessarily preserves $\nabla$.

\begin{lemma}
Let $\rad$ be the Euler vector field of real vector space $\ste$. If a $C^{2}$ diffeomorphism $\phi:\ste \to \ste$ satisfies $\phi^{\ast}(\rad) = \rad$, then $\phi \in GL(\ste)$.
\end{lemma}

\begin{proof}
Let $z = \sum_{i = 0}^{n}x^{i}e_{i}$ be the affine coordinates on $\ste$ corresponding with a basis $\{e_{1}, \dots, e_{n}\}$. Write $\phi(p) = \sum_{i = 1}^{n}\phi^{i}(p)e_{i}$. The component functions $\phi^{i}$ are in $C^{2}(\ste)$. That $\phi^{\ast}(\rad) = \rad$ means $T\phi(p)(\rad_{p}) = \rad_{\phi(p)}$ for all $p \in \ste$. In coordinates $\sum_{j = 1}^{n}x^{j}\tfrac{\pr \phi^{i}}{\pr x^{j}} = \phi^{i}$. Differentiating this shows $\sum_{k = 1}^{n}x^{k}\tfrac{\pr^{2} \phi^{i}}{\pr x^{j} \pr x^{k}} =  0$. If $f \in C^{1}(\ste)$ satisfies $df(\rad) = 0$, then $f$ is constant on rays emanating from the origin, so $f(x) = \lim_{t \to -\infty}f(e^{t}x) = f(0)$, showing $f$ is constant. Hence $\tfrac{\pr \phi^{i}}{\pr x^{j}}$ is constant for all $i$ and $j$. This shows $\phi \in GL(\ste)$.
\end{proof}
The automorphism group of the equiaffine radiant structure $(\ste, \nabla, \rad, \Psi)$ is the special linear subgroup $SL(\ste) = SL(n, \rea)$
\end{example}

\begin{lemma} 
On an $n$-dimensional vector space $\ste$, let $X$ be an Euler-like vector field having an isolated zero at $0 \in \ste$.
\begin{enumerate}
\item\label{eev1} $X = \rad + B$ where $B$ is a smooth vector field vanishing to second order at $0$.
\item\label{eev2} If $X$ is also Euler with respect to the volume form of the standard equiaffine structure $(\nabla, \Psi, \rad)$ on $\ste$, then $X = \rad + B$, where $B$ is a smooth divergence-free vector field vanishing to second order at $0$.
\item\label{eev3} There is a diffeomorphism $\phi$ defined on an open neighborhood $U$ of $0 \in \ste$ such that $\phi^{\ast}(X) = \rad$ on $U$. 
\end{enumerate}
\end{lemma}
\begin{proof}
Let $x^{1}, \dots, x^{n}$ be coordinates  such that $dx^{1}, \dots, dx^{n}$ is a $\nabla$-parallel coframe and $\Psi = dx^{1}\wedge \dots \wedge dx^{n}$. In such coordinates, $\rad = x^{i}\tfrac{\pr}{\pr x^{i}}$. That $X$ have an isolated zero at $0$ means there are constants $a_{j}\,^{i}$ and functions $b^{i}(x)$ vanishing to second order at $0$ such that $X = a_{p}\,^{i}x^{p}\tfrac{\pr}{\pr x^{i}} + B$ where $B = b^{i}(x)\tfrac{\pr}{\pr x^{i}}$. Because $X$ is Euler-like, $dx^{i}(X) - x^{i} = a_{p}\,^{i}x^{p} - x^{i} + b^{i}(x)$ must vanish to second order at $x = 0$ for $1 \leq i \leq n$. This forces $a_{i}\,^{j} = \delta_{i}\,^{j}$, so that $X = \rad + B$. This shows \eqref{eev1}. Since $\lie_{X}\Psi - n\Psi = (a_{p}\,^{p} + \tfrac{\pr b^{p}}{\pr x^{p}} - n)\Psi = \lie_{B}\Psi$, that $X$ be also Euler forces that $B$ be divergence free. This shows \eqref{eev2}. 

The remainder of the argument follows closely the proof of Lemma $2.4$ in \cite{Bursztyn-Lima-Meinrenken}.
Let $\tau_{t}(p) = e^{t}p$ be the flow generated by $\rad$. Define a time-dependent vector field $W:(0, \infty) \times \ste \to T\ste$ by $W(t, p) = t^{-1}\tau_{\log{t}}^{\ast}(B)_{p}$ for $t > 0$. In affine coordinates, $W(t, p)  = \sum_{i = 1}^{n}t^{-2}b^{i}(tx)\tfrac{\pr}{\pr x^{i}}e_{i}$. Because the $b^{i}$ vanish to second order at the origin, $W(t, p)$ extends smoothly to $t = 0$.

Let $\phi_{t}$ be the flow of the time-dependent vector field $W$, equal to the identity when $t = 0$ and satisfying $\tfrac{d}{dt}\phi_{t}(p) = W(t, \phi_{t}(p))$. 
Since $W(t, 0) = 0$ for all $t$, the integral curve $\phi_{t}(0) = 0$ is defined for all $t$ and there is an open neighborhood $U$ of $0$ such that, for $p \in U$, the integral curve $\phi_{t}(p)$ is defined for all $t \in [0, 1]$. By definition of $W$, for $a \geq 0$, $\tau_{\log{a}}^{\ast}(W)(t, p) = aW(at, p)$ and it follows that $a\phi_{at}(p) = \phi_{t}(ap)$ when both sides make sense, with the consequence that $\tau_{\log{a}}(U) \subset U$ for $a \in (0, 1]$.

Consider the vector field $Y(t,p) = tW(t, p) + \rad(p)$. When $t = 0$, $Y = \rad$, and when $t = 1$, $Y = B + \rad = X$. 
In general if $P$ and $Q$ are vector fields and $\psi_{t}$ is the flow of $P$, there holds $\tfrac{d}{dt}\psi_{t}^{\ast}(Q) = \psi_{t}^{\ast}([P, Q])$. Using this observation and $\tau_{\log{t}}^{\ast}(\rad) = \rad$ there results that, on $U$,
\begin{align}
\begin{split}
\tfrac{d}{dt}\phi_{t}^{\ast}(tW + \rad) & = \tfrac{d}{dt}\phi_{t}^{\ast}(\tau_{\log{t}}^{\ast}(B) + \tau_{\log{t}}^{\ast}(\rad) )\\
& = \phi_{t}^{\ast}\left([W, \tau_{\log{t}}^{\ast}(B + \rad)] + t^{-1}\tau_{\log{t}}^{\ast}([\rad, B + \rad]) \right) = \phi_{t}^{\ast}\left([W, tW + \rad] + [\rad, W] \right) = 0.
\end{split}
\end{align}
Thus $\phi_{t}^{\ast}(tW + \rad)$ is constant in $t$. At $t = 0$ it equals $\rad$ and at $t = 1$ it equals $\phi_{1}^{\ast}(X)$, so  $\phi_{1}^{\ast}(X) = \rad$ on $U$.
\end{proof}

\begin{example}\label{tautologicalbundleexample}
The standard equiaffine radiant structure on $\ste$ restricts to the submanifold $\ste \setminus \{0\}$. Let $\proj(\ste)$ be the projectivization of $\ste$ and let $\pi:\ste\setminus\{0\} \to \proj(\ste)$ be the canonical projection, so that $\pi(u) = [u]$ is the  line generated by $u \in \ste \setminus\{0\}$.
The \emph{tautological line bundle} $\pi:\taut \to \proj(\ste)$ has total space $\taut = \{(L, v)\in \proj(\ste)\times \ste: v \in L\}$ and projection $\pi(L, v) = L$, so that the fiber $\taut_{L}$ over $L \in \proj(\ste)$ is $L$. The group $\reat$ acts on the right on $\ste\setminus\{0\}$ and this action is principal for the projection $\pi:\ste\setminus\{0\} \to \proj(\ste)$, so the latter can be viewed as a principal $\reat$-bundle. 
The map sending the element $[v, t]$ of the bundle $\ste \setminus\{0\} \times_{\si}\rea \to \proj(\ste)$ associated with $\pi:\ste\setminus\{0\} \to \proj(\ste)$ via the representation $\si(r) = r$ of $\reat$ to the element $[[v], tv] \in \taut$ is a line bundle isomorphism. 
A section $s$ of $\taut$ corresponds with a function $\tilde{s}$ on $\ste \setminus \{0\}$ equivariant with respect to $\si$ in the sense that $\tilde{s}(rv) = \si(r^{-1})\tilde{s}(L, v) = r^{-1}\tilde{s}(v)$ for $v \in \ste \setminus\{0\}$ and $r \in \reat$. That is $\tilde{s}$ has homogeneity $-1$. Consequently homogeneous degree $k$ polynomials on $\ste$ determine sections of the $k$-fold tensor power $(\taut^{\ast})^{k} = \taut^{-k}$ of the dual line bundle $\taut^{\ast}$. Because sections are smooth, the converse claim is false; in general there are smooth sections of $\taut^{-k}$ for which the corresponding homogeneity $k$ function on $\ste \setminus\{0\}$ is not polynomial. A simple example is $\tilde{s}(x, y) = \sqrt{x^{4} + y^{4}}$.

There holds $T\proj(\ste) \simeq \hom(\taut, \ste/\taut) \simeq \taut^{\ast}\tensor \ste/\taut$. Here notation is slightly abused, and $\ste$ is used to denote the trivial vector bundle $\proj(\ste) \times \ste$ as well as the vector space $\ste$. Dualizing and taking determinants yields $\Det T^{\ast}\proj(\ste) \simeq \taut^{\tensor n}\tensor \Det(\ste/\taut)^{\ast}$. On the other hand, $\ste \simeq \taut \oplus \ste/\taut$ and taking determinants yields $\Det\ste \simeq \taut \tensor \Det(\taut/\ste)$.  Combining the preceding shows $\Det T^{\ast}\proj(\ste)  \simeq \Det T^{\ast}\proj(\ste) \tensor \Det \std \simeq \taut^{n+1}$, where $n = \dim \ste$ and the first isomorphism results because $\Det \std$ is a trivial line bundle. However, the trivialization is not canonical; fixing the standard volume form on $\ste$ determines an isomorphism $\taut \simeq  \Det(\ste/\taut)^{\ast}$, that fixes a particular isomorphism of $\Det T^{\ast}\proj(\ste)$ with $\taut^{\tensor n+1}$.

Since sections of $\Det T^{\ast}\proj(\ste)$ have weight $1$ in the sense that they correspond to functions on the complement of the zero section, $\Det T^{\ast}\proj(\ste) \setminus \zs(\proj(\ste))$, homogeneous of degree $1$, sections of $\taut$ should be considered to have weight $1/(n+1)$.  
\end{example}

\begin{example}\label{volumificationexample}
There is an Euler manifold $(\vdens \to M, \Psi^{\vdens}, \rad)$ canonically associated with an $n$-dimensional smooth manifold $M$. Its definition is motivated by Example \ref{tautologicalbundleexample}.
Let $\vdens = \Det \ctm \setminus \zs(M) \to M$ be the bundle of $n$-forms with the image of the zero section $\zs:M \to \Det \ctm$ deleted, regarded as a principal $\reat$-bundle with principal $\reat$-action $R_{r}$. Let $\veul$ be the vector field generating the fiber dilations $R_{e^{t}}(u) = e^{t}u$ on $\vdens$. The \emph{tautological $n$-form} $\mu \in \Ga(\ext^{n} (T^{\ast}\vdens))$, defined by 
\begin{align}\label{mudensdefined}
\mu_{s}(X_{1}, \dots, X_{n}) = s(T\rho(s)(X_{1}), \dots, T\rho(s)(X_{n}))
\end{align} 
for $X_{i} \in T_{s}\vdens$, has the following properties, which determine it uniquely: 
\begin{enumerate*}
\item $s^{\ast}(\mu) = s$ for any local section $s$ of $\vdens$;
\item $R_{r}^{\ast}(\mu) = r\mu$ for $r \in \reat$;
\item $i(\veul)\mu = 0$. 
\end{enumerate*}
For any principal $\reat$-connection $\be$ on $\vdens$ there holds $d\mu = \be \wedge \mu$ as follows from evaluating $i(\veul)d\mu = \lie_{\veul}\mu = \mu$ on the $\be$-horizontal lift of a local frame on $M$. As $\be \wedge \mu$ is evidently a volume form, this shows that $\Psi^{\vdens} = d\mu$ is a volume form on $\vdens$ satisfying $i(\veul)\Psi^{\vdens} = \mu$. 
Although $(\veul, \Psi)$ is not an Euler structure, because $\lie_{\veul}\Psi$ equals $\Psi$ rather than $(n+1)\Psi$, replacing $\veul$ with $\rad = (n+1)\veul$ yields the desired Euler structure $(\vdens \to M, \Psi^{\vdens}, \rad)$. The triple $(\vdens \to M, \Psi^{\vdens}, \rad)$ could be called the \emph{volumnification} of $M$ by analogy with the symplectification of a contact manifold as in Example \ref{symplectificationexample}.

Local coordinates $x^{1}, \dots, x^{n}$ on $U \subset M$ determine a trivialization $dx^{1}\wedge \dots \wedge dx^{n}$ of $\vdens$ over $U$ and a coordinate $t:U \to \reat$ on the fiber such that $\mu = t\rho^{\ast}(dx^{1})\wedge \dots \wedge \rho^{\ast}(dx^{n})$ and $\veul$ has the expression $t \pr_{t}$. The given trivialization is determined by a local section parallel with respect to the (locally defined) principal connection $\be = d\log t$. 
The vertical vector field $\rad = (n+1)\veul$ corresponds to the modified $\reat$ action $R_{\sign(r)|r|^{n+1}}$ on $\vdens$ (in the sense that its flow is $R_{e^{(n+1)t}}$), for which the given trivialization corresponds to the principal connection $\tfrac{1}{n+1}d\log t$ corresponding with the fiber coordinate $\sign(t)|t|^{1/(n+1)}$.
\end{example}

\begin{example}\label{symplectificationexample}
The notion of Euler structure is partly modeled on the notion of a Liouville structure on a symplectic manifold \cite{Cieliebak-Eliashberg, Eliashberg-Gromov, Eliashberg-Kim-Polterovich}.
A vector field $X$ on a symplectic manifold $(M, \Om)$ is \emph{Liouville} if $\lie_{X}\Om = \Om$. A \emph{Liouville structure} on a symplectic manifold $(M, \Om)$ is a Liouville vector field, $X$, usually also required to be complete and sometimes also required to satisfy some additional convexity condition.

From the point of view taken here, in the definition of a Liouville vector field it would be better to require $\lie_{X}\Om = 2\Om$, the constant $2$ playing the role of $n = \dim M$ in the definition of an Euler vector field. If the definition of a Liouville vector field is modified to be that $\lie_{X}\Om = 2\Om$, then every Liouville vector field gives rise to an Euler vector field. If the symplectic manifold has dimension $2n$, let $\Psi = \Om^{n}$. Since $\imt(X)\Om^{n} = n\imt(X)\Om \wedge \Om^{n-1}$, $\lie_{X}\Psi = d\imt(X)\Om^{n} = n\lie_{X}\Om \wedge \Om^{n-1} = 2n\Psi$, so $(X, \Om^{n})$ is an Euler pair.

For an $n$-manifold $N$, endow $M = \ctn$ with the the vector field $\rad$ equal to the generator of the dilations by $e^{2t}$ in the fibers and $\Psi = \Omega^{n}$, in which $\Omega$ is the canonical symplectic form. Then $(\rad, \Psi)$ is an Euler structure on $M$. This is a special case of the following observation that the symplectification of a contact manifold carries an Euler structure.
Let $C$ be a contact distribution on a $(2n-1)$-dimensional manifold $M$ and let $\ann C \subset \ctm$ be the line subbundle local sections of which annihilate $C$. On the total space of the complement $\rho:\F = \ann C \setminus \zs(M) \to M$ of the zero section $\zs(M)$ of $\ctm$ there is defined a tautological one-form $\theta$ by $\theta_{u}(X) = u(T\rho(u)(X))$ for $u \in \F$ and $X \in T_{u}\F$, and $\Omega = d\theta$ is a symplectic structure on $M$. Let $\rad$ generate dilations by $e^{2t}$ in the fibers of $\F$ and let $\Psi = \Omega^{n}$. Then $(\F, \rad, \Psi)$ is an Euler manifold.
\end{example}

Many examples of Euler structures arise on the total spaces of line bundles. Before giving the examples, some background related to line bundles is recalled.

For a real Lie group, $G$, the right action of $g \in G$ on the principal $G$-bundle $\rho:\F \to M$ is written $R_{g}(u) = u \cdot g$ for $u \in \F$. The vector bundle $V= \F \times_{\rho}\ste$ associated with the representation $\si:G \to \eno(\ste)$ is the quotient of $\F \times \ste$ by the equivalence relation $(u, v) \sim (u\cdot g, \rho(g^{-1})v)$. The image in $\F \times_{\rho}\ste$ of $(u, v) \in \frameb \times \ste$ is written $[u, v]$. 
The \emph{frame bundle} $\pi:\frameb[E] \to M$ of a vector bundle $E \to M$ with fiber $\ste$ is the principal $GL(\ste)$-bundle such that the fiber $\pi^{-1}(p)$ over $p \in M$ comprises invertible linear maps $u:\ste \to E_{p}$ and $GL(\ste)$ acts on the right by precomposition. 
By construction, $V = \F \times_{\rho}\ste\to M$ is recovered from $\frameb[V]$ as the bundle associated with the standard representation of $GL(\ste)$. 
Sections $s \in \Ga(M, V)$ are in canonical bijection with functions $\tilde{s}:\F \to \ste$ which are $G$-equivariant in the sense that $\tilde{s}(R_{g}(u)) = \rho(g^{-1})\tilde{s}(u)$. The equivariant function $\tilde{s} \in \cinf(\rho^{-1}(U))$ corresponding with the local section $s \in \Ga(U, \frameb \times_{\rho}\ste)$ is defined by $s(p) = [u, \tilde{s}(u)]$ for $p \in M$ and any $u \in \pi^{-1}(p)$. 
If $G$ acts on the right on a space by $u \to u \cdot g$, then $G$ acts on the right on a space of functions $s$ on the space by $(g \cdot s)(u) = s(u \cdot g^{-1})$. Thus the equivariance says that $g\cdot \tilde{s} = \rho(g)\tilde{s}$.

The integer power $\lmf^{k}$ of a real line bundle $\lmf \to M$ means the $k$-fold tensor product of $\lmf$ with itself if $k$ is positive, and the $k$-fold tensor product with itself of the dual line bundle $\lmf^{-1} = \lmf^{\ast}$ if $k$ is negative (when $k = 0$ it means the trivial line bundle).
The map $[u, t] \to tu(1)$ establishes a vector bundle isomorphism $\frameb[\lmf] \times_{\si}\rea \to \lmf$ where $\si(r) = r$ is the standard representation of $\reat$. Conversely, the complement $\lmf \setminus \zsec(M)$ of the zero section is identified with $\frameb[\lmf]$ via the map sending $l \in \lmf_{p}\setminus\{0\}$ to $\hat{l}:\rea \to \lmf_{p}$ defined by $\hat{l}(t) = tl$. Note that this means that sections $u$ of $\lmf^{k}$ correspond with functions $\tilde{u}$ on $\frameb[\lmf]$ of homogeneity $-k$. 

The $O(1) = \zmodtwo$ reduction of $\frameb[\lmf]$ is a $\zmodtwo$-principal bundle $\orientb[\lmf] \to M$. It is a double cover of $M$ that is connected if and only if $\lmf$ is orientable. The line bundle $\lmf$ is recovered as the associated bundle $\orientb[\lmf] \times_{\si}\rea$ where $\si$ is the standard representation of $\zmodtwo$ on $\rea$.
By the preceding, the spaces of isomorphism classes of real line bundles on $M$, isomorphism classes of principal $\reat$-bundles on $M$, and isomorphism classes of $\zmodtwo$-principal bundles on $M$ are in pairwise bijection in canonical ways. The transition functions of a real line bundle $\lmf \to M$ over a good cover determine a closed Cech $1$-cocyle so an element of the Cech cohomology $H^{1}(M; \reat)$ of the locally constant sheaf $\reat$. The short exact sequence of locally constant sheaves $0 \to \rea \stackrel{t \to e^{t}}{\to} \reat \to \zmodtwo \to 0$ determines a long exact sequence in cohomology and because $\reat$ is a fine sheaf, there results $H^{1}(M; \reat) \simeq H^{1}(M; \zmodtwo)$. The image in $H^{1}(M; \zmodtwo)$ of the Cech cocycle associated with $\lmf$ is the first Stiefel-Whitney characteristic class $w_{1}(\lmf) \in H^{1}(M; \zmodtwo)$. The usual classifying theory for bundles establishes that any class in $H^{1}(M; \rea)$ is the first Stiefel-Whitney class of some real line bundle and that two real line bundles are isomorphic if and only if their first Stiefel-Whitney classes are equal. There are written also $w_{1}(\lmf) = w_{1}(\frameb[\lmf]) = w_{1}(\orientb[\lmf])$. The bundle $\lmf$ is orientable if and only if $w_{1}(\lmf) = 0$.

Examples \ref{tautologicalbundleexample} and \ref{volumificationexample} motivate Definition \ref{tautdefinition}. These notions are are used in Sections \ref{thomassection} and \ref{codazzisection}. 

\begin{definition}\label{tautdefinition}
Let $M$ be an $n$-manifold.
\begin{enumerate}
\item A \emph{hyperplane line bundle} on $M$ is a real line bundle $\emf \to M$ equipped with a line bundle isomorphism $\emf^{n+1} \simeq \Det TM$.
\item A \emph{pseudo-hyperplane line bundle} on $M$ is a real line bundle $\emf \to M$ equipped with a line bundle isomorphism $\emf^{2n+2} \simeq (\Det TM)^{\tensor 2}$.
\end{enumerate}
The line bundle dual to a (pseudo)-hyperplane line bundle is called a \emph{(pseudo)-tautological line bundle}.
\end{definition}
Although in either case it forms part of the structure, the isomorphism is not indicated explicitly in the notation. By definition, a hyperplane line bundle is a pseudo-hyperplane line bundle. If $M$ admits an orientable hyperplane line bundle, then $M$ is orientable, and, similarly, if $n$ is odd the existence of a hyperplane bundle forces $M$ to be orientable, for $w_{1}(\ctm) = (n+1)w_{1}(\emf) = 0$. The notion of pseudo-hyperplane bundle is convenient because it accommodates the case of a nonorientable odd-dimensional manifold.

\begin{example}
Example \ref{tautologicalbundleexample} shows that the hyperplane line bundle $\taut^{\ast} \to \proj(\ste)$ and tautological line bundle $\taut \to \proj(\ste)$ are a hyperplane and tautological line bundle in the sense of Definition \ref{tautdefinition}.
\end{example}

A pseudo-hyperplane line bundle is determined by two requirements. Its tensor square is trivial as real line bundle and sections of its $(n+1)$st power transform locally like sections of $\Det TM$. 

\begin{lemma}
Let $M$ be a $n$-manifold and let $\omega \in H^{1}(M; \zmodtwo)$. Up to isomorphism of real line bundles there is a unique pseudo-hyperplane line bundle $\emf \to M$ such that $w_{1}(\emf) = \omega$. In particular, the orientable pseudo-hyperplane line bundle $|\Det TM|^{1/(n+1)}$ represents the isomorphism class of real line bundles having vanishing first Stiefel-Whitney class. 
\end{lemma}
\begin{proof}
Given $\omega \in H^{1}(M;\zmodtwo)$ there is a $\zmodtwo$-principal bundle $\P \to M$ such that $w_{1}(\P) = \omega$. Define the total space of $\emf$ to be the quotient of $\P \times |\Det TM|^{1/(n+1)}$ by the action of $\zmodtwo$ given by $-(p, \mu) = (-p, -\mu)$ for $(p, \mu) $ in a fiber of $\P \times |\Det TM|^{1/(n+1)}$. Alternatively, let $\lmf = \P \times_{\si}\rea$ where $\si$ is the standard representation of $\zmodtwo$ on $\rea$, and define $\emf = \lmf \tensor |\Det TM|^{1/(n+1)}$. This second description makes apparent that $\emf^{2(n+1)} \simeq (\Det TM)^{2}$. Since $|\Det TM|^{1/(n+1)}$ is orientable, its first Stiefel-Whitney class is trivial, so $w_{1}(\emf) = w_{1}(\lmf) = \omega$, and the uniqueness follows from the fact that the first Stiefel-Whitney class determines a real line bundle up to line bundle isomorphism.
\end{proof}

\begin{example}
The notions of (pseudo-)tautological bundle of Definition \ref{tautdefinition} extend the notion of tautological bundle on a flat real projective manifold introduced by Loftin in \cite{Loftin-affinespheres, Loftin-riemannian, Loftin-affinekahler}. 
Although a manifold necessarily admits an orientable pseudo-hyperplane line bundle, in general it need not admit an orientable hyperplane line bundle. However, \cite[Proposition $2.2.1$]{Loftin-affinespheres} shows that a manifold equipped with a properly convex flat real projective structure admits an orientable hyperplane line bundle.
\end{example}

\begin{example}
Viewed as the real projective line, $S^{1}$ admits two pseudo-tautological line bundles, the topologically trivial bundle $\lmf_{1}$ of $1/2$-densities and the usual tautological line bundle $\lmf_{2}$. View the $n$-dimensional torus as the $n$fold product $\torus_{n} = S^{1} \times \dots \times S^{1}$. By the Künneth formula, any pseudo-tautological bundle on $\torus_{n}$ is the tensor product of the bundle of $1/(n+1)$-densities $|\Det T^{\ast}\torus_{n}|^{1/(n+1)}$ with a direct sum of pullbacks $\oplus_{i = 1}^{n}p_{i}^{\ast}(\lmf_{\ep_{i}})$, $\ep_{i} \in \{1, 2\}$, where $p_{i}:\torus_{n} \to S^{1}$ is the projection onto the $i$th factor. 
\end{example}

A \emph{volume density} is a nonvanishing $1$-density. A \emph{pseudo-Euler structure} on the $n$-manifold $M$ is a pair $(\rad, \Psi)$ such that $\Psi$ is a volume density and $\rad$ is an Euler-like vector field satisfying $\lie_{\rad}\Psi = n \Psi$. What distinguishes a pseudo-hyperplane line bundle $\emf$ among arbitrary real line bundles is that the total space of $\F = \frameb[\emf^{-1}]$ is equipped in a canonical way with a volume density compatible with the fundamental vector field generating the principal action.

\begin{lemma}\label{canonicaleulerlemma}
For a pseudo-hyperplane bundle $\emf \to M$, let $\rho: \F = \frameb[\emf^{-1}] \to M$ be the frame bundle of $\emf^{-1}$ and let $\eul^{\F}$ be the fundamental vector field generating dilations by $e^{t}$ in the fibers of $\F$.
\begin{enumerate}
\item If $\emf \to M$ is a hyperplane line bundle, then there is a canonically determined volume form $\Psi^{\F}$ on the total space $\F$ such that $(\F, \eul^{\F}, \Psi^{\F})$ is an Euler manifold satisfying $R_{r}^{\ast}\Psi^{\F}= r^{n+1}\Psi^{\F}$ for $r \in \reat$.
\item If $\emf \to M$ is a pseudo-hyperplane line bundle, then there is a canonically determined volume density $|\Psi^{\F}|$ on the total space $\F$ such that $(\F, \eul^{\F}, |\Psi^{\F}|)$ is a pseudo-Euler manifold satisfying $R_{r}^{\ast}|\Psi^{\F}|= |r|^{n+1}|\Psi^{\F}|$ for $r \in \reat$.
\end{enumerate}
\end{lemma}

\begin{proof}
Suppose $\emf$ is a hyperplane line bundle.
The frame bundle $\F$ can be viewed as a $(n+1)$st root of the principal $\reat$-bundle $\vdens= \frameb[\Det \ctm] \simeq \Det \ctm \setminus \zsec(M)$ discussed in Example \ref{volumificationexample} and so acquires a structure of an Euler manifold as follows. The isomorphism $\emf^{-n-1}\simeq \Det \ctm$ induces a principal $\reat$-bundle morphism $\Q:\F \to \vdens$ that satisfies $\Q(ru) = r^{n+1}\Q(u)$ for all $r \in \reat$ and is either a diffeomorphism onto its image or a local diffeomorphism that is $2-1$ onto its image, as $n$ is even or odd. Because $T\Q(\eul^{\F}) = (n+1)\eul^{\vdens}$, the pullback $\mu^{\F} = \Q^{\ast}(\mu)$ of the tautological $n$-form $\mu$ on $\vdens$ satisfies $\imt(\eul^{\F})d\mu^{\F} = (n+1)\mu^{\F}$, so $\lie_{\eul^{\F}}\Psi^{\F} = (n+1)\Psi^{\F}$ where $\Psi^{\F} = d\mu^{\F} = \Q^{\ast}(\Psi^{\vdens})$. This shows that $(\F, \eul^{\F}, \Psi^{\F})$ is an Euler manifold. 
Explicitly, if $u \in \F$ and $X_{1}, \dots, X_{n} \in T_{u}\F$, 
\begin{align}\label{psifdefined}
\Psi^{\F}_{u}(\rad, X_{1}, \dots, X_{n}) = \lb u^{n+1}, T\rho(u)(X_{1})\wedge \dots \wedge T\rho(u)(X_{n})\ra,
\end{align}
where $\lb\dum, \dum\ra$ indicates the pairing of the volume form $u^{n+1} \in \Det T_{\rho(u)}^{\ast}M$ with the $n$-vector $T\rho(u)(X_{1})\wedge \dots \wedge T\rho(u)(X_{n})\in \Det T_{\rho(u)}M$ .
If $\emf$ is only a pseudo-hyperplane line bundle, then the preceding discussion applies on an open neighborhood of $M$ over which $\emf$ and $\F$ are trivial with the proviso that the pulled-back forms $\mu^{\F}$ and $d\mu^{\F}$ are determined only up to sign. Taking the absolute value of $d\mu^{\F}$ yields the volume density $|\Psi^{\F}|$. Concretely, in this case the right-hand side of \eqref{psifdefined} is defined up to a sign, and taking its absolute value defines $|\Psi^{\F}|_{u}(\rad, X_{1}, \dots, X_{n})$.
That $\lie_{\eul^{\F}}|\Psi^{\F}| = (n+1)|\Psi^{\F}|$ follows from the hyperplane line bundle case because this identity is a purely local claim. 
That $R_{r}^{\ast}\Psi^{\F}= r^{n+1}|\Psi^{\F}|$ and $R_{r}^{\ast}|\Psi^{\F}|= |r|^{n+1}|\Psi^{\F}|$ for $r \in \reat$ are immediate from \eqref{psifdefined}.
\end{proof}

\section{Conelike radiant structures}\label{conelikesection}
Let $(\nabla, \rad)$ be a radiant structure on $M$. (The reader is reminded that by the definition of a radiant structure $(\nabla, \rad)$, its underlying connection $\nabla$ is torsion-free.) A smoothly immersed surface $\Sigma$ in $M$ is \emph{planelike} if it is totally geodesic and $\rad_{p} \in T_{p}\Sigma$ for all $p \in \Sigma$. A planelike surface $\Sigma$ is a \emph{plane} if, moreover, the connection induced on $\Sigma$ by $\nabla$ is flat. A radiant structure is \emph{conelike} if it \emph{admits a complete set of planelike surfaces} in the sense that for every $p \in M$ and every two-dimensional subspace $L \subset T_{p}M$ containing $\rad_{p}$ there is a smoothly immersed planelike surface $\Sigma \subset M$ containing $p$ and such that $T_{p}\Sigma = L$. 

A planelike surface inherits a radiant structure, so, by Theorem \ref{2dtorustheorem}, a compact immersed planelike surface is topologically a torus or a Klein bottle.

\begin{lemma}\label{conetensorlemma}
Let $\Pi_{ij}\,^{k} = \Pi_{(ij)}\,^{k}$ and $X^{i}$ be given and suppose that $X^{p}\Pi_{ip}\,^{j} = 0$. Let $S_{ij}\,^{kl} = \Pi_{ij}\,^{[k}X^{l]}$. If $S_{(ij}\,^{[ab}\delta_{k)}\,^{c]} = 0$, then there exists a smooth tensor $Q_{ij} = Q_{(ij)}$ on $M_{\ast} = \{p \in M: X_{p} \neq 0\}$ such that
\begin{align}\label{conetensor}
\Pi_{ij}\,^{k} = Q_{ij}X^{k} - 2\delta_{(i}\,^{k}Q_{j)p}X^{p}= X^{p}\left(Q_{ij}\delta_{p}\,^{k} - 2Q_{p(i}\delta_{j)}\,^{k}\right)
\end{align} 
on $M_{\ast}$. Moreover, $\Pi_{ip}\,^{p} = -nQ_{ip}X^{p}$ and $X^{p}X^{q}Q_{pq} = 0$.
\end{lemma}
\begin{proof}
Tracing $S_{(ij}\,^{[ab}\delta_{k)}\,^{c]} = 0$ in $c$ and $k$ yields $nS_{ij}\,^{kl} = -4\delta_{(i}\,^{[k}S_{j)p}\,^{l]p}$. Substituting into this the definition of $\Pi_{ij}\,^{k}$ yields $n\Pi_{ij}\,^{[k}X^{l]}  - 2\delta_{(i}\,^{[k}X^{l]}\Pi_{j)p}\,^{p} = -2 \delta_{(i}\,^{[k}\Pi_{j)p}\,^{l]}X^{p} = 0$, the last equality because $X^{p}\Pi_{ip}\,^{j} = 0$. This shows that for any vector field $Y^{i}$ there vanishes the wedge product with $X$ of the vector field $Y^{p}Y^{q}(n\Pi_{pq}\,^{i} - 2\delta_{(p}\,^{i}\Pi_{q)a}\,^{a})$. Hence there is $Q_{ij} \in \Ga(S^{2}T^{\ast} M_{\ast})$ such that $n\Pi_{ij}\,^{k} - 2\delta_{(i}\,^{k}\Pi_{j)p}\,^{p} = nQ_{ij}X^{k}$ on $M_{\ast}$. Tracing this shows that $\Pi_{ip}\,^{p} = -nQ_{ip}X^{p}$, and \eqref{conetensor} and $X^{p}X^{q}Q_{pq} = 0$ follow.
\end{proof}

A radiant structure is \emph{nonsingular} if $\rad$ is nonsingular. By Lemma \ref{radiantisolatedlemma} this is automatic if $M$ is compact and odd-dimensional. Because, by Lemma \ref{radiantisolatedlemma}, the zeros of a radiant vector field are isolated, results about nonsingular radiant structures apply to arbitrary radiant structures over the open submanifold $M_{\ast} = \{p \in M: \rad_{p} \neq 0\}$.

\begin{lemma}\label{coneconditionlemma}
A nonsingular radiant structure $(\nabla, \rad)$ on an $n$-manifold is conelike if and only if there is $Q_{ij} \in \Ga(S^{2}\ctm)$ such that $n\rad^{p}Q_{ip} = \er_{i}$ and $\rad^{p}\rad^{q}Q_{pq} = 0$ and satisfying the equivalent identities
\begin{align}
\label{cc2} &(\lie_{\rad}\nabla)_{ij}\,^{k} = \rad^{p}R_{pij}\,^{k} = Q_{ij}\rad^{k} - 2\rad^{p}Q_{p(i}\delta_{j)}\,^{k}.
\end{align}
In this case,
\begin{align}
\label{conelieric} &(n-2)Q_{ij} + (\lie_{\rad}Q)_{ij} + \tfrac{n-2}{n}\nabla_{(i}\er_{j)} = (\lie_{\rad}\ric)_{(ij)}, & &d\er_{ij}  = 2(\lie_{\rad}\ric)_{[ij]}.
\end{align}
If, moreover, $\nabla$ has symmetric Ricci tensor then $\rad^{p}Q_{ip} = 0$ and $\rad^{p}R_{pij}\,^{k} = Q_{ij}\rad^{k}$.
\end{lemma}

\begin{proof}
Suppose $(\nabla, \rad)$ is conelike. Let $p \in M$, let $L \subset T_{p}M$ contain $\rad_{p}$, and let $\Sigma$ be a smoothly immersed planelike surface containing $p$ and tangent to $L$. 
Choose an open neighborhood $U \subset M$ of $p$ and a vector field $A$ such that over $\Sigma \cap U$ the vector fields $\rad$ and $A$ span $T(\Sigma \cap U)$. Since $\Sigma$ is totally geodesic $\nabla_{A}A$ and $\nabla_{\rad}A$ are tangent to $\Sigma\cap U$, from which it follows that $\rad \wedge A\wedge R(\rad, A)A = 0$ on $U\cap \Sigma$. Since $p$ and $L$ are arbitrary (except that $L$ must contain $\rad_{p}$) this shows that $\rad \wedge A \wedge R(\rad, A)A = 0$ on $M$ for all $A \in \Ga(TM)$.

By \eqref{radid} and the algebraic Bianchi identity, the tensor $\Pi_{ij}\,^{k} = \rad^{p}R_{pij}\,^{k}$ satisfies $\rad^{p}\Pi_{ip}\,^{j} = 0$ and $\Pi_{[ij]}\,^{k} = 0$. The preceding paragraph shows that the tensor $S_{ij}\,^{kl} = \Pi_{ij}\,^{[k}\rad^{l]} = \rad^{p}R_{pij}\,^{[k}\rad^{l]}$ satisfies $S_{(ij}\,^{[ab}\delta_{k)}\,^{c]} = 0$. By Lemma \ref{conetensorlemma} there is a tensor $Q_{ij}$ satisfying \eqref{cc2}. Tracing \eqref{cc2} in $i$ and $k$ shows $n\rad^{p}Q_{pj} = \rad^{p}R_{pj} = \er_{j}$. 

Suppose $(\nabla, \rad)$ is a nonsingular radiant structure and there is a tensor $Q_{ij}$ as in the statement of the lemma. Fix $p \in M$ and $\bar{X} \in T_{p}M$. Let $\si(t)$ be the maximal integral curve of $\rad$ such that $\si(0) = p$, and let $X(t)$ be a parallel vector field along $\si$ such that $X(0)= \bar{X}$. By \eqref{cc2}, 
\begin{align}\label{nrxs}
\begin{split}
nR(X, \dot{\si})\dot{\si} &= - n R(\rad, X)\rad =n(Q(X, \rad)\rad - Q(\rad, X)\rad - Q(\rad, \rad)x) = 0. 
\end{split}
\end{align}
By Lemma \ref{maxgeodesiclemma} the maximal integral curve $C$ of $\rad$ passing through $p$ is contained in the image of the maximal geodesic of $\nabla$ passing through $p$, which is parametrized by $\si$ since $\dot{\si}\wedge \nabla_{d/dt}\dot{\si} = 0$. Together with \eqref{nrxs} and \eqref{projjacobi} this shows that $X$ is a projective Jacobi field along $C$. 
It follows that $X$ generates a deformation of the image of $\si$ through projective geodesics, and that the submanifold swept out by these geodesics is a totally geodesic submanifold passing through $p$ and tangent to $\spn\{\rad_{p}, X\}$ at $p$. Since this argument can be applied at any $p \in M$, it shows $(\nabla, \rad)$ is conelike.

Suppose there is $Q_{ij}$ as in the statement of the lemma and satisfying \eqref{cc2}. Rewriting \eqref{cc2} yields
\begin{align}
\label{cc1} &(\lie_{\rad}\nabla)_{ij}\,^{k} - \tfrac{2}{n+1}\delta_{(i}\,^{k}(\lie_{\rad}\nabla)_{j)p}\,^{p} = \rad^{p}R_{pij}\,^{k} + \tfrac{2}{n}\er_{(i}\delta_{j)}\,^{k} = Q_{ij}\rad^{k}. 
\end{align}
Differentiating the second equality of \eqref{cc1} yields
\begin{align}\label{cci1}
\begin{split}
R_{ijk}\,^{l} + \rad^{p}\nabla_{i}R_{pjk}\,^{l} + \tfrac{2}{n}\nabla_{i}\er_{(j}\delta_{k)}\,^{l} = nQ_{jk}\delta_{i}\,^{l} + \rad^{l}\nabla_{i}Q_{jk}.
\end{split}
\end{align}
The differential Bianchi identity yields 
\begin{align}\label{cci2}
\rad^{p}\nabla_{q}R_{pjk}\,^{q} = \rad^{p}\nabla_{p}R_{jk} - \rad^{p}\nabla_{j}R_{pk} = (\lie_{\rad}\ric)_{jk} - R_{jk} - \nabla_{j}\er_{k}.
\end{align}
Tracing \eqref{cci1} in $i$ and $l$ and substituting \eqref{cci2} in the result yields 
\begin{align}\label{conelieric0}
\begin{split}
(\lie_{\rad}\ric)_{ij} = (n-2)Q_{ij} + (\lie_{\rad}Q)_{ij} + \tfrac{n-1}{n}\nabla_{i}\er_{j} + \tfrac{1}{n}\nabla_{j}\er_{i}.
\end{split}
\end{align}
Decomposing \eqref{conelieric0} by symmetries yields \eqref{conelieric}. If $\nabla$ is Ricci symmetric, then $\er_{j} = 0$, and in \eqref{cc1} this yields the final claim.
\end{proof}

\begin{corollary}\label{invariantconelikecorollary}
A nonsingular radiant structure $(\nabla, \rad)$ satisfying $\lie_{\rad}\nabla = 0$ is conelike with $\er = 0$ and admits a complete set of planes.
\end{corollary}
\begin{proof}
That $(\nabla, \rad)$ is conelike with $\er =0$ is immediate from Lemma \ref{coneconditionlemma}. Because $\rad^{p}R_{pij}\,^{k} = 0$, the connection induced on a planelike surface $\Sigma$ is flat, so the given radiant structure admits a complete set of planes.
\end{proof}

Two conelike radiant structures with the same radiant vector field \emph{have the same planelike surfaces} if each planelike surface of one radiant structure is a planelike surface of the other radiant structure.

\begin{lemma}\label{conesameplaneslemma}
Conelike nonsingular radiant structures $(M, \nabla, \rad)$ and $(M, \tnabla, \rad)$ have the same planelike surfaces if and only if their difference tensor $\Pi = \tnabla - \nabla$ has the form $\Pi_{ij}\,^{k} = Q_{ij}\rad^{k} -2\rad^{p}Q_{p(i}\delta_{j)}\,^{k}$ with $Q_{[ij]} = 0$ and $\rad^{p}\rad^{q}Q_{pq} = 0$.
In this case, if $\tnabla$ and $\nabla$ induce the same connection on some, and hence any, bundle of densities of nontrivial weight, then $\rad^{p}Q_{ip} = 0$, so that $\Pi_{ij}\,^{k} = Q_{ij}\rad^{k}$.
\end{lemma}
\begin{proof}
Fix $p \in M$ and $L \in T_{p}M$ containing $\rad_{p}$, and let $\Sigma$ be a smoothly immersed $\nabla$-planelike surface passing through $p$ and tangent to $L$. Let $A$ be a vector field tangent to $\Sigma$ near $p$. Then $\tnabla_{A}A = \nabla_{A}A + Q(A, A)\rad + 2Q(\rad, A)A$ is tangent to $\Sigma$ because $\nabla_{A}A$ is tangent to $\Sigma$. Since $\tnabla_{\rad}A = \nabla_{\rad}A$ and $\tnabla \rad = \nabla \rad$ this suffices to show that $\Sigma$ is $\tnabla$-totally geodesic, and it follows that $\tnabla$ and $\nabla$ have the same planelike surfaces.

Suppose the given radiant structures have the same planelike surfaces. The difference tensor $\Pi_{ij}\,^{k} = \tnabla - \nabla$ satisfies $\Pi_{[ij]}\,^{k} = 0$, because both connections are torsion-free, and $\rad^{p}\Pi_{ip}\,^{j} = \tnabla_{i}\rad^{j} - \nabla_{i}\rad^{j} = 0$. Let $\ga(t)$ be a $\nabla$-geodesic such that $\ga(0) \in M$ and $\dot{\ga}(0) \neq \rad_{\ga(0)}$, let $L = \spn\{\rad_{\ga(0)}, \dot{\ga}(0)\}$, and let $\Sigma$ be a planelike surface through $\ga(0)$ tangent to $L$. By assumption the $\tnabla$-geodesic $\tilde{\ga}$ such that $\tilde{\ga}(0) = \ga(0)$ and $\dot{\tilde{\ga}}(0) = \dot{\ga}(0)$ lies on $\Sigma$. Hence $0 = \tnabla_{d/dt}\dot{\tilde{\ga}} = \nabla_{d/dt}\dot{\tilde{\ga}} + \Pi(\dot{\tilde{\ga}}, \dot{\tilde{\ga}})$. At $t = 0$ this gives $\Pi(\dot{\ga}(0), \dot{\ga}(0)) \wedge \rad_{\ga(0)}\wedge \dot{\ga}(0) = 0$. Since $\ga$ is an arbitrary geodesic, this shows that at every point of $M$ there holds $\Pi(A, A)\wedge \rad \wedge A = 0$ for all vectors $A$ transverse to $\rad$. The preceding sentence shows $S_{ij}\,^{ab} = \Pi_{ij}\,^{[a}\rad^{b]}$ satisfies $S{(ij}\,^{[ab}\delta_{k)}\,^{c]} = 0$. By Lemma \ref{conetensorlemma} there is a tensor $Q_{ij}$ with the claimed properties. If $\tnabla$ and $\nabla$ induce the same connection on some bundle of densities of nontrivial weight, then $-n\rad^{p}Q_{ip} = \Pi_{ip}\,^{p} =0$, which shows the final claim.
\end{proof}

\begin{lemma}\label{conelikedifferencelemma}
Let $(\nabla, \rad)$ be a radiant structure on the $n$-manifold $M$. Suppose $Q_{ij} \in \Ga(S^{2}\ctm)$ is such that $\rad^{p}\rad^{q}Q_{pq} = 0$ and define $\tnabla = \nabla + \Pi$ where $\Pi_{ij}\,^{k} = Q_{ij}\rad^{k}  -2q_{(i}\delta_{j)}\,^{k}$ with $q_{i} = \rad^{p}Q_{pi}$. Then $(\tnabla, \rad)$ is a radiant structure and the curvature tensors of $\tnabla$ and $\nabla$ are related by
\begin{align}
\label{conecurvdiff}
\begin{split}
\tilde{R}_{ijk}\,^{l} & = R_{ijk}\,^{l} + \delta_{i}\,^{l}\left(Q_{jk} + \nabla_{j}q_{k} + q_{j}q_{k} \right) - \delta_{j}\,^{l}\left(Q_{ik} + \nabla_{i}q_{k} + q_{i}q_{k} \right)  \\
& \qquad + \rad^{l}\left(2\nabla_{[i}Q_{j]k} + 2q_{[i}Q_{j]k} \right) - dq_{ij}\delta_{k}\,^{l},
\end{split}\\
\label{trconecurvdiff}
&\tilde{R}_{ij}  = R_{ij} + (\lie_{\rad}Q)_{ij} + (n-2)\left(Q_{ij} + \nabla_{i}q_{j} + q_{i}q_{j}\right) + dq_{ij}, &&\tilde{R}_{[ij]} = R_{[ij]} + \tfrac{n}{2}dq_{ij},\\
\label{erdiff} \tilde{\er}_{i} & - \er_{i}= n\rad^{p}(\lie_{\rad}Q)_{pi}  =  n(\lie_{\rad}q)_{i}.
\end{align}
The connections $\tnabla$ and $\nabla$ induce the same connection on a line bundle of densities of nontrivial weight over $M$ if and only if $q_{i} = 0$. 
In this case $\tilde{\er}_{i} = \er_{i}$. 
\end{lemma}
\begin{proof}
Calculations using $\tilde{R}_{ijk}\,^{l} - R_{ijk}\,^{l} = 2\nabla_{[i}\Pi_{j]k}\,^{l} + 2 \Pi_{p[i}\,^{l}\Pi_{j]k}\,^{p}$ yield \eqref{conecurvdiff} and \eqref{trconecurvdiff}.
Since $q_{p}\rad^{p} = 0$, $\rad^{p}(\lie_{\rad}Q)_{pi} = (\lie_{\rad}q)_{i}  = \rad^{p}dq_{pi} = \rad^{p}\nabla_{p}q_{i} - \rad^{p}\nabla_{i}q_{p} = \rad^{p}\nabla_{p}q_{i} + q_{i}$. Combining this with \eqref{conecurvdiff} yields \eqref{erdiff}. Since $\Pi_{ip}\,^{p} = -nq_{i}$, the connections $\tnabla$ and $\nabla$ induce the same connection on any line bundle of densities of nontrivial weight, e.g. $\Det \ctm$, if and only if $q_{i} = 0$. 
\end{proof}

\begin{example}
Suppose the radiant structure $(\nabla, \rad)$ on an $n$-manifold satisfies $\er_{i} = 0$ and its Ricci tensor has positive homogeneity $\la \neq 2-n$, so that $(\lie_{\rad}\ric)_{ij} = \la R_{ij}$. By Lemma \ref{conelikedifferencelemma}, the connection $\bnabla = \nabla + \tfrac{1}{2 - n - \la}R_{(ij)}\rad^{k}$ forms with $\rad$ a radiant structure with Ricci tensor $\bar{R}_{ij}$ satisfying $\bar{R}_{(ij)} = 0$ and $\bar{R}_{[ij]} = R_{[ij]}$. In particular, if the Ricci tensor of $\nabla$ is symmetric and has positive homogeneity $\la \neq 2-n$, so that $(\lie_{\rad}\ric)_{ij} = \la R_{ij}$, then the connection $\bnabla = \nabla + \tfrac{1}{2 - n - \la}R_{ij}\rad^{k}$ forms with $\rad$ a Ricci-flat radiant structure.
\end{example}

By Lemma \ref{conesameplaneslemma}, the difference tensor of conelike nonsingular radiant structures $(\nabla, \rad)$ and $(\tnabla, \rad)$ having the same planelike surfaces satisfies the hypotheses of Lemma \ref{conelikedifferencelemma}.

\begin{theorem}\label{conenormalizationtheorem}
Let $M$ be a manifold of dimension $n > 2$. Let $(\nabla, \rad)$ be a conelike nonsingular radiant structure on $M$ such that $\er_{i} = 0$.
\begin{enumerate}
\item There is a unique $\rad$-invariant connection $\tnabla$ such that $(\tnabla, \rad)$ is a conelike nonsingular radiant structure having antisymmetric Ricci tensor, having the same planes as has $(\nabla, \rad)$, and inducing on $|\Det \ctm|$ the same connection as that induced by $\nabla$. 
\item\label{gcnn} If a Lie group $G$ acts on $M$ by automorphisms of $(\nabla, \rad)$, then $G$ acts by automorphisms of $(\tnabla, \rad)$.
\item If $(\nabla, \rad)$ has symmetric Ricci tensor, then there is a unique Ricci-flat, $\rad$-invariant connection $\tnabla$ such that $(\tnabla, \rad)$ is a conelike nonsingular radiant structure having the same planes as has $(\nabla, \rad)$ and inducing on $|\Det \ctm|$ the same connection as that induced by $\nabla$. 
\end{enumerate}
\end{theorem}

\begin{proof}
Fix a conelike nonsingular radiant structure $(\nabla, \rad)$ on $M$. By Lemma \ref{conesameplaneslemma}, the connection $\tnabla$ of the most general conelike radiant structure $(\tnabla, \rad)$ having the same planes as has $(\nabla, \rad)$ has the form $\tnabla = \nabla + \Pi_{ij}\,^{k}$ with $\Pi_{ij}\,^{k} =  Q_{ij}\rad^{k} - 2q_{(i}\delta_{j)}\,^{k}$ for $Q_{ij}$ and $q_{i}$ satisfying $Q_{[ij]} = 0$, $\rad^{j}Q_{ij} = q_{i}$, and $\rad^{p}q_{p} = 0$. By Lemma \ref{coneconditionlemma}, $\rad^{p}R_{pij}\,^{k} + \tfrac{2}{n}\er_{(i}\delta_{j)}\,^{k} = T_{ij}\rad^{k}$ for a tensor $T_{ij}$ satisfying $T_{[ij]} = 0$, $n\rad^{p}T_{pi} = \er_{i}$, and $\rad^{p}\rad^{q}T_{pq} = 0$. By \eqref{conelieric} of Lemma \ref{coneconditionlemma},
\begin{align}\label{cccd3}
\begin{split}
(\lie_{\rad}T)_{ij} & - (\lie_{\rad}\ric)_{ij} = (2-n)T_{ij} +  \tfrac{1-n}{n}\nabla_{i}\er_{j} + \tfrac{1}{n}\nabla_{j}\er_{j}.
\end{split}
\end{align}
From \eqref{lieradnabla} and \eqref{conecurvdiff} there follows
\begin{align}
\label{conecurvdiff3} 
\begin{split}
(\lie_{\rad}\tnabla)_{ij}\,^{k} & = \rad^{p}\tilde{R}_{pij}\,^{k}  = \rad^{p}R_{pij}\,^{k} + (\lie_{\rad}Q)_{ij}\rad^{k} - 2(\lie_{\rad}q)_{(i}\delta_{j)}\,^{k}\\
&= (T_{ij} + (\lie_{\rad}Q)_{ij})\rad^{k} - \delta_{i}\,^{k}\left((\lie_{\rad}q)_{j} + \tfrac{1}{n}\er_{j}\right)- \delta_{j}\,^{k}\left((\lie_{\rad}q)_{i} + \tfrac{1}{n}\er_{j}\right).
\end{split}
\end{align}
If $(\lie_{\rad}Q)_{ij} = -T_{ij}$ then $-n(\lie_{\rad}q)_{i} = -n\rad^{p}(\lie_{\rad}Q)_{ip} = n\rad^{p}T_{pi} = \er_{i}$, and so it follows from \eqref{conecurvdiff3} that $\tnabla$ is $\rad$-invariant if and only if $(\lie_{\rad}Q)_{ij} = -T_{ij}$. If this is the case, then by \eqref{trconecurvdiff} it must be that
\begin{align}\label{cccd4}
\tilde{R}_{ij} & = R_{ij} - T_{ij} + (n-2)\left(Q_{ij} + \nabla_{i}q_{j} + q_{i}q_{j}\right) + dq_{ij},
\end{align}
but it is not clear that there can always be found a tensor $Q_{ij}$ so that \eqref{cccd4} yields antisymmetric $\tilde{R}_{ij}$.

Now suppose that $\er_{i} = 0$. It follows from \eqref{cccd3} that if $Q_{ij}$ is defined by $(n-2)Q_{ij} = T_{ij} - R_{(ij)}$, then $(\lie_{\rad}Q)_{ij} = -T_{ij}$, so that $\tnabla$ is $\rad$-invariant by \eqref{conecurvdiff3}. In this case, $-q_{i} = - \rad^{p}Q_{ip} = \tfrac{1}{2n}\er_{j} = 0$, and so from \eqref{cccd4} there follows $\tilde{R}_{ij} = R_{[ij]}$. Since $\Pi_{ip}\,^{p} = -nq_{i} = 0$, the connection $\tnabla$ induces on $|\Det\ctm|$ the same connection as that induced by $\nabla$. This shows the existence claim of the theorem. 

On the other hand, if $\tnabla$ induces on $\Det \ctm$ the same connection as that induced by $\nabla$, then it must be that $-nq_{i} = \Pi_{ip}\,^{p} = 0$. If moreover the symmetric part of the Ricci tensor of $\tnabla$ vanishes, then by \eqref{cccd4} it must be that $(n-2)Q_{ij} = T_{ij} - R_{(ij)}$, and this shows the uniqueness of $\tnabla$.

If a Lie group $G$ acts on $M$ by automorphisms of $\nabla$ preserving $\rad$, then $G$ preserves $T_{ij}$ and $R_{(ij)}$, so also preserves $Q_{ij}$ and hence acts as automorphisms of $\tnabla$.

From \eqref{cccd4} it follows that the Ricci curvature of $\tnabla$ is $\tilde{R}_{ij} = R_{[ij]}$, so that if the initial $\nabla$ has symmetric Ricci tensor, then $\tnabla$ is Ricci-flat. 
\end{proof} 

\begin{remark}\label{conenormalizationremark}
In the setting of Theorem \ref{conenormalizationtheorem}, but with the Ricci tensor not assumed symmetric, one might try to first solve the antisymmetrization of \eqref{cccd4} to produce a conelike nonsingular radiant structure having the same planes and having symmetric Ricci tensor. The antisymmetrization of \eqref{cccd4} gives the equation $\tilde{R}_{[ij]} =R_{[ij]} + \tfrac{n}{2}dq_{ij}$. Since the antisymmetric part of the Ricci tensor is always exact, the equation $0 =R_{[ij]} + \tfrac{n}{2}dq_{ij}$ by itself always admits a solution. However, Example \ref{nosolutionexample} exhibits a conelike radiant structure for which there is no solution with $q_{i}$ satisfying also $\rad^{p}q_{p} = 0$. On the other hand, in Example \ref{nosolutionexample} the conelike radiant structure already has antisymmetric Ricci tensor, so it does not show the conclusion of Theorem \ref{conenormalizationtheorem} is unattainable when the Ricci tensor is not assumed symmetric.
\end{remark}

\begin{corollary}
On a manifold of dimension $n > 2$, if $(\nabla, \rad, \Psi)$ is an equiaffine nonsingular radiant structure such that $(\nabla, \rad)$ is conelike, then there is a unique $\rad$-invariant connection $\tnabla$ such that $(\tnabla, \rad, \Psi)$ is an equiaffine nonsingular radiant structure and $(\tnabla, \rad)$ is conelike, Ricci-flat, and has the same planes as $(\nabla, \rad)$.
\end{corollary}
\begin{proof}
An equiaffine radiant structure has symmetric Ricci tensor, so the claim follows from Theorem \ref{conenormalizationtheorem}.
\end{proof}

Said another way,  there is a unique Ricci-flat $\rad$-invariant conelike equiaffine radiant structure having the same planes as a given conelike equiaffine radiant structure.

A smooth map between foliated manifolds is \emph{foliated} if its differential maps the tangent bundle of the foliation of the domain manifold into the tangent bundle of the foliation of the target manifold. If $X$ is a nonsingular vector field on $\F$ a smooth submersion $\rho:\F \to M$ is \emph{foliated} if it is foliated as a map between $\F$ with the foliation by the orbits of $X$ and $M$ with the trivial foliation.

\begin{example}
The projection, $\rho:\ste \setminus \{0\} \to \projp(\ste)$, to the oriented projectivization of $\ste$ and the corresponding induced projection $\rho:\hopf^{n}(\la) \to \projp(\ste)$ are both foliated submersions. The fibers of the former are copies of $\reap$, while those the latter are copies of $S^{1}$.
\end{example}

\begin{definition}\label{fibereddefined}
\noindent
\begin{enumerate}
\item A nonsingular radiant structure $(\hnabla, \rad)$ on an $(n+1)$-dimensional manifold $\F$ \emph{fibers} over a projective structure $\en$ on an $n$-manifold $M$ if there is a surjective smooth submersion $\rho:\F \to M$ such that the image in $M$ of each geodesic of $\hnabla$ is contained in a projective geodesic of $\en$.

\item A nonsingular conelike radiant structure $(\hnabla, \rad)$ on an $(n+1)$-dimensional manifold $\F$ \emph{fibers} over a projective structure $\en$ on an $n$-manifold $M$ if there is a surjective smooth submersion $\rho:\F \to M$ such that the image in $M$ of each planelike surface is contained in a projective geodesic of $\en$. 
\end{enumerate}
\end{definition}

\begin{example}
The standard flat affine connection $\nabla$ on $\ste$ gives a conelike equiaffine radiant structure $(\stz, \nabla, \Psi, \rad)$ fibering over the standard flat projective structure on $\proj(\ste)$.
\end{example}


\begin{example}\label{nonuniquenessexample}
Let $(\ste, D, \rad, \Psi)$ be the standard equiaffine radiant structure on the $(n+1)$-dimensional vector space $\ste$ defined in Example \ref{standardradiantexample}. 
This example exhibits two Ricci-flat conelike nonsingular radiant structures $(\ste \setminus \{0\}, D, \rad)$ and $(\ste \setminus \{0\}, \nabla, \rad)$ that have the same planes and both fiber over the standard projective structure on $\proj(\ste)$. This shows the $\rad$-invariance hypothesis in Theorem \ref{conenormalizationtheorem} is necessary for the validity of the uniqueness claim. 

Let $h_{ij}$ be a \emph{Euclidean metric} on $\ste$, meaning a $D$-parallel Riemannian metric. Let $\rad^{\flat}_{i} = \rad^{p}h_{ip}$ and $\be_{i} = |\rad|^{-2}\rad^{\flat}_{i} = d\log|\rad|_{i}$ so that $\be_{i}\rad^{i} = 1$. Using $D_{i}|\rad|^{\al} = \al |\rad|^{\al}\be_{i}$ with $\al = 2$ yields $D_{i}\be_{j} = |\rad|^{-2}h_{ij} - 2\be_{i}\be_{j}$. Antisymmetrizing this relation yields $d\be_{ij} = 0$. Contracting it with $\rad^{i}$ yields $\rad^{p}D_{p}\be_{j} = - \be_{j}$, so $(\lie_{\rad}\be)_{i} = \rad^{p}d\be_{pi} = \rad^{p}D_{p}\be_{i} - \rad^{p}D_{i}\be_{p} = -\be_{i} + \be_{i} = 0$. In particular, $\be_{i}$ can be viewed as a flat principal $\reat$-connection on the principal $\reat$ bundle $\ste\setminus\{0\} \to \proj(\ste)$. 

Let $\si_{i}$ be a smooth one-form on $\ste \setminus\{0\}$ such that $\si_{p}\rad^{p} = 0$. 
Define $\Pi_{ij}\,^{k} = Q_{ij}\rad^{k} - 2\si_{(i}\delta_{j)}\,^{k}$ where
\begin{align}
\begin{split}
Q_{ij} & = |\rad|^{\al}(D_{i}\be_{j} + \be_{i}\be_{j})  + 2\si_{(i}\be_{j)}=  |\rad|^{\al-1}D_{i}D_{j}|\rad| +  2\si_{(i}\be_{j)}
 = |\rad|^{\al - 4}\left(|\rad|^{2}h_{ij} - \rad^{\flat}_{i}\rad^{\flat}_{j}\right) + 2\si_{(i}\be_{j)},
\end{split}
\end{align}
for some $\al \in \rea$.
Since $Q_{[ij]} = 0$, the affine connection $\nabla = D + \Pi_{ij}\,^{k}$ is torsion-free. 
As $\rad^{p}D_{p}\be_{i} = -\be_{i}$, there holds $\rad^{p}Q_{ip} = \si_{i}$, so $\Pi_{ip}\,^{k}\rad^{p} = 0$ and $\nabla_{i}\rad^{j} = \delta_{i}\,^{j}$, and $(\nabla, \rad)$ is a radiant structure. 

Because $\Pi_{ip}\,^{p} = -n\si_{i}$, $\Psi$ is $\nabla$-parallel if and only if $\si_{i} = 0$. In this case $(\nabla, \rad, \Psi)$ is an equiaffine radiant structure on $\ste \setminus \{0\}$.
Let $R_{ijk}\,^{l}$ be the curvature of $\nabla$. By construction, $(\lie_{\rad}Q)_{ij} = \al Q_{ij} - 2\al\si_{(i}\be_{j)} + 2(\lie_{\rad}\si)_{(i}\be_{j)}$, so, by Lemma \ref{conelikedifferencelemma},
\begin{align}\label{exric}
\begin{split}
R_{ij} & = (\lie_{\rad}Q)_{ij} + (n-1)(Q_{ij} + \nabla_{i}\si_{j} + \si_{i}\si_{j}) + d\si_{ij}\\
& = (n-1+\al)Q_{ij}   - 2\al\si_{(i}\be_{j)} + 2(\lie_{\rad}\si)_{(i}\be_{j)}+ (n-1)(D_{i}\si_{j} - \si_{i}\si_{j})+ d\si_{ij},\\
\rad^{p}R_{pij}\,^{k} & = \left(\al Q_{ij}  - 2\al\si_{(i}\be_{j)} + 2(\lie_{\rad}\si)_{(i}\be_{j)}\right)\rad^{k} - 2(\lie_{\rad}\si)_{(i}\delta_{j)}\,^{k}.
\end{split}
\end{align}
Because $\rad^{p}D_{p}\si_{i} = (\lie_{\rad}\si)_{i} - \si_{i}$ and $\rad^{p}d\si_{pi} = (\lie_{\rad}\si)_{i}$, 
\begin{align}\label{ccexer}
\er_{i}& = \rad^{p}R_{pi} =  (n-1+\al)\si_{i} + (n-1)\rad^{p}D_{p}\si_{i} - \al \si_{i} + (\lie_{\rad}\si)_{i} + \rad^{p}d\si_{pi}  = (n+1)(\lie_{\rad}\si)_{i},
\end{align}
as follows also from Lemma \ref{conelikedifferencelemma}.
The tensor called $Q_{ij}$ in Lemma \ref{coneconditionlemma} corresponds with the tensor $\al Q_{ij} - 2\al\si_{(i}\be_{j)} + 2(\lie_{\rad}\si)_{(i}\be_{j)}$ in the present example, and to apply Lemma \ref{coneconditionlemma} to conclude that $(\nabla, \rad)$ is conelike it has to be checked that the contraction of this tensor with $\rad^{p}$ equals $\tfrac{1}{n+1}\er_{i}$. This follows from \eqref{ccexer}, for 
\begin{align}
\rad^{p}\left( \al Q_{ip}  - 2\al\si_{(i}\be_{p)} + 2(\lie_{\rad}\si)_{(i}\be_{p)}\right) = \al \si_{i} - \al \si_{i} + (\lie_{\rad}\si)_{i} = (\lie_{\rad}\si)_{i} = \tfrac{1}{n+1}\er_{i}.
\end{align}
Hence, by Lemma \ref{coneconditionlemma}, $(\nabla, \rad)$ is conelike.

By Lemma \ref{conesameplaneslemma}, the planelike surfaces of the conelike radiant structure $(\nabla, \rad)$ on $\ste \setminus \{0\}$ are the two-dimensional subspaces of $\ste$. This can be seen directly as follows. Any two-dimensional subspace $\Sigma \subset \ste$ is spanned by $\rad$ and a constant vector field $v^{i} \in \ste$. By definition of $\nabla$, $\nabla_{\rad}\rad = \rad$, $\nabla_{\rad}v = 0$, $\nabla_{v}\rad = v$, and $\nabla_{v}v = Q(v, v)\rad - 2\si(v)v$, so that $\Sigma_{0} = \Sigma \cap (\ste \setminus \{0\})$ is $\nabla$-totally geodesic. Given $p \in \ste$ and a two-dimensional subspace $L \subset T_{p}\ste$ containing $\rad_{p}$ there is $v \in T_{p}\ste$ such that $L$ is spanned by $\rad_{p}$ and $v$. The constant vector field obtained by $D$-parallel transporting $v$, also denoted $v$, spans with $\rad$ a two-dimensional subspace $\Sigma_{0} \subset \ste\setminus\{0\}$ tangent at $p$ to $L$ and, by the preceding, totally geodesic. 

Suppose $\al = 1 - n$ and $\si_{i} =0$.
In this case, by the preceding and \eqref{exric}, $(\nabla, \rad)$ is Ricci-flat and conelike, but $\lie_{\rad}\nabla$ does not vanish and $\nabla$ is not $\rad$-invariant. The Ricci-flat conelike nonsingular radiant structures $(\ste \setminus \{0\}, D, \rad)$ and $(\ste \setminus \{0\}, \nabla, \rad)$ have the same planes and fiber over the standard projective structure on $\proj(\ste)$. This shows that the uniqueness statement in Theorem \ref{conenormalizationtheorem} is false without the assumption of the $\rad$-invariance of the connection.

Although in this case $\nabla$ and $D$ have the same planes, their geodesics behave quite differently. This is described in more detail than is strictly necessary because it is interesting in its own right and because it suggests that $\nabla$ is interesting rather than pathological.
In $D$-affine coordinates $x^{0}, \dots, x^{n}$, the equations for a $\nabla$-geodesic are
\begin{align}\label{cff1}
0 = \nabla_{d/dt}\dot{x} = \ddot{x} + |x|^{-n-3}\left(|\dot{x}|^{2}|x|^{2} - \lb x, \dot{x}\ra^{2}\right)x.
\end{align}
The analysis of \eqref{cff1} reduces to the classical problem of motion in a central force field with a power law potential.
Define $r(t) = |x(t)|$. 
Differentiating $|\dot{x}|^{2}|x|^{2} - \lb x, \dot{x}\ra^{2}$ using \eqref{cff1} shows that there is $c \in \rea$ such that $|\dot{x}|^{2}|x|^{2} - \lb x, \dot{x}\ra^{2} = c^{2}$, while differentiating and using \eqref{cff1} shows that the energy
\begin{align}\label{cff2}
E(x) = \tfrac{1}{2}|\dot{x}|^{2} - \tfrac{c^{2}}{n+1}|x|^{-n-1} = \tfrac{1}{2}\dot{r}^{2} + \tfrac{c^{2}}{2}r^{-2} - \tfrac{c^{2}}{n+1}r^{-n-1}= \tfrac{1}{2}\dot{r}^{2} + \tfrac{c^{2}}{2}r^{-2} + U(r) = \tfrac{1}{2}\dot{r}^{2} + V(r),
\end{align}
is constant along a solution, where $U(r) = -\tfrac{c^{2}}{n+1}r^{-n-1}$ and $V(r) = U(r) + \tfrac{c^{2}}{2}r^{-2}$. The preceding shows that a solution of \eqref{cff1} can be interpreted as a motion in a central force field with the potential energy $U(r)$ and effective potential energy $V(r)$. The latter is so named because differentiating \eqref{cff2} shows $\ddot{r} = c^{2}(r^{-3} - r^{-n-2}) = - \tfrac{\pr V}{\pr r}$. A nonrigorous qualitative discussion of the behavior of solutions to this problem is given in standard mechanics textbooks such as \cite[Chapter $4$]{Whittaker-analyticaldynamics}, \cite[Chapter $3$]{Goldstein}, or \cite[Section $14$]{Landau-Lifschitz-mechanics}. 

If $c =0$ then \eqref{cff1} shows $\ddot{x} = 0$, so the radial lines $x(t) = ta$, $a \in \stz$, are $\nabla$-geodesics. If $c \neq 0$ then, replacing $x(t)$ by the reparametrization $x(|c|^{-1}t)$, it can be and is assumed hereforth that $c^{2} = 1$. In this case $x(t)$ and $\dot{x}(t)$ are linearly independent for all $t$. From \eqref{cff1} it follows that the two-form $x\wedge \dot{x}$ is constant in time, so there is an $h$-orthonormal basis $\{a_{1}, a_{2}\}$ of $\spn\{x(0), \dot{x}(0)\}$ such that $x(t) = f_{1}(t)a_{1} + f_{2}(t)a_{2}$ for some functions $f_{1}(t)$ and $f_{2}(t)$ that, by \eqref{cff1}, satisfy the differential equations $\ddot{f}_{i} + (f_{1}^{2} + f_{2}^{2})^{-(n+3)/2}f_{i} = 0$, $i = 1, 2$.  Because $x \wedge \dot{x} =  (f_{1}\dot{f}_{2} - \dot{f}_{1} f_{2})a_{1}\wedge a_{2}$, $f_{1}\dot{f}_{2} - \dot{f}_{1} f_{2}$ equals $\pm c = \pm 1$. Because the order of $a_{1}$ and $a_{2}$ is arbitrary, it can be and is assumed that $f_{1}\dot{f}_{2} - \dot{f}_{1} f_{2} = 1$. (The same conclusion follows more conceptually from the observation that, by construction, two-dimensional subspaces are totally geodesic for $\nabla$.) Taking $f_{1}(t) =  \cos{t}$ and $f_{2}(t) = \sin{t}$ shows that the unit radius circles centered on the origin, $x = \cos(t)a + \sin(t)b$, are $\nabla$-geodesics.
 Writing $f_{1} = r\cos \theta$ and $f_{2} = r\sin \theta$ and differentiating shows that $\dot{\theta}r^{2} = f_{1}\dot{f}_{2} - \dot{f}_{1} f_{2} =1$, which shows that the constancy of $c^{2}$ can be regarded as conservation of angular momentum.

On the domain $r > 0$, the effective potential $V(r) = \tfrac{1}{2}r^{-2} - \tfrac{1}{n+1}r^{-n-1}$ has a unique maximum at $r = 1$, with value $V(1) = \tfrac{n-1}{2(n+1)}$. In qualitative terms, the behavior of a solution with energy $E_{0}$ is determined by the relation of the initial value $r_{0} = r(0)$ with the roots of the equation $E_{0} = V(r)$, for, because $\dot{r}^{2} = 2(E_{0} - V(r))$, the latter correspond to turning points of $r(t)$.

Consider a solution with initial conditions $r(0) = r_{0}$ and $\dot{r}(0) = \dot{r}_{0}$. These are related to the initial position $x(0)$ and velocity $\dot{x}(0)$ by $r_{0} = |x(0)|$ and $\dot{r}_{0}^{2} + r_{0}^{-2} = |\dot{x}(0)|^{2}$. By \eqref{cff2}, the energy $E_{0}$ is given by $2E_{0} = \dot{r}_{0}^{2} + 2V(r_{0})$. If $r_{0} = 1$ and $E_{0}$ equals the maximal value $\tfrac{n-1}{2(n+1)} = V(1)$ of $V(r)$, then $\dot{r}(0)^{2} = 2(E_{0}- V(r_{0})) = 0$. The unique solution of $\ddot{r} = r^{-3} - r^{-n-2}$ with initial conditions $r(0) = 1$ and $\dot{r}(0) = 0$ is $r(t) = 1$. This yields the circular solutions $x = \cos(t)a + \sin(t)b$ already mentioned. With the same energy, so that $\dot{r}_{0} = 0$, if $r_{0} < 1$, then from $\ddot{r} = r^{-3} - r^{-n-2}$ it follows that $\ddot{r} < 0$, so $\dot{r} < 0$ as well for $t > 0$, and both $\dot{r}$ and $r$ decrease for $t > 0$. Hence $\dot{r} = -(2(E_{0} - V(r))^{1/2}$ and so $r(t) \to 0$ as $t \to T$ for $T = - \int_{r_{0}}^{0}(2(E_{0} + \tfrac{1}{n+1}s^{-n-1}- s^{-2})^{-1/2}\,ds$. If instead $r_{0} > 1$, then the same reasoning shows that $r$ and $\dot{r}$ are increasing for $t > 0$, and $r(t) \to \infty$ as $t \to \infty$. If $E_{0} > V(1)$, then $\dot{r}(t) \neq 0$ for all valid $t$, and $r(t)$ increases or decreases. If $E_{0} < V(1)$, then there are $r_{\min} < 1 < r_{\max}$ such that $V(r_{\min}) = E_{0} = V(r_{\max})$ and the behavior of $r(t)$ depends on the relation of $r_{0}$ to the interval $[r_{\min}, r_{\max}]$. If $r_{0} < r_{\min}$, $r(t)$ tends to $0$ in finite time, if $r_{0} > r_{\max}$, then $r(t)$ tends to $\infty$ as $t \to \infty$, while if $r_{0} \in [r_{\min}, r_{\max}]$ then $r(t) \in [r_{\min}, r_{\max}]$ for $t > 0$, so the solution stays bounded, although it is not closed. More precise description of the phase curve can also be obtained by writing $u = r^{-1}$ (as in \cite{Whittaker-analyticaldynamics}) and using $\tfrac{d}{dt} = r^{-2}\tfrac{d}{d\theta}=  u^{2}\tfrac{d}{d\theta}$ to rewrite $\ddot{r} = r^{-3} - r^{-n-2}$ to obtain $\tfrac{d^{2}u}{d\theta^{2}} = u^{n} - u = u(u-1)(u^{n-2} + u^{n-1} + \dots + u + 1)$, which describes the solution in polar coordinates. The corresponding expression for the energy is $(\tfrac{du}{d\theta})^{2} = \tfrac{2}{n+1}u^{n+1} - u^{2} + 2E_{0}$.
More detailed discussion of the solutions is omitted, but the preceding suffices to show that the geodesics of $\nabla$ behave quite differently than those of $D$ (which are all unbounded, being lines). 

Independent of the considerations of this section, it seems interesting that there is a family of affine connections with nice properties (e.g. Ricci-flat) and whose geodesics are motions in a central field with power law potential. The author has not found this in the literature. 
\end{example}

\section{Extended projective structures and the associated cone connections}\label{extendedprojectivesection}
This section introduces the notion of \emph{extended projective structure}, which is a projective structure coupled to a class of principal connections on a principal bundle $\rho:N \to M$ with one-dimensional structure group, and establishes one of the main results of the paper, Theorem \ref{extendedthomastheorem}, that shows that there is a canonical correspondence between extended projective structures on $M$ and certain conelike radiant structures on $N$. 

It is convenient to recall some definitions.
For a principal $G$-bundle $\rho:\F \to M$, $\prin(\F)$ denotes the space of principal $G$-connections on $\F$. 
By a principal bundle $\rho:\F \to M$ with one-dimensional structure group is meant either a principal $\reat$-bundle or a principal $S^{1}$-bundle. In either case the Lie algebra of the structure group is identified with $\rea$. A principal $\reap$-bundle is regarded as a principal $\reat$-bundle via the identity inclusion of the structure groups. 
Because the structure group is abelian, a connection $\be \in \prin(\F)$ is by definition a one-form $\be \in \ext^{1}(T^{\ast}\F)$ preserved by the principal action and such that $\be(\eul) = 1$, in which $\eul$ is the fundamental vector field generating the principal action $R_{e^{t}}$. Since the principal action preserves $\be$ there holds $\lie_{\eul}\be = 0$, and so also $i(\eul)d\be = 0$; that is $d\be$ is horizontal and there is a closed two-form $\om \in \Ga(\ext^{2}\ctm)$, the \emph{curvature} of $\be$, such that $d\be = \rho^{\ast}(\om)$.
A connection $\be \in \prin(\F)$ determines and is determined by the homogeneous horizontal subbundle $H = \ker \be \subset T\F$, where \emph{horizontal} means that $\ker T\rho(u) \cap H_{u} = \{0\}$ and $T\rho(u)(H_{u}) = T_{\rho(u)}M$ for all $u \in \F$, and \emph{homogeneous} means that $H$ is preserved by the differential $TR_{r}$ of the principal action. The horizontal lift $X \in \Ga(TM) \to \hat{X} \in \Ga(T\F)$ determined $\be$ is the unique section of $\Ga(H)$ covering $X$. The uniqueness implies that the map $X \to \hat{X}$ is a $\cinf(M)$-module map, where $f \in \cinf(M)$ acts on $\Ga(T\F)$ by multiplication by $\rho^{\ast}(f)$. The difference of $\tilde{\be} - \be$ of $\tilde{\be}, \be \in \prin(\F)$ is a homogeneity $0$ one-form annihilating the vertical so has the form $\rho^{\ast}(\ga)$ for some $\ga \in \Ga(\ctm)$. The $\tilde{\be}$ horizontal lift of $X \in \Ga(TM)$ is expressible in terms of its $\be$-horizontal lift $\hat{X}$ as $\hat{X} - \ga(X)\eul$.

By \eqref{projvary} the antisymmetric parts of the Ricci tensors of projectively equivalent connections are cohomologous two-forms. The resulting cohomology class is trivial because the antisymmetric part of the Ricci tensor of a torsion-free affine connection $\nabla$ is exact, for if $\mu$ is any volume density, then $\nabla_{i}\mu = \si_{i}\mu$ for some one-form $\si$, and $2R_{[ij]}\mu = -R_{ijp}\,^{p}\mu = 2\nabla_{[i}\nabla_{j]}\mu = 2\nabla_{[i}\si_{j]}\mu = d\si_{ij} \mu$, so $2R_{[ij]} = d\si_{ij}$.

Let $\rho:N \to M$ be a principal $S^{1}$-bundle or principal $\reat$-bundle. Since any two elements of $\prin(\rho:N \to M)$ differ by the pullback via $\rho$ of a one-form on $M$, their curvatures determine a cohomology class $[\om] \in H^{2}(M; \rea)$.  (When the structure group is $S^{1}$, the cohomology class $\tfrac{1}{2\pi}[\om] \in H^{2}(M;\integer)$ is integral, equal to the first Chern class of the bundle \cite{Chern-circlebundles, Kobayashi-circlebundles}.)

The idea behind the definition of extended projective structure is to link representatives $\nabla \in \en$ and $\be \in \prin(\rho:M \to N)$ in such a way that a two-form representing the cohomology class $[\om]$ can be associated with the linked pair.

A projective structure $\en$ is a torsor for the abelian additive group $\Ga(\ctm)$ with $\ga \in \Ga(\ctm)$ acting on $\nabla \in \en$ by $\ga \cdot \nabla = \nabla + 2\ga_{(i}\delta_{j)}\,^{k}$. The space $\prin(\rho:N \to M)$ of principal connections on $\rho:N \to M$ is also a torsor for $\Ga(\ctm)$, with $\ga \in \Ga(\ctm)$ acting on $\be \in \prin(\rho:N \to M)$ by $\ga \cdot \be = \be - \rho^{\ast}(\ga)$. It follows that $\Ga(\ctm)$ acts on $\en \times \prin(\rho:N \to M)$ by $\ga\cdot(\nabla, \be) =( \nabla + 2\ga_{(i}\delta_{j)}\,^{k},\be - \rho^{\ast}(\ga))$. The orbit of $(\nabla, \be) \in \en \times \prin(\rho:N \to M)$ under this action is denoted $[\nabla, \be]$ and is called an \emph{extended projective structure} on $\rho:N \to M$. Pairs $(\tnabla, \tilde{\be})$ and $(\nabla, \be)$ are \emph{projectively equivalent} if there is $\ga \in \Ga(\ctm)$ such that $\ga \cdot (\nabla, \be) = (\tnabla, \tilde{\be})$. 


An extended projective structure $[\nabla, \be]$ on $\rho:N \to M$ determines an underlying projective structure $\en$ on $M$ and associates with each representative connection $\nabla \in \en$ a principal connection $\be \in \prin(\rho:N \to M)$. The different possible extended projective structures $[\nabla, \be]$ determining a given projective structure $\en$ can be parametrized by the choice of $\be \in \prin(\rho:N \to M)$, in the sense that if $[\nabla, \be]$ is an extended projective structure then so is $[\nabla, \be + \rho^{\ast}(\si)]$ for any $\si \in \Ga(\ctm)$, but there is no canonical choice of $\be$ unless $\rho:N \to M$ has some additional structure (in particular is somehow determined by the smooth structure on $M$, as is the case, for example, when $N$ is taken to be $\vdens$). More succinctly, the space of extended projective structures on $\rho:N \to M$ is a torsor for $\Ga(\ctm)$.

\begin{lemma}\label{epcdefinedlemma}
Let $\rho:N \to M$ be a principal $S^{1}$-bundle or principal $\reat$-bundle over the $n$-manifold $M$.
Given an extended projective structure $[\nabla, \be]$ on $\rho:N \to M$, the two-form
\begin{align}\label{fcurvdefined}
\epc_{ij} = \om_{ij} - \tfrac{2}{n+1}R_{[ij]} = \om_{ij} + 2P_{[ij]}.
\end{align}
associated with a representative $(\nabla, \be) \in [\nabla, \be]$ does not depend on the choice of $(\nabla, \be)$, so determines a canonical representative of the cohomology class $[\om] \in H^{2}(M; \rea)$ determined by $\rho:N \to M$.
\end{lemma}
\begin{proof}
By \eqref{projvary}, the two-form $\epc$ defined in \eqref{fcurvdefined} does not depend on the choice of $(\nabla, \be) \in [\nabla, \be]$. Because the antisymmetric Ricci tensor $R_{[ij]}$ of $\nabla$ is exact, $\epc$ represents the cohomology class $[\om]$.
\end{proof}

The two-form $\epc$ associated with the extended projective structure $[\nabla, \be]$ by \eqref{fcurvdefined} is the the \emph{associated two-form} of $[\nabla, \be]$.

\begin{theorem}\label{extendedthomastheorem}
Let $M$ be an $n$-manifold, let $\rho:N \to M$ be a principal $S^{1}$-bundle or principal $\reat$-bundle, and let $\rad$ be the fundamental vertical vector field generated by the principal action. Let $[\nabla, \be]$ be an extended projective structure on $\rho:N \to M$ with associated two-form $\epc$. 

\begin{enumerate}
\item\label{extendedthomas1} There is a unique torsion-free affine connection $\hnabla$ on $N$ having antisymmetric Ricci tensor and inducing on $M$ the given extended projective structure $[\nabla, \be]$ and so that $(\hnabla, \rad)$ is a conelike radiant structure invariant under the principal action (and so $\rad$-invariant) and fibering over $(M, \en)$. 

For any $(\nabla, \be) \in [\nabla, \be]$, the connection $\hnabla$ is given by
\begin{align}\label{tripleconnection}
&\hnabla_{\hat{X}}\hat{Y} = \widehat{\nabla_{X}Y} + \rho^{\ast}(Q(X, Y))\rad,&&\hnabla_{\hat{X}}\rad = \hat{X} = \hnabla_{\rad}\hat{X},& &\hnabla_{\rad}\rad = \rad. 
\end{align}
where $\hat{X} \in \Ga(TN)$ is the $\be$-horizontal lift of $X \in \Ga(TM)$, $\om \in \Ga(\ext^{2}\ctm)$ is the curvature of $\be$, and $Q\in \Ga(\tensor^{2}\ctm)$ is defined by
\begin{align}\label{qdefined}
Q_{ij} = \tfrac{1}{1-n}R_{(ij)} - \tfrac{1}{2}\om_{ij} = P_{(ij)} - \tfrac{1}{2}\om_{ij} = P_{ij} + \tfrac{1}{n+1}R_{[ij]} - \tfrac{1}{2}\om_{ij} = P_{ij} - \tfrac{1}{2}\epc_{ij},
\end{align} 
where $R_{ij}$ and $P_{ij}$ are the Ricci tensor and projective Schouten tensor of $\nabla$, and $\hnabla$ satisfies
\begin{align}\label{hnablabe}
&\hnabla_{I} \be_{J} + \be_{I}\be_{J} = -\rho^{\ast}(Q)_{IJ}, & &2Q_{[ij]} =- \om_{ij}.
\end{align}
Moreover, $(\hnabla, \rad)$ has the following properties:
\begin{enumerate}
\item\label{ett1} The planelike surfaces of $(\hnabla, \rad)$ are the $\rho$-preimages of the projective geodesics of $M$ and the base curve $\ga = \rho\circ \hat{\ga}:I \to M$ of a parametrized curve $\hat{\ga}:I \to N$ is a projective parametrization of a projective geodesic of $\en$ if and only if $\hat{\ga}$ is a $\hnabla$-geodesic.
\item\label{etepc} The Ricci curvature of $\hnabla$ equals $ -\tfrac{n+1}{2}\rho^{\ast}(\epc)$.
\item For any $(\nabla, \be) \in [\nabla, \be]$, the curvature tensor $\hat{R}_{IJK}\,^{L}$ of $\hnabla$ satisfies
\begin{align}\label{hnablabecurvatures}
\begin{aligned}
&\rad^{I}\hat{R}_{IJK}\,^{L} = 0=  \rad^{K}\hat{R}_{IJK}\,^{L}, \qquad \hat{R}_{IJK}\,^{A}\be_{A} = \rho^{\ast}(C)_{IJK} + \tfrac{1}{2}\rho^{\ast}(\nabla \epc)_{KIJ},&\\
&\hat{R}_{ijk}\,^{l} = B_{ijk}\,^{l} - \delta_{[i}\,^{l}\rho^{\ast}(\epc)_{j]k} + \delta_{k}\,^{l}\rho^{\ast}(\epc)_{ij} = R_{ijk}\,^{l} + \delta_{i}\,^{l}P_{(jk)} - \delta_{j}\,^{l}P_{(ik)} - \delta_{[i}\,^{l}\om_{j]k} + \delta_{k}\,^{l}\om_{ij},&
\end{aligned}
\end{align}
where $B_{ijk}\,^{l}$ and $C_{ijk}$ are the projective Weyl and projective Cotton tensors of $\nabla$ and the splitting of $TN$ determined by $\be$ is used in writing $\hat{R}_{ijk}\,^{l}$.
\end{enumerate}
\item\label{extendedthomas2} Any conelike radiant structure $(\hat{D}, \rad)$ on $N$ invariant under the principal action on $\rho:N \to M$ induces on $\rho:N \to M$ an extended projective structure $[\nabla, \be]$ for which it fibers over the underlying projective structure $\en$ on $M$ and has the same planelike surfaces as the associated Ricci antisymmetric, principal action invariant, conelike radiant structure $(\hnabla, \rad)$ associated with $[\nabla, \be]$ as in \eqref{extendedthomas1}.
\end{enumerate}
\end{theorem}

\begin{definition}
The connection $\hnabla$ associated with the extended projective structure $[\nabla, \be]$ by Theorem \ref{extendedthomastheorem} is the \emph{cone connection determined by $[\nabla, \be]$}. 
\end{definition}


\begin{remark}
If $\om_{ij} = \om_{[ij]}$ is an antisymmetric two-form on a two-dimensional vector space, then $2\delta_{[i}\,^{l}\om_{j]k} = -\om_{ij}\delta_{k}\,^{l}$.
To see this, let $X$ and $Y$ span the vector space and let $Z = aX + bY$. Then $2X^{i}Y^{j}Z^{k}\delta_{[i}\,^{l}\om_{j]k} = \om_{jk}Y^{j}Z^{k}X^{l} - \om_{ik}X^{i}Z^{k}Y^{l} = a\om_{jk}Y^{j}X^{k}X^{l} - b\om_{ik}X^{i}Y^{k}Y^{l} = -\om_{ij}X^{i}Y^{j}Z^{l}$.
Consequently, when $n = 2$, the last identity of \eqref{hnablabecurvatures} simplifies to $\hat{R}_{ijk}\,^{l} = \tfrac{3}{2}\rho^{\ast}(\epc)_{ij}\delta_{k}\,^{l} = -\hat{R}_{ij}\delta_{k}\,^{l}$.
\end{remark}

The remainder of the section is devoted to the proof of Theorem \ref{extendedthomastheorem}. This is organized as follows:
\begin{itemize}
\item Lemmas \ref{invariantinducedlemma} and \ref{sameinducedlemma} shows that a radiant structure on $N$ induces an extended projective structure on $M$ and that conelike radiant structures with the same planelike surfaces induce the same extended projective structure.
\item Lemmas \ref{preextendedthomaslemma} and \ref{extendedshiftlemma} define the cone connection of a pair $(\nabla, \be)$ and show that it depends only on the extended projective structure generated by $(\nabla, \be)$.
\item Lemmas \ref{geodesiclemma} and \ref{projectiveparametrizationlemma} show that the cone connection of an extended projective structure $[\nabla, \be]$ fibers over the underlying projective structure $\en$. 
\end{itemize}
The section concludes with the proof of Theorem \ref{extendedthomastheorem} and some remarks.

\begin{lemma}\label{invariantinducedlemma}
Let $\rho:N \to M$ be a principal $S^{1}$-bundle or principal $\reat$-bundle and let $\rad$ be the fundamental vertical vector field generated by the principal action. A torsion-free affine connection $\hnabla$ on $N$ invariant under the principal action and constituting with $\rad$ a radiant structure determines on $M$ an extended projective structure $[\nabla, \be]$. More precisely:
\begin{enumerate}
\item\label{iil1} For each $\be \in \prin(\rho:N \to M)$ with curvature $\om$ there is a torsion-free affine connection $\nabla$ on $M$ 
satisfying \eqref{tripleconnection}, where $\hat{X} \in \Ga(TN)$ is the $\be$-horizontal lift of $X \in \Ga(TM)$, and $Q\in \Ga(\tensor^{2}\ctm)$ satisfies \eqref{hnablabe}.
\item\label{iil2} The connections $\tnabla$ and $\nabla$ determined by $\hnabla$ together with the principal connections $\tilde{\be} = \be - \rho^{\ast}(\ga)$ and $\be$ as in \eqref{iil1} satisfy $\tnabla = \nabla + 2\ga_{(i}\delta_{j)}\,^{k}$, so $(\tnabla, \tilde{\be})$ and $(\nabla, \be)$ are projectively equivalent. 
\end{enumerate}
\end{lemma}

\begin{proof}
For $\be \in \prin(\rho:N \to M)$, let $\hat{X}$ be the $\be$-horizontal lift of $X \in \Ga(TM)$. Because $(\hnabla, \rad)$ is radiant and $[\rad, \hat{X}] = 0$ there hold $\hat{X} = \hnabla_{\hat{X}}\rad = \hnabla_{\rad}\hat{X}$ and $\hnabla_{\rad}\rad = \rad$. Because $\hnabla$ is invariant under the principal action, the vector field $\hnabla_{\hat{X}}\hat{Y}$ is likewise invariant, so there are $\nabla_{X}Y \in \Ga(TM)$ and $Q(X, Y) \in \cinf(M)$ such that $\hnabla_{\hat{X}}\hat{Y} = \widehat{\nabla_{X}Y} + \rho^{\ast}(Q(X, Y))\rad$. The expression $\nabla_{X}Y$ is a $\cinf(M)$-module map in the first argument, is linear in the second argument, and satisfies a Leibiz rule in the second argument, and the expression $Q(X, Y)$ is a $\cinf(M)$-module map in both arguments, so $\nabla$ is an affine connection on $M$, $Q \in \Ga(\tensor^{2}\ctm)$, and there holds \eqref{tripleconnection}. Straightforward computation using \eqref{tripleconnection} shows the first identity of \eqref{hnablabe}.  Antisymmetrizing the first identity of \eqref{hnablabe} shows the second identity of \eqref{hnablabe}. 
This shows \eqref{iil1}.

Because the $\tilde{\be} = \be - \rho^{\ast}(\ga)$ horizontal lift of $X$ is $\hat{X} + \rho^{\ast}(\ga(X))\rad$, the connection $\tnabla$ determined by $\hnabla$ and $\tilde{\be}$ as in \eqref{tripleconnection} is related to $\nabla$ by $\tnabla = \nabla + 2\ga_{(i}\delta_{j)}\,^{k}$. This shows $\hnabla$ determines on $M$ an extended projective structure $[\nabla, \be]$. This shows \eqref{iil2}.
\end{proof}

\begin{lemma}\label{sameinducedlemma}
Let $\rho:N \to M$ be a principal $S^{1}$-bundle or principal $\reat$-bundle and let $\rad$ be the fundamental vertical vector field generated by the principal action. If $\hnabla$ and $\hD$ are torsion-free affine connections on $N$ that are invariant under the principal action and constitute with $\rad$ conelike radiant structures such that $\er(\hnabla) = 0$ and $\er(\hD) = 0$, then $\hnabla$ and $\hD$ induce the same extended projective structure on $\rho:N \to M$.
\end{lemma}

\begin{proof}
Because $\hnabla$ and $\hD$ induce on $|\Det T^{\ast}N|$ the same connection, by Lemma \ref{conesameplaneslemma} there is $\hat{T} \in \Ga(S^{2}T^{\ast}N)$ such that $\hnabla - \hD = \hat{T}_{IJ}\rad\,^{K}$ and $\rad^{P}\hat{T}_{PI} = 0$. By the invariance of $\hnabla$ and $\hD$ under the principal action there is $T \in \Ga(S^{2}\ctm)$ such that $\hat{T} = \rho^{\ast}(T)$, so that $\hnabla - \hD = \rho^{\ast}(T)_{IJ}\rad\,^{K}$. For $\be \in \prin(\rho:N \to M)$ let $\nabla$ and $D$ be the connection induced on $M$ as in Lemma \ref{invariantinducedlemma}. For $X, Y \in \Ga(TM)$, 
\begin{align}
\widehat{\nabla_{X}Y} = \hnabla_{\hat{X}}\hat{Y} - \be(\hnabla_{\hat{X}}\hat{Y})\rad = \hD_{\hat{X}}\hat{Y} + T(X, Y)\rad - \be(\hD_{\hat{X}}\hat{Y} + T(X, Y)\rad) = \widehat{D_{X}Y}.
\end{align}
which proves the claim. 
\end{proof}

Given a principal bundle $\rho:N \to M$ with one-dimensional structure group, a tensor $Q \in \Ga(\tensor^{2}\ctm)$ is \emph{compatible} with $\be \in \prin(\rho:N \to M)$ if $2Q_{[ij]} = -\om_{ij}$ where $\om$ is the curvature of $\be$.

\begin{lemma}\label{preextendedthomaslemma}
Let $M$ be an $n$-manifold, let $\rho:N \to M$ be a principal $S^{1}$-bundle or principal $\reat$-bundle, and let $\rad$ be the fundamental vertical vector field generated by the principal action. Let $\be \in \prin(\rho:N \to M)$ be a principal connection having curvature $\om_{ij} \in \Ga(\ext^{2}\ctm)$ and let $\nabla$ be a torsion-free affine connection on $M$. 
\begin{enumerate}
\item\label{pet1} For any $Q_{ij} \in \Ga(\tensor^{2}\ctm)$ compatible with $\be$,  
the affine connection $\hnabla$ on $N$ defined by \eqref{tripleconnection}, where $\hat{X} \in \Ga(TN)$ is the $\be$-horizontal lift of $X \in \Ga(TM)$, is torsion-free and forms with $\rad$ a conelike radiant structure invariant under the principal action (so $\lie_{\rad}\hnabla = 0$). 
\item\label{pet2} Replacing $Q_{ij} \in \Ga(\tensor^{2}\ctm)$ in \eqref{tripleconnection} by any $\tilde{Q}_{ij} \in \Ga(\tensor^{2}\ctm)$ compatible with $\be$ determines a conelike radiant structure having the same planelike surfaces as has that determined by $Q_{ij}$.
\item\label{pet3} There is a unique choice of $Q_{ij} \in \Ga(\tensor^{2}\ctm)$ compatible with $\be$ such that the resulting conelike radiant structure $(\hnabla, \rad)$ has antisymmetric Ricci tensor. It is given by \eqref{qdefined}, where $R_{ij}$ and $P_{ij}$ are the Ricci tensor and projective Schouten tensor of $\nabla$. In this case the curvature tensor $\hat{R}_{IJK}\,^{L}$ and Ricci tensor $\widehat{\ric}_{IJ}$ of $\hnabla$ satisfy
\begin{align}\label{qcurvatures}
\begin{aligned}
&\hat{R}_{ijk}\,^{l} = B_{ijk}\,^{l} - \delta_{[i}\,^{l}\rho^{\ast}(\epc)_{j]k} + \delta_{k}\,^{l}\rho^{\ast}(\epc)_{ij},&\\
&\hat{R}_{ijk}\,^{A}\be_{A} = \rho^{\ast}(C)_{ijk} + \tfrac{1}{2}\rho^{\ast}(\nabla \epc)_{kij},
& \widehat{\ric}_{IJ} = -\tfrac{n+1}{2}\rho^{\ast}(\epc)_{IJ},
\end{aligned}
\end{align}
where $B_{ijk}\,^{l}$ and $C_{ijk}$ are the projective Weyl and projective Cotton tensors of $\nabla$, $\epc_{ij} = \om_{ij} - \tfrac{2}{n+1}R_{[ij]} = \om_{ij} + 2P_{[ij]}$, and the splitting of $TN$ determined by $\be$ is used in writing $\hat{R}_{ijk}\,^{l}$. 
\end{enumerate}
\end{lemma}

\begin{proof}
Uppercase Latin indices indicate tensors on $N$, while lowercase Latin indices indicate tensors on $M$. 
The principal connection $\be\in \prin(\rho:N \to M)$ can be viewed as a one-form satisfying $\be_{Q}\rad^{Q} = 1$ and $\rad^{Q}d\be_{QI} = 0$.  Fixing $\be$ determines a splitting $TN \simeq \spn\{\rad\} \oplus \ker \be$ and so it makes sense to decorate sections of tensor powers of $\ker \be$ and its dual with lowercase Latin indices, and to indicate sections of tensor powers of $\spn\{\rad\}$ its dual with the single index $\infty$. Let $\om_{ij}$ be the curvature two-form of $\be$ defined by $\rho^{\ast}(\om) = d\be$. For $X \in \Ga(TM)$ let $\hat{X}$ denote the $\be$-horizontal lift of $X$.

Let $\nabla$ be a torsion-free connection on $M$. Given an arbitrary $Q_{ij} \in \Ga(\tensor^{2}\ctm)$, define an affine connection $\hnabla$ on $N$ associated with the triple $(\nabla, Q, \be)$ by requiring that it satisfy the identities \eqref{tripleconnection}. Because $[\hat{X}, \hat{Y}] = \widehat{[X, Y]} - d\beta(X, Y)\rad$, the connection $\hnabla$ is torsion-free if and only if $Q$ is compatible with $\be$, and this is assumed in what follows. In this case, $(\hnabla, \rad)$ is a radiant structure on $N$. 

Let $R_{ijk}\,^{l}$ and $\hat{R}_{IJK}\,^{L}$ be the curvature tensors of $\nabla$ and $\hnabla$. Using the splitting of $TN$ given by the span of $\rad$ and the kernel of $\be$, it makes sense to write $\hat{R}_{ijk}\,^{l}$ for $\hat{R}_{ijk}\,^{L} - \hat{R}_{ijk}\,^{P}\be_{P}\rad^{L}$. It follows from \eqref{qdefined} and \eqref{hnablabe} that $\rad^{P}\hat{R}_{PIJ}\,^{K} = 0$. Since $\hnabla_{I}\rad^{J} = 0$ this implies $\lie_{\rad}\hnabla = 0$. From the construction \eqref{tripleconnection} there is apparent the validity of the stronger statement that the principal action on $N$ is by automorphisms of $\hnabla$. By Lemma \ref{coneconditionlemma}, that $\rad^{P}\hat{R}_{PIJ}\,^{K} = 0$ implies that $(\hnabla, \rad)$ is conelike. This shows \eqref{pet1}. 

By Lemma \ref{conesameplaneslemma} changing the choice of $Q$, subject to the requirement of compatibility with $\be$, determines a conelike radiant structure having the same planelike surfaces as has $(\hnabla, \rad)$. This shows \eqref{pet2}.

Routine computations show that the nonvanishing components of the curvature of $\hnabla$ and its traces are
\begin{align}
\label{tc1}&\hat{R}_{ijk}\,^{l} = R_{ijk}\,^{l} + 2\delta_{[i}\,^{l}Q_{j]k} + \om_{ij}\delta_{k}\,^{l},& &\hat{R}_{ijk}\,^{P}\be_{P} = 2\nabla_{[i}Q_{j]k},\\
\label{tc2}&\hat{R}_{ij} = R_{ij} + (n-1)Q_{ij} - \om_{ij},& &\hat{R}_{[ij]} = R_{[ij]} - \tfrac{n+1}{2}\om_{ij}.
\end{align}
in which the notations indicating pullback by $\rho$ have been omitted for readability.
It is evident from \eqref{tc2} that in general the skew-part of the Ricci tensor of $\hnabla$ need not vanish. By \eqref{tc2}, requiring that the symmetric part of the Ricci tensor of $\hnabla$ vanish determines $Q_{ij}$ uniquely as in \eqref{qdefined}. In this case
\begin{align}
\label{tc3}
\begin{split}
\hat{R}_{ijk}\,^{l} &= B_{ijk}\,^{l} + \tfrac{1}{n+1}\left(\delta_{i}\,^{l}R_{[jk]} - \delta_{j}\,^{l}R_{[ik]} - 2\delta_{k}\,^{l}R_{[ij]}\right) - \delta_{[i}\,^{l}\om_{j]k} + \om_{ij}\delta_{k}\,^{l}\\
& = B_{ijk}\,^{l} -\delta_{[i}\,^{l}\epc_{j]k} + \delta_{k}\,^{l}\epc_{ij},\\
\hat{R}_{ijk}\,^{P}\be_{P}  &= C_{ijk} - \tfrac{1}{n+1}\nabla_{k}R_{[ij]} + \tfrac{1}{2}\nabla_{k}\om_{ij} = C_{ijk} + \tfrac{1}{2}\nabla_{k}\epc_{ij},\\
\hat{R}_{ij} &=  R_{[ij]} - \tfrac{n+1}{2}\om_{ij} = -\tfrac{n+1}{2}\epc_{ij}.
\end{split}
\end{align}
This shows \eqref{qcurvatures} and completes the proof of \eqref{pet3}.
\end{proof}

\begin{definition}
The unique connection $\hnabla$ associated with the pair $(\nabla, \be)$ and having antisymmetric Ricci tensor as in \eqref{pet3} of Lemma \ref{preextendedthomaslemma} is the \emph{cone connection} determined by $\nabla$ and $\be$. 
\end{definition}

\begin{lemma}\label{extendedshiftlemma}
Let $\nabla$ be a torsion-free affine connection on the $n$-manifold $M$, let $\rho:N \to M$ be a principal $S^{1}$-bundle or principal $\reat$-bundle, and let $\rad$ be the fundamental vertical vector field generated by the principal action. The cone connections of any two pairs $(\nabla, \be)$ and $(\tnabla, \tilde{\be})$ representing the same extended projective structure $[\nabla, \be]$ on $\rho:N \to M$ are the same, so it makes sense to speak of the \emph{cone connection of $[\nabla, \be]$}.  
\end{lemma}

\begin{proof}
For $\ga \in \Ga(\ctm)$,  straightforward computations using the definition of the cone connection, \eqref{qdefined}, \eqref{tripleconnection}, the remark that the horizontal lift of $X \in \Ga(TM)$ with respect to $\be + \rho^{\ast}(\ga)$ is $\hat{X} - \rho^{\ast}(\ga(X))\rad$, and the remark that the projective Schouten tensors $\tilde{P}_{ij}$ and $P_{ij}$ of $\nabla + 2\ga_{(i}\delta_{j)}\,^{k}$ and $\nabla$ are related by $\tilde{P}_{ij} = P_{ij} + \nabla_{i}\ga_{j} - \ga_{i}\ga_{j}$ show that the cone connections of $(\nabla, \be + \rho^{\ast}(\ga))$ and $(\nabla + 2\ga_{(i}\delta_{j)}\,^{k}, \be)$ both equal 
\begin{align}\label{shiftedthomas}
\hnabla + 2\rho^{\ast}(\ga)_{(I}\delta_{J)}\,^{k} - 2\rho^{\ast}(\ga)_{(I}\be_{J)}\rad^{K} + \left(\rho^{\ast}(\nabla \ga)_{(IJ)} - \rho^{\ast}(\ga)_{I}\rho^{\ast}(\ga)_{J}\right)\rad^{K},
\end{align}
where $\hnabla$ is the cone connection of $(\nabla, \be)$. It follows that the cone connection associated with $(\nabla, \be)$ and the cone connection associated with $(\tnabla = \nabla + 2\ga_{(i}\delta_{j)}\,^{k}, \tilde{\be} =\be - \rho^{\ast}(\ga))$ are equal.
\end{proof}

Lemmas \ref{curvedevlemma} and \ref{gliftlemma} are needed in the proof of Lemma \ref{geodesiclemma}.
\begin{lemma}\label{curvedevlemma}
Let $G$ be a Lie group with identity element $e$, Lie algebra $\g$, and left-invariant Maurer-Cartan form $\mc$. Let $\si$ be a smooth $\g$-valued one-form on a nonempty open subinterval $I \subset \rea$.  For every $a_{0} \in I$ there exists a unique smooth immersion $c:I \to G$ such that $c(a_{0}) = e$ and $c^{\ast}(\omega_{G}) = \si$.
\end{lemma}
\begin{proof}
This is a special case of the more general statement with $I$ replaced by an open domain in $\rea^{n}$, and $\si$ replaced by a smooth $G$-valued one-form satisfying the Maurer-Cartan identity $d\si + [\si, \si] = 0$, whose proof can be found in \cite[section $3.5$]{Koszul-lectures} or \cite[Theorem $6.1$]{Sharpe}.
\end{proof}

Let $G$ be either $\reat$ or $S^{1}$ and in either case identify its Lie algebra with $\rea$. Let $\rho:N \to M$ be a principal $G$-bundle. Let $I \subset \rea$ be an interval containing $0$. The \emph{base curve} of a smoothly immersed curve $\hat{\ga}:I \to N$ is the curve $\ga = \rho \circ \hat{\ga}:I \to M$ in $M$. It is an immersion if $\hat{\ga}$ is everywhere transverse to the vertical. A \emph{lift} to $N$ of a smooth immersion $\ga:I \to M$ is a smooth immersion $\hat{\ga}:I \to N$ such that $\rho \circ \hat{\ga} = \ga$. If $u \in \rho^{-1}(\ga(0))$, the lift is \emph{based} at $u$ if $\hat{\ga}(0) = u$. 
If $\be$ is a principal $G$-connection on $N$, a \emph{$\be$-horizontal lift} of $\ga$ is a lift $\hat{\ga}$ of $\ga$ such that $\hat{\ga}^{\ast}(\be) = 0$.

Since $G$ is abelian its left-invariant Maurer-Cartan form $\mc$ is simply an invariant one-form invariant. For $t \in \rea$ the exponential map $\exp_{G}$ takes the forms $\exp_{\reat}(t) = e^{t}$ and $\exp_{S^{1}}(t) = e^{\j t}$, and in either case $\mc = dt$. Let $\rad$ be the fundamental vector field generating the principal $G$-action. If $\hat{\ga}$ is a lift of $\ga$ then so is $R_{c}(\hat{\ga})$ for any $c \in \cinf(I, G)$. In this case
\begin{align}\label{drcg}
\tfrac{d}{dt}R_{c}(\hat{\ga}) = TR_{c}(\hat{\ga})(\dot{\hat{\ga}}) + c^{\ast}(\mc)(\tfrac{d}{d t})\rad_{R_{c}(\hat{\ga})},
\end{align}
in which, for example, $c^{\ast}(\omega_{\reat})(\tfrac{d}{d t}) = c^{-1}\dot{c}$.

\begin{lemma}\label{gliftlemma}
Let $\rho:N \to M$ be a principal $G$-bundle and let $\be$ be a principal $G$-connection on $N$, where $G$ is $\reat$ or $S^{1}$. Let $I \subset \rea$ be a nonempty open interval and let $\ga:I \to M$ be a smooth immersion. Given $u \in \rho^{-1}(\ga(0))$ there is a unique $\be$-horizontal lift of $\ga$ based at $u$. 
\end{lemma}

\begin{proof}
Choose an open neighborhood $U$ of $\ga(0)$ over which $N$ trivializes, so that there is a smooth $G$-equivariant diffeomorphism $\phi: U \times G \to \rho^{-1}(U)$ such that $\rho \circ \phi$ equals the projection $U\times G \to U$ onto the first factor. Let $I$ be an open subinterval of $\rea$ containing $0$ and such that $\ga(I) \subset U$. Define a lift $\hat{\ga}:I \to \rho^{-1}(U)$ of $\ga:I \to U$ by $\hat{\ga}(t) = \phi(\ga(t), e)$. 
By Lemma \ref{curvedevlemma} there is a unique smooth immersion $c:I \to G$ such that $c^{\ast}(\mc) = -\hat{\ga}^{\ast}(\be)$ and $c(0) = e$. Define a lift $\bar{\ga}:I \to \rho^{-1}(U)$ of $\ga:I \to U$ by $\bar{\ga}(t) = R_{c(t)}\hat{\ga}(t)$. Then $\rho\circ \bar{\ga}(0) = \rho \circ \hat{\ga}(0) = \ga(0)$ and, by \eqref{drcg} and the construction of $c$,
\begin{align}\label{belift}
\bar{\ga}^{\ast}(\be)(\tfrac{d}{dt}) = \be(\tfrac{d}{dt}R_{c}(\hat{\ga})) = R_{c}^{\ast}(\be)(\dot{\hat{\ga}}) + c^{\ast}(\mc)(\tfrac{d}{d t}) = \hat{\ga}^{\ast}(\be)(\tfrac{d}{dt})+ c^{\ast}(\mc)(\tfrac{d}{d t}) = 0.
\end{align}
This shows there exists a $\be$-horizontal lift of $\ga$ based at $u$. The computation \eqref{belift} also shows the uniqueness, for if $\tilde{\ga}:I \to \rho^{-1}(U)$ is another $\be$-horizontal lift based at $u$, then there is a smooth immersion $c:I \to G$ such that $\bar{\ga} = R_{c}\tilde{\ga}$ and, as in \eqref{belift}, $0 =\bar{\ga}^{\ast}(\be) = \tilde{\ga}^{\ast}(\be) + c^{\ast}(\mc) = c^{\ast}(\mc)$, so that $c$ is a constant map. Since $\tilde{\ga}(0) = \bar{\ga}(0)$, $c(0) = e$, so $c(I) = e$ and $\tilde{\ga} = \bar{\ga}$.
\end{proof}

Lemma \ref{geodesiclemma} is needed to prove that a radiant structure $(\hnabla, \rad)$ on the total space of $\rho:N \to M$ fibers over the projective structure it induces on $M$ (in the sense of Definition \ref{fibereddefined}). Although this is geometrically clear, a precise statement requires some care. 

\begin{lemma}\label{geodesiclemma}
Let $\rho:N \to M$ be a principal $S^{1}$-bundle or principal $\reat$-bundle and let $\rad$ be the fundamental vertical vector field generated by the principal action. Let $\hnabla$ be a torsion-free affine connection on $N$ invariant under the principal action and constituting with $\rad$ a radiant structure. Let the torsion-free affine connection $\nabla$ and the tensor $Q \in \Ga(\tensor^{2}\ctm)$ be associated with $\be \in \prin(\rho:N \to M)$ as in \eqref{tripleconnection} of Lemma \ref{invariantinducedlemma}. Let $\en$ be the projective structure generated by $\nabla$.
Let $I \subset \rea$ be a nonempty open subinterval containing $0$. 
\begin{enumerate}
\item\label{liftedprojgeodesic} Let $\hat{\ga}:I \to N$ be a smooth immersion with base curve $\ga:I \to M$. The image $\hat{\ga}(I)$ of $\hat{\ga}$ is a projective geodesic of $\hnabla$ if and only if there is $b \in \cinf(I)$ such that
\begin{align}\label{hnablaprojconditions}
&\nabla_{d/dt}\dot{\ga} = (b-2A)\dot{\ga},& &Ab = \dot{A} + A^{2} + Q(\dot{\ga}, \dot{\ga}),
\end{align} 
where $A = \be(\dot{\hat{\ga}}) \in \cinf(I)$.
In this case $\hnabla_{\dot{\hat{\ga}}}\dot{\hat{\ga}} = b\dot{\hat{\ga}}$ and the image $\ga(I)$ is contained in a projective geodesic of $\nabla$. 
\item\label{liftedprojgeodesic2} Suppose that the path of the smooth immersion $\ga:I \to M$ is a projective geodesic of $\nabla$. For each $u \in \rho^{-1}(\ga(0))$ there is a subinterval $J \subset I$ containing $0$ in its interior and a lift $\hat{\ga}:J \to N$ of $\ga$ based at $u$ and such that $\dot{\hat{\ga}}\wedge \hnabla_{\dot{\hat{\ga}}}\dot{\hat{\ga}} = 0$ on $J$ (so that $\hat{\ga}(J)$ is a projective geodesic of $\hnabla$).
\item\label{fiberingcorollary} $(\hnabla, \rad)$ is a conelike radiant structure fibering over $(M, \en)$, where $\en$ is the projective structure generated by $\nabla$. 
\end{enumerate}
\end{lemma}

\begin{proof}
Let $\hat{\ga}:I \to N$ be a smooth immersion with base curve $\ga:I \to M$. Choose a smooth vector field $X$ and smooth function $f$ defined in a neighborhood of the image of $\ga(I)$ such that $\dot{\ga} = X_{\ga}$ and $f\circ \ga = A= \be(\dot{\hat{\gamma}}) \in \cinf(I)$, so that there holds $\dot{\hat{\ga}} = (\hat{X} + \rho^{\ast}(f)\rad)_{\hat{\ga}}$. By \eqref{tripleconnection}, 
\begin{align}\label{fxlifted}
\hnabla_{\hat{X} + \rho^{\ast}(f)\rad}(\hat{X} + \rho^{\ast}(f)\rad) = \widehat{\nabla_{X}X} + 2\rho^{\ast}(f)\hat{X} + \rho^{\ast}\left(Q(X, X) + df(X) + f^{2}\right)\rad.
\end{align}
Differentiating $f \circ \ga = A$ gives $df(X)_{\ga} = \dot{A}$, and so \eqref{fxlifted} yields
\begin{align}\label{liftedgeodesics}
\hnabla_{d/dt}\dot{\hat{\ga}} &= \widehat{\nabla_{d/dt}\dot{\ga}} + 2A\hat{\dot{\ga}} + \left( \dot{A} + A^{2} + Q(\dot{\ga}, \dot{\ga})\right)\rad_{\hat{\ga}},\\
\label{liftedgeodesicswedge}
\dot{\hat{\ga}} \wedge \hnabla_{d/dt}\dot{\hat{\ga}} &=\widehat{\dot{\ga} }\wedge \widehat{\nabla_{d/dt}\dot{\ga}} + \left(\dot{A} - A^{2} + Q(\dot{\ga}, \dot{\ga}) \right) \widehat{\dot{\ga}}\wedge \rad_{\hat{\ga}} - A\widehat{\nabla_{d/dt}\dot{\ga}}\wedge \rad_{\hat{\ga}}.
\end{align}
If \eqref{liftedgeodesicswedge} vanishes on $I$, then $\dot{\ga}\wedge \nabla_{d/dt}\dot{\ga} = 0$, so there is $g \in \cinf(I)$ such that $\nabla_{d/dt}\dot{\ga} = g\dot{\ga}$, and $\dot{A} - A^{2} + Q(\dot{\ga}, \dot{\ga}) = gA$ on $I$, and, conversely, if there is such a $g$, then \eqref{liftedgeodesicswedge} vanishes. Writing $g = b- 2A$ with $b \in \cinf(I)$, it follows that $\dot{\hat{\ga}} \wedge \hnabla_{d/dt}\dot{\hat{\ga}} = 0$ on $I$ if and only if there hold \eqref{hnablaprojconditions}. Substituting \eqref{hnablaprojconditions} into \eqref{liftedgeodesics} yields
\begin{align}\label{hatsideriv0}
\hnabla_{d/dt}\dot{\hat{\ga}} & = b\widehat{\dot{\ga}} + \left( \dot{A} + A^{2} + Q(\dot{\ga}, \dot{\ga})\right)\rad_{\hat{\ga}} = b\left(\widehat{\dot{\ga}} + A\rad_{\hat{\ga}}\right)=  b\dot{\hat{\ga}}.
\end{align}
This proves \eqref{liftedprojgeodesic}.

Now suppose that the path of the smooth immersion $\ga:I \to M$ is a projective geodesic of the projective structure $\en$ generated by $\nabla$ and fix $u \in \rho^{-1}(\ga(0))$. By the assumption that $\ga(I)$ is a projective geodesic there is $g \in \cinf(I)$ such that $\nabla_{d/dt}\dot{\ga} = g \dot{\ga}$.

Choose a vector field $X$ such that $\dot{\ga} = X_{\ga}$. Then the integral curve $\hat{\ga}$ of the $\be$-horizontal lift $\hat{X}$ such that $\hat{\ga}(0) = u$ is a lift of $\ga$ based at $u$. Any other lift $\hat{\si}$ of $\ga$ based at $u$ has the form $\hat{\si} = R_{c}(\hat{\ga})$ for some $c \in \cinf(I, G)$ satisfying $c(0) = 1$. By \eqref{drcg}, 
\begin{align}\label{dothatsi}
\dot{\hat{\si}} = TR_{c}(\hat{\ga})(\dot{\hat{\ga}}) + c^{\ast}(\mc)(\tfrac{\pr}{\pr t})\rad_{\hat{\si}} = TR_{c}(\hat{\ga})(\hat{X}_{\hat{\ga}}) + c^{\ast}(\mc)(\tfrac{\pr}{\pr t})\rad_{\hat{\si}} = \hat{X}_{\hat{\si}} + c^{\ast}(\mc)(\tfrac{\pr}{\pr t})\rad_{\hat{\si}},
\end{align}
so that $\hat{\si}$ is an integral curve of $\hat{X} + f\rad$ where $f$ is any smooth function defined on a neighborhood of $\ga(I)$ and satisfying $f \circ \ga = c^{\ast}(\mc)(\tfrac{\pr}{\pr t})$. Let $A = \be(\dot{\hat{\si}})$. By \eqref{dothatsi}, $A$ equals $c^{\ast}(\mc)(\tfrac{\pr}{\pr t})$. 
There holds \eqref{fxlifted},
and differentiating $f \circ \ga = A$ gives $df(X)_{\ga} = \dot{A}$, and so
\begin{align}\label{hatsideriv}
\begin{split}
\hnabla_{\dot{\hat{\si}}}\dot{\hat{\si}} & = (g + 2A)\dot{\hat{\si}} + \left( \dot{A} - A^{2} + Q(\dot{\ga}, \dot{\ga}) - gA\right)\rad_{\hat{\si}}\\
& = b\dot{\hat{\si}} + \tfrac{1}{2}\left( \dot{b} - \tfrac{1}{2}b^{2}- \left( \dot{g} - \tfrac{1}{2}g^{2} - 2Q(\dot{\ga}, \dot{\ga})\right)\right)\rad_{\hat{\si}},
\end{split}
\end{align}
where $b = g + 2A$. By \eqref{hatsideriv}, the image of $\hat{\si}(I)$ is contained in a projective geodesic of $\hnabla$ if and only if
\begin{align}\label{derin2}
\dot{b} - \tfrac{1}{2}b^{2}- \left( \dot{g} - \tfrac{1}{2}g^{2} - 2Q(\dot{\ga}, \dot{\ga})\right) = 0.
\end{align}
There is an open interval $J \subset I$ containing $0$ on which \eqref{derin2} has a unique solution $b \in \cinf(J)$ such that $b(0) = 0$. 
A $\hat{\si}$ which is a projective geodesic is found by solving $A =  c^{\ast}(\mc)(\tfrac{\pr}{\pr t})$ for $c$ with $A = \tfrac{1}{2}(b-g)$. Precisely, once $b$ has been obtained, $A$ is determined from $b$ and $g$ by $A = (b - g)/2$, and, by Lemma \ref{curvedevlemma}, there is a unique $c \in \cinf(J, G)$ such that $c(0)= e$ and $c^{\ast}(\mc)(\tfrac{\pr}{\pr t}) = A$. (In the case $G = \reat$, the solution is given explicitly as $c(t) = \exp{\left(\tfrac{1}{2}\int_{0}^{t}(b(v) - g(v))\,dv\right)}$.) The resulting lift $\hat{\si}$ solves $\hnabla_{\dot{\hat{\si}}}\dot{\hat{\si}} = (g+ 2A)\dot{\hat{\si}}= b\dot{\hat{\si}}$ and is based at $u$. This proves \eqref{liftedprojgeodesic2}.
By Corollary \ref{invariantconelikecorollary}, $(\hnabla, \rad)$ is conelike, so to show \eqref{fiberingcorollary} it is needed only to show that it fibers over $(M, \en)$.

For $p \in M$, let $\Sigma$ be a planelike surface in $N$ containing $u \in \rho^{-1}(p)$. Let $X \in T_{u}\Sigma$ be transverse to $\rad_{u}$. Let $\hat{\ga}:I \to N$ be a $\hnabla$-geodesic such that $\hat{\ga}(0) = u$ and $\dot{\hat{\ga}}(0) = X_{u}$. Because $\Sigma$ is $\hnabla$-totally geodesic, $\hat{\ga}(I) \subset \Sigma$. By \eqref{liftedprojgeodesic}, the image $\ga(I)$ of the base curve $\ga = \rho \circ \hat{\ga}:I \to M$ is a $\en$-projective geodesic and $\ga:I \to M$ is a parametrization of this projective geodesic. Given $Y = X + f\rad_{u} \in T_{u}\Sigma$, let $\tilde{\ga}:J \to M$ be a $\hnabla$-geodesic such that $\tilde{\ga}(0) = u$ and $\dot{\tilde{\ga}}(0) = Y_{u}$. As before, the base curve $\rho \circ \tilde{\ga}$ is a projective parametrization of a $\en$-projective geodesic. As 
\begin{align}
\begin{split}
\tfrac{d}{dt}\big|_{t = 0}\rho(\tilde{\ga}(t)) &= T\rho(u)(\dot{\tilde{\ga}}(0)) = T\rho(u)(Y_{u}) = T \rho(u)(X_{u}) = T\rho(u)(\dot{\hat{\ga}}(0)) = \tfrac{d}{dt}\big|_{t = 0}\rho(\hat{\ga}(t))
\end{split}
\end{align}
the base curves $\rho\circ \hat{\ga}$ and $\rho\circ \tilde{\ga}$ equal $p$ and have the same direction at $t = 0$; since they parametrize $\en$-projective geodesics they must parametrize the same projective geodesic. If $Y$ is replaced by its a multiple by a function not zero at $u$ then the $\hnabla$-geodesic $\tilde{\ga}$ is replaced by an appropriate linear reparametrization of itself, and the image of its base curve is unchanged. Since the intersection of $\Sigma$ with a $\hnabla$-geodesically convex neighborhood of $u$ is foliated by $\hnabla$-geodesics, the preceding shows that $\rho(\Sigma)$ is contained in a $\en$-projective geodesic.
\end{proof}

\begin{lemma}\label{projectiveparametrizationlemma}
Let $\rho:N \to M$ be a $\reat$ or principal $S^{1}$ bundle and let $\hnabla$ be the cone connection associated with a torsion-free affine connection $\nabla$ on $M$ and a principal connection $\be$ on $N$. Let $I \subset \rea$ be a nonempty open subinterval containing $0$. 
\begin{enumerate}
\item\label{liftedprojgeodesicb} If the immersion $\hat{\ga}:I \to N$ is a $\hnabla$-geodesic, then its base curve $\ga:I \to M$ is a projective parametrization of the projective geodesic $\ga(I)$.
\item\label{liftedprojgeodesic3} Suppose that the smooth immersion $\ga:I \to M$ is a projective parametrization of a projective geodesic of $\nabla$. For each $u \in \rho^{-1}(\ga(0))$ there is a subinterval $J \subset I$ containing $0$ in its interior and a lift $\hat{\ga}:J \to N$ of $\ga$ based at $u$ and such that $\hnabla_{\dot{\hat{\ga}}}\dot{\hat{\ga}} = 0$ on $J$. In particular, there is a reparametrization of $\ga$ defined in an open neighborhood of $0$ so that there is a lift of $\ga$ based at $u$ which is a $\hnabla$-geodesic, and this parametrization is determined up to precomposition with a linear fractional transformation. 
\end{enumerate}
\end{lemma}

\begin{proof}
Let the notations be as in the proof of \eqref{liftedprojgeodesic} of Lemma \ref{geodesiclemma}. Because $\hat{\ga}:I \to N$ is a $\hnabla$-geodesic, by \eqref{liftedprojgeodesic} of Lemma \ref{geodesiclemma} there holds \eqref{hnablaprojconditions} with $b = 0$ and, by \eqref{qdefined}, $Q(\dot{\ga}, \dot{\ga}) = P(\dot{\ga}, \dot{\ga})$, where $P$ is the projective Schouten tensor of $\en$. In this case, $g = 2A$ and \eqref{hnablaprojconditions} yields $\dot{g} - \tfrac{1}{2}g^{2} - 2P(\dot{\ga}, \dot{\ga}) = 0$, so that $\ga:I \to M$ is a projective parametrization of its image. This proves \eqref{liftedprojgeodesicb}.

Now suppose the smooth immersion $\ga:I \to M$ is a projective parametrization of a projective geodesic of the projective structure $\en$ generated by $\nabla$ and fix $u \in \rho^{-1}(\ga(0))$. By the assumption that $\ga(I)$ is a projective geodesic there is $g \in \cinf(I)$ such that $\nabla_{d/dt}\dot{\ga} = g \dot{\ga}$. Let the notations be as in the proof of \eqref{liftedprojgeodesic2} of Lemma \ref{geodesiclemma}. By the proof of \eqref{liftedprojgeodesic2} of Lemma \ref{geodesiclemma}, any other lift $\hat{\si}$ of $\ga$ based at $u$ has the form $\hat{\si} = R_{c}(\hat{\ga})$ for some $c \in \cinf(I, G)$ satisfying $c(0) = 1$, and the function $A = \be(\dot{\hat{\si}})$ satisfies \eqref{derin2} where $b = g + 2A$. Since $\ga:I \to M$ is a projective parametrization of $\ga(I)$, $\dot{g} - \tfrac{1}{2}g^{2} - 2P(\dot{\ga}, \dot{\ga}) = 0$, and the unique solution, $b(t)$, of \eqref{derin2} with initial condition $b(0) = 0$ is the function that is identically zero; that is $b(t) = 0$ for all $t$. The resulting lift $\hat{\si}$ solves $\hnabla_{\dot{\hat{\si}}}\dot{\hat{\si}} = (g+ 2A)\dot{\hat{\si}}= b\dot{\hat{\si}}= 0$ and is based at $u$. 
Moreover, by Lemma \ref{1dprojstructurelemma}, if $\ga:I \to M$ is any parametrization of a projective geodesic of $\nabla$ there is, after possibly passing to a subinterval $J \subset I$, a reparametrization, uniquely determined up to a linear fractional transformation, that is a projective parametrization. This proves \eqref{liftedprojgeodesic3}.
\end{proof}

\begin{proof}[Proof of Theorem \ref{extendedthomastheorem}]
By Lemmas \ref{preextendedthomaslemma} and \ref{extendedshiftlemma} there is associated with the extended projective structure $[\nabla, \be]$ a unique torsion-free affine connection $\hnabla$ on $N$ that with $\rad$ constitutes a conelike radiant structure invariant under the principal action, having antisymmetric Ricci tensor, and inducing $[\nabla, \be]$. By \eqref{fiberingcorollary} of Lemma \ref{geodesiclemma}, $(\hnabla, \rad)$ fibers over $(M, \en)$, and the $\rho$-preimages of the projective geodesics of the projective structure $\en$ generated by $\nabla$ are the planelike surfaces in $N$. From Lemma \ref{projectiveparametrizationlemma} there follows the stronger statement that the base curve of a parametrized curve in $N$ is a projective parametrization of a projective geodesic of $\en$ if and only if the curve in $N$ is a $\hnabla$-geodesic. That the curvature of $\hnabla$ has the form \eqref{hnablabecurvatures} follows from \eqref{qcurvatures} of Lemma \ref{preextendedthomaslemma}. This proves \eqref{extendedthomas1}. 

There remains to prove \eqref{extendedthomas2}. Suppose given a conelike radiant structure $(\hat{D}, \rad)$ on $N$ invariant under the principal action. Since it is invariant under the principal action, by Lemma \ref{invariantinducedlemma}, $(\hD, \rad)$ induces an extended projective structure $[\nabla, \be]$ on $\rho:N \to M$, and, by \eqref{fiberingcorollary} of Lemma \ref{geodesiclemma}, $(\hD, \rad)$ fibers over the underlying projective structure $\en$ on $M$. Because $\hD$ is $\rad$-invariant, $\er(\hD)_{i} = 0$ by \eqref{lieradnabla}, so by Theorem \ref{conenormalizationtheorem}, there is a conelike $\rad$-invariant radiant structure $(\tnabla, \rad)$ having antisymmetric Ricci tensor, having the same planelike surfaces as $(\hD, \rad)$, and inducing on $|\Det T^{\ast}N|$ the same connection as $\hD$. By \eqref{gcnn} of Theorem \ref{conenormalizationtheorem}, $(\tnabla, \rad)$ is invariant under the principal action.
Since it is invariant under the principal action, by Lemma \ref{invariantinducedlemma}, $(\tnabla, \rad)$ induces an extended projective structure on $M$, and, by \eqref{fiberingcorollary} of Lemma \ref{geodesiclemma}, $(\tnabla, \rad)$ fibers over its underlying projective structure on $M$. Since $(\hat{D}, \rad)$ fibers over $(M, \en)$ and $(\tnabla, \rad)$ has the same planelike surfaces as $(\hD, \rad)$, this underlying projective structure must be $\en$. This shows that $\hat{D}$ and $\tnabla$ induce on $M$ extended projective structures having the same underlying projective structure $\en$. By Lemma \ref{sameinducedlemma}, the induced extended projective structures are the same.
\end{proof}

\begin{remark}
By \eqref{etepc} of Theorem \ref{extendedthomastheorem} (which follows from the last equality of \eqref{tc3}), the cone connection of $(\nabla, \be)$ is Ricci-flat if and only if $2R_{[ij]} = (n+1)\om_{ij}$, that is the curvature of the connection induced by $\nabla$ on the density bundle $|\Det M|^{1/(n+1)}$ equals the curvature of $\be$. Since the antisymmetric part of the Ricci tensor is always exact, every torsion-free affine connection is projectively equivalent to one with symmetric Ricci tensor. It follows that if the bundle $\rho:N \to M$ admits a flat connection then there are pairs $(\nabla, \be)$ with vanishing associated two-form, so with Ricci-flat cone connection. 
\end{remark}

\begin{remark}
The geometric content of the vanishing of the symmetric part of the Ricci tensor in Theorem \ref{extendedthomastheorem} is that this condition is what forces the linkage between the parametrization of the geodesics of $\hnabla$ and the projective parametrization of the projective geodesics of $\en$ in \eqref{ett1} of Theorem \ref{extendedthomastheorem}.
\end{remark}

\section{Classical Thomas connection revisited}\label{thomassection}
Theorem \ref{extendedthomastheorem} can be viewed as generalizing to $S^{1}$-bundles an essentially local result that goes back to T. Y. Thomas \cite{Thomas-affine, Thomas-book, Veblen-Thomas-paths}, that associates with a projective structure on an $n$-manifold a Ricci-flat affine connection $\hnabla$ on the total space of a trivial $\reat$ principal bundle over $M$. Here $\hnabla$ is called the \emph{(classical) Thomas connection} of $\en$. 
An appropriate specialization of Theorem \ref{extendedthomastheorem} yields Theorem \ref{classicalthomastheorem}, which gives a gauge equivariant, global, functorial formulation of the classical Thomas connection. For modern accounts of Thomas's construction and the closely related tractor formalism for parabolic geometries, see, for example, \cite{Bailey-Eastwood-Gover, Cap-Gover, Dhooghe, Dhooghe-VanVlierden}, and in particular \cite{Gover-Neusser-Willse} which gives a streamlined and careful account situated in the general context of parabolic geometries.

The Thomas construction is revisited here primarily to elucidate two points. The first is the functorial nature of the construction. As the construction is essentially local, this amounts to clarifying what characterizes the classical Thomas connection in the sense of what conditions on $\hnabla$ determine it uniquely. The usual accounts establish existence but are somewhat vague with respect to uniqueness. Clarifying this point is relevant for the problem of the recognition of a Thomas connection, that is, given a connection, deciding whether it is a Thomas connection. Theorem \ref{conjugatethomastheorem} gives an example where such a characterization is used to decide when the conjugate of a Thomas connection with respect to a metric is again a Thomas connection. 

The second point is to clarify the global nature of the construction, in particular that the construction makes sense on real line bundles that need not be trivial or canonically trivialized (as is often assumed). This amounts to describing how to extend the construction to line bundles over $M$ that transform like density bundles without being density bundles, in particular in the sense that they need not admit any natural action of $\diff(M)$. This amounts to working with a pseudo-hyperplane bundle as in Definition \ref{tautdefinition} instead of the necessarily topologically trivial density bundle $|\Det TM|^{1/(n+1)}$. In the setting of Theorem \ref{extendedthomastheorem}, if $N = |\Det \ctm|^{1/(n+1)} \setminus \zs(M)$ and if $\be$ is taken to be the principal connection induced by a connection $\nabla$ on $M$, then $(\nabla, \be)$ determines an extended projective structure and the cone connection of $[\nabla, \be]$ depends only on the projective equivalence class $\en$. This recovers the classical Thomas connection of the projective structure $\en$. The special feature of the bundle of nonvanishing densities is that an affine connection induces a principal connection on this bundle, so any projective structure determines in a canonical way an extended projective structure. Although this means that the original construction of Thomas works for any $\reat$-principle bundle with the property that its sections transform locally like densities, it is what fails for an arbitrary principal $\reat$-bundle having no additional structure. The additional structure needed is formally analogous to a spin structure or square root of the canonical line bundle on a Kähler manifold, namely it is an identification of some power of the principal $\reat$-bundle, or some associated line bundle, with some density bundle, and this is what the notion of pseudo-hyperplane bundle as in Definition \ref{tautdefinition} formalizes. (This discussion can be situated in the general context of parabolic geometries as in \cite[Sections $3.5$ and $4.15$]{Cap-Gover}, but the more concrete approach taken here, while less powerful, is useful in applications.) 

A torsion-free affine connection $\nabla$ on $M$ induces a covariant derivative, also denoted $\nabla$, on a pseudo-hyperplane line bundle $\emf \to M$ as follows. For a local section $e$ of $\emf$ and a vector field $X$ on $M$, define 
\begin{align}\label{nablaemfdefined}
\nabla_{X}e = \tfrac{1}{2n+2}e^{-2n-1}\tensor\nabla_{X}(e^{2n+2}), 
\end{align}
where $\nabla_{X}(e^{2n+2})$ denotes the covariant derivative of the section local section $e^{2n+2}$ of $(\det \ctm)^{\tensor 2}$.
A tensor taking values in $\emf^{k}$ is said to have \emph{p-weight} $k$ (where \emph{p-weight} abbreviates \emph{projective weight}). Locally a section of $\emf^{k}$ behaves like a $-k/(n+1)$-density.
In particular, if $\tnabla = \nabla + 2\ga_{(i}\delta_{j)}\,^{k}$ and $e \in \Ga(\emf^{k})$, then
\begin{align}\label{pwktransform}
&\tnabla_{i}e = \nabla_{i}e + k \ga_{i}e,&
&\nabla_{[i}\nabla_{j]}e = \tfrac{k}{2(n+1)}R_{ijp}\,^{p}e = -\tfrac{k}{n+1}R_{[ij]}e = kP_{[ij]}e.
\end{align}
It is straightforward to check that the the right-hand side of \eqref{nablaemfdefined} with the Lie derivative $\lie_{X}$ in place of $\nabla_{X}$ defines an action of the Lie algebra $\vect(M)$ of vector fields on $\Ga(\emf)$. This means that sections of $\emf$ locally behave as densities, so $\Ga(\emf)$ is a module over$\vect(M)$, even though in general $\diff(M)$ need not act on $M$.

\begin{lemma}\label{inducedeplemma}
Let $M$ be an $n$-manifold equipped with a pseudo-hyperplane bundle $\emf \to M$ and let $\rho: \F = \frameb(\emf^{-1}) \to M$ be the frame bundle of the dual pseudo-tautological bundle. 
\begin{enumerate}
\item \label{inducedcurv} A torsion-free affine connection $\nabla$ on $M$ induces a principal $\reat$-connection $\tau \in \prin(\F)$, and the curvature $\om \in \Ga(\ext^{2}\ctm)$ of $\tau$ satisfies $\om_{ij} = -\tfrac{1}{n+1} R_{ijp}\,^{p} = \tfrac{2}{n+1} R_{[ij]} = -2P_{[ij]}$. 
\item A projective structure $\en$ on $M$ determines in a canonical way an extended projective structure $[\nabla, \tau]$ on the principal $\reat$-bundle $\rho:\F  \to M$ for which the associated two-form $\epc_{ij}$ vanishes identically.
\item\label{linftdefined} A projective structure $\en$ on $M$ determines a Lie algebra embedding $\linft^{\F}:\Ga(TM) \to \Ga(T\F)$ satisfying $T\rho(\linft^{\F}(X)) = X$ and $R_{r}^{\ast}(\linft^{\F}(X)) = \linft(X)$, and such that, for $\nabla \in \en$ inducing $\tau \in \prin(\F)$, there holds
\begin{align}\label{wtdf}
\linft^{\F}(X) = \hat{X} - \tfrac{1}{n+1}\rho^{\ast}(\div_{\nabla}(X))\eul^{\F},
\end{align}
where $X \in \Ga(TM) \to \hat{X} \in \Ga(\F)$ is the horizontal lift determined by $\tau$.
If $\emf\to M$ is a hyperplane line bundle, then $\lie_{\linft^{\F}(X)}\mu^{\F} = 0$, so that $\lie_{\linft^{\F}(X)}\Psi^{\F} = 0$, while if $\emf \to M$ is a pseudo-hyperplane line bundle, then $\lie_{\linft^{\F}(X)}|\Psi^{\F}| = 0$.
\end{enumerate}
\end{lemma}
\begin{proof}
Regarding $\F$ as $\emf^{-1} \setminus \{0\}$, $\tau \in \prin(\F)$ is defined by the requirement that if $\tilde{u} \in \cinf(\F)$ is the homogeneity $1$ function corresponding with a smooth section $u \in \Ga(\emf)$, then the homogeneity $1$ smooth function on $\F$ corresponding with $\nabla_{X}u$ is $d\tilde{u}(\hat{X})$, where $\hat{X}$ is the $\tau$-horizontal lift of $X \in \Ga(TM)$. By definition the curvature $\om \in \Ga(\ext^{2}\ctm)$ satisfies $d\tau = \rho^{\ast}(\om)$. On the one hand, by \eqref{pwktransform}, $2\nabla_{[i}\nabla_{j]}u = \tfrac{1}{n+1} R_{ijp}\,^{p}u = -\tfrac{2}{n+1}R_{[ij]}u$. On the other hand, for $X, Y \in \Ga(TM)$, the homogeneity $1$ function on $\F$ corresponding with $\nabla_{X}\nabla_{Y}u - \nabla_{Y}\nabla_{X}u - \nabla_{[X, Y]}u$ equals
\begin{align}
\begin{split}
d\widetilde{\nabla_{Y}u}(\hat{X}) &- d\widetilde{\nabla_{X}u}(\hat{Y}) - d\tilde{u}(\widehat{[X, Y]})\\
& = \hat{X}d\tilde{u}(\hat{Y}) - \hat{Y}d\tilde{u}(\hat{X}) - d\tilde{u}\left([\hat{X}, \hat{Y}] + d\tau(\hat{X}, \hat{Y})\eul^{\F}\right) = -d\tau(\hat{X}, \hat{Y})\tilde{u} = -\om(X, Y)\tilde{u},
\end{split}
\end{align}
where there has been used $d\tilde{u}(\eul^{\F}) = \tilde{u}$. This shows \eqref{inducedcurv}. 

Given a projective structure $\en$, each $\nabla \in \en$ induces $\tau \in \prin(\F)$. Let $u  \in \Ga(\emf)$.
If $\bnabla = \nabla + 2\ga_{(i}\delta_{j)}\,^{k}$, then, by \eqref{pwktransform}, $\bnabla_{i}u = \nabla_{i}u + \ga_{i}u$, and it follows that the horizontal lift of $X$ with respect to the principal connection $\bar{\tau}\in \prin(\F)$ induced by $\bnabla$ is $\hat{X}  + \rho^{\ast}(\ga(X))\eul^{\F}$ and $\bar{\tau} = \tau - \rho^{\ast}(\ga)$. These identities imply $[\nabla, \tau]$ is an extended projective structure on $\rho: \F \to M$ and also that  \eqref{wtdf} does not depend on the choice of $\nabla \in \en$. 
By \eqref{inducedcurv}, $\om_{ij} = \tfrac{2}{n+1}R_{[ij]}$, where $d\tau = \rho^{\ast}(\om)$, so the two-form $\epc_{ij}$ associated with $[\nabla, \tau]$ satisfies $\epc_{ij} = \om_{ij} - \tfrac{2}{n+1}R_{[ij]} = 0$. The properties $T\rho(\linft^{\F}(X)) = X$ and $R_{r}^{\ast}(\linft^{\F}(X)) = \linft(X)$ hold by the definition \eqref{wtdf}.
That \eqref{wtdf} defines a Lie algebra embedding can be checked directly using the identities
\begin{align}\label{riccocycle}
\begin{aligned}
\lie_{X}\div_{\nabla}(Y) &- \lie_{Y}\div_{\nabla}(X) -\div_{\nabla}([X, Y]) = -\ric(X, Y) + \ric(Y, X),\\
[\hat{X}, \hat{Y}] & = \widehat{[X, Y]} - \tfrac{1}{n+1}(\ric(X, Y) - \ric(Y, X))\eul^{\F},
\end{aligned}
\end{align}
the first holding for any torsion-free affine connection $\nabla$ and the second following from Lemma \ref{inducedeplemma}. 

Suppose $\emf\to M$ is a hyperplane line bundle. From \eqref{wtdf} it follows that $T\Q(\linft^{\F}(X)) = \linft(X)$ where $\Q:\F \to \vdens$ is the map induced from the isomorphism $\emf^{n+1}\simeq \Det TM$, and the property $\lie_{\linft^{\F}(X)}\mu^{\F} = 0$ follows. Taking the exterior derivative shows $\lie_{\linft^{\F}(X)}\Psi^{\F} = 0$. If $\emf \to M$ is a pseudo-hyperplane line bundle, then the preceding argument applies locally to show that $\lie_{\linft^{\F}(X)}|\Psi^{\F}| = 0$.
\end{proof}

The conclusion of Lemma \ref{invariancelemma} requires argument because $\reat$ is disconnected.
\begin{lemma}\label{invariancelemma}
On an $n$-manifold $M$, let $\emf \to M$ be a pseudo-hyperplane line bundle and let $\rho: \F = \frameb(\emf^{-1}) \to M$ be equipped with the volume density $|\Psi^{\F}|$.
Let $\rad = \eul^{\F}$. A Ricci-flat, $\rad$-invariant, equiaffine conelike radiant structure $(\hnabla, \rad)$ on $(\F, \rad, |\Psi^{\F}|)$ is invariant under the principal action on $\F$.
\end{lemma}
\begin{proof}
Because $\hnabla$ is $\rad$-invariant, it is invariant under the principal action $R_{r}$ for $r > 0$, so it suffices to prove that $R_{-1}^{\ast}(\hnabla) = \hnabla$. Because $\rad$ and $|\Psi^{\F}|$ are invariant under the principal action, $R_{-1}^{\ast}(\hnabla)$ is a Ricci-flat, $\rad$-invariant, equiaffine conelike radiant structure on $(\F, \rad, |\Psi^{\F}|)$. Because $R_{-1}$ preserves the planes of $\hnabla$, $R_{-1}^{\ast}(\hnabla)$ has the same planes as $\hnabla$, and because $R_{1}^{\ast}(\hnabla)$ and $\hnabla$ preserve the same volume density $|\Psi^{\F}|$ they induce on $|\Det T^{\ast}\F|$ the same connection, so, by Theorem \ref{conenormalizationtheorem}, $R_{-1}^{\ast}(\hnabla) = \hnabla$.
\end{proof}

\begin{theorem}\label{classicalthomastheorem}
On an $n$-manifold $M$, let $\emf \to M$ be a pseudo-hyperplane line bundle and let $\rho: \F = \frameb(\emf^{-1}) \to M$ be the frame bundle of the dual pseudo-tautological line bundle equipped with the pseudo-Euler structure determined by the volume density $|\Psi^{\F}|$ and radiant Euler vector field $\rad = \eul^{\F}$.

There is a bijection between the set of projective structures on $M$ and the set of Ricci-flat, $\rad$-invariant, conelike equiaffine radiant structures on $(\F, \rad, |\Psi^{\F}|)$ associating with a projective structure $\en$ on $M$ a unique torsion-free affine connection $\hnabla$ on $\F$ such that $(\hnabla, \rad, |\Psi^{\F}|)$ is a Ricci-flat, conelike equiaffine radiant structure invariant under the principal action on $\F$ and fibering over $(M, \en)$.
\begin{enumerate}
\item\label{ctt1} The curvature $\hat{R}$ of $\hnabla$ satisfies
\begin{align}\label{thomascurvature}
&\hat{R}(\hat{X}, \hat{Y})\hat{Z}  = \widehat{B(X, Y)Z} + C(X, Y, Z)\rad,& 
&\hat{R}(\hat{X}, \hat{Y})\rad  = \hat{R}(\rad, \hat{X})\hat{Y} = \hat{R}(\hat{X}, \rad)\rad  = 0,
\end{align}
where $B$ and $C$ are the projective Weyl and Cotton tensors of $\en$ and the horizontal lifts of $X, Y, Z \in \Ga(TM)$ are with respect to any principal $\reat$-connection on $\F$. 
In particular, $\hnabla$ is flat if and only if $\en$ is projectively flat.
\item\label{ctt2} The planelike surfaces of $(\hnabla, \rad)$ are the $\rho$-preimages of the projective geodesics of $M$ and the base curve $\ga = \rho\circ \hat{\ga}:I \to M$ of a parametrized curve $\hat{\ga}:I \to \F$ is a projective parametrization of a projective geodesic of $\en$ if and only if $\hat{\ga}$ is a $\hnabla$-geodesic.
\item\label{ctt4} If $\emf \to M$ is a hyperplane line bundle, then $\hnabla$ preserves the volume form $\Psi^{\F} = d\mu^{\F}$. 
\item\label{ctt5} For all $X \in \Ga(TM)$, $\div_{\hnabla}(\linft^{\F}(X)) = 0$.
\end{enumerate}
\end{theorem}

The connection $\hnabla$ of Theorem \ref{classicalthomastheorem} is the (classical) \emph{Thomas connection} associated with $\en$.

\begin{proof}
By Lemma \ref{inducedeplemma}, $\en$ determines an extended projective structure $[\nabla, \be]$ on $\F$ in such a way that for $(\nabla, \be) \in [\nabla, \be]$, $\be$ is induced by $\nabla$, and such that $\epc_{ij} =0$.
Because $\epc_{ij} = 0$, by Theorem \ref{extendedthomastheorem} there is a unique torsion-free affine connection $\hnabla$ on $\F$ that with $\rad$ constitutes a Ricci-flat, conelike radiant structure invariant under the principal action on $\F$ and fibering over $(M, \en)$. Because $\epc_{ij} =0$, by \eqref{tripleconnection} and \eqref{qdefined} of Theorem \ref{extendedthomastheorem}, $\hnabla$ is defined by $\hnabla_{\hat{X}}\hat{Y} = \widehat{\nabla_{X}Y} + P(X, Y)\rad$ and $\hnabla \rad = \delta$, where $\hat{X}$ is the $\be$-horizontal lift of $X \in \Ga(TM)$, and satisfies \eqref{ctt1} and \eqref{ctt2}, in particular the curvature of $\hnabla$ has the form \eqref{thomascurvature}. 

Suppose $\emf \to M$ is a hyperplane line bundle.
Consider the Euler structure $(\F, \rad, \Psi^{\F}= d\mu^{\F})$ defined in Lemma \ref{canonicaleulerlemma} and let $\hnabla$ be the Thomas connection of $\en$. By Lemma \ref{liesslemma}, $\hnabla_{\rad}\Psi^{\F} = \lie_{\rad}\Psi^{\F} - (n+1)\Psi^{\F} = 0$, the last equality by the construction of $\Psi^{\F}$. Because $\imt(\rad)\mu^{\F} =0$ and $\lie_{\rad}\mu = (n+1)\mu$, for $X \in \Ga(TM)$,
\begin{align}
\imt(\rad)\left(\hnabla_{\hat{X}}\Psi^{\F}\right) = \lie_{\hat{X}}\mu^{\F} - \rho^{\ast}(\div_{\nabla}X)\mu^{\F} = \lie_{\linft^{\F}(X)}\mu^{\F} = 0,
\end{align}
the penultimate equality by \eqref{wtdf} and the last equality by \eqref{linftdefined} of Lemma \ref{inducedeplemma}. This shows that $\hnabla_{\hat{X}}\Psi^{\F}= 0$ for all $X \in \Ga(TM)$ and with $\hnabla_{\rad}\Psi^{\F} = 0$, this shows $\hnabla \Psi^{\F} = 0$. For a general pseudo-hyperplane line bundle $\emf\to M$, essentially the same argument shows that $\hnabla|\Psi^{\F}| = 0$. This shows that $\hnabla$ is equiaffine and shows \eqref{ctt4}. Said differently, the claim that $\hnabla$ is equiaffine is essentially local so that its validity for a general pseudo-hyperplane line bundle follows from its validity for a hyperplane line bundle.

The preceding establishes that associated with a projective structure $\en$ on $M$ there is a Ricci-flat, conelike equiaffine radiant structure on $\F$ invariant under the principal action and fibering over $(M, \en)$

By Lemma \ref{invariancelemma} a Ricci-flat, $\rad$-invariant, conelike equiaffine radiant structure $(\hD, \rad, |\Psi^{\F}|)$ on $\F$ is necessarily invariant under the principal action.
By \eqref{fiberingcorollary} of Lemma \ref{geodesiclemma}, $(\hD, \rad)$ fibers over the projective structure $\en$ it induces on $M$, so $(\hD, \rad, |\Psi^{\F}|)$ and $(\hnabla, \rad, |\Psi^{\F}|)$ are Ricci-flat, conelike equiaffine radiant structures on $(\F, \rad, |\Psi^{\F})$ invariant under the principal action and fibering over the same projective structure $\en$. By the uniqueness of $\hnabla$, $\hD = \hnabla$.

For $X \in \Ga(TM)$, by definition of the divergence and \eqref{linftdefined} of Lemma \ref{inducedeplemma}, $\div_{\hnabla}(\linft^{\F}(X))\Psi^{\F} = \lie_{\linft^{\F}(X)}\Psi^{F} - \hnabla_{\linft^{\F}(X)}|\Psi^{\F}| = d\lie_{\linft^{\F}(X)}\mu^{\F} =0$. This proves \eqref{ctt5} when $\emf$ is a hyperplane line bundle. For a general pseudo-hyperplane line bundle $\emf\to M$, $\hnabla|\Psi^{\F}| = 0$, so that $\div_{\hnabla}(\linft^{\F}(X))|\Psi^{\F}| = \lie_{\linft^{\F}(X)}|\Psi^{F}| - \hnabla_{\linft^{\F}(X)}|\Psi^{\F}| =0$. 
\end{proof}

\begin{remark}
The algebraic and differential Bianchi identities for the Thomas connection $\hnabla$ of $\en$ yield \eqref{projectivebianchi} and $\nabla_{[i}C_{jk]l} + P_{[i|p|}B_{jk]l}\,^{p} = 0$. 
\end{remark}

\begin{lemma}\label{thomasliftlemma}
Let $M$ be an $n$-manifold equipped with a pseudo-hyperplane bundle $\emf \to M$ and let $\rho: \F = \frameb(\emf^{-1}) \to M$. Let $\en$ be a projective structure on $M$ with Thomas connection $\hnabla$.  If $k \neq n+1$, for a trace-free $A_{i_{1}\dots i_{k}}\,^{j} \in \Ga(S^{k}\ctm\tensor TM)$ there is a unique \emph{$\en$-invariant lift} $\linft(A)_{I_{1}\dots I_{k}}\,^{J} \in \Ga(S^{k}T^{\ast}\F \tensor T\F)$ satisfying:
\begin{enumerate}
\item $\linft(A)$ is a \emph{horizontal lift} of $A$, meaning $\imt(\rad)\linft(A) =0$ and for all $X_{1}, \dots X_{k} \in \Ga(TM)$ and $\theta \in \Ga(\ctm)$ there holds $\rho^{\ast}(\theta)(\linft(A)(\linft(X_{1}), \dots, \linft(X_{k}))) = A(X_{1}, \dots, X_{k})$;
\item $\linft(A)$ is $\hnabla$-divergence free, $\hnabla_{J}\linft(A)_{I_{1}\dots I_{K}}\,^{J} =0$.
\end{enumerate}
\end{lemma}

\begin{proof}
Fix $\nabla \in \en$ and consider the associated invariant lift of vector fields.
If $\linft(A)$ is an invariant lift of $A$ then there is $B \in \Ga(S^{k}\ctm)$ such that, for all $X_{1}, \dots, X_{k} \in \Ga(TM)$,
\begin{align}\label{linftadefined}
\linft(A)(\hat{X}_{1}, \dots, \hat{X}_{k}) = \widehat{A(X_{1}, \dots, X_{k})} + \rho^{\ast}(B(X_{1}, \dots, X_{k}))\rad.
\end{align}
From the identities $\imt(\rad)(\hnabla_{\rad}\linft(A)) =0$, $(\hnabla_{\rad}\linft(A)) = (1-k)\linft(A)$, and
\begin{align}
\begin{aligned}
(\hnabla_{\hat{X}}\linft(A))(\rad, \hat{X}_{1}, \dots, \hat{X}_{k-1}) &= - \linft(A)(\hat{X}, \hat{X}_{1}, \dots, \hat{X}_{k-1}),\\
(\hnabla_{\hat{X}}\linft(A))(\hat{X}_{1}, \dots, \hat{X}_{k}) &=\widehat{(\nabla_{X}A)(X_{1}, \dots, X_{k})} +B(X_{1}, \dots, X_{k})\hat{X} \\&+ \left((\nabla_{X}B)(X_{1}, \dots, X_{k}) + P(X, A(X_{1}, \dots, X_{k})\right)\rad,
\end{aligned}
\end{align}
for $X, X_{1}, \dots, X_{k} \in \Ga(TM)$, it follows that $\hnabla_{J}\linft(A)_{I_{1}\dots I_{K}}\,^{J} =(n+1-k)B_{I_{1}\dots I_{K}} + \rho^{\ast}(\tr \nabla A)_{I_{1}\dots I_{k}}$, where $(\tr \nabla A)_{i_{1}\dots i_{k}}= \nabla_{p}A_{i_{1}\dots i_{k}}\,^{p}$. Consequently, when $k \neq n+1$, the condition $\hnabla_{J}\linft(A)_{I_{1}\dots I_{K}}\,^{J} =0$ determines $B$ in \eqref{linftadefined} uniquely by $B_{i_{1}\dots i_{k}} = -\tfrac{1}{n+1 - k}\nabla_{p}A_{i_{1}\dots i_{k}}\,^{p}$. Since the condition determining $B$ does not depend on the choice of $\nabla \in \en$, the resulting tensor $\linft(A)$ also does not depend on this choice. Alternatively, it is straightforward to check directly that $\linft(A)$ is well defined independently of the choice of $\nabla$.
\end{proof}

\begin{example}\label{curvatureliftexample}
By \eqref{projectivebianchi} and \eqref{thomascurvature}, if $n > 2$, the curvature of the Thomas connection $\hnabla$ of $\en$ is the $\en$-invariant lift of the projective Weyl tensor $B_{ijk}\,^{l}$ of $\en$.
\end{example}

\begin{corollary}[{\cite[Theorem $4.4$]{Gover-Neusser-Willse}}]\label{liftedliecorollary}
Let $M$ be an $n$-manifold equipped with a pseudo-hyperplane bundle $\emf \to M$ and let $\rho: \F = \frameb(\emf^{-1}) \to M$. Let $\en$ be a projective structure on $M$ with Thomas connection $\hnabla$. 
 For all $X \in \Ga(TM)$, the invariant lift $\linft(\lie_{X}\en)$ of $\lie_{X}\en$ is $\lie_{\linft(X)}\hnabla$. In particular, $X$ is an infinitesimal automorphism of $\en$ if and only if $\linft(X)$ is an infinitesimal automorphism of $\hnabla$.
\end{corollary}

\begin{proof}
By the uniqueness claim of Lemma \ref{thomasliftlemma}, it suffices to check that $\lie_{\linft(X)}\hnabla$ is a horizontal lift of $\lie_{X}\en$ that is $\hnabla$-divergence free. Although the details are omitted, a straightforward computation using \eqref{thomascurvature} and \eqref{riccocycle} shows that for $X, Y, Z \in \Ga(TM)$ and $\theta \in \Ga(\ctm)$, $(\lie_{\linft(X)}\hnabla)(\rad, \hat{Z}) = 0$ and $\rho^{\ast}(\theta)((\lie_{\linft(X)}\hnabla)(\hat{Y}, \hat{Z}) = (\lie_{X}\en)(Y, Z)$, so that $\lie_{\linft(X)}\hnabla$ is a horizontal lift of $\lie_{X}\en$. 
For any torsion-free affine connection $\nabla$ and any $X \in \Ga(TM)$ there holds
\begin{align}\label{trlienabla}
(\lie_{X}\nabla)_{ip}\,^{p} = \nabla_{i}\div_{\nabla}(X) - 2X^{p}R_{[pi]}.
\end{align}
Applying \eqref{trlienabla} with $\linft(X)$ and $\hnabla$ in place of $X$ and $\nabla$, using the Ricci-flatness of $\hnabla$, and using that, by \eqref{ctt4} of Theorem \ref{classicalthomastheorem}, $\div_{\hnabla}(\linft(X)) =0$, yields $(\lie_{\linft(X)}\hnabla)_{IP}\,^{P} = 0$.
\end{proof}

\begin{remark}
The essential local content of Lemma \ref{thomasliftlemma} and Corollary \ref{liftedliecorollary} goes back in some form to Thomas \cite{Thomas-book}. The invariant lift is a close relative of the notion of saturated cotractor used in \cite{Gover-projective}; for example in that language the application of Lemma \ref{thomasliftlemma} to the projective Weyl tensor as in Remark \ref{curvatureliftexample} is essentially \cite[Equation $33$]{Gover-projective}. In \cite[Theorem $4.4$]{Gover-Neusser-Willse} the result is proved essentially as stated here for $\emf = |\Det \ctm|^{1/(n+1)}$. 
\end{remark}

Let $\Aut(\F)$ denote the group of principal bundle automorphisms of the principal bundle $\rho: \F\to M$. 

\begin{lemma}
Let $M$ be an $n$-manifold and let $\mu$ be the tautological $n$-form on the total space of the principal $\reat$-bundle $\rho:\vdens = \frameb[\Det \ctm] \to M$. 

\begin{enumerate}
\item The map $\lin:\diff(M) \to \Aut(\vdens)$ associating with $\phi \in \diff(M)$ its \emph{tautological lift} $\lin(\phi) \in \Aut(\vdens)$ defined by 
\begin{align}\label{lindefined}
&\lin(\phi)(s)(X_{1}, \dots, X_{n}) = s(T\phi^{-1}(\phi(\rho(s))(X_{1}) \wedge \dots \wedge T\phi^{-1}(\rho(\phi(s))(X_{n})), & & \text{for}\,\, X_{i} \in T_{\phi(\rho(s))}M,
\end{align}
is an injective homomorphism whose image
\begin{align}\label{lindiffimage}
\lin(\diff(M)) = \{\Phi \in \Aut(\vdens): \Phi^{\ast}(\mu) = \mu\}
\end{align}
acts as automorphisms of the associated Euler structure $(\vdens, \Psi^{\vdens}, \rad^{\vdens})$. 

\item For any $s \in \Ga(\Det \ctm)$ and any $\phi \in \diff(M)$, the homogeneity $-1$ functions $\tilde{s}$ and $\widetilde{\phi^{\ast}s}$ on $\vdens$ corresponding with $s$ and $\phi^{\ast}s$ satisfy $\lin(\phi)^{\ast}\tilde{s} = \tilde{s} \circ \lin(\phi) = \widetilde{\phi^{\ast}s}$. 

\item The \emph{tautological lift} $\linft:\Ga(TM) \to \Ga(T\vdens)$ defined by setting $\linft(X)$ equal to the infinitesimal generator of the image under $\lin$ of the flow of $X$ is an injective Lie algebra homomorphism whose image is
\begin{align}\label{linftimage}
\linft(\Ga(TM)) = \{Z \in \Ga(T\vdens): R_{r}^{\ast}(Z) = Z, \lie_{Z}\mu = 0\}.
\end{align}
For a torsion-free affine connection $\nabla$ on $M$ inducing $\tau \in \prin(\vdens)$ and $X \in \Ga(TM)$, 
\begin{align}\label{wtd}
\linft(X) = \hat{X} - \rho^{\ast}(\div_{\nabla}(X))\eul^{\vdens},
\end{align}
where $\div_{\nabla}(X) = \tr \nabla X$.
\end{enumerate}
\end{lemma}

\begin{proof}
Using local trivializations of $\dens$ it is straightforward to check that $\lin(\phi)$ is a diffeomorphism of $\dens$.
That $\rho\circ \lin(\phi) = \phi \circ \rho$, $\lin(\phi) \circ R_{r} = R_{r}\circ \lin(\phi)$, and $\lin(\phi_{1}\circ \phi_{2}) = \lin(\phi_{1})\circ \lin(\phi_{2})$ follow directly from \eqref{lindefined}, while $\lin(\phi)^{\ast}(\mu) = \mu$ follows from \eqref{lindefined} together with \eqref{mudensdefined}.

If $\Phi \in \Aut(\vdens)$ covers the identity then there is a homogeneity $0$ function $\tilde{\phi}$ on $\vdens$ such that $\Phi(s) = R_{\tilde{\phi}(s)}s$, and it follows from \eqref{mudensdefined} that $\Phi^{\ast}(\mu)_{s} = \tilde{\phi}(s)\mu_{s}$. Consequently, $\Phi^{\ast}(\mu) = \mu$ if and only if $\tilde{\phi}$ is identically $1$, or, what is the same, $\Phi$ is the identity. 
If $\Phi \in \Aut(\vdens)$ covers $\phi \in \diff(M)$ then $\lin(\phi^{-1})\circ \Phi$ covers the identity, and it follows that $\Phi^{\ast}(\mu) = \mu$ if and only if $\Phi = \lin(\phi)$. This shows \eqref{lindiffimage}. Since $\lin(\phi)$ preserves $\veul$, this means that $\lin(\diff(M))$ acts as automorphisms of the associated Euler structure $(\vdens, \Psi^{\vdens}, \rad^{\vdens})$.

By definition of $\mu$, for any section $s \in \Ga(\dens)$, the homogeneity $-1$ function $\tilde{s}$ on $\vdens$ corresponding with $s$ satisfies $\rho^{\ast}(s) = \tilde{s}\mu$. Hence, for $\phi \in \diff(M)$,
\begin{align}\label{linphipullback}
(\widetilde{\phi^{\ast}s})\mu = \rho^{\ast}\phi^{\ast}s = (\phi \circ \rho)^{\ast}s = (\rho \circ \lin(\phi))^{\ast}s  = \lin(\phi)^{\ast}\rho^{\ast}s = \lin(\phi)^{\ast}(\tilde{s}\mu) = (\tilde{s}\circ \lin(\phi))\mu,
\end{align}
which shows that $\lin(\phi)^{\ast}\tilde{s} = \tilde{s} \circ \lin(\phi) = \widetilde{\phi^{\ast}s}$. 

That for $X \in \Ga(TM)$ there hold $T\rho(\linft(X)) = X$, $R_{r}^{\ast}(\linft(X)) = \linft(X)$, and $\lie_{\linft(X)}\mu = 0$ is immediate from the definition. If $Z \in \Ga(T\vdens)$ is as in \eqref{linftimage}, then it generates a local flow by principal $\reat$-bundle automorphisms of $\vdens$ preserving $\mu$, which by \eqref{lindiffimage} are in the image of $\lin$, and so differentiating along the flow shows that $Z$ is in the image of $\linft$. 

For $X \in \Ga(TM)$, differentiating $\lin(\phi_{t})^{\ast}\tilde{s} = \widetilde{\phi_{t}^{\ast}s}$ along the flow $\phi_{t}$ of $X$ and using \eqref{linphipullback} shows that $d\tilde{s}(\linft(X)) = \widetilde{\lie_{X}s}$.
Let $\nabla$ be a torsion-free affine connection on $M$ inducing $\tau \in \prin(\vdens)$. For a local section $s$ of $\Det \ctm$, $\nabla_{X}s = \lie_{X}s - \div_{\nabla}(X)s$. Because $\tilde{s}$ has homogeneity $-1$, 
\begin{align}
d\tilde{s}(\hat{X}) - d\tilde{s}(\linft(X))  = \widetilde{\nabla_{X}s} - \widetilde{\lie_{X}s} = -  \rho^{\ast}(\div_{\nabla}(X))\widetilde{s} =  d\tilde{s}\left(\rho^{\ast}(\div_{\nabla}(X))\eul^{\vdens} \right),
\end{align}
where $\hat{X}$ is the $\tau$-horizontal lift of $X$. This shows \eqref{wtd}.
\end{proof}

\begin{remark}
By its definition \eqref{lindefined}, $\lin(\phi)$ extends to an element of $\diff(\Det \ctm)$ that fixes the image of the zero section. For $s \in \Ga(\Det \ctm)$ and $\phi \in \diff(M)$, $\lin(\phi^{-1})\circ s \circ \phi$ is a section of $\Det \ctm$, and
\begin{align}\label{phiadjoint}
\lin(\phi^{-1})\circ s \circ \phi  = (\lin(\phi^{-1})\circ s \circ \phi )^{\ast}\mu = \phi^{\ast}s^{\ast}\lin(\phi^{-1})^{\ast}\mu = \phi^{\ast}s^{\ast}\mu= \phi^{\ast}s,
\end{align}
by the properties characterizing $\mu$.
\end{remark}

For any representation $\si:\reat \to \reat$, the principal bundle automorphism $\lin(\phi) \in \Aut(\vdens)$ induces a line bundle automorphism of the associated bundle $\vdens\times_{\si}\rea$ also denoted $\lin(\phi)$. In Corollary \ref{diffeocorollary} this is applied to the bundle $|\Det \ctm|^{-1/(n+1)}$ corresponding with $\si(t) = |t|^{1/(n+1)}$.

\begin{corollary}\label{diffeocorollary}
On an $n$-manifold $M$, let $\emf = |\Det \ctm|^{-1/(n+1)} \to M$ and let $\rho: \F = \frameb(\emf^{-1}) \to M$.
The assignment $\en \to \hnabla$ is equivariant with respect to the action of $\diff(M)$ in the sense that the tautological lift of $\phi \in \diff(M)$ satisfies $\lin(\phi)^{\ast}(\hnabla) = \widehat{\phi^{\ast}(\en)}$, and $\lin(\phi) \in \Aut(\hnabla,\Psi)$ if and only if $\phi \in \Aut(\en)$.
\end{corollary}
\begin{proof}
When  $\emf = |\Det \ctm|^{-1/(n+1)}$, if $\phi \in \diff(M)$, then $\lin(\phi)^{\ast}(\hnabla)$ constitutes with $\rad$ a Ricci-flat $\rad$-invariant conelike radiant structure fibering over $(M, \phi^{\ast}(\en))$, so, by the uniqueness part of Theorem \ref{classicalthomastheorem}, equals the Thomas connection of $\phi^{\ast}(\en)$. 
\end{proof}

\begin{example}\label{thomashopfexample}
This example generalizes Example \ref{hopfexample}. Let $M$ be an $n$-manifold equipped with a pseudo-hyperplane bundle $\emf \to M$ and let $\rho: \F = \frameb(\emf^{-1}) \to M$. Let $\en$ be a projective structure on $M$ with Thomas connection $\hnabla$. Fix $1 < \la \in \rea$ and let $\bar{M}$ be the quotient of $\F$ by the action of $\integer \times \zmodtwo$ on $\F$ given via the principal action $R_{r}$ of $r \in \reat$ on $\F$ by $(k, \ep) \to R_{\ep\la^{k}}$. Because the principal action is by automorphisms of $\hnabla$, there are a torsion-free affine connection $\bnabla$ and a vector field $\bar{\rad}$ on $\bar{M}$ such that $(\F, \hnabla, \rad) \to (\bar{M}, \bnabla, \bar{\rad})$ is a submersion of radiant structures. In general, $(\bar{M}, \bnabla, \bar{\rad})$ is a Ricci-flat conelike radiant structure. Its planelike surface are the images in $\bar{M}$ of the planelike surfaces of $(\hnabla, \rad)$. If $\en$ is flat, then $(\bar{M}, \bnabla, \bar{\rad})$ is flat; if moreover $\emf = |\Det TM|^{-1/(n+1)}$, then it is affinely isomorphic to the radiant flat affine manifold constructed in Example \ref{hopfexample}.

Note that the parallel volume density on $\F$ does not descend to $\bar{M}$ and $(\bar{M}, \nabla, \bar{\rad})$ admits no parallel volume by Lemma \ref{paralleloneformlemma}. This is consistent with the main theorem of \cite{Carriere-Dalbo-Meigniez}, which shows that a three-manifold fibered by circles over a compact surface of genus at least two admits no flat unimodular affine structure. 
\end{example}

\section{Metrics compatible with radiant structures and radiant statistical structures}\label{codazzisection}
A connection $\nabla$ on $TM$ determines a dual connection, customarily also denoted by $\nabla$, on $\ctm$, and defined by $(\nabla_{X}\mu)(Y) = \lie_{X}(\mu(Y)) - \mu(\nabla_{X}Y)$ for $X, Y \in \Ga(TM)$ and $\mu \in \Ga(\ctm)$. Using a metric $h$ on $TM$ the dual connection can be transported back to $TM$, and the result is the \emph{$h$-conjugate connection} $\bnabla$ defined for $X\in \Ga(TM)$ and $\mu \in \Ga(\ctm)$ by $\bnabla_{X}\mu= (\nabla_{X}\mu^{\flat})^{\sharp}$. Expanding this shows that $\bnabla$ is defined by $\lie_{X}h(A, B) = h(\nabla_{X}A, B) + h(A, \bnabla_{X}B)$ for $X, A, B \in \Ga(TM)$. Alternatively,
\begin{align}\label{conjugatediff}
\bnabla - \nabla = h^{kp}\nabla_{i}h_{jp},
\end{align}
and from \eqref{conjugatediff} it follows that $h$-conjugacy is an involution on the space of affine connections. By \eqref{conjugatediff}, $\bnabla_{i}h_{jk} = -\nabla_{i}h_{jk}$. In particular, a connection $\nabla$ is $h$-self-conjugate, meaning $\bnabla = \nabla$, if and only if $\nabla$ is the Levi-Civita connection of $h$. 
However, as follows from \eqref{conjugatediff}, the torsion $\btau_{ij}\,^{k}$ of $\bnabla$ is expressed in terms of the torsion $\tau_{ij}\,^{k}$ of $\nabla$ by 
\begin{align}\label{btau}
\btau_{ij}\,^{k} = \tau_{ij}\,^{k} + 2h^{kp}\nabla_{[i}h_{j]p},
\end{align}
so that $h$-conjugacy does not preserve the affine subspace of torsion-free affine connections. For this reason it is convenient to work with the \emph{$h$-opposite connection} $\onabla$ of $\nabla$ defined by $h(\onabla_{X}Y, Z)  = -h(\nabla_{X}Y, Z) + h([X, Y], Z) - h([X, Z], Y) - h([Y, Z], X) + \lie_{X}h(Y, Z)  + \lie_{Y}h(X, Z) - \lie_{Z}h(X, Y)$ for $X, Y, Z \in \Ga(TM)$. 
Alternatively,
\begin{align}\label{opdiff}
\onabla - \nabla = h^{kp}\left(\nabla_{i}h_{jp} + \nabla_{j}h_{ip} - \nabla_{p}h_{ij}\right) - \tau_{ij}\,^{k} + 2h_{q(i}\tau_{j)p}\,^{q}h^{pk},
\end{align}
The interpretation of \eqref{opdiff} is that the opposite connection of a torsion-free $\nabla$ is the reflection of $\nabla$ through the Levi-Civita connection of $h$. It follows from \eqref{opdiff} that opposition is an involution in the sense that $\op{(\onabla)} = \nabla$. Note that $\onabla_{i}h_{jk} = -\nabla_{i}h_{jk}$.
By \eqref{opdiff}, the torsion $\op{\tau}_{ij}\,^{k}$ is given by $\op{\tau}_{ij}\,^{k} = -\tau_{ij}\,^{k}$. In particular, the $h$-opposite connection of a torsion-free connection is again torsion-free. 

\begin{lemma}\label{oppositeconjugatelemma}
For an affine connection $\nabla$ and a pseudo-Riemannian metric $h$ on $M$ any two of the following implies the other two:
\begin{enumerate*}
\item\label{oc1} $\nabla$ is torsion-free;
\item\label{oc2} $\bnabla$ is torsion-free;
\item\label{oc3} $\nabla_{[i}h_{j]k} = 0$; 
\item\label{oc4} $\onabla = \bnabla$.
\end{enumerate*}
\end{lemma}
\begin{proof}
Combining \eqref{conjugatediff}, \eqref{opdiff}, and \eqref{btau} shows
\begin{align}\label{opconjdiff}
\onabla - \bnabla
= h^{kq}\left(2\nabla_{[j}h_{q]i}- \tau_{ijq} + \tau_{jqi} + \tau_{iqj}\right) = h^{kq}\left(\bar{\tau}_{jqi} - \tau_{ijq} - \tau_{qij}\right),
\end{align}
where $\tau_{ijk} = \tau_{ij}\,^{p}h_{pk}$ and $\bar{\tau}_{ijk} = \bar{\tau}_{ij}\,^{p}h_{pk}$. That any two of \eqref{oc1}-\eqref{oc3} implies the third is immediate from \eqref{btau}. That in these cases there holds \eqref{oc4} follows from \eqref{opconjdiff}. That \eqref{oc1} and \eqref{oc4} imply \eqref{oc3} is immediate from the first equality of \eqref{btau} and, in this case, the second equality of \eqref{oc4} implies that $\bnabla$ is torsion-free so there holds \eqref{oc2}.
If there holds \eqref{oc2} then $2\nabla_{[i}h_{j]k} = -\tau_{ijk}$, so $\tau_{[ijk]} =0$. In \eqref{opconjdiff} this yields $\onabla - \bnabla = h^{kp}\tau_{jpi} = -2h^{kp}\nabla_{[j}h_{p]i}$ from which it follows that \eqref{oc2} and \eqref{oc4} imply \eqref{oc1} and \eqref{oc3}.
If there hold \eqref{oc3} and \eqref{oc4} then by \eqref{opconjdiff} there holds $0 = -\tau_{jki} + \tau_{ijk} + \tau_{kij}$. Antisymmetrizing this shows $\tau_{[jki]} = 0$ which yields $\tau_{jki} = \tau_{ijk} + \tau_{kij} = -\tau_{jki}$ so that $\nabla$ is torsion-free. By \eqref{btau}, $\bnabla$ is torsion-free as well. This shows \eqref{oc3} and \eqref{oc4} imply \eqref{oc1} and \eqref{oc2}.
\end{proof}

A \emph{statistical structure} is a pair $(\nabla, h)$ comprising a torsion-free affine connection, $\nabla$, and a pseudo-Riemannian metric, $h_{ij}$, such that $\nabla_{[i}h_{j]k} = 0$. Statistical structures are also called \emph{Codazzi structures}, for example in \cite{Shima}. For background, in addition to \cite{Shima} see \cite{Amari, Ay-Jost-Le-Schwachhofer, Kurose-dual, Matsuzoe-conformallyprojectively}.

By Lemma \ref{oppositeconjugatelemma}, if $(\nabla, h)$ is a statistical structure, then $(\onabla, h) = (\bnabla, h)$ is also a statistical structure, the \emph{conjugate statistical structure}. Conjugation is an involution on the space of statistical structures with a fixed underlying metric.

A statistical structure is \emph{special} if $0 = h^{pq}\nabla_{i}h_{pq} = \nabla_{i}|\det h|$. 
Since $\bnabla_{i}h_{jk} = -\nabla_{i}h_{jk}$, a statistical structure is special if and only if the conjugate statistical structure is special.
Because $|\det h|^{-1}\nabla_{[i}\nabla_{j]}|\det h| = -R_{ijp}\,^{p} = 2R_{[ij]}$, a special statistical structure has symmetric Ricci tensor.

Next there are examined conditions of compatibility between a radiant structure and a pseudo-Riemannian metric. These conditions extend notions of compatibility for a \emph{flat} affine connection and a metric usually called Hessian metrics or affine Kähler metrics as in \cite{Fox-schwarz}. The related literature is large. A recent survey of some aspects is \cite{Cardoso-Mohaupt}.

The homogeneity condition $\lie_{\rad}h = 2h$, or the stronger condition that the metric $h$ has homogeneity $2$ with respect to the flow of $\rad$, plays a role in the treatment of the ambient metric in \cite{Fefferman-Graham-ambient} and \cite{Rodnianski-Shlapentokh-Rothman}. Definition \ref{selfsimilardefined} adapts terminology introduced in \cite{Rodnianski-Shlapentokh-Rothman}.

\begin{definition}\label{selfsimilardefined}
A pair $(h, \rad)$ comprising a pseudo-Riemannian metric $h$ and a vector field $\rad$ is a \emph{self-similar metric structure} if $\lie_{\rad}h = 2h$.
\end{definition}
By definition, the vector field of a self-similar metric structure $(h, \rad)$ is conformally Killing for $h$. A special case of Lemma \ref{liesslemma} yields
\begin{align}
\label{nrhlrh} 
\rad^{p}\nabla_{p}h_{ij}  & = (\lie_{\rad}h)_{ij} - 2h_{ij},
\end{align}
from which it follows that a pair $(h, \rad)$ is self-similar if and only if $\nabla_{\rad}h = 0$.

A radiant structure $(\nabla, \rad)$ and a metric $h_{ij}$ determine the one-form $\rad^{\flat}_{i} = \rad^{p}h_{pi}$ and the function $v = \rad^{p}\rad^{q}h_{pq}$. Note that if $h$ is of indefinite signature, then it can be that $v$ vanishes even though $\rad$ is nonsingular, for it could be that $\rad$ is $h$-null. 
Direct computations yield
\begin{align}
\label{dvrad} &dv_{i}  = 2\rad^{\flat}_{i} + \rad^{a}\rad^{b}\nabla_{i}h_{ab}, &&
\nabla_{i}dv_{j}  = 2\nabla_{i}\rad^{\flat}_{j} + 2\rad^{p}\nabla_{j}h_{ip} + \rad^{a}\rad^{b}\nabla_{i}\nabla_{j}h_{ab},
\end{align}
which suggest investigating conditions implying $dv = 2\rad^{\flat}$ and $\nabla dv = 2\nabla\rad^{\flat}$.

\begin{lemma}\label{radiantmetriclemma}
On an $n$-manifold $M$, let $(\nabla, \rad)$ be a radiant structure, let $h_{ij}$ be a pseudo-Riemannian metric, and let $\onabla$ be the opposite connection of $\nabla$. Define $\rad^{\flat}_{i} = \rad^{p}h_{pi}$. There hold:
\begin{align}
\label{crinter}
\nabla_{i}\rad^{\flat}_{j} - h_{ij} &= \rad^{p}\nabla_{i}h_{jp}  = 2\rad^{p}\nabla_{[i}h_{p]j} + \rad^{p}\nabla_{p}h_{ij} = 2\rad^{p}\nabla_{[i}h_{p]j} + (\lie_{\rad}h)_{ij} - 2h_{ij},\\
\label{dradflat} d\rad^{\flat}_{ij} &= 2\nabla_{[i}\rad^{\flat}_{j]} = 2\rad^{p}\nabla_{[i}h_{j]p},\\
\label{crinter2}
h_{jp}(\onabla_{i}\rad^{p} - \delta_{i}\,^{p}) & = d\rad^{\flat}_{ij} + (\lie_{\rad}h)_{ij} - 2h_{ij},\\
\label{crinter3sym}
\onabla_{(i}\rad^{\flat}_{j)} - h_{ij} &  
= -(\nabla_{(i}\rad^{\flat}_{j)} - h_{ij}) + (\lie_{\rad}h)_{ij} - 2h_{ij}.
\end{align}
In particular:
\begin{enumerate}

\item\label{codazziclosed} $\rad^{\flat}$ is closed if $\nabla_{[i}h_{j]k} = 0$.
\item\label{cq0} Any two of the following imply the third:
\begin{enumerate*}
\item \label{cq3} $(h, \rad)$ is self-similar; 
\item \label{cq2} $\rad^{\flat}$ is closed;
\item \label{cq1} $(\onabla, \rad)$ is a radiant structure.
\end{enumerate*}
If these hold, $2\rad^{\flat}_{i} = dv_{i}$ where $v = \rad^{p}\rad^{q}h_{pq}$.
\item\label{opradiant} Any two of the following imply the third:
\begin{enumerate*}
\item \label{ca1} $(h, \rad)$ is self-similar;
\item \label{ca2} $\nabla_{(i} \rad^{\flat}_{j)} = h_{ij}$;
\item \label{ca3} $\onabla_{(i}\rad^{\flat}_{j)} = h_{ij}$.
\end{enumerate*}
\end{enumerate}
\end{lemma}

\begin{proof}
Direct computation and \eqref{nrhlrh} yield \eqref{crinter}. Antisymmetrizing the first equality of \eqref{crinter} gives \eqref{dradflat}. 
Claim \eqref{codazziclosed} follows from \eqref{dradflat}.
Computations using $\onabla_{i}h_{jk} = -\nabla_{i}h_{jk}$ show 
\begin{align}
\label{crinter2pre}
h_{jp}(\onabla_{i}\rad^{p} - \delta_{i}\,^{p}) &= h_{jp}\onabla_{i}\rad^{p} - h_{ij} =     2\rad^{p}\nabla_{[i}h_{j]p} + \rad^{p}\nabla_{p}h_{ij} = d\rad^{\flat}_{ij} + (\lie_{\rad}h)_{ij} - 2h_{ij},
\end{align}
which is \eqref{crinter2}. From \eqref{crinter2} it is apparent that any two of \eqref{cq3}-\eqref{cq1} imply the third. In the case there hold \eqref{cq3}-\eqref{cq1}, then, since $\rad^{\flat}$ is closed, by \eqref{dradflat} and \eqref{dvrad}, $dv_{i}  = 2\rad^{\flat}_{i} + \rad^{a}\rad^{b}\nabla_{a}h_{ib} = \rad^{p}(\lie_{\rad}h)_{ip} = 2\rad^{\flat}_{i}$.
Computations using \eqref{crinter}, \eqref{crinter2}, and $\onabla_{i}h_{jk} = -\nabla_{i}h_{jk}$ yield 
\begin{align}
\label{crinter3}
\begin{split}
\onabla_{i}\rad^{\flat}_{j} - h_{ij} &  =  \rad^{p}\onabla_{i}h_{jp} + h_{jp}\onabla_{i}\rad^{p} - h_{ij}\\
& = - \rad^{p}\nabla_{i}h_{jp}  + d\rad^{\flat}_{ij} + (\lie_{\rad}h)_{ij} - 2h_{ij} 
= d\rad^{\flat}_{ij} - (\nabla_{i}\rad^{\flat}_{j} - h_{ij}) +  (\lie_{\rad}h)_{ij} - 2h_{ij}.
\end{split}
\end{align}
Symmetrizing \eqref{crinter3} yields \eqref{crinter3sym}. 
From \eqref{crinter3sym} it is apparent that any two of \eqref{ca1}-\eqref{ca3} imply the third.
\end{proof}

\begin{lemma}\label{thurstonlemma}
If a compact manifold $M$ admits a nonsingular radiant structure $(\nabla, \rad)$ and a metric $h_{ij}$ such that $\rad^{\flat}_{i} = \rad^{p}h_{pi}$ is closed, then $M$ is a fiber bundle over $S^{1}$. Moreover, $\rad^{\flat}$ is the pullback of the generator of $H^{1}(S^{1}; \rea)$ if and only if $\tfrac{1}{2\pi}\rad^{\flat} \in H^{1}(M; \integer)$.
\end{lemma}

\begin{proof}
By \cite[Theorem $1$]{Tischler} a compact manifold admitting a nowhere vanishing closed one-form is a fiber bundle over $S^{1}$. The second claim is \cite[Lemma $2.1$]{Farber}.
\end{proof}

\begin{lemma}\label{hopflikelemma}
A manifold admitting a radiant structure $(\nabla, \rad)$ and a metric $h_{ij}$ such that the function $v = \rad^{a}\rad^{b}h_{ab}$ satisfies $\nabla_{i}dv_{j} = 2h_{ij}$ is not compact.
\end{lemma}

\begin{proof}
By the nondegeneracy of $h$, $v$ is a Morse function and all its critical points have the same index, but if $M$ were compact, then $v$ would have both a maximum and a minimum.
\end{proof}

\begin{lemma}
Let $(\nabla, \rad)$ be a radiant structure on $M$. Let $h$ be a metric and define $v = \rad^{p}\rad^{q}h_{pq}$. If there holds $\nabla_{(i}\rad^{\flat}_{j)} = h_{ij}$, then $dv_{i} = 2\rad^{\flat}_{i}$ if and only if $\lie_{\rad}\rad^{\flat} = 2\rad^{\flat}$, in which case $\nabla_{i}dv_{j} = 2h_{ij}$. 
\end{lemma}

\begin{proof}
From \eqref{dvrad} and \eqref{dradflat} there follows
\begin{align}
\label{lieradradflat} (\lie_{\rad}\rad^{\flat})_{i} & = \rad^{a}d\rad^{\flat}_{ai} + dv_{i} = \rad^{a}\rad^{b}\nabla_{a}h_{bi} + 2\rad_{i}^{\flat} = \rad^{p}(\lie_{\rad}h)_{pi}.
\end{align}
From \eqref{lieradradflat} it follows that the condition $\lie_{\rad}h = 2h$ implies $\lie_{\rad}\rad^{\flat} = 2\rad^{\flat}$.
It follows from the definition of $v$, $\nabla_{i}\rad^{j} = \delta_{i}\,^{j}$, and \eqref{lieradradflat} that
\begin{align}
\begin{split}
2\rad^{p}(\nabla_{(i}\rad^{\flat}_{j)} - h_{ip}) & = \rad^{p}\nabla_{i}\rad^{\flat}_{p} + \rad^{p}\nabla_{p}\rad^{\flat}_{i} - 2\rad^{\flat}_{i}\\
& = dv_{i} - \rad^{\flat}_{p}\nabla_{i}\rad^{p} + (\lie_{\rad}\rad^{\flat})_{i} - 3\rad^{\flat}_{i} = (dv_{i} - 2\rad^{\flat}_{i}) + ((\lie_{\rad}\rad^{\flat})_{i} - 2\rad^{\flat}_{i}).
\end{split}
\end{align} 
from which the claimed equivalence is evident. 
\end{proof}

\begin{lemma}\label{preradcodazzilemma}
On a smooth manifold, $M$, let $(h, \rad)$ be a self-similar metric structure such that $d\rad^{\flat}_{ij} = 0$. 
For a torsion-free affine connection, $\nabla$, the following are equivalent:
\begin{enumerate}
\item \label{cr1} $(\nabla, \rad)$ is a radiant structure and $\nabla_{i} \rad^{\flat}_{j} = h_{ij}$.
\item \label{cr2} $(\onabla, \rad)$ is a radiant structure and $\onabla_{i} \rad^{\flat}_{j} = h_{ij}$.
\end{enumerate}
When these conditions hold:
\begin{enumerate}
\item The function $v = h(\rad, \rad)$ satisfies $dv_{i} = 2\rad^{\flat}_{i}$ and so also $2h_{ij} = 2\nabla_{i}\rad^{\flat}_{j} = \nabla_{i} dv_{j} = \onabla_{i}dv_{j}$. 
\item The curvatures $R_{ijk}\,^{l}$ of $\nabla$ and $\op{R}_{ijk}\,^{l}$ of $\onabla$ satisfy
\begin{align}
&R_{ijk}\,^{p}\rad^{\flat}_{p} = -2\nabla_{[i}h_{j]k}, & &\op{R}_{ijk}\,^{p}\rad^{\flat}_{p} = -2\onabla_{[i}h_{j]k} = 2\nabla_{[i}h_{j]k}.
\end{align}
\item On $\hat{M} = \{p \in M: v(p) \neq 0\}$, the closed one-form $\be_{i} = \tfrac{1}{2}v^{-1}dv_{i} = v^{-1}\rad^{\flat}_{i}$ satisfies $(\lie_{\rad}\be)_{i} = 0$ and $\be(\rad) = 1$. The restriction to $\hat{M}$ of the metric $h_{ij}$ has the form
\begin{align}\label{nablabebebe}
h_{ij} = v\left(\nabla_{i}\be_{j} + 2\be_{i}\be_{j}\right) = v\left(\onabla_{i}\be_{j} + 2\be_{i}\be_{j}\right).
\end{align}
\end{enumerate}
\end{lemma}

\begin{proof}
The equivalence of \eqref{cr1} and \eqref{cr2} follows from \eqref{opradiant} of Lemma \ref{radiantmetriclemma}, \eqref{crinter3}, and the observation that antisymmetrizing $\nabla_{i} \rad^{\flat}_{j} = h_{ij}$ implies $\rad^{\flat}$ is closed.
If there hold \eqref{cr1} and \eqref{cr2} then $dv_{i} = 2\rad^{a}\nabla_{i}\rad^{\flat}_{a} = 2\rad^{a}h_{ia} = 2\rad^{\flat}_{i}$. By the Ricci identity, $R_{ijk}\,^{p}\rad^{\flat}_{p} = -2\nabla_{[i}\nabla_{j]}\rad^{\flat}_{k} = -2\nabla_{[i}h_{j]k}$ and similarly with $\onabla$ in place of $\nabla$. The first equality of \eqref{nablabebebe} follows from 
\begin{align}
\label{vnablabe}&v\nabla_{i}\be_{j} = \nabla_{i}\rad^{\flat}_{j} - dv_{i}\be_{j}.
\end{align}
Since $\be_{i}\rad^{i} = v^{-1}\rad^{p}\rad^{\flat}_{p} = 1$ and $d\rad^{\flat}_{ij} =0$, skewing \eqref{vnablabe} shows $v(\lie_{\rad}\be)_{i} =v \rad^{p}d\be_{pi} = 0$.
The claims for $\onabla$ follow by taking $\onabla$ in place of $\nabla$.
\end{proof}

\begin{definition}
A \emph{radiant Hessian structure} is a triple $(\nabla, \rad, h)$ such that $(\nabla, \rad)$ is a radiant structure, $(h, \rad)$ is a self-similar metric structure, and $\nabla_{i}\rad^{\flat}_{j} = h_{ij}$.
\end{definition}

The definition of a radiant Hessian structure $(\nabla, \rad, h)$ implies that $\rad^{\flat}_{i}$ is closed. By Lemma \ref{hopflikelemma}, a manifold admitting a radiant Hessian structure is not compact. 
This conclusion is analogous to the statement that a complex Hopf surface admits no Kähler structure.

By Lemma \ref{preradcodazzilemma}, if $(\nabla, \rad, h)$ is a radiant Hessian structure, then $v = \rad^{p}\rad^{q}h_{pq}$ satisfies $dv_{i} = 2\rad^{\flat}_{i}$ and $\nabla_{i}dv_{j} = 2h_{ij}$, and $(\onabla, \rad, h)$ is also a radiant Hessian structure, the \emph{opposite} radiant Hessian structure. 

\begin{lemma}\label{radcodazzilemma}
On a smooth manifold, $M$, let $(\nabla, \rad, h)$ be a radiant Hessian structure with opposite radiant Hessian structure $(\onabla, \rad, h)$.  Let $v = \rad^{p}\rad^{q}h_{pq}$ and let $\be_{i} =\tfrac{1}{2}v^{-1}dv_{i}$ where $v \neq 0$.
For $t \in \rea$ the connection 
\begin{align}\label{nablatdefined}
\nabla(t) = (1-t)\nabla + t \onabla = \nabla + th^{kp}\left(\nabla_{i}h_{jp} + 2\nabla_{[j}h_{p]i}\right)
\end{align} 
satisfies $\tfrac{1}{2}\nabla(t)_{i}dv_{j} = \nabla(t)_{i}\rad^{\flat}_{j} = h_{ij}$. 
In particular, the Levi-Civita connection, $D = \nabla(1/2)$ of $h_{ij}$ satisfies $D_{i}\rad^{\flat}_{j} = h_{ij}$. In this case, around any point $p \in M$ at which $v  \neq 0$ there is an open neighborhood $U$ such that the pseudo-Riemannian manifold $(U, h)$ is isometric to a warped product metric of the form 
\begin{align}\label{warped}
\ep dr \tensor dr + r^{2}\pi^{\ast}(g)
\end{align}
on $(a, b) \times \Ga$ where $\Ga = \{q \in U: v(q) = v(p)\}$, $\ep = \sign(v(p))$, $\pi:(a, b) \times \Ga \to \Ga$ is the projection on the second factor, $r = \sign(v)|v|^{1/2}$, and the metric $\ep g_{ij}$ is the restriction to $\Ga$ of the tensor
\begin{align}\label{nablathessian}
\nabla(t)_{i}\be_{j} + \be_{i}\be_{j} = \tfrac{1}{2}v^{-1}\nabla(t)_{i}dv_{j} - \tfrac{1}{4}v^{-1}dv_{i}dv_{j} = \tfrac{1}{2}\nabla(t)_{i}d\log{|v|}_{j},
\end{align}
which does not depend on $t$.
\end{lemma}

\begin{proof}
By \eqref{opdiff}, $\nabla(t) = \nabla + th^{kp}\left(\nabla_{i}h_{jp}+ 2\nabla_{[j}h_{p]i}\right)$, and so, 
\begin{align}
\nabla(t)_{i}\rad^{\flat}_{j} = \nabla_{i}\rad^{\flat}_{j} - t \rad^{p}\left(\nabla_{i}h_{jp}  + 2\nabla_{[j}h_{p]i}\right)  = h_{ij},
\end{align}
the final equality by \eqref{nrhlrh} and \eqref{crinter} of Lemma \ref{radiantmetriclemma}. Hence $(\nabla(t), \rad, h)$ is a radiant Hessian structure for all $t$. In particular, $\nabla(t)_{i}dv_{j} = 2h_{ij}$ and $\nabla(t)_{i}\be_{j} + 2\be_{i}\be_{j} = vh_{ij}$ for all $t$. 

That $D_{i}dv_{j} = 2h_{ij}$ implies that $(M, h)$ is locally isometric to a warped product is well-known. See \cite[pages $192-194$]{Cheeger-Colding} for the proof of a more general statement in the Riemannian setting. As in the applications considered here the signature of $h_{ij}$ is indefinite, this is briefly sketched so as to get straight the signs. Also it is useful to know the relations between $\be$, $v$, $r$, and $h$. 

Let $U$ be an open neighborhood of $p$ on which the flow $\phi_{t}$ of $\rad$ is defined for all $t$ in some interval $(a, b)$ containing $0$ and $v$ does not vanish and set $\ep = \sign(v(p))$. Define $r = \ep|v|^{1/2}$ on $U$.
Define $\tilde{g}_{ij} = \ep( \nabla_{i}\be_{j} + \be_{i}\be_{j})$ where $\ep = \sign(v(p))$. Because $\nabla(t)_{i}dv_{j}$ and $\nabla(t)_{i}\be_{j}$ do not depend on $t$, $\tilde{g}_{ij}$ equals the expressions in \eqref{nablathessian}. By Lemma \ref{preradcodazzilemma},
\begin{align}\label{hhess1}
\begin{split}
h_{ij} & = v(\nabla_{i}\be_{j} + 2\be_{i}\be_{j}) = \ep v\tilde{g}_{ij} + \tfrac{1}{4}v^{-1}dv_{i}dv_{j} = r^{2} \tilde{g}_{ij} + \ep dr_{i}dr_{j}.
\end{split}
\end{align}
Note that $dv_{i} = 2rdr_{i}$ and $dr(\rad) = r$, so that $\lie_{\rad}r = r$. From \eqref{hhess1} there follows $0 = \lie_{\rad}h - 2h = r^{2}\lie_{\rad}\tilde{g}$, so that $\lie_{\rad}\tilde{g} = 0$. Define $g_{ij}$ to be the restriction to $\Ga = \{q \in U: v(q) = v(p)\}$ of $\tilde{g}_{ij}$.  Restricting \eqref{hhess1} to $\Ga$ shows that along $\Ga$ there holds \eqref{warped}, and the preceding shows that $\phi_{t}^{\ast}(\tilde{g})_{q} = g_{q}$ for all $q \in \Ga$, or, what is the same $\pi^{\ast}(g) = \tilde{g}$ on $U$. Thus the map $\Psi:(a, b) \times \Ga \to M$ defined by $\Psi(t, q) = \phi_{t}(q)$ satisfies 
\begin{align}
\begin{split}
\Psi^{\ast}(h) & = \Psi^{\ast}(\ep v \tilde{g} + \tfrac{1}{4}v^{-1}dv \tensor dv) = \ep v_{0} e^{2t}\phi_{t}^{\ast}(\tilde{g}) + v_{0}e^{2t}dt \tensor dt\\
& = \ep v_{0}e^{2t}\left( \pi^{\ast}(g) + \ep dt \tensor dt\right) = r^{2}\pi^{\ast}(g) + \ep dr \tensor dr,
\end{split}
\end{align}
where $v(\phi_{t}(q)) = v_{0}e^{2t}$ and $r(\phi_{t}(q)) = \ep\sqrt{|v_{0}|}e^{t}$.
\end{proof}

\begin{lemma}\label{radcodazziconjugatelemma}
For a triple $(\nabla, \rad, h)$ such that $(\nabla, \rad)$ is a radiant structure and $(\nabla, h)$ is a statistical structure, the $h$-conjugate connection $\bnabla$ forms with $\rad$ a radiant structure if and only if $(h, \rad)$ is a self-similar metric structure.
In this case, $(\bnabla, h)$ is a statistical structure.
\end{lemma}

\begin{proof}
If $(\nabla, h)$ is statistical, then, by \eqref{cq0} of Lemma \ref{radiantmetriclemma}, $\bnabla = \onabla$, and by \eqref{opradiant} of Lemma \ref{radiantmetriclemma}, $\onabla = \bnabla$ forms with $\rad$ a radiant structure if and only if $\lie_{\rad}h = 2h$. In this case, because $\bnabla_{i}h_{jk} = -\nabla_{i}h_{jk}$, $(\bnabla, h)$ is statistical.
\end{proof}

\begin{definition}
A \emph{radiant statistical structure} is a triple $(\nabla, \rad, h)$ such that $(\nabla, \rad)$ is a radiant structure, $(\nabla, h)$ is a statistical structure, and $(h, \rad)$ is a self-similar metric structure. 
\end{definition}

If $(\nabla, \rad, h)$ is a radiant statistical structure and $\bnabla$ is the $h$-conjugate connection of $\nabla$, then by Lemma \ref{radcodazziconjugatelemma}, $(\bnabla, \rad, h)$ is also a radiant statistical structure, called the \emph{conjugate} radiant statistical structure.

A radiant statistical structure $(\nabla, \rad, h)$ is a radiant Hessian structure, and in this case $\bnabla = \onabla$, so the opposite radiant Hessian structure is the conjugate radiant statistical structure.

A radiant statistical structure $(\nabla, \rad, h)$ is \emph{special} if the underlying statistical structure $(\nabla, h)$ is special; it is \emph{conelike} if the underlying radiant structure $(\nabla, \rad)$ is conelike.

On an oriented manifold $M$, a radiant statistical structure $(\nabla, \rad, h)$ determines a radiant Euler structure $(\nabla, \rad, \vol_{h})$, where $\vol_{h}$ is the volume form determined by $h$ and the given orientation of $M$. That $\lie_{\rad}h = 2h$ implies $\lie_{\rad}\det h = 2n \det h$, and it follows that $\lie_{\rad}\vol_{h} = n \vol_{n}$.

Similarly, if on an oriented manifold $M$, the radiant statistical structure $(\nabla, \rad, h)$ is special, then $(\nabla, \rad, \vol_{h})$ is an equiaffine Euler structure. 

\begin{example}
If $D$ is the Levi-Civita connection of a pseudo-Riemannian metric $g_{ij}$ and $X^{i}$ is a concurrent vector field, meaning $D_{i}X_{j}^{\flat} = g_{ij}$ where $X_{j}^{\flat} = X^{p}g_{pj}$, then $(\lie_{X}g)_{ij} = 2D_{(i}X_{j)}^{\flat} = 2g_{ij}$, so $(D, X, g)$ is a radiant statistical structure.

The self-conjugate radiant statistical structures, those for which $(\nabla, \rad, h) = (\bnabla, \rad, h)$, are exactly those for which $\nabla$ is the Levi-Civita connection of $h$ and $\rad$ is a concurrent vector field.
\end{example}

\begin{corollary}[Corollary of Lemma \ref{preradcodazzilemma}]\label{radcodazzicorollary}
On a smooth manifold, $M$, for a radiant statistical structure $(\nabla, \rad, h)$ there hold:
\begin{enumerate}
\item The function $v = h(\rad, \rad)$ satisfies $dv_{i} = 2\rad^{\flat}_{i}$ and $ \nabla_{i} dv_{j} = 2h_{ij} = \onabla_{i}dv_{j}$.
\item The curvature $R_{ijk}\,^{l}$ of $\nabla$ satisfies $R_{ijk}\,^{p}\rad^{\flat}_{p} = 0$. 
\item\label{radcodbe} On $\hat{M} = \{p \in M: v(p) \neq 0\}$, the closed one-form $\be_{i} = \tfrac{1}{2}v^{-1}dv_{i} = v^{-1}\rad^{\flat}_{i}$ satisfies $(\lie_{\rad}\be)_{i} = 0$ and $\be(\rad) = 1$. The restriction to $\hat{M}$ of the metric $h_{ij}$ has the form $h_{ij} = v\left(\nabla_{i}\be_{j} + 2\be_{i}\be_{j}\right)$.
\end{enumerate}
All the same statements hold for the conjugate radiant statistical structure $(\bnabla, \rad, h)$.
\end{corollary}

\begin{example}
By Propositions $2.4$ and $3.4$ of Fefferman and Graham's \cite{Fefferman-Graham-ambient}, a pre-ambient space that is \emph{straight} constitutes a radiant statistical structure for which $dv_{i} = 2\rad^{\flat}_{i}$. Here $\rad^{i}$ is the vector field called $T$ in \cite{Fefferman-Graham-ambient}, that generates dilations by $s^{2}$ in the fibers of the bundle of metrics. 
\end{example}

A statistical structure $(\nabla, h)$ is \emph{equiaffine} if there is a nonvanishing volume density $\Psi \in \Ga(|\Det \ctm|)$ such that $\nabla \Psi = 0$. In this definition, no relation is supposed between $\Psi$ and $h$. A radiant statistical structure is equiaffine if its underlying radiant structure is equiaffine.

For a statistical structure $(\nabla, h)$, define $R_{ijkl} = R_{ijk}\,^{p}h_{pl}$. The curvature tensor and those derived from it associated with the conjugate statistical structure $(\bnabla, h)$ are decorated with $\bar{\dum}$.

\begin{lemma}\label{conjugateradiantcodazzilemma}
Let $(\nabla, h)$ be a statistical structure with conjugate statistical structure $(\bnabla, h)$. 
\begin{enumerate}
\item The curvature and Ricci tensors $R_{ijkl} = R_{ijk}\,^{p}h_{pl}$, $\bar{R}_{ijkl} = \bar{R}_{ijk}\,^{p}h_{pl}$, $R_{ij}$, and $\bar{R}_{ij}$ satisfy
\begin{align}
\label{conjradcodcurv}&\bar{R}_{ijkl} = -R_{ijlk},&
& \bar{R}_{ij}=  R_{ipqj}h^{pq},&
&\bar{R}_{[ij]} = - R_{[ij]},&
& h^{ij}\bar{R}_{ij} = h^{ij}R_{ij}.&
\end{align}
\item\label{rcsricflat} $(\bnabla, h)$ is Ricci-flat if and only if $h^{pq}R_{ipq}\,^{j} = 0$.
\item\label{rcssymric} $(\bnabla,h)$  has symmetric Ricci tensor if and only if $(\nabla, \rad, h)$ has symmetric Ricci tensor.
\item\label{rcsequiaffine} $(\nabla, h)$ is equiaffine for the volume density $\Psi$ if and only if $(\bnabla, h)$ is equiaffine for the volume density $A\Psi$ where $\det h = A\Psi^{\tensor 2}$ (so $h^{pq}\nabla_{i}h_{pq}= A^{-1}dA_{i}$).
\item\label{rcsspecial} The following are equivalent:
\begin{enumerate*}
\item $(\bnabla, h)$ is special;
\item $(\nabla,  h)$ is special;
\item $(\bnabla, h)$ and $(\nabla,  h)$ are simultaneously equiaffine for the same volume density $\Psi$.
\end{enumerate*}
If these hold, $\det h = c \Psi^{\tensor 2}$ for some $c \in \reat$.
\item If $(\nabla, \rad, h)$ is a radiant statistical structure with conjugate radiant statistical structure $(\bnabla, \rad, h)$, then
\begin{align}
\label{rcsradzero}
&\bar{\er}_{i} = - \er_{i},&
&\bar{R}_{ijk}\,^{p}\rad_{p}^{\flat} = 0,& &R_{ijk}\,^{p}\rad_{p}^{\flat} = 0.
\end{align}
\end{enumerate}
\end{lemma}

\begin{proof}
Since $\Pi_{ij}\,^{k} = \bnabla - \nabla =  h^{kp}\nabla_{i}h_{jp} = h^{kp}\nabla_{p}h_{ij}$, the curvature tensors are related by
\begin{align}\label{conjugatecurvature}
\begin{split}
\bar{R}_{ijkl} &= R_{ijkl} + 2h_{lp}\nabla_{[i}\Pi_{j]k}\,^{p} - 2h_{lq}\Pi_{p[i}\,^{q}\Pi_{j]k}\,^{p}\\
& = R_{ijkl} + 2\nabla_{[i}\nabla_{j]}h_{kl} + 2h^{pq}\left( \nabla_{[j}h_{q]l}\nabla_{i}h_{kp} - \nabla_{[i}h_{q]l}\nabla_{j}h_{kp}\right) = -R_{ijlk}.
\end{split}
\end{align}
the last equality by the Ricci identity and the statistical property. Tracing \eqref{conjugatecurvature} shows $\bar{R}_{ij} = R_{ipqj}h^{pq}$. 
 By the algebraic Bianchi identity and \eqref{conjugatecurvature}, $-2\bar{R}_{[ij]} = \bar{R}_{ijp}\,^{p} = - R_{ijp}\,^{p} = 2R_{[ij]}$. Claim \eqref{rcsricflat} is immediate from \eqref{conjradcodcurv}. By \eqref{conjradcodcurv}, $\bnabla$ is Ricci symmetric if and only if $\nabla$ is Ricci symmetric. This shows \eqref{rcssymric}. 

Given a volume density, $\Psi$, there is a nonvanishing $A \in \cinf(M)$ such that $\det h= A \Psi^{\tensor 2}$. If $\nabla \Psi = 0$, then
\begin{align}
\begin{split}
(h^{pq}\nabla_{i}h_{pq})A \Psi^{\tensor 2} & = (h^{pq}\nabla_{i}h_{pq})\det h = \nabla_{i}\det h = dA_{i} \Psi^{\tensor 2},
\end{split}
\end{align}
so $h^{pq}\nabla_{i}h_{pq} = A^{-1}dA_{i}$. By \eqref{conjugatediff},
\begin{align}
\begin{split}
\bnabla_{i}(A\Psi) = \nabla_{i}(A\Psi) - (h^{pq}\nabla_{i}h_{pq})A\Psi = (dA_{i} - Ah^{pq}\nabla_{i}h_{pq})\Psi = 0.
\end{split}
\end{align}
This shows \eqref{rcsequiaffine}. It follows that $(\nabla, \rad, h)$ and $(\bnabla, \rad, h)$ are both equiaffine for the same volume density $\Psi$ if and only if $A$ is constant, which holds if and only if $h^{pq}\nabla_{i}h_{pq} = A^{-1}dA_{i}= 0$, that is, $(\nabla, \rad, h)$ is special. Because $\bnabla_{i}h_{jk} = -\nabla_{i}h_{jk}$, $h^{jk}\bnabla_{i}h_{jk} = 0$ if and only if $h^{jk}\nabla_{i}h_{jk} = 0$. This shows \eqref{rcsspecial}.

If $(\nabla, \rad, h)$ is a radiant statistical structure with conjugate $(\bnabla, \rad, h)$, contracting \eqref{conjradcodcurv} with $\rad^{j}$ yields $\bar{\er}_{i} = 2\rad^{p}\bar{R}_{[pi]} = 2\rad^{p}R_{[ip]} = -\er_{i}$. From \eqref{conjugatecurvature} there follows $\bar{R}_{ijk}\,^{p}\rad^{\flat}_{p} = - R_{ijpk}\rad^{p} = 0$, the last equality by \eqref{radid}. By duality, $R_{ijk}\,^{p}\rad^{\flat}_{p} =0$ too. This shows \eqref{rcsradzero}.
\end{proof}

A radiant statistical structure $(\nabla, \rad, h)$ is \emph{nonsingular} if $\rad^{i}$ is nonsingular.

\begin{lemma}\label{conjugateconelikelemma}
Let $(\nabla, \rad, h)$ be a nonsingular radiant statistical structure with conjugate radiant statistical structure $(\bnabla, \rad, h)$. 
Then $(\nabla, \rad, h)$ and $(\bnabla, \rad, h)$ are both conelike if and only if $\nabla$ and $\bnabla$ are both $\rad$-invariant (equivalently, $\rad^{p}\bar{R}_{pijk} = 0$ and $\rad^{p}R_{pijk} = 0$). 
In this case, $(\nabla, \rad, h)$ and $(\bnabla, \rad, h)$ have the same planelike surfaces if and only if $\nabla = \bnabla$ is the Levi-Civita connection of $h$ (in which case $\rad$ is a concurrent vector field).
\end{lemma}

\begin{proof} 
By Lemma \ref{coneconditionlemma}, $(\nabla, \rad, h)$ and $(\bnabla, \rad, h)$ are both conelike if and only if there are $Q_{ij}, \bar{Q}_{ij} \in \Ga(S^{2}\ctm)$ such that $\rad^{p}\rad^{q}Q_{pq} = 0$, $\rad^{p}\rad^{q}\bar{Q}_{pq} = 0$, and 
\begin{align}\label{rcscc0}
&\rad^{p}R_{pijk} = Q_{ij}\rad^{\flat}_{k} - 2q_{(i}h_{j)k}, & &\rad^{p}\bar{R}_{pijk} = \bar{Q}_{ij}\rad^{\flat}_{k} - 2\bar{q}_{(i}h_{j)k}
\end{align}
where $q_{i} = \rad^{p}Q_{ip}$ and $\bar{q}_{i} = \rad^{p}\bar{Q}_{ip}$.
If this is the case, then by \eqref{conjugatecurvature}, 
\begin{align}
\begin{split}
\bar{Q}_{ij}v - 2\bar{q}_{(i}h_{j)k} = \rad^{p}\bar{R}_{pijk} = - \rad^{p}R_{pikj} = -Q_{ik}\rad^{\flat}_{k} + 2q_{(i}h_{k)j},
\end{split}
\end{align}
so that
\begin{align}\label{rcscc1}
\bar{Q}_{ij}\rad^{\flat}_{k} + Q_{ik}\rad^{\flat}_{j} = (\bar{q}_{i} + q_{i})h_{jk} + \bar{q}_{j}h_{ki} + q_{k}h_{ji}.
\end{align}
Tracing \eqref{rcscc1} in $jk$ gives $(n+1)(q_{i} + \bar{q}_{i}) = 0$ where $n = \dim M$. Hence $\bar{q}_{i} = -q_{i}$ and \eqref{rcscc1} becomes
\begin{align}\label{rcscc2}
\bar{Q}_{ij}\rad^{\flat}_{k} + Q_{ik}\rad^{\flat}_{j} =- q_{j}h_{ki} + q_{k}h_{ji}.
\end{align}
Tracing \eqref{rcscc2} in $ij$ and in $ik$ gives $h^{ij}\bar{Q}_{ij}\rad^{\flat}_{k} = (n-1)q_{k}$ and $h^{ik}Q_{ik}\rad^{\flat}_{j} = (1-n)q_{j}$, so that $q_{i} = f\rad^{\flat}_{i}$ where $-h^{pq}\bar{Q}_{pq} =  h^{pq}Q_{pq} = (1-n)f$. Hence $0 = \rad^{p}q_{p} = fv $ where $v  = \rad^{p}\rad^{q}h_{pq}$. In \eqref{rcscc2} this yields
\begin{align}\label{rcscc3}
\bar{Q}_{ij}\rad^{\flat}_{k} + Q_{ik}\rad^{\flat}_{j} =f (\rad^{\flat}_{k}h_{ji} - \rad^{\flat}_{j}h_{ki}).
\end{align}
Contracting \eqref{rcscc3} with $\rad^{k}$ and $\rad^{j}$ gives $v \bar{Q}_{ij} = fv h_{ij} - f\rad^{\flat}_{i}\rad^{\flat}_{j} = - f\rad^{\flat}_{i}\rad^{\flat}_{j}$ and $v Q_{ik} =  f\rad^{\flat}_{i}\rad^{\flat}_{k} - fv h_{ik}  =  f\rad^{\flat}_{i}\rad^{\flat}_{k}$, so that $v \bar{Q}_{ij} =- f\rad^{\flat}_{i}\rad^{\flat}_{j} = -v Q_{ik}$. Wherever $v  \neq 0$, $f$ vanishes and so $Q_{ij} = 0$ and $\bar{Q}_{ij} = 0$. Where $v  =0$ it follows that $f\rad^{\flat}_{i}\rad^{\flat}_{j} = 0$. Because $\rad^{i}$ is nonsingular, for $p \in M$ there is $w^{i} \in T_{p}M_{\ast}$ such that $\rad^{\flat}_{i}w^{i} = 1$, and so $0 = f\rad^{\flat}_{i}\rad^{\flat}_{j}w^{i}w^{j} = f$ at $p$. This shows $f$ vanishes identically. Consequently $q_{i}$ and $\bar{q}_{i}$ also vanish identically. In \eqref{rcscc3} this yields
\begin{align}\label{rcscc4}
\bar{Q}_{ij}\rad^{\flat}_{k} + Q_{ik}\rad^{\flat}_{j} =0.
\end{align}
Also $v \bar{Q}_{ij}  = -v Q_{ik} = 0$. From \eqref{rcscc4} it follows that $Q_{ij}\rad^{\flat}_{k} = Q_{jk}\rad^{\flat}_{i} = Q_{ki}\rad^{\flat}_{j}$. The relation $Q_{ij}\rad^{\flat}_{k} = Q_{ik}\rad^{\flat}_{j}$ implies that there is a one-form $\al_{i}$ such that $Q_{ij} = \al_{i}\rad^{\flat}_{j}$. Since $Q_{[ij]} =0$, there is $z \in \cinf(M)$ such that $\al_{i} = z\rad^{\flat}_{i}$ on $M$. This shows $Q_{ij} = z\rad^{\flat}_{i}\rad^{\flat}_{j}$ and from \eqref{rcscc4} it follows that $\bar{Q}_{ij} = -z\rad^{\flat}_{i}\rad^{\flat}_{j} = -Q_{ij}$. It follows that $0 = \rad^{p}\rad^{q}Q_{pq} = zv^{2}$, so that $z$ vanishes wherever $v $ does not. Since, by Lemma \ref{preradcodazzilemma}, $dv  = \rad^{\flat}_{i}$, and $\rad^{i}$ is by assumption nonvanishing, the zero set of $v $ is contained in a union of codimension one submanifolds of $M$; because $z$ is smooth and vanishes off this union, $z$ vanishes identically. This proves that $Q_{ij} =0$ and $\bar{Q}_{ij} =0$. In \eqref{rcscc0} there results $\rad^{p}R_{pijk} = 0$ and $\rad^{p}\bar{R}_{pijk} =0$, so that both $\nabla$ and $\bnabla$ are $\rad$-invariant.

Now suppose $\rad$ is nonsingular and the conjugate radiant statistical structures $(\nabla, \rad, h)$ and $(\bnabla, \rad, h)$ are conelike and have the same planelike surfaces. By Lemma \ref{coneconditionlemma} and \eqref{btau} there is $Q_{ij} \in \Ga(S^{2}\ctm)$ satisfying $\rad^{p}\rad^{q}Q_{pq} = 0$ and such that $h^{kp}\nabla_{i}h_{jp} = \bnabla - \nabla = Q_{ij}\rad^{k} - 2q_{(i}\delta_{j)}\,^{k}$ where $q_{i} = \rad^{p}Q_{ip}$. Hence $0 = (\lie_{\rad}h)_{ij} - 2h_{ij} = \rad^{p}\nabla_{p}h_{ij} = \rad^{p}(Q_{pi}\rad^{\flat}_{j} - q_{p}h_{ij} - q_{i}h_{pj}) = -q_{i}\rad^{\flat}_{j}$. Since $\rad^{i}$ is nonsingular this implies $q_{i} = 0$, so that $\nabla_{i}h_{jk} = Q_{ij}\rad^{\flat}_{k}$. Since $\nabla_{i}h_{jk}$ is completely symmetric, this implies $Q_{ij}\rad^{\flat}_{k} = Q_{jk}\rad^{\flat}_{i} = Q_{ki}\rad^{\flat}_{j}$. Arguing as in the previous paragraph, this implies there is $z \in \cinf(M)$ such that $Q_{ij} = z\rad^{\flat}_{i}\rad^{\flat}_{j}$ and $zv  = 0$. Again as in the previous paragraph, this implies $Q_{ij} =0$, so that $\nabla_{i}h_{jk} = 0$ and $\bnabla = \nabla$.
\end{proof}

\section{AH structures and Einstein equations for statistical structures}\label{einsteinstatisticalsection}
The notion of AH structure introduced and studied by the author in \cite{Fox-ahs, Fox-2dahs, Fox-ricweyl, Fox-crm, Fox-schwarz} is recalled because it provides a convenient language for results described in Section \ref{ewsection} that generalize known constructions associating Einstein-Weyl structures with constant scalar curvature Kähler metrics. (For background on Einstein-Weyl structures, consult the survey \cite{Calderbank-Pedersen}.) What is needed is a generalization of the notion of Einstein-Weyl structure. Informally, AH structures are to statistical structures as Weyl structures are to metric structures. The definition of AH structure given here is stated in a different way than in \cite{Fox-ahs, Fox-2dahs, Fox-ricweyl, Fox-crm, Fox-schwarz} but the formulation given here, although equivalent, is more economical and for it more readable.

An \emph{AH structure} on an $n$-dimensional manifold, $M$, is a pair $(\nabla, [g])$ comprising a torsion-free affine connection, $\nabla$, and a conformal class $[g]$ of pseudo-Riemannian metrics, that satisfy two compatibility conditions:
\begin{enumerate}
\item\label{ahdef1} For each $g \in [g]$ there is a one-form $\chi_{i} \in \Ga(\ctm)$ such that $\nabla_{[i}g_{j]k} = \chi_{[i}g_{j]k}$.
\item\label{ahdef2} The pair is \emph{aligned}, meaning that for each $g \in [g]$ there holds $g^{pq}\nabla_{i}g_{pq} = ng^{pq}\nabla_{p}g_{qi}$.
\end{enumerate}

It is straightforward to check that each of the two conditions is well-posed. If $\tilde{g} = e^{f}g$ with $f \in \cinf(M)$, then $\nabla_{i}\tilde{g}_{jk} = e^{f}(\nabla_{i}g_{jk} + df_{i}g_{jk})$. Tracing this in two ways shows $\tilde{g}^{pq}(\tnabla_{i}\tilde{g}_{pq} - n \tnabla_{p}\tilde{g}_{qi}) = g^{pq}(\nabla_{i}g_{pq} - n \nabla_{p}g_{qi})$, so that the alignment condition is well-defined, while antisymmetrizing it in $ij$ shows that if $\nabla_{[i}g_{j]k} = \chi_{[i}g_{j]k}$, then $\tnabla_{[i}\tilde{g}_{j]k} = \tilde{\chi}_{[i}\tilde{g}_{j]k}$ with $\tilde{\chi}_{i} = \chi_{i} + df_{i}$. Note that it follows that $d\chi_{ij}$ does not depend on the choice of $g \in [g]$.

\begin{remark}
The letters \emph{AH} abbreviate \emph{affine hypersurface}; see \cite{Fox-ahs, Fox-2dahs, Fox-crm, Fox-schwarz} for motivation.
\end{remark}

\begin{remark}
The conformal structure $[g]$ can be identified with the scale-invariant weighted tensor $G_{ij} = |\det g|^{-1/n}g_{ij}$. The alignment condition is equivalent to the vanishing of all traces of the covariant derivative of $G_{ij}$, $G^{ij}\nabla_{i}G_{jk} = 0$ (because $|\det G| = 1$, $G^{jk}\nabla_{i}G_{jk} = 0$ is automatic). In \cite{Fox-ahs, Fox-2dahs, Fox-crm} everything is written in terms of $G_{ij}$; while economical and conceptually clean, it appears that this impedes understandability.
\end{remark}

The curvature of an AH structure $(\nabla, [g])$ means the curvature of $\nabla$.
For an AH structure $(\nabla, [g])$, the one-form $\chi_{i}$ associated with $g \in [g]$ is given by $\chi_{i} = g^{pq}\nabla_{p}g_{qi} = \tfrac{1}{n}g^{pq}\nabla_{i}g_{pq} = \tfrac{1}{n}(\det g)^{-1}\nabla_{i}\det g$. 
 It follows that the curvature of $(\nabla, [g])$ satisfies
\begin{align}\label{faraday1}
nd\chi_{ij} = 2n\nabla_{[i}\chi_{j]} = 2(\det g)^{-1}\nabla_{[i}\nabla_{j]}\det g = - 2R_{ijp}\,^{p} = 4R_{[ij]}.
\end{align}
Following \cite{Calderbank-faraday} in the context of Weyl structures, define the \emph{Faraday curvature} $F_{ij} \in \Ga(\ext^{2}\ctm)$ of the AH structure to be the curvature of the connection induced by $\nabla$ on the line bundle of $-(1/n)$-densities. Because $|\det g|^{-1/(2n)}$ is $-(1/n)$-density it follows from \eqref{faraday1} that $F_{ij} = -\tfrac{1}{2}d\chi_{ij}$.

An AH structure $(\nabla, [g])$ is \emph{exact} if the one-form $\chi_{i}$ associated with any $g \in [g]$ is exact. Equivalently, there is $g \in [g]$ such that the associated one-form $\chi_{i}$ is identically zero. In this case $g$ is a \emph{distinguished} representative of $[g]$. A distinguished representative is determined only up to positive homothety. Because $\nabla_{i}g^{ij} = -\chi^{j}$, an AH structure is exact if and only if there is $g \in [g]$ so that the divergence $\nabla_{i}g^{ij}$ vanishes. An AH structure is \emph{closed} if its Faraday curvature is identically zero (equivalently, $\chi$ is closed for any $g \in [g]$). A closed AH structure is locally exact, but it need not be globally so.

By definition, a distinguished representative $g$ of an exact AH structure $(\nabla, [g])$ determines with $\nabla$ a special statistical structure.  Lemma \ref{statisticalahlemma} implies that a statistical structure $(\nabla, g)$ generates an AH structure $(\tnabla, [g])$ and the AH structure generated by a statistical structure $(\nabla, g)$ is exact if and only if $(\nabla, g)$ is a special statistical structure. Consequently, an exact AH structure can be viewed as a positive homothety class of special statistical structures.

\begin{lemma}\label{statisticalahlemma}
On an $n$-manifold $M$, if $(\nabla, g)$ satisfies $\nabla_{[i}g_{j]k} = \ga_{[i}g_{j]k}$ for some one-form $\ga_{i}$, then the connection $\tnabla$ defined by $\tnabla = \nabla + 2\si_{(i}\delta_{j)}\,^{k}$ where $\si_{i}= \tfrac{1}{n+2}\left(g^{pq}\nabla_{i}g_{pq} - n \ga_{i}\right)$ generates with the conformal structure $[g]$ an AH structure, the AH structure \emph{generated} by $(\nabla, g)$. Moreover, the AH structure generated by a statistical structure $(\nabla, g)$ is exact with distinguished representative $g$ if and only if $(\nabla, g)$ is a special statistical structure.
\end{lemma}

\begin{proof}
Suppose $\tnabla = \nabla + 2\si_{(i}\delta_{j)}\,^{k}$ for some one-form $\si_{i}$. Straightforward computation using $\tnabla_{[i}g_{j]k} = \ga_{[i}g_{j]k} - \si_{[i}g_{j]k}$ shows 
\begin{align}\label{alignedtransform}
g^{pq}(\tnabla_{i}g_{pq} - n\tnabla_{p}g_{qi}) = g^{pq}(\nabla_{i}g_{pq} - n\nabla_{p}g_{qi}) + (n-1)(n+2)\si_{i}. 
\end{align}
Because $\nabla_{[i}g_{j]k} = \ga_{[i}g_{j]k}$, $g^{pq}\nabla_{p}g_{qi} = g^{pq}\nabla_{i}g_{pq} + (n-1)\ga_{i}$, so from \eqref{alignedtransform} there results $g^{pq}(\tnabla_{i}g_{pq} - n\tnabla_{p}g_{qi}) = (n-1)((n+2)\si_{i} - g^{pq}\nabla_{i}g_{pq} + n\ga_{i})$, so that $\tnabla$ is aligned with respect to $[g]$ and $(\tnabla, [g])$ is an AH structure if and only if $\si_{i}= \tfrac{1}{n+2}\left(g^{pq}\nabla_{i}g_{pq} - n \ga_{i}\right)$.

If $(\nabla, g)$ is a special statistical structure, then $\ga_{i} = 0$ and $g^{pq}\nabla_{i}g_{pq} = 0$, so $\si_{i} = 0$, $\tnabla = \nabla$ is exact, and $g$ is a distinguished representative of the AH structure $(\nabla, [g])$. 

If $(\nabla, g)$ is a statistical structure, then $\ga_{i} = 0$ by hypothesis, and the AH structure $(\tnabla, [g])$ it generates is determined by $\si_{i} =  \tfrac{1}{n+2}g^{pq}\nabla_{i}g_{pq}$. On the other hand, the Faraday primitive of $(\tnabla, [g])$ associated with $g$ is $\chi_{i} = \tfrac{1}{n}g^{pq}\tnabla_{i}g_{pq} = \tfrac{n+2}{n}\si_{i}$. It follows that if $(\tnabla, [g])$ is exact with distinguished representative $g$, then $\si_{i} =0$ and $(\nabla, g)$ is special.
\end{proof}

It follows from \eqref{alignedtransform} that given a conformal structure $[g]$ any torsion-free affine connection is projectively equivalent to a unique such connection that is aligned with respect to $[g]$.

Lemma \ref{locallycodazzilemma} shows that an AH structure is the same thing as a \emph{locally statistical structure} in the sense defined in its statement.
\begin{lemma}\label{locallycodazzilemma}
For a pair $(\en, [g])$ comprising a projective structure $\en$ and conformal structure $[g]$ on a manifold $M$ the following are equivalent.
\begin{enumerate}
\item\label{cpah2} $(\en, [g])$ is \emph{locally statistical} in the sense that every $p \in M$ is contained in an open neighborhood $U \subset M$ on which there is a representative $\nabla \in \en$ (not necessarily aligned with respect to $[g]$) and a representative metric $g \in [g]$ such that $\nabla_{[i}g_{j]k} = 0$ on $U$, that is such that $(\nabla, g)$ is a statistical structure on $U$.
\item\label{cpah3} For the aligned representative $\nabla \in \en$, $(\nabla, [g])$ is an AH structure. 
\item\label{cpah5} For any $g \in [g]$ there is $\tnabla \in \en$ such that $(\tnabla, g)$ is a statistical structure.
\end{enumerate}
\end{lemma}

\begin{proof}
Suppose $(\en, [g])$ is locally statistical. Let $\nabla \in \en$ be the unique representative aligned with respect to $[g]$ and fix a reference metric $g \in [g]$. By assumption there is an open cover $\{U_{a}\}$ of $M$ such that on $U_{a}$ there are a connection $\nabla(a) = \nabla + 2\si(a)_{(i}\delta_{j)}\,^{k}$ and a metric $g(a)_{ij} = e^{f(a)}g_{ij}$ representing the restrictions to $U_{a}$ of $\en$ and $[g]$ and satisfying $0 = \nabla(a)_{[i}g(a)_{j]k}$. Antisymmetrizing $\nabla(a)_{i}g(a)_{jk} = e^{f(a)}(\nabla_{i}g_{jk} + df(a)_{i}h_{jk} - 2\si(a)_{i}g_{jk} - 2\si(a)_{(j}g_{k)i})$ yields
\begin{align}
0 = \nabla(a)_{[i}g(a)_{j]k} = e^{f(a)}\left(\nabla_{[i}g_{jk]} + df(a)_{[i}g_{j]k} - \si(a)_{i}g_{jk}\right),
\end{align}
so that $\nabla_{[i}g_{jk]}  = \tau(a)_{[i}g_{j]k}$ where $\tau(a)_{i} = \si(a)_{i} - df(a)_{i}$. On the overlap $U_{a}\cap U_{b}$ there holds $\tau(b)_{[i}g_{j]k} = \nabla_{[i}g_{jk]}  = \tau(a)_{[i}g_{j]k}$. Tracing this in $jk$ yields $\tau(a)_{i} = \tau(b)_{i}$ so the one-forms $\tau(a)$ patch together to yield a globally defined one-form $\tau$ such that $\nabla_{[i}g_{j]k} = \tau_{[i}g_{j]k}$. This shows that \eqref{cpah2} implies \eqref{cpah3}. If there holds \eqref{cpah3}, then, for any $g \in [g]$, $\tnabla = \nabla + 2\si_{(i}\delta_{j)}\,^{k}$ satisfies $\tnabla_{[i}g_{j]k} = \nabla_{[i}g_{j]k} - \si_{[i}g_{j]k}  = \tau_{[i}g_{j]k}- \si_{[i}g_{j]k}$, so $\tnabla_{[i}g_{j]k} =\tilde{\tau}_{[i}g_{j]k}$ with $\tilde{\tau}_{i} = \tau_{i} - \si_{i}$. This shows that \eqref{cpah3} implies \eqref{cpah2}. 
If $(\en, [g])$ is an AH structure, then by definition, $\nabla_{[i}g_{j]k} = \chi_{[i}g_{j]k}$. The connection $\tnabla = \nabla + 2\chi_{(i}\delta_{j)}\,^{k}$ then satisfies $\tnabla_{[i}h_{j]k}  = 0$. This shows that \eqref{cpah3} implies \eqref{cpah5}. It is immediate that \eqref{cpah5} implies \eqref{cpah2}.
\end{proof}

\begin{lemma}
A radiant statistical structure $(\nabla, \rad, h)$ on $M$ is special if and only if $(\nabla, [g])$ is an AH structure on $\hat{M} = \{p \in M: \rad_{p} \neq 0\}$, where $g_{ij} = v^{-1}h_{ij}$ with $v = \rad^{p}\rad^{q}h_{pq}$.
\end{lemma}

\begin{proof}
By Corollary \ref{radcodazzicorollary}, if $(\nabla, \rad, h)$ is a radiant statistical structure, then $g_{ij}= v^{-1}h_{ij} = \nabla_{i}\be_{j} + 2\be_{i}\be_{j}$, where $\be_{i} = (1/2)v^{-1}dv_{i}$, satisfies $\nabla_{[i}g_{j]k} = \nabla_{[i}\nabla_{j]}\be_{k} - 2\be_{[i}\nabla_{j]}\be_{k} = - \tfrac{1}{2}R_{ijk}\,^{p}\be_{p} - 2\be_{[i}g_{j]k} = -2\be_{[i}g_{j]k}$, so, by Lemma \ref{statisticalahlemma}, $\nabla$ and $g_{ij}$ generate an AH structure $(\tnabla, [g])$ on $\hat{M}$. Moreover, by the proof of Lemma \ref{statisticalahlemma}, $\tnabla = \nabla$ if and only if $g^{pq}\nabla_{i}g_{pq} = -2n\be_{i} = -nv^{-1}dv_{i}$. This last condition is equivalent to $v^{n}\det g = \det h$ being $\nabla$-parallel on $\hat{M}$, and because $\hat{M}$ is open and dense in $M$, this holds if and only if $\det h$ is $\nabla$-parallel on $M$. 
\end{proof}

\begin{remark}
Note that given $(\nabla, [g])$ as in Lemma \ref{statisticalahlemma} there is no way to reconstruct $h$ without the additional data of a function $v$ such that $v^{n}\det g$ is $\nabla$-parallel. Constructing such a function from $\nabla$ and $g$ amounts to solving a Monge-Ampère equation. 
\end{remark}

Given an AH structure $(\nabla, [g])$ and $g \in [g]$ with associated one-form $\chi_{i}$, the tensor $\bt_{ijk}$ defined by
\begin{align}\label{cubicformdefined}
\bt_{ijk} = \nabla_{i}g_{jk} - \chi_{i}g_{jk}
\end{align}
is completely symmetric and completely $g$-trace-free. It is the \emph{cubic form} of $\nabla$ with respect to the representative $g \in [g]$. The tensor $\bt_{ij}\,^{k} = g^{kp}\bt_{ijp}$ does not depend on the choice of $g \in [g]$ and is called the \emph{cubic torsion} of $(\nabla, [g])$. It is straightforward to check that the connection $\anabla = \nabla + \bt_{ij}\,^{k}$ is aligned with respect to $[g]$ and generates with $[g]$ an AH structure having cubic torsion $\wideparen{\bt}_{ij}\,^{k} = -\bt_{ij}\,^{k}$. The AH structure $(\anabla, [g])$ is said to be \emph{conjugate} to $(\nabla, [g])$. Conjugacy is an involution on the space of AH structures.  The curvature of $(\anabla, [g])$ and the tensors derived from it are indicated by decoration with $\wideparen{\dum}$, as in $\wideparen{R}_{ijk}\,^{l}$.

\begin{remark}
The AH structure generated by the statistical structure conjugate to a given statistical structure is not equal to the AH structure conjugate to the AH structure generated by the given statistical structure unless the given statistical structure is special. In general the resulting conjugate connections differ by a tensor of the form $2\si_{(i}\delta_{j)}\,^{k} - g_{ij}\si^{k}$, where $\si_{i}$ is a multiple of the covariant derivative of $\det g$. 
\end{remark}

\begin{example}
A \emph{Weyl structure} is a pair $(\nabla, [g])$ comprising a torsion-free affine connection, $\nabla$, and a conformal class $[g]$ of pseudo-Riemannian metrics such that, for each $g \in [g]$, there is a one-form $\chi_{i} \in \Ga(\ctm)$ such that $\nabla_{i}g_{jk} = \chi_{i}g_{jk}$. By definition a Weyl structure is an AH structure with vanishing cubic form and the Weyl structures are exactly those AH structures that are self-conjugate, equivalently, that have vanishing cubic form. 
\end{example}

\begin{example}\label{affinehypersurfaceexample}
Although the language used here is a bit different, the geometric content of this example is fully present in H. Matsuzoe's \cite{Matsuzoe-conformallyprojectively, Matsuzoe}. For details of the claims that follow in the language used here see \cite{Fox-ahs, Fox-2dahs, Fox-crm}. The Blaschke metric $h$ of a cooriented nondegenerate hypersurface immersion in flat affine space constitutes with the connection $\nabla$ induced via the affine normal a special statistical structure, so generates an exact AH structure for which $h$ is a distinguished metric. The pullback via the conormal Gauss map of the flat projective structure on the projectivization of the vector space dual to the ambient flat affine space yields a projective structure that has a unique representative $\anabla$ aligned with respect to the equiaffine metric $h$, and the AH structure $(\anabla, [h])$ it generates with $h$ is that conjugate to $(\nabla, [h])$. 
\end{example}

\begin{remark}
H. Matsuzoe \cite{Matsuzoe} called \emph{semi-Weyl} a pair $(\nabla, g)$ satisfying the first condition \eqref{ahdef1} and observed the relation with geometric structures induced on hypersurfaces in flat affine space described in Example \ref{affinehypersurfaceexample}. The alignment condition \eqref{ahdef2} seems to be important for properly formulating the Einstein equations defined later in Definition \ref{einsteinahdefinition}. Matsuzoe also defined the cubic form of $(\nabla, g)$ by \eqref{cubicformdefined}. However, without the alignment condition, $\bt_{ijk}$ need not be trace-free. On the other hand, in Matsuzoe's setting, the vanishing of $\bt_{ijk}$ implies the alignment condition, so this condition is somehow hidden from view. 
\end{remark}

The \emph{scalar curvature} of the AH structure $(\nabla, [g])$ associated with the representative $g \in [g]$ is the function $\sc = g^{pq}R_{pq}$.

\begin{lemma}
For $g \in [g]$ with $\nabla_{[i}g_{j]k} = \chi_{[i}g_{j]k}$, the curvatures $R_{ijkl} = R_{ijk}\,^{p}g_{pl}$ and $\wideparen{R}_{ijkl} = \wideparen{R}_{ijk}\,^{p}g_{pl}$ of the AH structure $(\nabla, [g])$ and the conjugate AH structure $(\anabla, [g])$ are related by:
\begin{align}\label{ahcurv}
\begin{aligned}
\wideparen{R}_{ijkl}& = R_{ijkl} + 2\nabla_{[i}\bt_{j]kl} - 2\chi_{[i}\bt_{j]kl}= -R_{ijlk} - d\chi_{ij} g_{kl}, \\
\wideparen{R}_{ij}  &= R_{ij} + \nabla_{p}\bt_{ij}\,^{p} - \bt_{ip}\,^{q}\bt_{jq}\,^{p}= g^{pq}R_{ipqj} + d\chi_{ij}.
\end{aligned}
\end{align}
In particular,  for any $g \in [g]$, $(\nabla, [g])$ and $(\anabla, [g])$ have the same scalar curvature $\wideparen{\sc} = g^{pq}\wideparen{R}_{pq} = g^{pq}R_{pq} = \sc$.
\end{lemma}
\begin{proof}
By definition of $\bt_{ijk}$,
\begin{align}
\begin{split}
\wideparen{R}_{ijkl} & = R_{ijkl} + 2g_{lp}\nabla_{[i}\bt_{j]k}\,^{p} + 2\bt_{pl[i}\bt_{j]k}\,^{p} = R_{ijkl} + 2\nabla_{[i}\bt_{j]kl}  - 2\bt_{k[j}\,^{p}\nabla_{i]}g_{lp} + 2\bt_{pl[i}\bt_{j]k}\,^{p}\\
& = R_{ijkl} + 2\nabla_{[i}\bt_{j]kl} - 2\chi_{[i}\bt_{j]kl}\\
& = R_{ijkl} + 2\nabla_{[i}\nabla_{j]}g_{kl} - 2\nabla_{[i}\chi_{j]}g_{kl}  - 2\chi_{[j}\nabla_{i]}g_{kl} - 2\chi_{[i}\bt_{j]kl} = - R_{ijlk} - d\chi_{ij}g_{kl}.
\end{split}
\end{align}
The remaining identities follow by taking traces and using $\nabla_{i}g^{ij} = -\chi^{j} = - g^{jp}\chi_{p}$.
\end{proof}

An AH structure $(\nabla, [g])$ has \emph{self-conjugate} curvature if the curvature tensor of the conjugate AH structure equals that of $(\nabla, [g])$. For example, by \eqref{ahcurv} a Weyl structure has self-conjugate curvature because its cubic torsion vanishes.

A tensor derived from the curvature tensor of an AH structure that is unchanged under conjugacy is said to be \emph{self-conjugate}. For example, the scalar curvature of an AH structure $(\nabla, [g])$ associated with $g \in [g]$ is self-conjugate.
\begin{lemma}\label{ahconservationlemma}
Let $(\nabla, [g])$ be an AH structure. Let $\chi_{i}$ and $\sc$ be the one-form and scalar curvature associated with the representative $g \in [g]$. The one-forms defined by
$d\sc_{i} + \sc \chi_{i}$ and $\tfrac{n}{2}g^{pq}\nabla_{p}d\chi_{qi}$ do not depend on the choice of $g$, so are associated with the AH structure $(\nabla, [g])$.
\end{lemma}

\begin{proof}
If $\tilde{g} = fg$ with $0 < f \in \cinf(M)$, then $\tilde{\sc} = \tilde{g}^{ij}R_{ij} = f^{-1}\sc$ and $\nabla_{[i}\tilde{g}_{j]k} = \tilde{\chi}_{[i}\tilde{g}_{j]k}$ with $\tilde{\chi}_{i} = \chi_{i} + d\log f_{i}$. It follows that $d\tilde{\sc}_{i} + \tilde{\sc}\tilde{\chi}_{i} = d\sc_{i} + \sc \chi_{i}$ and $d\tilde{\chi}_{ij} = d\chi_{ij}$.
\end{proof}

For an AH structure $(\nabla, [g])$ and $g \in [g]$ define a symmetric tensor by
\begin{align}\label{stdefined}
\T_{ij} = \bt_{ip}\,^{q}\bt_{jq}\,^{p} - \tfrac{1}{2}|\bt|^{2}g_{ij},
\end{align}
where $|\bt|^{2} = \bt^{ijk}\bt_{ijk}$. The expression \eqref{stdefined} does not depend on the choice of $g \in [g]$, so $\T$ is associated with the AH structure $(\nabla, [g])$. 

\begin{lemma}\label{conservationconditionlemma}
On an $n$-manifold, let $(\nabla, [g])$ be an AH structure with conjugate AH structure $(\anabla, [h])$. For $g \in [g]$, with associated one-form $\chi_{i}$ and scalar curvature $\sc$, there holds
\begin{align}\label{conserve}
\begin{aligned}
\bt^{abc}R_{i(abc)} &= \tfrac{n-2}{n}\left(d\sc_{i} + \sc \chi_{i} + \tfrac{n}{2}g^{pq}\nabla_{p}d\chi_{qi}\right) - g_{ia}g^{pq}\left(\anabla_{p}\wideparen{S}_{q}\,^{a} + \nabla_{p}S_{q}\,^{a}\right) - \chi^{p}(S_{ip} + \wideparen{S}_{ip})\\
& = \tfrac{1}{2}\div(\T)_{i} + \tfrac{2-n}{2n}\bt_{i}\,^{ab}\div(\bt)_{ab} - \tfrac{1}{n}\bt_{i}\,^{ab}(\wideparen{S}_{ab}-  S_{ab}),
\end{aligned}
\end{align}
where $S_{ij} = R_{(ij)} - \tfrac{\sc}{n}g_{ij}$ is the trace-free symmetrized Ricci tensor of $(\nabla, [g])$, $S_{i}\,^{j} = g^{ja}S_{ia}$, $\chi^{i} = g^{ip}\chi_{p}$, and $\sT_{ij}$ is the tensor defined in \eqref{stdefined}. 
\end{lemma}

\begin{proof}
The proof  is an adaptation of the usual argument showing that the Einstein tensor of a metric is divergence free.

Observe that $R_{ij} = \tfrac{\sc}{n}g_{ij} + \tfrac{n}{4}d\chi_{ij} + S_{ij}$, so that $S_{[ij]} = 0$ and $\wideparen{R}_{ij} - R_{ij} = \wideparen{S}_{ij} - S_{ij}$. From \eqref{ahcurv} it follows that $R_{ip}\,^{p}\,_{j} = \wideparen{R}_{ij} - d\chi_{ij} = \tfrac{\sc}{n}g_{ij} + \tfrac{n-4}{4}d\chi_{ij} + \wideparen{S}_{ij}$. Differentiating $R_{ij} = \tfrac{\sc}{n}g_{ij} + \tfrac{n}{4}d\chi_{ij} + S_{ij}$ and antisymmetrizing the result yields
\begin{align}\label{conserve1}
\begin{split}
\nabla_{i}R_{jk} &  = \tfrac{1}{n}(d\sc_{i} + \sc \chi_{i})g_{jk} + \tfrac{\sc}{n}\bt_{ijk} + \tfrac{n}{4}\nabla_{i}d\chi_{jk} + \nabla_{i}S_{jk},\\
2\nabla_{[i}R_{j]k} &  = \tfrac{2}{n}(d\sc_{[i}g_{j]k} + \sc \chi_{[i}g_{j]k}) - \tfrac{n}{4}\nabla_{k}d\chi_{ij} + 2\nabla_{[i}S_{j]k}.
\end{split}
\end{align}
Contracting \eqref{conserve1} with $g^{jk}$ yields
\begin{align}\label{conserve2}
\begin{split}
2g^{jk}\nabla_{[i}R_{j]k} &= \tfrac{n-1}{n}(d\sc_{i} + \sc \chi_{i}) + \tfrac{n}{4}g^{pq}\nabla_{p}d\chi_{qi}  - g^{pq}\nabla_{p}S_{iq}+ \bt_{i}\,^{pq}S_{pq}\\
& =  \tfrac{n-1}{n}(d\sc_{i} + \sc \chi_{i}) + \tfrac{n}{4}g^{pq}\nabla_{p}d\chi_{qi} - g_{ia}g^{pq}\nabla_{p}S_{q}\,^{a} - \chi^{p}S_{qi}, 
\end{split}
\end{align}
in which the last inequality follows from $\nabla_{i}g^{jk} = -\bt_{i}\,^{jk} - \chi_{i}g^{jk}$ and that $S_{ij}$ is trace-free. On the other hand, by the differential Bianchi identity and $R_{ip}\,^{p}\,_{j} = \tfrac{\sc}{n}g_{ij} + \tfrac{n-4}{4}d\chi_{ij} + \wideparen{S}_{ij}$, 
\begin{align}\label{conserve3}
\begin{split}
2g^{jk}\nabla_{[i}R_{j]k} &= g^{jk}\nabla_{p}R_{ijk}\,^{p} = \nabla_{p}R_{ia}\,^{ap} - R_{ijk}\,^{p}\nabla_{p}g^{jk}\\
&= \nabla_{p}R_{ia}\,^{ap} + \chi^{p}R_{ia}\,^{a}\,_{p} + \bt^{abc}R_{ibca}  = g^{pq}\nabla_{p}R_{ia}\,^{a}\,_{q} + \bt^{abc}R_{i(abc)}\\ 
& = \tfrac{1}{n}(d\sc_{i} + \sc\chi_{i}) + \tfrac{4-n}{4}g^{pq}\nabla_{p}d\chi_{qi} + g^{pq}\nabla_{p}\wideparen{S}_{qi} + \bt_{i}\,^{pq}\wideparen{S}_{pq}+ \bt^{abc}R_{i(abc)}\\ 
& = \tfrac{1}{n}(d\sc_{i} + \sc\chi_{i}) + \tfrac{4-n}{4}g^{pq}\nabla_{p}d\chi_{qi} + g_{ia}g^{pq}\anabla_{p}\wideparen{S}_{q}\,^{a} + \chi^{p}\wideparen{S}_{pi} + \bt^{abc}R_{i(abc)}, 
\end{split}
\end{align}
in which the last equality follows from $\anabla_{i}g^{jk} = \bt_{i}\,^{jk} - \chi_{i}g^{jk}$
Combining \eqref{conserve2} and \eqref{conserve3} yields the first equality of \eqref{conserve}.

The Levi-Civita connection $D$ of $g$ has the form
\begin{align}\label{dnabladifference}
D = \nabla + \tfrac{1}{2}\left(\bt_{ij}\,^{k}  + \chi_{i}\delta_{j}\,^{k} + \chi_{j}\delta_{i}\,^{k} - g_{ij}\chi^{k}\right).
\end{align}
From \eqref{dnabladifference} there follow
\begin{align}\label{dbtids}
\begin{aligned}
D_{i}\bt_{jkl} & = \nabla_{i}\bt_{jkl} - \tfrac{3}{2}\bt_{i(j}\,^{p}\bt_{kl)p} - \tfrac{3}{2}\chi_{i}\bt_{jkl} - \tfrac{3}{2}\chi_{(j}\bt_{kl)i} + \tfrac{3}{2}g_{i(j}\bt_{kl)}\,^{p}\chi_{p},\\
D_{[i}\bt_{j]kl} & = \nabla_{[i}\bt_{j]kl} - \chi_{[i}\bt_{j]kl} + \tfrac{1}{2}g_{k[i}\bt_{j]l}\,^{p}\chi_{p} +\tfrac{1}{2}g_{l[i}\bt_{j]k}\,^{p}\chi_{p},\\
\div(\bt)_{ij}& =   \nabla_{p}\bt_{ij}\,^{p} - \bt_{ip}\,^{q}\bt_{jq}\,^{p} + \tfrac{n}{2}\bt_{ij}\,^{p}\chi_{p},
\end{aligned}
\end{align}
where $\div(\bt)_{ij} = D_{p}\bt_{ij}\,^{p}$. Combining \eqref{dbtids} with \eqref{ahcurv} yields
\begin{align}\label{Dcurv}R_{ij(kl)} & = -\nabla_{[i}\bt_{j]kl}  - \tfrac{1}{2}d\chi_{ij} g_{kl}  + \chi_{[i}\bt_{j]kl}= -D_{[i}\bt_{j]kl}  - \tfrac{1}{2}d\chi_{ij} g_{kl}  + \tfrac{1}{2}g_{k[i}\bt_{j]lp}\chi^{p} +\tfrac{1}{2}g_{l[i}\bt_{j]kp}\chi^{p},\\
\label{Driccurv}\wideparen{R}_{ij} & = R_{ij} + \div(\bt)_{ij} -  \tfrac{n}{2}\bt_{ij}\,^{p}\chi_{p}.
\end{align}
The divergence $\div(\T)_{i} = D^{p}\T_{ip}$ satisfies
\begin{align}\label{divT}
\div(\T)_{i}  &  = \bt_{i}\,^{ab}\div(\bt)_{ab} + \bt^{abc}D_{a}\bt_{bci} - \tfrac{1}{2}D_{i}|\bt|^{2}.
\end{align}
Contracting $\bt^{abc}$ and $R_{iabc}$ and simplifying the result using \eqref{Dcurv} and \eqref{divT} yields directly
\begin{align}\label{pdt1}
\begin{aligned}
\bt^{abc}R_{iabc} & = \bt^{abc}R_{i(abc)} = -\tfrac{1}{2}\bt^{abc}D_{i}\bt_{abc} + \tfrac{1}{2}\bt^{abc}D_{a}\bt_{bci} + \tfrac{1}{2}\bt^{abc}\left(g_{b[i}\bt_{a]c}\,^{p}\chi_{i} + g_{c[i}\bt_{a]b}\,^{p}\chi_{i}\right)\\
& =-\tfrac{1}{4}D_{i}|\bt|^{2}  + \tfrac{1}{2}\left(\div(\T)_{i} - \bt_{i}\,^{ab}\div(\bt)_{ab} + \tfrac{1}{2}D_{i}|\bt|^{2}\right) + \tfrac{1}{2}\bt_{ip}\,^{q}\bt_{aq}\,^{p}\chi^{a}\\
& = \tfrac{1}{2}\div(\T)_{i} - \tfrac{1}{2}\bt_{i}\,^{ab}\div(\bt)_{ab} + \tfrac{1}{2}\bt_{ip}\,^{q}\bt_{aq}\,^{p}\chi^{a}.
\end{aligned}
\end{align}
Substituting \eqref{Driccurv} in \eqref{pdt1} yields the second equality of \eqref{conserve}.
\end{proof}

Lemma \ref{ahconservationlemma} shows that the condition \eqref{conservationcondition} in Definition \ref{einsteinahdefinition} is well posed.

\begin{definition}\label{einsteinahdefinition}
\noindent
\begin{itemize}
\item An AH structure $(\nabla, [g])$ is \emph{half naive Einstein} if the symmetric part of its Ricci curvature is a multiple of any $g \in [g]$. That is, for any $g \in [g]$, $R_{(ij)} = \tfrac{\sc}{n}g_{ij}$.
\item An AH structure is \emph{conjugate half naive Einstein} if the conjugate AH structure is half naive Einstein.
\item An AH structure $(\nabla, [g])$ is \emph{naive Einstein} if it is both half naive Einstein and conjugate half naive Einstein - the symmetric parts of its Ricci curvature and the Ricci curvature of the conjugate AH structure are multiples of any $g \in [g]$. That is, for any $g \in [g]$, $R_{(ij)} = \tfrac{\sc}{n} g_{ij} = \wideparen{R}_{(ij)}$.
\item An AH structure $(\nabla, [g])$ is \emph{Einstein} if it is naive Einstein and for every $g \in [g]$ with associated one-form $\chi_{i}$ and scalar curvature $\sc$, there holds the \emph{conservation condition}
\begin{align}\label{conservationcondition}
0 = d\sc_{i} + \sc \chi_{i} + \tfrac{n}{2}g^{pq}\nabla_{p}d\chi_{qi}.
\end{align}
\end{itemize}
\end{definition}

\begin{lemma}\label{conservativenaivelemma}
On an $n$-manifold, for a naive Einstein AH structure $(\nabla, [g])$, for $g \in [g]$ with associated one-form $\chi_{i}$, scalar curvature $\sc$, and tensor $\T_{ij}$, there hold
\begin{align}\label{conservenaive}
\tfrac{n-2}{n}\left(d\sc_{i} + \sc \chi_{i} + \tfrac{n}{2}g^{pq}\nabla_{p}d\chi_{qi}\right) = \bt^{abc}R_{i(abc)} = \tfrac{1}{2}\div(\T)_{i} + \tfrac{2-n}{4}\bt_{ip}\,^{q}\bt_{aq}\,^{p}\chi^{a}.
\end{align}
Consequently, for a naive Einstein AH structure on a manifold of dimension $n \geq 3$ the following are equivalent:
\begin{enumerate}
\item It is Einstein.
\item The invariantly defined one-form $\bt^{abc}R_{iabc}$ vanishes identically. 
\item There holds $2\div(\T)_{i} = (n-2)\bt_{ip}\,^{q}\bt_{aq}\,^{p}\chi^{a}$.
\end{enumerate}
\end{lemma}
\begin{proof}
By the naive Einstein condition there vanish $S_{ij}$ and $\wideparen{S}_{ij}$. By \eqref{Driccurv} this implies $2\div(\bt)_{ij} = n\bt_{ij}\,^{p}\chi_{p}$ and with \eqref{conserve} there results \eqref{conservenaive}. 
The claimed equivalences follows from \eqref{conservenaive}.
\end{proof}

The condition \eqref{conservationcondition} of Definition \ref{einsteinahdefinition} and the terminology \emph{conservation condition} need motivation. This is given below, following Corollary \ref{selfconjugatecorollary}. For an exact AH structure the condition \eqref{conservationcondition} is simply the constancy of the scalar curvature, and the issue is that this needs to be imposed, as it does not follow from the naive Einstein condition (Example \ref{naiveexample} gives an example showing that naive Einstein does not imply Einstein). 

\begin{lemma}\label{conjugateeinsteiahlemma}
An AH structure is Einstein if and only if the conjugate AH structure is Einstein.
\end{lemma}

\begin{proof}
That conjugacy preserves the naive Einstein condition is immediate. 
Let $(\anabla, [g])$ be the conjugate AH structure of the naive Einstein AH structure $(\nabla, [g])$. The scalar curvature $\sc$ and one-form $\chi_{i}$ associated with $\nabla$ and $g \in [g]$ are the same as those associated with $\anabla$ and $g$, and $g^{pq}\anabla_{p}d\chi_{qi} = g^{pq}(\nabla_{p}d\chi_{qi} - \bt_{pq}\,^{a}d\chi_{ai} - \bt_{pi}\,^{a}d\chi_{qa}) = g^{pq}\nabla_{p}d\chi_{qi} - \bt_{i}\,^{pq}d\chi_{pq} = g^{pq}\nabla_{p}d\chi_{qi}$, which suffices to show that \eqref{conservationcondition} holds for $(\nabla, [g])$ if and only if it holds for $(\anabla, [g])$. Alternatively, the right-hand side of \eqref{conservenaive} is self-conjugate, so the left-hand side is as well.
\end{proof}

Corollary \ref{selfconjugatecorollary} shows that, with the additional condition of self-conjugacy of the curvature, the naive Einstein condition implies the conservation condition when $n > 2$. 

\begin{corollary}\label{selfconjugatecorollary}
On a manifold of dimension $n > 2$, a naive Einstein AH structure with self-conjugate curvature satisfies \eqref{conservationcondition}, so is Einstein.
\end{corollary}

\begin{proof}
If the curvature of $(\nabla, [g])$ is self-conjugate, then, by \eqref{ahcurv}, $R_{ij(kl)} = \tfrac{1}{2}d\chi_{ij}g_{kl}$, so $\bt^{abc}R_{iabc} = 0$ and $(\nabla, [g])$ is conservative. With \eqref{conserve} this implies the claim.
\end{proof}

The conservation condition \eqref{conservationcondition} generalizes the constancy of the scalar curvature of a metric. It follows from the traced Bianchi identities that if the dimension $n$ is greater than $2$, then the usual Einstein condition for a metric implies the scalar curvature is constant. In dimension $2$ this fails and the constancy of the scalar curvature is the best substitute for the Einstein equations. In \cite{Calderbank-mobius}, Calderbank made this observation the basis of the definition of Einstein equations for Weyl structures in $2$-dimensions. With the notations used here, he defined a $2$-dimensional Weyl structure to be Einstein if it satisfies \eqref{conservationcondition}. This definition is justified by the observation that for a Weyl structure on a manifold of dimension $n > 2$, the naive Einstein equations imply \eqref{conservationcondition}, by essentially the same argument as for ordinary metrics. Since Einstein equations for AH structures should specialize to the usual Einstein equations for Weyl structures when the cubic torsion vanishes, this suggests that the definition of Einstein equations for AH structures should force or require \eqref{conservationcondition}.

Lemma \ref{conservativenaivelemma} suggests and Example \ref{naiveexample} confirms that for AH structures the naive Einstein equations do not imply \eqref{conservationcondition} when $n > 2$. 

For a metric $g$ with Ricci and scalar curvature $\ric(g)_{ij}$ and $\sR(g)$, the \emph{Einstein tensor} $\G_{ij} = \ric(g)_{ij} - \tfrac{1}{2}\sR(g)g_{ij}$ is divergence free by the traced differential Bianchi identity. Calling \eqref{conservationcondition} a \emph{conservation} condition is motivated by thinking of the constancy of the scalar curvature of an Einstein metric (in the usual sense) as the vanishing of the divergence of the Einstein tensor, which can be viewed as a conservation law. Corollary \ref{conservativestatisticalcorollary} shows that for a naive Einstein special statistical structure $(\nabla, g)$ to be Einstein is equivalent to the vanishing of the divergence of the tensor $\T_{ij}$ defined in \eqref{stdefined} and this is in turn equivalent to $g$ solving the Einstein field equations with stress-energy tensor $\T_{ij}$. These observations seem the most convincing arguments for the requiring the conservation condition as part of the definition of Einstein.

Now these statements are made precise for special statistical structures.

\begin{definition}\label{einsteinspecialstatisticaldefinition}
A special statistical structure $(\nabla, g)$ is \emph{naive Einstein} or \emph{Einstein} if the AH structure that it generates has the same property.  In this case it is said to have scalar curvature equal to the scalar curvature $\sc$ corresponding with $g$.

A special statistical structure or an exact AH structure for which there vanish both the Ricci curvature and the conjugate Ricci curvature satisfies \eqref{conservationcondition} so is Einstein. Such a structure is said to be \emph{scalar-flat Einstein}.
\end{definition}

The \emph{stress-energy} tensor of the special statistical structure $(\nabla, g)$ is the tensor $\T_{ij}$ defined by \eqref{stdefined}. This terminology is justified by Corollary \ref{conservativestatisticalcorollary}. Some preliminary calculations are needed for its proof.

\begin{lemma}
On an $n$-manifold, let $(\nabla, [g])$ be an AH structure and consider a representative metric $g \in [g]$ with Levi-Civita connection $D$, scalar curvature $\sc$, associated one-form $\chi_{i}$, and associated tensor $\sT_{ij}$ as in \eqref{stdefined}. 
The Einstein tensor, $\G_{ij} = \ric(g)_{ij} - \tfrac{1}{2}\sc(g)g_{ij}$, of $g$ satisfies:
\begin{align}\label{aheinsteintensor}
\begin{aligned}
\G_{ij} & = R_{(ij)} - \tfrac{\sc}{2}g_{ij} + \tfrac{1}{4}\T_{ij} + \tfrac{1}{2}(\div(\bt)_{ij} - \bt_{ij}\,^{p}\chi_{p})\\
&\quad + \tfrac{2-n}{2}\left(D_{(i}\chi_{j)} -\div(\chi)g_{ij} + \tfrac{1}{2}\chi_{i}\chi_{j} + \tfrac{n-3}{4}|\chi|^{2}g_{ij}\right),
\end{aligned}
\end{align}  
where $\div(\chi) = g^{pq}D_{p}\chi_{q}$.
\end{lemma}
\begin{proof}
Straightforward computations using the expression \eqref{dnabladifference} for the difference tensor $D - \nabla$ show the Ricci and scalar curvatures, $\ric(g)$ and $\sc(g)$, satisfy
\begin{align}
\begin{aligned}
\ric(g)_{ij} & = R_{(ij)} + \tfrac{1}{4}\bt_{ip}\,^{q}\bt_{jq}\,^{p}+ \tfrac{1}{2}(\div(\bt)_{ij} - \bt_{ij}\,^{p}\chi_{p})\\
& \quad + \tfrac{2-n}{2}\left(D_{(i}\chi_{j)} + \tfrac{1}{2}\chi_{i}\chi_{j} -\tfrac{1}{2}|\chi|^{2}g_{ij}\right) -\tfrac{1}{2}\div(\chi)g_{ij},\\
\sc(g) & = \sc + \tfrac{1}{4}|\bt|^{2} + (1-n)\div(\chi) + \tfrac{(n-1)(n-2)}{4}|\chi|^{2}.
\end{aligned}
\end{align}  
Combining these yields \eqref{aheinsteintensor}.
\end{proof}

\begin{corollary}\label{conservativestatisticalcorollary}
On an $n$-manifold, for a special statistical structure $(\nabla, g)$ with stress-energy tensor $\T_{ij}$ and scalar curvature $\sc$ there holds
\begin{align}\label{conservenaive2}
2(n-2)d\sc_{i} = 2n\bt^{abc}R_{i(abc)} = n\div(\T)_{i}.
\end{align}
For a naive Einstein special statistical structure on a manifold of dimension $n \geq 3$ the following are equivalent:
\begin{enumerate}
\item It is Einstein.
\item Its stress-energy tensor is divergence free.
\item The metric $g$ satisfies the Einstein field equations $\G_{ij} = \tfrac{1}{4}\sT_{ij}$.
\end{enumerate}
\end{corollary}
\begin{proof}
For any special statistical structure $(\nabla, g)$, specializing \eqref{aheinsteintensor} shows
\begin{align}\label{prefield}
\begin{aligned}
R_{(ij)} - \tfrac{\sc}{2} g_{ij} & = \G_{ij} - \tfrac{1}{4}\T_{ij} - \tfrac{1}{2}\div(\bt)_{ij}.
\end{aligned}
\end{align}
By \eqref{Driccurv} the naive Einstein condition implies $\bt_{ijk}$ is divergence free. The claimed equivalences follow from Lemma \ref{conservativenaivelemma} together with \eqref{prefield}.
\end{proof}

\begin{remark}
Definition \ref{einsteinspecialstatisticaldefinition} makes sense for statistical structures that are not special, but it is not clear whether this is the most natural notion of Einstein equations for general statistical structures. It is possible that it should be augmented by some additional condition, for example that the vector field $\chi^{i}$ be Killing.
\end{remark}

\begin{example}\label{vanishcottonexample}
Let $\nabla$ be a torsion-free affine connection with nondegenerate symmetric Ricci tensor $R_{ij} = R_{(ij)}$ and define $g_{ij} = P_{(ij)} = \tfrac{1}{1-n}R_{(ij)}$. If $C_{ijk} = 0$, $\nabla \det P = 0$, and $g^{jk}R_{ijk}\,^{l} = 0$, then $(\nabla, g)$ is an Einstein special statistical structure. The following observation motivates the general construction: the AH structures induced on a cooriented nondegenerate hypersurface in flat affine space as in Example \ref{affinehypersurfaceexample} are Einstein if and only if the affine hypersurface is an affine sphere. See \cite{Fox-ahs, Fox-2dahs, Fox-crm} for proofs. 
\end{example}

\begin{example}
Coupled with the main results of \cite{Fox-simplicial} and \cite{Fox-cubicpoly}, \cite[Lemma $1.18$]{Fox-cubicpoly} yields Einstein AH structures that are not locally equivalent to those induced on an affine sphere. If $D$ is the Levi-Civita connection of a Euclidean metric $g_{ij}$ and $(\nabla, [g])$ is an exact AH structure with $D$-parallel cubic torsion, then the Einstein AH equations become purely algebraic conditions on the cubic torsion, which can be interpreted as the structure tensor of a commutative, not necessarily associative, nonunital algebra for which the Killing type trace-from is invariant. The purely algebraic problem of constructing solutions that are not locally equivalent to affine spheres is tractable and is resolved affirmatively in all dimensions greater than $3$ in \cite{Fox-simplicial, Fox-cubicpoly}.
\end{example}

\begin{example}\label{naiveexample}
Here is an example of a special statistical structure that is naive Einstein but not Einstein.

Consider $\rea^{3}$ equipped with a Euclidean metric $g_{ij}$ and its Levi-Civita connection $D$. Let $x_{1}, x_{2}, x_{3}$ be coordinates such that $dx_{i}$, $dx_{2}$, and $dx_{3}$ constitute an orthonormal coframe, and let $\{\pr_{1}, \pr_{2}, \pr_{3}\}$ be the dual frame. Write $\pr_{ijk} = \pr_{i}\tensor \pr_{j}\tensor \pr_{k}$. Define a trace-free symmetric cubic tensor $L$ by
\begin{align}
\begin{aligned}
L & = 2(x_{1} + x_{3})\left(\pr_{112} + \pr_{121} + \pr_{211} - \pr_{113} - \pr_{131} - \pr_{311}\right)\\
& + 2(x_{1} - x_{3})\left(\pr_{123} + \pr_{231} + \pr_{312} + \pr_{213} + \pr_{321} + \pr_{132}\right).
\end{aligned}
\end{align}
Let $L_{ij}\,^{k} = g^{kp}L_{ijp}$.
As $L$ is trace-free, $\nabla = D - \tfrac{1}{2}L_{ij}\,^{k}$ determines with $g$ a special statistical structure. Straightforward calculations show that $L$ is divergence-free and satisfies
\begin{align}
&\T =-8(x_{1}^{2} + x_{3}^{2})g,& & \div(\T) = -16(x_{1}\pr_{1} + x_{3}\pr_{3}).
\end{align}
Because $L$ is divergence-free $(\nabla, g)$ has self-conjugate Ricci curvature by \eqref{Driccurv}. Because $g$ is flat, by \eqref{prefield}, $R_{(ij)} = -\tfrac{1}{4}\T_{ij} = 2(x_{1}^{2} + x_{3}^{2})g_{ij}$, so $(\nabla, g)$ is naive Einstein. On the other hand, by Corollary \ref{conservativestatisticalcorollary}, because $\div(\T)$ is not identically zero, $(\nabla, g)$ is not Einstein.
\end{example}

\section{Einstein statistical and AH structures on principal bundles with one-dimensional fibers}\label{ewsection}
This section treats a metric $G_{IJ}$ and a cone connection $\hnabla$ on the total space of a principal bundle $\rho:N \to M$ with one-dimensional structure group making use of the notions related to statistical structures introduced in Section \ref{einsteinstatisticalsection}. Uppercase Latin letters are used for abstract indices on $N$, and indices are raised and lowered with $G_{IJ}$. For example $\hR_{IJKL} = \hR_{IJK}\,^{P}G_{PL}$ is the curvature of $\hnabla$ with the last index lowered by $G_{IJ}$.

The general setting considered is the following. Let $M$ be an $n$-manifold, let $\rho:N \to M$ be a principal $S^{1}$-bundle or principal $\reat$-bundle, and let $\rad$ be the fundamental vertical vector field generated by the principal action. Let $\nabla$ be a torsion-free affine connection on $M$ and let $\be$ be a principal connection on $\rho:N \to M$. Let $\hnabla$ be the cone connection on $\rho:N \to M$ of $(\nabla, \be)$. 
Let $g_{ij} = P_{(ij)}$ be the symmetrized projective Schouten tensor of $\nabla$, let $\rho^{\ast}(\om)_{IJ} = d\be_{IJ}$ be the curvature of $\be$, and let $\epc_{ij} = 2P_{[ij]} + \om_{ij}$. For $t \in \rea$, define a symmetric tensor $G_{IJ}$ on $N$ by 
\begin{align}\label{liftedmetric1}
G_{IJ} = \hnabla_{(I}\be_{J)} + (2+t)\be_{I}\be_{J} = \hnabla_{I}\be_{J} - \tfrac{1}{2}d\be_{IJ} + (2+t)\be_{I}\be_{J} = (1+t)\be_{I}\be_{J} - \rho^{\ast}(P)_{(IJ)}.
\end{align}
By Lemma \ref{liftedmetricpreliminarieslemma}, $G_{IJ}$ is a metric if $t \neq -1$. The form of $G_{IJ}$ is a generalization of the construction used to produce the Berger metrics on a three sphere viewed as the total space of the Hopf fibration over the two sphere, and is commonly used in the construction of Einstein-Weyl metrics on $S^{1}$-bundles \cite{Calderbank-Pedersen, Pedersen-Swann-submersions} where it is sometimes called the \emph{canonical variation}, and $1+t$ is usually written $r^{2}$. 

Lemma \ref{liftedmetricpreliminarieslemma} records basic computations relating the cone connection $\hnabla$ and the metric $G_{IJ}$. Corollary \ref{flatEinsteinliftcorollary} shows that when $N = M \times G$ is trivial and $\be$ is flat, $\hnabla$ generates with $[G]$ a closed AH structure of which it is the aligned representative, and this AH structure is scalar-flat Einstein. In the case that $\be$ is exact, the conformal metric $H = e^{2f}G$, where $f$ is a primitive such that $df = \be$, constitutes with $\hnabla$ a scalar-flat Einstein special statistical structure. 

This motivates Theorem \ref{conjugatethomastheorem} which shows that given a metric $H$ on the frame bundle $\F$ of a pseudo-tautological line bundle the $H$-conjugate connection of the Thomas connection $\nabla$ of a projective structure $\en$ on $M$ is again the Thomas connection of a projective structure $\ben$ on $M$ if and only if $(\hnabla, \rad, H)$ is a scalar-flat Einstein special radiant statistical structure. Moreover, in this case, $v = H(\rad, \rad)$ is a potential for $2H$, meaning $\hnabla dv = 2H$, and solves a Monge-Ampère equation $\det \hnabla dv = c|\Psi^{\F}|^{2}$ where $|\Psi^{\F}|$ is the canonical volume density on $\F$. With the minor additional technical hypothesis that the closed one-form $\be = H(\rad, \rad)^{-1}\rad^{\flat}$ is a principal connection, the unique representatives $\nabla \in \en$ and $\bnabla \in \ben$ inducing $\be$ generate with their projective Schouten tensors conjugate special statistical structures that are Einstein with negative scalar curvature. As in explained in Example \ref{chengyauexample}, when $\nabla$ is a properly convex flat projective structure this recovers the picture, generalizing Example \ref{projflatexample}, relating such structures with hyperbolic affine spheres. It is more general in several senses. For one, it makes sense for non-Riemannian signatures. For another, it allows $\en$ to be not projectively flat, or, equivalently, $\hnabla$ to be not flat. 

Lemma \ref{liftedmetriclemma} shows that in the general setting where $\rho:N \to M$ is a principal $S^{1}$-bundle or principal $\reat$-bundle, the connection $\hnabla$ can be modified in a way that depends on $\be$ and $G$ to obtain a connection $\hDs$ that has some possibility, in the presence of further geometric conditions, of constituting with $[G]$ an Einstein AH structure. Theorem \ref{liftedahtheorem} shows that when $\hnabla$ is the cone connection of an Einstein special statistical structure that admits an appropriately compatible almost complex structure, then $\hDs$ is aligned with respect to $G$ and $(\hDs, [G])$ is an Einstein AH structure. This generalizes Corollary \ref{flatEinsteinliftcorollary} to the case of nontrivial bundles. For a properly convex flat real projective structure on an oriented compact surface of genus at least $2$, the compatible complex structure is that determined by the Cheng-Yau metric and the given orientation, and the desired $\hDs$ exists provided the Euler number $e(N)$ of $N$ satisfies $|e(N)| \leq -\chi(M)$. This is explained in Example \ref{convexprojectiveexample} and summarized in Theorem \ref{convexprojectivetheorem}. As is explained in Example \ref{ewexample}, when the geometric data on $M$ comprises a Kähler-Einstein metric of nonzero scalar curvature, the resulting $(\hDs, [G])$ are Einstein-Weyl structures whose construction is well known.

There is a long tradition of constructing Einstein metrics on the total spaces of principal bundles in which the Kaluza-Klein construction is paradigmatic. See \cite{Bourguignon-kaluzaklein} for a mathematically oriented survey. The constructions described in this section are similar in flavor to these constructions, although hewing to the model of Example \ref{projflatexample} and based on the cone connection of an extended projective structure, and finally yielding Einstein statistical structures and Einstein AH structures rather than Einstein metrics as such. The presentation of Example \ref{projflatexample} was made to conform with the statement of Lemma \ref{liftedmetriclemma} and is guided by seeking analogues of the Fefferman metric \cite{Fefferman-monge}, \cite[Section $5$]{Graham-Lee}, and \cite{Lee-feffermanmetric} in the context of projective structures.

\begin{lemma}\label{liftedmetricpreliminarieslemma}
Let $M$ be an $n$-manifold, let $\rho:N \to M$ be a principal $S^{1}$-bundle or principal $\reat$-bundle, and let $\rad$ be the fundamental vertical vector field generated by the principal action. Let $\nabla$ be a torsion-free affine connection on $M$ and let $\be$ be a principal connection on $\rho:N \to M$. Let $\hnabla$ be the cone connection on $\rho:N \to M$ of $(\nabla, \be)$. 
Let $g_{ij} = P_{(ij)}$ be the symmetrized projective Schouten tensor of $\nabla$, let $\rho^{\ast}(\om)_{IJ} = d\be_{IJ}$ be the curvature of $\be$, and let $\epc_{ij} = 2P_{[ij]} + \om_{ij}$.
\begin{enumerate}
\item For $t \in \rea$, the tensor $G_{IJ}$ on $N$ defined by \eqref{liftedmetric1} is invariant under the principal action (so, in particular, $\lie_{\rad}G = 0$) and $G_{IJ}$ is nondegenerate if and only if $g_{ij} = P_{(ij)}$ is nondegenerate and $t \neq -1$. 
\item With $\tpc_{i} = g^{pq}\nabla_{i}g_{pq}$ and $\spc_{i} = g^{pq}\nabla_{p}g_{qi}$, there hold
\begin{align}
\label{hnablaGIJK1}
\hnabla_{I}G_{JK} &= - \rho^{\ast}(\nabla g)_{IJK} - 2\be_{I}G_{JK} + 2t G_{I(J}\be_{K)}+ (1 + t)d\be_{I(J}\be_{K)} - 2t(t+1)\be_{I}\be_{J}\be_{K},\\
\label{hnablagskew}
\hnabla_{[I}G_{J]K}& - \tfrac{1+t}{2}\hnabla_{[I}d\be_{J]K}  = -(2+t)\be_{[I}G_{J]K} - \tfrac{1}{2}\rho^{\ast}(C)_{IJK}  - \tfrac{1}{4}\rho^{\ast}(\nabla \epc)_{KIJ} + \tfrac{2+t}{4}\rho^{\ast}(\nabla \om)_{KIJ},\\\label{hnabladetG}
G^{JK}\hnabla_{I}G_{JK} &= -2(n+1)\be_{I}  -G^{JK}\rho^{\ast}(\nabla G)_{IJK} = -2(n+1)\be_{I}  + \rho^{\ast}(\tpc)_{I},\\
\label{hnabladetG2}
G^{JK}\hnabla_{J}G_{KI} &=(nt-2)\be_{I} - G^{JK}\rho^{\ast}(\nabla G)_{JKI} =(nt-2)\be_{I} + \rho^{\ast}(\spc)_{I},\\
\label{hnablaGaligned}
G^{JK}(\hnabla_{I}G_{JK} & - (n+1)\hnabla_{J}G_{KI}) = -n(n+1)t\be_{I} + \rho^{\ast}(\tpc - (n+1)\spc)_{I}.
\end{align}
\item The curvature $\hat{R}_{IJK}\,^{L}$ of $\hnabla$ satisfies
\begin{align}
\label{hnablaGcurv}
\hat{R}_{IJK}\,^{Q}G_{QL} & = (1+t)(\rho^{\ast}(C)_{IJK} + \tfrac{1}{2}\rho^{\ast}(\nabla \epc)_{KIJ}) - \rho^{\ast}(S)_{IJKL},\\
\label{hnablaGconjugatericci}
G^{AB}\hat{R}_{IAB}\,^{Q}G_{QJ} & = (1+t)\rho^{\ast}(\spc - \tpc + \tfrac{1}{2}\wpc )_{I}\be_{J} + \rho^{\ast}(T)_{IJ},
\end{align}
where $\wpc_{i} = g^{pq}\nabla_{p}\om_{qi}$ and
\begin{align}
S_{ijkl} & = B_{ijkl} - g_{l[i}\epc_{j]k} + g_{kl}\epc_{ij} = R_{ijkl} + 2g_{l[i}g_{j]k} - g_{l[i}\om_{j]k} - g_{kl}\om_{ij},\\
T_{ij} &= g^{pq}R_{ipqj} + (n-1)g_{ij} + \tfrac{3}{2}\om_{ij},
\end{align}
with $B_{ijkl} = B_{ijk}\,^{p}g_{pl}$ and $R_{ijkl} = R_{ijk}\,^{p}g_{pl}$.
\end{enumerate}
\end{lemma}

\begin{proof}
The last equality of \eqref{liftedmetric1} follows from \eqref{hnablabe} and implies the invariance of $G$ under the principal action. That $G_{IJ}$ is nondegenerate if and only if $g_{ij}$ is nondegenerate and $t \neq -1$ is apparent from \eqref{liftedmetric1}. 

For any $S_{ij} \in \Ga(\tensor^{2}\ctm)$, a straightforward calculation using \eqref{tripleconnection} shows 
\begin{align}\label{hnablarhos}
\hnabla_{I}\rho^{\ast}(S)_{JK} & = \rho^{\ast}(\nabla S)_{IJK} - \rho^{\ast}(S)_{JI}\be_{K} - \rho^{\ast}(S)_{IK}\be_{J} - 2\rho^{\ast}(S)_{JK}\be_{I}.
\end{align}
Taking $S_{ij} = \om_{ij}$ in \eqref{hnablarhos} yields 
\begin{align}
\label{hnablarhoom}
\hnabla_{I}d\be_{JK} & = \rho^{\ast}(\nabla \om)_{IJK}  - 2\be_{I}d\be_{JK} + 2\be_{[J}d\be_{K]I}.
\end{align}
Taking $S_{ij} = g_{ij}$ in \eqref{hnablarhos} and substituting \eqref{liftedmetric1} in the result yields
\begin{align}\label{hnablarhog}
\begin{aligned}
\hnabla_{I}\rho^{\ast}(g)_{JK} & = \rho^{\ast}(\nabla g)_{IJK} + 2\be_{(J}G_{K)I} + 2\be_{I}G_{JK} - 4(1+t)\be_{I}\be_{J}\be_{K}.
\end{aligned}
\end{align}
Combining \eqref{hnablarhog} with
\begin{align}\label{hnablabebe}
\hnabla_{I}(\be_{J}\be_{K}) = 2\be_{(J}\hnabla_{|I|}\be_{K)} = 2\be_{(J}G_{K)I} + d\be_{I(J}\be_{K)} - 2(2+t)\be_{I}\be_{J}\be_{K},
\end{align}
yields \eqref{hnablaGIJK1}.
Antisymmetrizing \eqref{hnablaGIJK1} in $IJ$, noting that
\begin{align}\label{skewgijk}
2\nabla_{[i}g_{j]k} = C_{ijk} + \nabla_{k}P_{[ij]} =  C_{ijk} + \tfrac{1}{2}\nabla_{k}\epc_{ij} - \tfrac{1}{2}\nabla_{k}\om_{ij}, 
\end{align}
and using \eqref{hnablarhoom} yields
\begin{align}\label{lm2}
\begin{aligned}
\hnabla_{[I}G_{J]K}& =- (2+t)\be_{[I}G_{J]K} + \tfrac{1+t}{2}\left(\be_{K}d\be_{IJ}  -\be_{[I}d\be_{J]K}\right) \\
& \quad - \tfrac{1}{2}\rho^{\ast}(C)_{IJK}  - \tfrac{1}{4}\rho^{\ast}(\nabla \epc)_{KIJ} + \tfrac{1}{4}\rho^{\ast}(\nabla \om)_{KIJ}\\
& = - (2+t)\be_{[I}G_{J]K} - \tfrac{1}{2}\rho^{\ast}(C)_{IJK}  - \tfrac{1}{4}\rho^{\ast}(\nabla \epc)_{KIJ} + \tfrac{2+t}{4}\rho^{\ast}(\nabla \om)_{KIJ}- \tfrac{1+t}{4}\hnabla_{K}d\be_{IJ},
\end{aligned}
\end{align}
which can be rewritten as \eqref{hnablagskew}.
Let $\{E(1), \dots, E(n)\}$ be a $g$-unimodular orthogonal local frame on $M$. Then $\{|1 + t|^{-1/2}\rad, \widehat{E(1)}, \dots, \widehat{E(n)}\}$ is a $G$-unimodular orthogonal local frame on $N$. The tensor corresponding with the inverse symmetric bivector $G^{IJ}$ is 
\begin{align}\label{preliftedcontract}
G^{IJ} = \sum_{i  = 1}^{n}G(\widehat{E(i)}, \widehat{E(i)})\widehat{E(i)}^{I} \widehat{E(i)}^{J} + \tfrac{1}{1+t}\rad^{I} \rad^{J} = -\sum_{i  = 1}^{n}g(E(i), E(i))\widehat{E(i)}^{I} \widehat{E(i)}^{J} +\tfrac{1}{1+t} \rad^{I} \rad^{J}
\end{align}
It follows that, for $k \geq 2$ and any $S_{i_{1}\dots i_{k}} \in \Ga(\tensor^{k}\ctm)$ there holds 
\begin{align}\label{liftedcontract}
G^{AB}\rho^{\ast}(S)_{I_{1}\dots A \dots B \dots I_{k-2}} = -\rho^{\ast}(\si)_{I_{1}\dots I_{k-2}},
\end{align}
where $\si_{i_{1}\dots i_{k-2}} = g^{ab}S_{i_{1}\dots a \dots b \dots i_{k-2}}$. Using this remark and contracting \eqref{hnablarhog}, \eqref{hnablabebe}, and \eqref{hnablaGIJK1} with $G^{JK}$ yields \eqref{hnabladetG}, \eqref{hnabladetG2}, and \eqref{hnablaGaligned}.
The identities \eqref{hnablaGcurv} and \eqref{hnablaGconjugatericci} follow straightforwardly from \eqref{hnablabecurvatures} and \eqref{liftedmetric1} together with \eqref{liftedcontract}.
\end{proof}

\begin{remark}
If the inertial indices of $g_{ij} = P_{(ij)}$ are $(i_{+}, i_{-}, i_{0})$, then those of $G_{IJ}$ are $(1 + i_{-}, i_{+}, i_{0})$ if $t > -1$ and $(i_{-}, 1 + i_{+}, i_{0})$ if $t < -1$. When applying Lemma \ref{liftedmetricpreliminarieslemma} it should be kept in mind that $g_{ij}$ can be negative definite.
\end{remark}

\begin{corollary}\label{flatEinsteinliftcorollary}
Let $M$ be an $n$-manifold, let $\rho:N = M \times G \to M$ be a trivial principal $G$-bundle where $G$ is $S^{1}$ or $\reat$, and let $\rad$ be the fundamental vertical vector field generated by the principal action. Let $\nabla$ be a torsion-free affine connection on $M$ and let $\be$ be a flat principal connection on $N$. Let $\hnabla$ be the cone connection of the extended projective structure $[\nabla, \be]$ on $\rho:N \to M$ generated by $(\nabla, \beta)$. 

Suppose the symmetrized projective Schouten tensor $g_{ij} = P_{(ij)}$ of $\nabla$ is nondegenerate and there hold $\nabla \det g = 0$, $\nabla_{[i}g_{j]k} = 0$, and $g^{pq}R_{(i|pq|j)} = (1-n)g_{ij}$, so that $(\nabla, g)$ is an Einstein special statistical structure with negative scalar curvature. Let $G_{IJ}$ be the metric on $N$ defined by \eqref{liftedmetric1} with $t = 0$.
\begin{enumerate}
\item $\hnabla$ is aligned with respect to the metric $G_{IJ}$ and $(\hnabla, [G])$ is a scalar-flat Einstein closed AH structure on $N$. 
\item If $\be$ is exact and $f \in \cinf(N)$ is a primitive of $\be$, then $H_{IJ} = e^{2f}G_{IJ}$ constitutes with $\hnabla$ a scalar-flat Einstein special statistical structure.
\end{enumerate}
\end{corollary}

\begin{proof}
The assumptions on $\nabla$ imply the vanishing of $\tpc$, $\spc$, $\wpc$, $\om_{ij}$, and $P_{[ij]}$, and, by \eqref{skewgijk}, of $C_{ijk} + \tfrac{1}{2}\nabla_{k}\epc_{ij}$. Because $t = 0$, by \eqref{hnablagskew}, $\hnabla_{[I}G_{J]k} = -2\be_{[I}G_{J]K}$, while by \eqref{hnablaGaligned}, $G^{JK}(\hnabla_{I}G_{JK} -(n+1)\hnabla_{J}G_{KI}) = 0$, so $\hnabla$ is aligned with respect to $G$ and $(\hnabla, [G])$ is an AH structure that is closed because $\be$ is flat. From \eqref{hnablaGconjugatericci} and the hypothesis that $(\nabla, g)$ is scalar-flat Einstein AH it follows that $G^{AB}\hat{R}_{IAB}\,^{Q}G_{QJ}$ vanishes, so that $(\hnabla, [G])$ is naive Einstein. A scalar-flat closed naive Einstein AH structure automatically satisfies the condition \eqref{conservationcondition}, so $(\hnabla, [G])$ is Einstein. If $\be = df$ and $H_{IJ} = e^{2f}G_{IJ}$, then $\hnabla_{[I}H_{J]K} = 0$ and, by \eqref{hnabladetG}, $H^{JK}\hnabla_{I}H_{JK} = 2(n+1)df_{I} - 2(n+1)\be_{I} = 0$, so, by the preceding, $(\hnabla, H)$ is a special statistical structure that is scalar-flat Einstein.
\end{proof}

\begin{remark}
The assumptions of Corollary \ref{flatEinsteinliftcorollary} are satisfied by the connection induced on a hyperbolic affine sphere via its equiaffine normal, and so also by the statistical structure on a properly convex flat real projective manifold generated by the Cheng-Yau metric and the aligned representative of the flat projective structure.
\end{remark}

Theorem \ref{conjugatethomastheorem} refines the special case of Corollary \ref{flatEinsteinliftcorollary} where $N$ is the total space $\F$ of the frame bundle of a pseudo-tautological line bundle $\emf^{-1}$ over an $n$-manifold, $M$, and $\be$ is exact.
For a statistical structure $(\hnabla, H)$ on $\F$ for which $\hnabla$ is a Thomas connection, Theorem \ref{conjugatethomastheorem} shows that the connection conjugate to the $\hnabla$ is itself a Thomas connection if and only if $(\hnabla, H)$ is a scalar-flat Einstein special statistical structure. 

A metric $H_{IJ}$ on the total space $\F$ has \emph{homogeneity $2$} if $R^{\F}_{r}(H) = r^{2}H$ for all $r \in \reat$. If $H$ is assumed self-similar, this amounts to supposing additionally that $R_{-1}^{\ast}(H) = H$.

\begin{theorem}\label{conjugatethomastheorem}
On an $n$-manifold $M$, let $\emf \to M$ be a pseudo-hyperplane line bundle, let $\rho: \F = \frameb(\emf^{-1}) \to M$ be the frame bundle of $\emf^{-1}$, and let $\rad = \eul^{\F}$.  Let $\hnabla$ be the Thomas connection of a projective structure $\en$ on $M$. For a pseudo-Riemannian metric $H_{IJ}$ on $\F$, the following are equivalent:
\begin{enumerate}
\item\label{cott1} The $H$-conjugate connection, $\chnabla$, of $\hnabla$ is the Thomas connection of a projective structure $\ben$ on $M$.
\item\label{cott2} $(\hnabla, \rad, H)$ is a special radiant statistical structure such that $(\hnabla, H)$ is scalar-flat Einstein as a special statistical structure.
\end{enumerate}
In the case there hold \eqref{cott1} and \eqref{cott2}:
\begin{enumerate}
\setcounter{enumi}{2}
\item The special radiant statistical structure $(\chnabla, \rad, H)$ is also scalar-flat Einstein as a special statistical structure.
\item\label{cttclosedoneform} For $v = H(\rad, \rad)$, the closed one-form $\be_{I} = \tfrac{1}{2}v^{-1}dv_{I}$ satisfies 
\begin{align}\label{cttclosedoneformb} 
\tfrac{1}{2}v^{-1}\hnabla_{I}dv_{J} = \hnabla_{I}\be_{J} + 2\be_{I}\be_{J} = v^{-1}H_{IJ} = \chnabla_{I}\be_{J} + 2\be_{I}\be_{J} = \tfrac{1}{2}v^{-1}\chnabla_{I}dv_{J}
\end{align}
\item There is $c \in \reat$ such that $2^{-n-1}\det \hnabla dv = \det H = c |\Psi^{\F}|^{2}$, where $|\Psi^{\F}|$ is the volume density defined in Lemma \ref{canonicaleulerlemma}.
\item\label{differencelift} The $\en$-invariant lift of the trace-free tensor $\ben - \en = \Pi_{ij}\,^{k}\in \Ga(S^{2}\ctm \tensor TM)$ is $\linft(\Pi)_{IJ}\,^{K} = \hat{\Pi}_{IJ}\,^{K} - \tfrac{1}{n-1}\hat{\Pi}_{IA}\,^{B}\hat{\Pi}_{JB}\,^{A}\rad^{K}$ where $\hat{\Pi}_{IJ}\,^{K} = \chnabla - \hnabla = H^{KA}\hnabla_{I}H_{JA}$ .
\end{enumerate}
If there hold \eqref{cott1} and \eqref{cott2} and additionally $H_{IJ}$ has homogeneity $2$, then:
\begin{enumerate}
\setcounter{enumi}{6}
\item \label{einsteinclaim1} The closed one-form $\be_{I} = H(\rad, \rad)^{-1}\rad^{P}H_{PI}$ is a principal $\reat$-connection on $\F$.
\item \label{einsteinclaim2} The unique representatives $\nabla \in \en$ and $\bnabla \in \ben$ inducing $\be$ have vanishing projective Cotton tensors and nondegenerate symmetric projective Schouten tensors $P_{ij}$ and $\bar{P}_{ij}$ satisfying
\begin{align}\label{einsteinclaim2b}
&-\hnabla_{I}\be_{J} - \be_{I}\be_{J} = \rho^{\ast}(P)_{IJ},&&-\chnabla_{I}\be_{J} - \be_{I}\be_{J} = \rho^{\ast}(\bar{P})_{IJ}
\end{align}
and $\nabla_{i}\det P = 0$ and $\bnabla_{i}\det \bar{P} = 0$. In particular each of $(\nabla, P)$ and $(\bnabla, \bar{P})$ is a special statistical structure.
\item\label{einsteinclaim3} The special statistical structures $(\nabla, P)$ and $(\bnabla, \bar{P})$ are conjugate and Einstein with negative scalar curvature.
\end{enumerate}
\end{theorem}

\begin{proof}
The condition \eqref{ctt2} can be restated as that $(\hnabla, \rad, H)$ is a special radiant statistical structure with curvature satisfying $H^{AB}\hat{R}_{IABJ} = 0$. In this case, by Lemma \ref{radiantmetriclemma}, $(\chnabla, \rad, H)$ is a special radiant statistical structure with vanishing Ricci curvature and conjugate Ricci curvature. Since $\hnabla$ is a Thomas connection it is $\rad$-invariant, so $\rad^{A}\hat{R}_{PIJ}\,^{B}H_{BK} = 0$. By \eqref{rcsradzero} of Lemma \ref{conjugateradiantcodazzilemma}, $\chnabla$ is also $\rad$-invariant and by Lemma \ref{conjugateconelikelemma}, $(\chnabla, \rad, H)$ is conelike. By Theorem \ref{classicalthomastheorem}, $\chnabla$ is the Thomas connection of some projective structure on $M$.

Now suppose $\chnabla$ is the Thomas connection of some projective structure on $M$. Since $\chnabla$ is torsion-free, there holds $\hnabla_{[I}H_{J]K} = 0$, and since $\chnabla$ is radiant, by \eqref{cq0} of Lemma \ref{radiantmetriclemma} this implies $(g, \rad)$ is self-similar, so $(\hnabla, \rad, H)$ is a radiant statistical structure. Since both $\chnabla$ and $\hnabla$ preserve $|\Psi^{\F}|$, $H^{AB}\hnabla_{I}H_{AB} = 0$, so $(\hnabla, \rad, H)$ is special. Because $\chnabla$ is Ricci-flat, by \eqref{conjradcodcurv} of Lemma \ref{conjugateradiantcodazzilemma} there holds $H^{AB}\hat{R}_{IABJ} = 0$, so $(\hnabla, \rad, H)$ satisfies \eqref{ctt2}.

Suppose there hold \eqref{cott1}-\eqref{cott2}. 
In this case it is immediate from Lemma \ref{conjugateradiantcodazzilemma} that $(\chnabla, \rad, H)$ is also Ricci-flat and conjugate Ricci-flat special radiant statistical structure, and claim \eqref{cttclosedoneform} is immediate from Corollary \ref{radcodazzicorollary}. By \eqref{rcsspecial} of Lemma \ref{conjugateradiantcodazzilemma}, the statements that $(\hnabla, \rad, H)$ is special, $(\chnabla, \rad, H)$ is special, and both $\hnabla$ and $\chnabla$ are equiaffine are all equivalent. Since $\hnabla |\Psi^{\F}| = 0$, these conditions are equivalent to the existence of $c \neq 0$ such that $\det H = c |\Psi^{\F}|^{2}$.

Because $(\hnabla, H)$ is statistical, $\hat{\Pi}_{IJ}\,^{K} = \chnabla - \hnabla = H^{KA}\hnabla_{I}H_{JA}$ satisfies $\hat{\Pi}_{[IJ]}\,^{K} = 0$. Because $g$ is self-similar, $\rad^{P}\hnabla_{P}H_{IJ} = (\lie_{\rad}H)_{IJ} - 2H_{IJ} = 0$. Hence $\rad^{P}\hat{\Pi}_{PI}\,^{J} = 0$. Because $(\hnabla, H)$ is special, $\hat{\Pi}_{IQ}\,^{Q} = H^{PQ}\hnabla_{I}H_{PQ} = 0$ and $\hnabla_{Q}H^{IQ} = 0$. These observations imply that $\hat{\Pi}_{IJ}\,^{K}$ is the horizontal lift of some $\Pi_{ij}\,^{k} \in \Ga(S^{2}\ctm \tensor TM)$. Let $\bnabla \in \ben$ and $\nabla \in \en$ be the representatives inducing $\be \in \prin(\F)$. By definition of the Thomas connections, $\rho^{\ast}(\theta)(\hat{\Pi}(\hat{X}, \hat{Y})) = \rho^{\ast}(\theta(\Pi(X, Y)))$ for all $X, Y \in \Ga(TM)$ and $\theta \in \Ga(\ctm)$, so $\Pi_{ij}\,^{k} = \bnabla - \nabla$. Because $\bnabla$ and $\nabla$ induce the same principal connection, $\Pi_{ip}\,^{p} = 0$, so $\Pi_{ij}\,^{k}$ is trace-free and $\Pi_{ij}\,^{k} = \ben - \en$.

Calculating using the preceding observations yields
\begin{align}
\begin{split}
\hnabla_{K}\hat{\Pi}_{IJ}\,^{K} &= \hnabla_{K}H^{KA}\hnabla_{I}H_{JA} + H^{KA}\hnabla_{K}\hnabla_{I}H_{JA} \\
&= H^{KA}\left(\hnabla_{I}\hnabla_{K}H_{JA} - \hat{R}_{KIJ}\,^{B}H_{BA} - \hat{R}_{KIA}\,^{B}H_{JB}\right)\\
&= H^{KA}\hnabla_{I}\hnabla_{J}H_{KA} - \hat{R}_{KIJ}\,^{K} + H^{KA}\hat{R}_{IKA}\,^{B}H_{JB}\\
& = \hnabla_{I}(H^{KA}\hnabla_{J}H_{KA}) - \hnabla_{I}H^{KA}\hnabla_{J}H_{KA} = H^{KP}H^{AQ}\hnabla_{I}H_{PQ}\hnabla_{J}H_{KA} = \hat{\Pi}_{IA}\,^{B}\hat{\Pi}_{JB}\,^{A}.
\end{split}
\end{align}
Thus $\hat{\Pi}_{IJ}\,^{K}$ is not the $\en$-invariant lift of $\Pi_{ij}\,^{k}$. 
Because $\chnabla$ and $\hnabla$ are invariant under the principal $\reat$-action so is $\hat{\Pi}_{IJ}\,^{K}$. Hence
\begin{align}
\begin{split}
\hnabla_{K}&\left(\hat{\Pi}_{IA}\,^{B}\hat{\Pi}_{JB}\,^{A}\rad^{K}\right) = \hat{\Pi}_{IA}\,^{B}\hat{\Pi}_{JB}\,^{A}\hnabla_{K}\rad^{K} + \rad^{K}\hnabla_{K}\left(\hat{\Pi}_{IA}\,^{B}\hat{\Pi}_{JB}\,^{A}\right)\\
& = (n+1)\hat{\Pi}_{IA}\,^{B}\hat{\Pi}_{JB}\,^{A} + \lie_{\rad}\left(\hat{\Pi}_{IA}\,^{B}\hat{\Pi}_{JB}\,^{A}\right) - 2\hat{\Pi}_{IA}\,^{B}\hat{\Pi}_{JB}\,^{A} 
= (n-1)\hat{\Pi}_{IA}\,^{B}\hat{\Pi}_{JB}\,^{A},
\end{split}
\end{align}
where the penultimate equality follows from \eqref{liederivative}. 
It follows that $\hat{\Pi}_{IJ}\,^{K} - \tfrac{1}{n-1}\hat{\Pi}_{IA}\,^{B}\hat{\Pi}_{JB}\,^{A}\rad^{K}$ is a horizontal lift of $\Pi_{ij}\,^{k}$ and $\hnabla$-divergence free, so, by Lemma \ref{thomasliftlemma}, it is the $\en$-invariant lift of $\Pi_{ij}\,^{k}$.

Now suppose $H_{IJ}$ has homogeneity $2$ and let $v = H(\rad, \rad)$.
Since $H_{IJ}$ has homogeneity $2$, $\be_{I} = v^{-1}\rad^{P}H_{PI}$ is invariant under the principal $\reat$ action on $\F$, so, by \eqref{radcodbe} of Corollary \ref{radcodazzicorollary}, is a principal $\reat$-connection.
The vanishing of the projective Cotton tensors in claim \eqref{einsteinclaim2} follows from \eqref{thomascurvature} of Theorem \ref{classicalthomastheorem} and \eqref{rcsradzero} of Lemma \ref{conjugateradiantcodazzilemma}. That there holds \eqref{einsteinclaim2b} follows from the observation that $\hnabla_{I}\be_{J} + \be_{I}\be_{J}$ is invariant under the principal connection, so has the form $\rho^{\ast}(k)_{IJ}$ for some $k_{ij} \in \Ga(S^{2}\ctm)$ and an explicit computation using the definition of the Thomas connection in the proof of Theorem \ref{classicalthomastheorem} that shows $k_{ij} = -P_{ij}$. Because, by \eqref{cttclosedoneformb} and \eqref{einsteinclaim2b}, $v^{-1}H_{IJ} = \hnabla_{I}\be_{J} + 2\be_{I}\be_{J} = -\rho^{\ast}(P)_{IJ} + \be_{I}\be_{J}$, $P_{ij}$ is nondegenerate.

Because $d\be_{IJ} =0$, $P_{[ij]} = 0$. Let $G_{IJ} = v^{-1}H_{IJ} = - \rho^{\ast}(P)_{IJ} + \be_{I}\be_{J}$ and $\tau_{i} = P^{ab}\nabla_{i}P_{ab}$. By \eqref{hnabladetG} of Lemma \ref{liftedmetricpreliminarieslemma} and because $2\be_{I} = v^{-1}dv_{I}$,
\begin{align}\label{ecb1}
 - G^{KA}\rho^{\ast}(\nabla P)_{IJA} & = G^{KA}\hnabla_{I}G_{JA} + 2\be_{I}\delta_{J}\,^{K}= H^{KA}\hnabla_{I}H_{JA}.
 \end{align}
 Tracing \eqref{ecb1} in $JK$ and using \eqref{liftedcontract} show that $\rho^{\ast}(\tau)_{I} = 0$, so $\tau_{i} = 0$ and $(\nabla, P)$ is a special statistical structure. Similarly, by \eqref{preliftedcontract} and \eqref{ecb1}, $- G^{KA}\rho^{\ast}(\nabla P)_{IJA} = H^{KA}\hnabla_{I}H_{JA}$ is the horizontal lift of $P^{ka}\nabla_{i}P_{ja}$, and by the proof of \eqref{differencelift}, $P^{ka}\nabla_{i}P_{ja} = \bnabla - \nabla$. This shows $(\bnabla, P)$ is the conjugate statistical structure of $(\nabla, P)$. It follows that $(\nabla, P)$ and $(\bnabla, P)$ are conjugate special statistical structures that are naive Einstein, so Einstein. Because, by definition, $R_{ij} = (1-n)P_{ij}$, they have negative scalar curvature.
\end{proof}

\begin{remark}
The metric $H_{IJ}$ of Theorem \ref{conjugatethomastheorem} must be self-similar in the sense of Definition \ref{selfsimilardefined}, but in the setting of the theorem it is not clear that this necessarily implies that it have homogeneity $2$. The issue is that although $\chnabla$ and $\hnabla$ are both invariant under the principal action and $g$ is positively homogeneous of degree $2$, these conditions do not seem to be sufficient to imply that $g$ be homogeneous of degree $2$, although the author does not know an example demonstrating their insufficiency.
\end{remark}

\begin{remark}
The tractor connection of the regular, normal Cartan connection canonically associated with a projective structure can be constructed by descending the Thomas connection to the rank $n+1$ tractor bundle on $M$ obtained by quotienting $T\emf^{-1}$ by an appropriate $\reap$ action, and the condition that $H$ have homogeneity $2$ characterizes those metrics which arise by lifting metrics from the tractor bundle.
\end{remark}

\begin{example}\label{chengyauexample}
Let $\hnabla$ be the flat affine connection determined by the affine structure on the $(n+1)$-dimensional real vector space $\ste$ and let $\cone \subset \ste$ be a pointed convex cone. Let $\Psi$ be the standard volume form on $\ste$ and let $\rad$ be the Euler field generating dilations by $e^{t}$. By a theorem of Cheng-Yau (see \cite{Fox-schwarz, Loftin-affinekahler, Loftin-survey} for details and references), there is a unique smooth function $F$ on $\cone$ solving $\det \hnabla dF = e^{2F}\Psi^{\tensor 2}$, tending to $+\infty$ on the boundary of $\cone$, and such that $\hnabla_{I}dF_{J}$ is a complete Riemannian metric on $\cone$. Moreover, the function $e^{F}$ has positive homogeneity $-n-1$. (The level sets of $F$ are affine spheres asymptotic to $\cone$ and foliating its interior.) The positive homogeneity $2$ function $v=-\tfrac{n+1}{2} e^{-\tfrac{2}{n+1}F}$, one-form $\be_{I} = \tfrac{1}{2}v^{-1}dv_{I} = -\tfrac{1}{n+1}dF_{I}$, and tensor $H_{IJ} = \tfrac{1}{2}\hnabla_{I}dv_{J}$ satisfy $\hnabla_{I}\be_{J} = v^{-1}H_{IJ} - 2\be_{I}\be_{J} = -\tfrac{1}{n+1}\hnabla_{I}dF_{J}$, $2\hnabla_{[I}H_{J]K} = \hnabla_{[I}\hnabla_{J]}dv_{K} = 0$, and $2^{n+1}\det H = \det \hnabla dv = -\Psi^{\tensor 2}$. To justify the last claim, observe that $H_{IJ}$ is nondegenerate and the tensor inverse to $\hnabla_{I}dF_{J}$ is $-(n+1)^{-1}(vH^{IJ} - 2\rad^{I}\rad^{J})$ and use the formula for the determinant of a rank one perturbation to calculate $\det \hnabla_{I}dv_{J}$ in terms of $\det \hnabla_{I}dF_{J}$ (or see \cite[section $5$]{Fox-schwarz}). In particular, $(\hnabla, \rad, H)$ is an Einstein special radiant statistical structure.

Let $M$ be the set of rays in $\cone$, viewed as a subset of the oriented projectivization of $\ste$ (the projective sphere). The canonical projection $\rho:\cone \to M$ has the structure of principal $\reap$-bundle and can be viewed as the frame bundle of the tautological line bundle restricted to $M$. Because $(\hnabla, \rad, \Psi)$ is a flat conelike equiaffine radiant structure on $\cone$, it determines a flat projective structure $\en$ on $M$. Because $\cone$ is pointed, $\en$ is properly convex. The homothety class of the metric of the induced special statistical structure on $M$ comprises the Blaschke metrics of the affine spheres foliating the interior of $\cone$. The induced special statistical structure on $M$ is a Einstein. 
By Lemma \ref{conjugateradiantcodazzilemma} and Theorem \ref{conjugatethomastheorem}, the $H$-conjugate connection $\chnabla$ is a flat connection that with $\rad$ and $\Psi$ constitutes a flat conelike equiaffine radiant structure and with $\rad$ and $H$ constitutes a special radiant statistical structure inducing on $M$ a flat projective structure $\ben$ that is again properly convex. In fact, $\ben$ is the flat projective structure dual to $\en$ obtained via the preceding process applied to the dual cone $\cone^{\ast}$, and the homothety class of the metric of the special statistical structure induced on $M$ comprises the Blaschke metrics of the dual family of affine spheres foliating $\cone^{\ast}$. 

Because all of the preceding behaves well with respect to automorphisms of the structures involved, it can be applied to the cone over the universal cover of any properly flat convex projective structure.  The point is to indicate that Theorem \ref{conjugatethomastheorem} extends this picture based on Cheng-Yau's work on affine spheres and the related Monge-Ampere equations to a possibly nonflat setting. 
\end{example}

\begin{lemma}\label{liftedmetriclemma}
Let $M$ be an $n$-manifold, let $\rho:N \to M$ be a principal $S^{1}$-bundle or principal $\reat$-bundle, and let $\rad$ be the fundamental vertical vector field generated by the principal action. Let $\nabla$ be a torsion-free affine connection on $M$ and let $\be$ be a principal connection on $\rho:N \to M$. Let $\hnabla$ be the cone connection on $\rho:N \to M$ generated by $(\nabla, \be)$. 

Suppose the symmetrized projective Schouten tensor of $\nabla$, $g_{ij} = P_{(ij)}$, is nondegenerate and define the metric $G_{IJ}$ by \eqref{liftedmetric1}, let $\rho^{\ast}(\om)_{IJ} = d\be_{IJ}$ be the curvature of $\be$, and let $\epc_{ij} = 2P_{[ij]} + \om_{ij}$.

Define $\tpc_{i} = g^{ab}\nabla_{i}g_{ab} = (\det g)^{-1}\nabla_{i}\det g$, $\spc_{i} = g^{pq}\nabla_{p}g_{qi}$, $\wpc_{i} = g^{pq}\nabla_{p}\om_{qi}$, and $\conftor_{ijk}= 2\nabla_{[i}g_{j]k}$. 
For $s\in \rea$ and $-1 \neq t \in \rea$ let $G^{IJ}$ be the symmetric bivector inverse to $G_{IJ}$ and define $\hDs = \hnabla + \Om_{IJ}\,^{K}$ where
\begin{align}\label{hdsdefined}
\Om_{IJ}\,^{K} & =  (1+t)\be_{(I}d\be_{J)Q}G^{QK} - t\be_{I}\be_{J}\rad^{K}+ (s-1)\left(2\be_{(I}\delta_{J)}\,^{K}- G_{IJ}\rad^{K}\right) + 2st\beta_{(I}\delta_{J)}\,^{K}.
\end{align}
The modified connection $\hDs$ satisfies
\begin{align}
\label{hdg}
\hDs_{I}G_{JK} & = - 2s(t+1)\be_{I}G_{JK}  -\rho^{\ast}(\nabla g)_{IJK} ,\\
\label{hdsgskew}
\begin{split}
\hDs_{[I}G_{J]K} &= - 2s(1+ t)\be_{[I}G_{J]K} - \tfrac{1}{2}\rho^{\ast}(C)_{IJK}  - \tfrac{1}{4}\rho^{\ast}(\nabla \epc)_{KIJ} + \tfrac{1}{4}\rho^{\ast}(\nabla \om)_{KIJ}\\
&= -2s(1+ t)\be_{[I}G_{J]K} - \tfrac{1}{2}\rho^{\ast}(\conftor)_{IJK},
\end{split}\\
\label{hdsalignment}
G^{PQ}\hDs_{I}G_{PQ} & - (n+1)G^{PQ}\hDs_{P}G_{QI} = \rho^{\ast}(\tpc - (n+1)\spc)_{I},\\
\label{hdstotgeod}\hDs_{I}\rad^{J} & = (1 + t)(\tfrac{1}{2}d\beta_{I}\,^{J} + s\delta_{I}\,^{J}).
\end{align}
In particular, $\hDs_{\rad}\rad = s(1+t)\rad$, so the fibers of $\rho:N \to M$ are $\hDs$-totally geodesic.
The Ricci curvature $\ric(\hDs)_{IJ}$ of $\hDs$ is
\begin{align}\label{richds}
\begin{aligned}
\ric(\hDs)_{IJ} =  
&  -(n+1)\rho^{\ast}(P)_{[IJ]} - \tfrac{s(n+1+ (n+3)t)}{2}\rho^{\ast}(\om)_{IJ}  + (t+1)\left(\rho^{\ast}(\wpc)_{(I}\be_{J)}+  \be_{(I}\rho^{\ast}(\spc \circ \om)_{J)}\right)\\
&\quad + \left((n-1 + nt)s^{2} -n + 1 - t\right)\rho^{\ast}(g)_{IJ} -\tfrac{1+t}{2}\rho^{\ast}(\om \circ \om)_{IJ}\\
&\quad   + (t+1)\left(t(1 - s^{2}) + \tfrac{(t+1)}{4}\rho^{\ast}(|\om|^{2}_{g})\right)\be_{I}\be_{J},
\end{aligned}
\end{align}
\begin{align}\label{richdstrace}
G^{IJ}\ric(\hDs)_{IJ} =  &-\tfrac{1+t}{4}\rho^{\ast}(|\om|^{2}_{g}) - \tfrac{(s^{2}-n - 1)t^{2} + (n^{2} + 1)(s^{2} - 1)t + n(n-1)(s^{2} - 1)}{t+1} + \tfrac{ns^{2}t(t - n+1)}{t+1},
\end{align}
where $(\om \circ \om)_{ij} = \om_{ip}\om_{qj}g^{pq}$, $|\om|^{2}_{g} = \om_{ab}\om_{pq}g^{ap}g^{bq}$, and $(\spc \circ \om)_{i}= \spc_{p}\om_{iq}g^{pq}$.
\end{lemma}

\begin{proof}
In this proof indices are raised and lowered using $G_{IJ}$ and $G^{IJ}$. By \eqref{hnabladetG2},
\begin{align}
\label{hnablahinverse}
\hnabla_{Q}G^{IQ} & = G^{IQ}\rho^{\ast}(\spc)_{Q} + \tfrac{2-nt}{1+t}\rad^{I}.
\end{align}
By \eqref{hnablarhoom}, \eqref{hnabladetG2}, and \eqref{hnablahinverse},
\begin{align}\label{hnbaa}
\begin{aligned}
\hnabla_{A}d\be_{I}\,^{A} &= \hnabla_{A}(G^{AP}d\be_{IP}) = G^{AP}\hnabla_{A}d\be_{IP} + d\be_{IP}\hnabla_{A}G^{AP} \\
& =  G^{AB}\rho^{\ast}(\nabla \om)_{AIB} - d\be_{I}\,^{P}((nt-2)\be_{P} + \rho^{\ast}(\spc)_{P}) =  \rho^{\ast}(\wpc)_{I} -  d\be_{I}\,^{P}\rho^{\ast}(\spc)_{P} .
\end{aligned}
\end{align}
In the rest of the proof suppose $G_{IJ}$ is nondegenerate and let $G^{IJ}$ be the symmetric bivector inverse to $G_{IJ}$.
For $s\in \rea$, define the modified connection $\hDs = \hnabla + \Om_{IJ}\,^{K}$, where $\Om_{IJ}\,^{K}$ is as in \eqref{hdsdefined}.
From \eqref{hdsdefined} and \eqref{hnablaGIJK1} there follows \eqref{hdg}.
\begin{align}
\label{hdgb}
\hDs_{I}G_{JK} & = - 2s(1+t)\be_{I}G_{JK} -\rho^{\ast}(\nabla g)_{IJK}.
\end{align}
Antisymmetrizing \eqref{hdg} and using \eqref{hnablarhoom} and \eqref{lm2} to simplify the result yields the first equality of \eqref{hdsgskew}. The second equality of \eqref{hdsgskew} follows from the definitions of $C_{ijk}$ and $\epc_{ij}$ and the observation that $\nabla_{[i}P_{jk]} =0$. 
Tracing \eqref{hdgb} in two different ways shows \eqref{hdsalignment}. A straightforward computation using \eqref{hdsdefined} shows \eqref{hdstotgeod}.

Let $\tD = \hnabla + \Pi_{IJ}\,^{K}$ where 
\begin{align}
\Pi_{IJ}\,^{K}& =  (1+t)\be_{(I}d\be_{J)Q}G^{QK} - t\be_{I}\be_{J}\rad^{K}+ (s-1)\left(2\be_{(I}\delta_{J)}\,^{K}- G_{IJ}\rad^{K}\right) = \Om_{IJ}\,^{K} - 2st\be_{(I}\delta_{J)}\,^{K}.
\end{align}
The Ricci curvature $\ric(\tD)_{IJ}$ of $\tD$ is
\begin{align}\label{richds1}
\begin{aligned}
\ric(\tD)_{IJ} & = \hat{R}_{IJ} + \hnabla_{Q}\Pi_{IJ}\,^{Q} - \hnabla_{I}\Pi_{QJ}\,^{Q} + \Pi_{PQ}\,^{Q}\Pi_{IJ}\,^{P} - \Pi_{IP}\,^{Q}\Pi_{JQ}\,^{P}.
\end{aligned}
\end{align}
Straightforward computations using $\rad^{I}d\be_{IJ} =0$, and \eqref{hnbaa} in the final equality, yield
\begin{align}\label{richds2}
\begin{aligned}
\Pi_{IQ}\,^{Q} &=  ((s-1)(n+1) -st))\be_{I},\\
\hnabla_{I}\Pi_{JQ}\,^{Q} 
&=  ((s-1)(n+1) -st))(G_{IJ} + \tfrac{1}{2}d\be_{IJ} - (t+2)\be_{I}\be_{J}),\\
\Pi_{IJ}\,^{P}\Pi_{PQ}\,^{Q} 
&=  ((s-1)(n+1) -st))\left((1-s)G_{IJ} + (2(s-1) - t)\be_{I}\be_{J}\right),\\
\Pi_{IP}\,^{Q}\Pi_{JQ}\,^{P} &=  \tfrac{(t+1)^{2}}{4}\be_{I}\be_{J}d\be_{P}\,^{Q}d\be_{Q}\,^{P} - 2(s-1)^{2}G_{IJ} \\
&\quad + \left( t^{2} + 2t(t-1)(s-1) + (n+3 + t^{2})(s-1)^{2}\right)\be_{I}\be_{J},\\
\hnabla_{Q}\Pi_{IJ}\,^{Q} &= (t+1)\left(\rho^{\ast}(\wpc)_{(I}\be_{J)} -  \be_{(I}d\be_{J)}\,^{P}\rho^{\ast}(\spc)_{P}\right)  + \tfrac{t+1}{2}d\be_{I}\,^{A}d\be_{AJ} \\
&\quad + (3-n)(s-1)G_{IJ} - \left((2s + n - 3)t + 4(s-1)\right)\be_{I}\be_{J}.
\end{aligned}
\end{align}
Combining \eqref{richds1}, \eqref{richds2}, and Theorem \ref{extendedthomastheorem} yields
\begin{align}\label{richds3}
\begin{aligned}
\ric(\tD)_{IJ}  & =  -\tfrac{n+1}{2}\rho^{\ast}(\epc)_{IJ} + \tfrac{(1-s)(n+1) - st}{2}\rho^{\ast}(\om)_{IJ}+ (t+1)\left(\rho^{\ast}(\wpc)_{(I}\be_{J)} + \be_{(I}\rho^{\ast}(\spc \circ \om)_{J)}\right) \\
&\quad  - \tfrac{t+1}{2}\rho^{\ast}(\om \circ \om)_{IJ} + (t - (s^{2}  - 1)(n-1))G_{IJ}\\
&\quad + \left((s^{2} - 1)(n-1) - t(n-1 + 2s^{2}) - s^{2}t^{2} +  \tfrac{(t+1)^{2}}{4}\rho^{\ast}(|\om|^{2}_{g})\right)\be_{I}\be_{J}.
\end{aligned}
\end{align}
The Ricci curvature of $\hDs = \tD + st\be_{(I}\delta_{J)}\,^{K}$ can be calculated using \eqref{richds3} and \eqref{projvary}, and there results
\begin{align}
\begin{aligned}
\ric(\tD +st\be \tensor \delta +st\delta \tensor \be)_{IJ} 
& = \ric(\tD)_{IJ} + st\left(ns\left)(2+t)\be_{I}\be_{J} - G_{IJ}\right) - \tfrac{n+2}{2}d\be_{IJ} \right)\\
& = \ric(\tD)_{IJ} + st\left(ns\left(\be_{I}\be_{J}+  \rho^{\ast}(g)_{IJ}\right) - \tfrac{n+2}{2}d\be_{IJ} \right).
\end{aligned}
\end{align}
Combined with \eqref{richds3} this shows \eqref{richds}. Tracing \eqref{richds} yields \eqref{richdstrace}.
\end{proof}

\begin{theorem}\label{liftedahtheorem}
Let $M$ be an $n$-manifold, let $\rho:N \to M$ be a principal $S^{1}$-bundle or principal $\reat$-bundle, and let $\rad$ be the fundamental vertical vector field generated by the principal action. Let $\nabla$ be a torsion-free affine connection on $M$ and let $\be$ be a principal connection on $\rho:N \to M$. Let $\hnabla$ be the cone connection on $\rho:N \to M$ generated by $(\nabla, \be)$. 
Suppose the symmetrized projective Schouten tensor of $\nabla$, $g_{ij} = P_{(ij)}$, is nondegenerate, let $\rho^{\ast}(\om)_{IJ} = d\be_{IJ}$ be the curvature of $\be$, let $\epc_{ij} = 2P_{[ij]} + \om_{ij}$, and, for $t \neq -1$ and $s \in \rea$, define the metric $G_{IJ}$ by \eqref{liftedmetric1} and the connection $\hDs = \hnabla + \Om_{IJ}\,^{K}$ by \eqref{hdsdefined}.
\begin{enumerate}\item\label{ewt}
If $\nabla_{[i}g_{j]k}=0$ and $\nabla_{i}|\det g| =0$, so that $(\nabla, g)$ is a special statistical structure, and there is $\al \in \reat$ such that $\om_{ip}\om_{qj}g^{pq} =\al g_{ij}$, then $\hDs_{[I}G_{J]K}  = - 2s(t+1)\be_{[I}G_{J]K}$ and, if $\al > - \tfrac{4(n-1)}{(n+2)(t+1)}$, for 
\begin{align}\label{salcondition}
s = \pm \sqrt{\tfrac{1}{1+t} + \tfrac{(n+2)\al}{4(n - 1)}},
\end{align}
then $(\hDs, G_{IJ})$ is a half naive Einstein AH structure.
\item\label{eah} If $(\nabla, g)$ is an Einstein special statistical structure with scalar curvature $-n(n-1)$, and there is $\al \in \reat$ such that $\om_{ip}\om_{qj}g^{pq} =\al g_{ij}$, then $(\hDs, G_{IJ})$ generates a naive Einstein AH structure for which $\hDs$ is the aligned representative. If $t = 0$, then $(\hDs, G_{IJ})$ is moreover an Einstein AH structure for which the scalar curvature associated with $G_{IJ}$ is $-\tfrac{n(n+1)\al}{4}$. 
\end{enumerate}
If $g$ has definite signature, the hypotheses in \eqref{ewt} and \eqref{eah} imply $\al < 0$ and $n$ is even.
\end{theorem}

\begin{proof}
The setting is as in the proof of Lemma \ref{liftedmetriclemma}. For tensors on $M$ raise and lower indices using $g_{ij}$ and $g^{ij}$. Suppose $\nabla_{[i}g_{j]k} = 0$ and suppose $g^{pq}\om_{ip}\om_{jq} = -\al g_{ij}$ for some $\al \neq 0$. If $g$ has definite signature, this last assumption means that $\al = -(1/n) |\om|^{2}_{g} < 0$ and a multiple of $\om_{ip}g^{pj}$ is an almost complex structure, so $M$ must have even dimension.
Differentiating $g^{pq}\om_{ip}\om_{jq} = -\al g_{ij}$ yields
\begin{align}\label{etw1}
\begin{split}
-\al \nabla_{k}g_{ij} & = -\om_{i}\,^{a}\om_{j}\,^{b}\nabla_{k}g_{ab} -2\om^{q}\,_{(i}\nabla_{|k|}\om_{j)q}.
\end{split}
\end{align}
Contracting \eqref{etw1} with $g^{ij}$ yields $- \al \tpc_{i} = \om^{pq}\nabla_{i}\om_{pq}$.
Contracting \eqref{etw1} with $g^{jk}$ and using $- \al \tpc_{i} = \om^{pq}\nabla_{i}\om_{pq}$ and $\nabla_{[i}\om_{jk]} =0$ yields
\begin{align}\label{etw3}
\begin{split}
- \al \spc_{i} &=  \om_{i}\,^{a}\om^{bk}\nabla_{k}g_{ab} - \om^{pq}\nabla_{p}\om_{qi}  + \om_{i}\,^{a}\wpc_{a} = \tfrac{1}{2}\om^{pq}\nabla_{i}\om_{pq}  + \om_{i}\,^{a}\wpc_{a} = -\tfrac{1}{2}\al\tpc_{i} + \om_{i}\,^{a}\wpc_{a}.
\end{split}
\end{align}
Contracting $\nabla_{[i}g_{j]k} = 0$ with $g^{jk}$ shows $\tpc_{i} = \spc_{i}$, so \eqref{etw3} implies $0 = \al(2\om_{i}\,^{p}\spc_{p} - \om_{i}\,^{p}\tpc_{p} + 2\wpc_{i}) = \al(\om_{i}\,^{p}\spc_{p} + 2\wpc_{i})$. If there is assumed $\tau_{i} = 0$ and $\al \neq 0$ this implies $\wpc_{i} =0$. 
Consequently, with the stated assumptions, by \eqref{hdsgskew}, $(\hDs, G)$ satisfies $\hDs_{[I}G_{J]K} = -2s(t+1)\be_{[I}G_{J]K}$ and, by \eqref{hdsalignment}, there holds $G^{PQ}\hDs_{I}G_{PQ} - (n+1)G^{PQ}\hDs_{P}G_{QI} = 0$, so that in this case $\hDs$ is the aligned representative of the AH structure that it generates with $G_{IJ}$. Specializing \eqref{richds} yields
\begin{align}\label{richdsew}
\begin{aligned}
\ric(\hDs)_{(IJ)} &= \left((n-1 + nt)s^{2} - n + 1 - t - \tfrac{\al (t+1)}{2}\right)\rho^{\ast}(g)_{IJ}  + (t+1)\left(t(1 - s^{2})- \tfrac{n\al (t+1)}{4}\right)\be_{I}\be_{J}\\
& = -\left(t(s^{2} - 1) + \tfrac{n\al(t+1)}{4} \right)G_{IJ} + \left((n-1)(t+1)s^{2} - (n-1) -\tfrac{(n+2)(t+1)\al}{4} \right)\rho^{\ast}(g)_{IJ}.
\end{aligned}
\end{align}
where the second equality follows from \eqref{liftedmetric1}. By \eqref{richdsew}, $\ric(\hDs)_{IJ}$  is a multiple of $G_{IJ}$ if and only if $s$ solves \eqref{salcondition}. This proves \eqref{ewt}.

The hypotheses of \eqref{eah} are the same as those of \eqref{eah} except there is assumed additionally that the conjugate Ricci curvature $\wideparen{R}_{ij}$ equals $R_{ij} = (1-n)g_{ij}$. By \eqref{ahcurv} this means $\nabla_{p}\bt_{ij}\,^{p} = \bt_{ip}\,^{q}\bt_{jq}\,^{p}$ where $\bt_{ijk} = \nabla_{i}g_{jk}$ is the cubic tensor of $(\nabla, g)$. Let $(\wideparen{\hDs}, [G])$ be the conjugate AH structure of $(\hDs, [G])$. By definition and \eqref{hdg} the cubic tensor of $(\hDs, [G])$ associated with $G_{IJ}$ is $\hDs_{I}G_{JK} + 2s(t+1)\be_{I}G_{JK} = - \rho^{\ast}(\nabla g)_{IJK} = -\rho^{\ast}(\bt)_{IJK}$. By \eqref{ahcurv}, the difference of the Ricci curvatures of $\wideparen{\hDs}$ and $\hDs$ is $-\hDs_{P}\rho^{\ast}(\bt)_{IJ}\,^{P} - \rho^{\ast}(\bt)_{IP}\,^{Q}\rho^{\ast}(\bt)_{JQ}\,^{P}$. It will now be shown that this is zero. On the one hand $\rho^{\ast}(\bt)_{IP}\,^{Q}\rho^{\ast}(\bt)_{JQ}\,^{P} = G^{AB}G^{CB}\rho^{\ast}(\nabla g)_{IAC}\rho^{\ast}(\nabla g)_{JCD}$, which is the pullback via $\rho$ of $\bt_{ip}\,^{q}\bt_{jq}\,^{p}$. On the other hand, because $\hDs_{P}G^{PI}$ is a multiple of $\rad^{I}$,
\begin{align}
\begin{split}
\hDs_{P}\rho^{\ast}(\bt)_{IJ}\,^{P} & = G^{PQ}\hDs_{P}\rho^{\ast}(\bt)_{IJQ}\\
& = G^{PQ}\hnabla_{P}\rho^{\ast}(\bt)_{IJQ} - 2\Om_{(I}\,^{PQ}\rho^{\ast}(\bt)_{J)PQ} - G^{AB}\Om_{AB}\,^{Q}\rho^{\ast}(\bt)_{IJQ}\\
& = G^{PQ}\hnabla_{P}\rho^{\ast}(\bt)_{IJQ},
\end{split}
\end{align}
in which considerable simplification occurs because of the form \eqref{hdsdefined} of $\Om_{IJ}\,^{K}$ (for example, $G^{AB}\Om_{AB}\,^{K}$ is a multiple of $\rad^{K}$ so its contraction with $\rho^{\ast}(\bt)_{IJK}$ vanishes). Computing as in \eqref{hnablarhoom} shows
\begin{align}
\begin{split}
\hnabla_{I}\rho^{\ast}(\bt)_{JKL} = \rho^{\ast}(\nabla \bt)_{IJKL} - 3\be_{I}\rho^{\ast}(\bt)_{JKL} - 3\be_{(J}\rho^{\ast}(\bt)_{KL)I},
\end{split}
\end{align}
from which it follows that $G^{PQ}\hnabla_{P}\rho^{\ast}(\bt)_{IJQ}$ is the pullback via $\rho$ of $-g^{pq}\nabla_{p}\bt_{ijq} = -\nabla_{p}\bt_{ij}\,^{p}$, the last equality because $(\nabla, g)$ is special. It follows that the difference of the Ricci curvatures of $\wideparen{\hDs}$ and $\hDs$ is the pullback via $\rho$ of $\nabla_{p}\bt_{ij}\,^{p} - \bt_{ip}\,^{q}\bt_{jq}\,^{p}$, which vanishes, as commented above. This shows that $(\hDs, [G])$ is a naive Einstein AH structure. To show that it is Einstein there remains to verify the condition \eqref{conservationcondition}.  

The one-form associated with $G$ qua representative of the AH structure $(\hDs, [G])$ is $-2s(t+1)\be_{I}$. Let $\hat{\sc}$ be its scalar curvature. By \eqref{hnablarhoom} and \eqref{hdsdefined},
\begin{align}
\begin{split}
G^{PQ}\hDs_{P}d\be_{QI} & = G^{PQ}\hnabla_{P}d\be_{QI} - \Om_{I}\,^{PQ}d\be_{PQ} = G^{PQ}\rho^{\ast}(\nabla \om)_{PQI} - \tfrac{(1+t)}{2}d\be_{PQ}d\be^{PQ}\be_{I}\\
& = - \rho(\spc)_{I} - \tfrac{1+t}{2}\rho^{\ast}(|\om|^{2}) = -\tfrac{(1+t)n\al}{2}.
\end{split}
\end{align}
Hence,
\begin{align}
d\hat{\sc}_{I} -2s(t+1)\hat{\sc}\be_{I} - 2s(t+1)\tfrac{n+1}{2}G^{PQ}\hDs_{P}d\be_{QI} = -2s(t+1)\left( \hat{\sc} + \tfrac{(1+t)(n+1)n\al}{4}\right)\be_{I}.
\end{align}
When $t = 0$, $\hat{\sc} = -n(n+1)\al/4$, and this yields $0$, completing the proof. 
\end{proof}

\begin{example}\label{ewexample}
Theorem \ref{liftedahtheorem} recovers special cases of results of Calderbank-Pedersen-Swann constructing an Einstein-Weyl manifold on a circle bundle over a positive scalar curvature Kähler-Einstein manifold. On a $2n$-dimensional manifold $M$, let $\nabla$ be the Levi-Civita connection of a Kähler-Einstein metric $h_{ij}$ having nonzero scalar curvature $\sR \in \reat$. Let $\rho:N \to M$ be a principal $S^{1}$-bundle with principal connection $\be$ whose curvature $\om_{ij}$ is a constant multiple of the Kähler form. The assumptions imply $\nabla_{i}\om_{jk} = 0$, $\epc_{ij} = - \om_{ij}$, and $C_{ijk} = 0$, so, by Theorem \ref{liftedahtheorem}, the connections $\hDs$ and the metric $G_{IJ}$ on $N$ defined by \eqref{hdsdefined} and \eqref{liftedmetric1} for $t = 0$ and the two values of $s$ as in \eqref{salcondition} constitute a pair of Einstein-Weyl structures. The tensor $g_{ij}$ of Theorem \ref{liftedahtheorem} equals $-\tfrac{1}{2n-1}\sR h_{ij}$, so if $\sR > 0$ then $G_{IJ}$ is Riemannian, while if $\sR < 0$, then $G_{IJ}$ has Lorentzian signature. In the case $\sR > 0$, this recovers the special case of \cite[Theorems $3.2$, $4.2$]{Pedersen-Swann-submersions} stated as \cite[Theorem $5.8$]{Calderbank-Pedersen}. Moreover, these references give examples of $S^{1}$-bundles to which the construction can be applied.
\end{example}

\begin{example}\label{convexprojectiveexample}
Let $M$ be an oriented, compact surface of genus $g \geq 2$ and fix an integer $k$ satisfying $|k| \leq 2(g-1) = |\chi(M)|$. Let $\en$ be a convex flat real projective structure on $M$. It follows from a theorem of Cheng-Yau that $\en$ is properly convex and there are a Riemannian metric $h$ on $M$ and a representative $\nabla \in \en$ such that $\nabla_{[i}h_{j]k} = 0$, $\nabla_{i}|\det h| = 0$, and the Ricci curvature $R_{ij}$ of $\nabla$ satisfies $R_{ij} = \tfrac{c}{2}h_{ij}$ for a negative constant $c \in \rea$; see \cite[Theorem $7.3$]{Fox-2dahs} and \cite{Loftin-affinekahler, Loftin-survey}. The Riemannian metric $g_{ij} = P_{(ij)} = -R_{(ij)} = -(c/2)h_{ij}$ is homothetic to $h_{ij}$ so satisfies $\nabla_{[i}g_{j]k} = 0$ and $\nabla_{i}|\det g| = 0$. Let $J_{i}\,^{j}$ be the complex structure determined by rotation by $\pi/2$ with respect to the metric $g_{ij}$ in the sense of the given orientation of $M$. The two-form defined by $\om_{ij} = \tfrac{2\pi k}{\vol_{g}(M)}J_{i}\,^{p}g_{pj}$ is a constant multiple of the volume form determined by $g$ and the given orientation of $M$ and satisfies $\om_{ip}\om_{pj}g^{pq} = \al g_{ij}$ where $\al = -\tfrac{4\pi^{2}k^{2}}{\vol_{g}(M)^{2}}$. Because $\nabla$ preserves the volume form of $g$, $\nabla_{i}\om_{jk} = 0$. Because $\tfrac{1}{2\pi}\int_{M}\om = k$ is an integer, there is a principal $S^{1}$-bundle $\rho:N \to M$ and a principal $S^{1}$-connection $\be$ on $N$ such that $d\be = \rho^{\ast}(\om)$ \cite{Kobayashi-circlebundles}. Moreover $k$ is the Euler number $e(N)$ of $\rho:N \to M$. Let $\hnabla$ be the cone connection of the extended projective structure generated by $(\nabla, \be)$ and, as in \eqref{liftedmetric1} of Lemma \ref{liftedmetricpreliminarieslemma}, define on $N$ the Lorentz signature metric $G_{IJ} = \be_{I}\be_{J} - \rho^{\ast}(g)_{IJ}= \be_{I}\be_{J} + \tfrac{c}{2} \rho^{\ast}(h)_{IJ}$ and the connection $\hDs$ as in \eqref{hdsdefined} (with $t = 0$ and $s \in \rea$). By Lemma \ref{liftedmetriclemma} and Theorem \ref{liftedahtheorem},  
\begin{align}\label{richds2d}
\begin{aligned}
\ric(\hDs)_{IJ} & =  - \tfrac{3}{2}s\rho^{\ast}(\om)_{IJ}  + \tfrac{2\pi^{2}k^{2}}{\vol_{g}(M)^{2}}\be_{I}\be_{J} + \left(s^{2} - 1  + \tfrac{2\pi^{2}k^{2}}{\vol_{g}(M)^{2}}\right)\rho^{\ast}(g)_{IJ}  \\
& = - \tfrac{3}{2}s\rho^{\ast}(\om)_{IJ}    + \tfrac{2\pi^{2}k^{2}}{\vol_{g}(M)^{2}}G_{IJ}+ \left(s^{2} - 1  + \tfrac{4\pi^{2}k^{2}}{\vol_{g}(M)^{2}}\right)\rho^{\ast}(g)_{IJ}.
\end{aligned}
\end{align}
By \cite[Theorem $7.3$]{Fox-2dahs}, $\vol_{g}(M) = -(c/2)\vol_{h}(M) \geq 4\pi(g-1)$, with equality if and only if $\nabla$ is the Levi-Civita connection of $h$. Hence, if $|k| \leq 2(g-1)$, then $\vol_{g}(M)^{2} \geq 4\pi^{2}k^{2}$, so there is $s$ solving $s^{2} = 1 - \tfrac{4\pi^{2}k^{2}}{\vol_{g}(M)^{2}}$, and $s$ equals $0$ if and only if $\nabla$ is the Levi-Civita connection of $h$. For each of the resulting roots $s = s^{\pm}$, by Theorem \ref{liftedahtheorem}, the resulting connection $\hDs$ satisfies $\hDs_{I}G_{JK} = -2s\be_{I}G_{JK} - \rho^{\ast}(\nabla g)_{IJK}$, $G^{PQ}\hDs_{I}G_{PQ} - 3G^{PQ}\hDs_{P}G_{QI} = 0$, and $\ric(\hDs)_{IJ} =  - \tfrac{3}{2}s\rho^{\ast}(\om)_{IJ} - \tfrac{2\pi^{2}k^{2}}{\vol_{g}(M)^{2}}G_{IJ}$ and $(\hDs, [G])$ is an Einstein AH structure that is not closed and has scalar curvature $- \tfrac{3\pi^{2}k^{2}}{2\vol_{g}(M)^{2}}$.
The preceding proves:
\begin{theorem}\label{convexprojectivetheorem}
Let $M$ be an oriented, compact surface of genus $g \geq 2$. Let $\rho:N \to M$ be a principal $S^{1}$-bundle with Euler number $e(N)$ satisfying $|e(N)| \leq 2(g-1) = -\chi(M)$. On $N$ there are a Lorentzian signature metric $G$ and a torsion-free affine connection $\hDs$ which is the aligned representative of the AH structure $(\hDs, [G])$ it generates with $G$ and having the following properties:
\begin{enumerate}
\item $(\hDs, [G])$ is an Einstein AH structure. 
\item The fibers of $\rho:N \to M$ are timelike and $\hDs$-totally geodesic.
\item There is a Riemannian metric $g$ on $M$ such that $\rho:(N, G) \to (M, -g)$ is a metric submersion.
\item The connection $\nabla$ on $M$ defined by $\nabla_{X}Y = T\rho(\hDs_{\hat{X}}\hat{Y})$, where $X \in \Ga(TM) \to \hat{X} \in \Ga(TN)$ is the horizontal lift with respect to a principal connection on $\rho:N \to M$, generates a properly convex flat projective structure on $M$ for which $g$ is homothetic to the Cheng-Yau metric. The symmetric part of the Ricci curvature of $\nabla$ is $-g$ and $(\nabla, g)$ is an Einstein special statistical structure with negative scalar curvature.
\item All properly convex flat projective structures on $M$ arise in this way, and the resulting Einstein special statistical structure $(\nabla, g)$ comprises a Riemannian metric of constant curvature and its Levi-Civita connection if and only if $e(N) = \pm \chi(M)$ and $(\hDs, [G])$ is closed.
\end{enumerate}
\end{theorem}
The condition on the Euler number is needed to construct $G$. It would be interesting to show that it is necessary. 
\end{example}

\section{Left-invariant conelike radiant structures on Lie groups}\label{leftinvariantsection}
A radiant structure $(\nabla, \rad)$ on $M$ is \emph{homogeneous} if there is a Lie subgroup $G \subset \Aut(\nabla)$ acting transitively on $M$ and preserving $\rad$. The simplest nontrivial examples of homogeneous radiant structures are left-invariant radiant structures on Lie groups. The main result of this section, Theorem \ref{gnonnulltheorem}, shows that an invariant symmetric bilinear form $k$ and an a $k$-anisotropic element $t$ in the Lie algebra $\g$ of a Lie group $G$ determine on $G$ in a canonical way a left-invariant conelike radiant connection having antisymmetric Ricci tensor. In the particular case that $k$ is the Killing form, Theorem \ref{killingnonnulltheorem} shows that the construction is equivariant with respect to the natural actions of $\Aut(G)$ on all the structures involved.

Let $\g$ be the Lie algebra of a real Lie group $G$ and, for $r \in \g$, let $E^{r}_{g} = \tfrac{d}{dt}\big|_{t =0}g\exp(tr)$ be the left-invariant vector field generated on $G$ by $r$. 
A radiant structure $(\nabla, \rad)$ on $G$ is \emph{left-invariant} if $\nabla$ and $\rad$ are left-invariant. In this case $\rad = E^{t}$ for some $t \in \g$ and so a left-invariant radiant structure is always nonsingular. For this reason, when discussing left-invariant radiant structures the qualifier \emph{nonsingular} is omitted.

That $\nabla$ be left-invariant implies there is $A \in \tensor^{2}\g^{\ast}\tensor \g$ such that $\nabla_{E^{r}}E^{s} = E^{A(r, s)}$ for all $r, s \in \g$, and  $\nabla$ is torsion-free if and only if $A(r, s) - A(s, r) = [r, s]$ for all $r, s \in \g$. This means $A(r, s) = \Pi(r, s) + \tfrac{1}{2}[r, s]$ for some $\Pi \in S^{2}\g^{\ast}\tensor \g$. That $(\nabla, \rad)$ be a left-invariant radiant structure on $G$ means that there is $t \in \g$ such that $\rad = E^{t}$ and $A(r, t) = r$ for all $r \in \g$. This implies that, for all $r \in \g$, $A(t, r) = r + [t, r]$, so that $\Pi(t, r) = r + \tfrac{1}{2}[t, r] = \Pi(r, t)$. In abstract index notation, $A_{ij}\,^{k} = \Pi_{ij}\,^{k} + \tfrac{1}{2}c_{ij}\,^{k}$ where $c_{ij}\,^{k}$ is the structure tensor of $\g$ defined by $a^{i}c_{ij}\,^{k} = \ad_{\g}(a)_{j}\,^{k}$.

Let $\ell(a) = \tr \ad_{\g}(a)$ be the one-form on $\g$ measuring the failure of unimodularity (so $\ell_{i} = c_{ip}\,^{p}$). Let $B_{\g}(a, b) = \tr \ad_{\g}(a)\ad_{\g}(b)$ be the Killing form of $\g$ (so $B_{ij} = c_{ip}\,^{q}c_{jq}\,^{p}$).

The adjoint action induces an action of $\g$ on any space of tensors on $\g$. For example, for $a \in \g$, $a\cdot B_{\g} = 0$, where, for $\Om \in \tensor^{2}\g^{\ast}\tensor \g$, the action of $a$ on $\Om$ is given by
\begin{align}
(a\cdot \Om)(b, c) = [a, \Om(b, c)] - \Om([a, b], c) - \Om(b, [a, c]).
\end{align}
Since the curvature of a left-invariant radiant structure is left-invariant, there is $\lr \in \tensor^{3}\g^{\ast}\tensor \g$ defined by $R(E^{a}, E^{b})E^{c} = E^{\lr(a, b)c}$. By definition,
\begin{align}\label{lrijkl}
\begin{split}
\lr_{ijk}\,^{l} & = 2A_{[i|p|}\,^{l}A_{j]k}\,^{p}  - c_{ij}\,^{p}A_{pk}\,^{l} = 2\Pi_{p[i}\,^{l}\Pi_{j]k}\,^{p} - c_{p[i}\,^{l}\Pi_{j]k}\,^{p} + \Pi_{p[i}\,^{l}c_{j]k}\,^{p} - c_{ij}\,^{p}\Pi_{pk}\,^{l} - \tfrac{1}{4}c_{ij}\,^{p}c_{pk}\,^{l}.
\end{split}
\end{align}
Equivalently,
\begin{align}\label{lrtrs}
\lr(t, r)s  = [t, \Pi(r, s)] - \Pi([t, r], s) - \Pi(r, [t, s]) = (t\cdot \Pi)(r, s)
\end{align}
for all $r, s \in \g$. Write $\ric(E^{a}, E^{b}) = \lric(a, b)$ for $a, b \in \g$ and $\er(a) = \lric(t, a)$. Contracting \eqref{lrijkl} in $il$ yields
\begin{align}\label{lricij}
&\lr_{(jk)}  = \Pi_{pq}\,^{q}\Pi_{jk}\,^{p} - \Pi_{jp}\,^{q}\Pi_{kq}\,^{p} - \tfrac{1}{2}\ell_{p}\Pi_{jk}\,^{p}- c_{p(j}\,^{q}\Pi_{k)q}\,^{p} - \tfrac{1}{4}B_{jk} ,&&
\lr_{[jk]} =  \tfrac{1}{2}c_{jk}\,^{p}\Pi_{pq}\,^{q}.
\end{align}
Suppose $n = \dim \g > 2$. By Lemma \ref{coneconditionlemma}, a left-invariant radiant structure is conelike if and only if there is $Q \in S^{2}\g^{\ast}$ such that $nQ(t, r) = \er(r)$, $Q(t, t) = 0$, and
\begin{align}\label{tpirs}
\begin{split}
(t\cdot \Pi)(r, s)&= \r(t, r)s = Q(r, s)t - Q(t, r)s - Q(t, s)r = Q(r, s)t - \tfrac{1}{n}\er(r)s - \tfrac{1}{n}\er(s)r,
\end{split}
\end{align}
for all $r, s \in \g$, and, in this case, there holds
\begin{align}\label{qrs}
\begin{split}
(n-2)Q(r, s) &- Q([t, r], s) - Q(r, [t, s]) \\
&= \tfrac{n-2}{n}\er(\Pi(r, s)) + \tfrac{1}{2}\left(\lric([t, r], s) + \lric(r, [t, s]) + \lric([t, s], r) + \lric(s, [t, r]) \right),
\end{split}
\end{align}
for all $r, s \in \g$, where there has been used $\lric(\dum, t) =0$. 

\begin{lemma}\label{ertrlemma}
Suppose $n  > 2$ and $G$ is a $n$-dimensional Lie group with Lie algebra $\g$. If the radiant vector field of a left-invariant conelike radiant structure on $G$ is generated by $t \in \g$, then $\er([t, r]) = 0$ for all $r \in \g$.
\end{lemma}
\begin{proof}
Taking $s = t$ in \eqref{qrs} yields
\begin{align}\label{ertr}
\begin{split}
\tfrac{n-2}{n}\er(t) - \tfrac{1}{n}\er([t, r])& =  (n-2)Q(r, t) - Q([t, r], t) - Q(r, [t, t]) \\
&= \tfrac{n-2}{n}\er(\Pi(r, t)) + \tfrac{1}{2}\left(\lric([t, r], t) + \lric(r, [t, t]) + \lric([t, t], r) + \lric(t, [t, r]) \right)\\
& = \tfrac{n-2}{n}\er(r) + \tfrac{n-2}{2n}\er([t, r]) + \tfrac{1}{2}\er([t, r]) = \tfrac{n-2}{n}\er(r) + \tfrac{n-1}{n}\er([t, r]),
\end{split}
\end{align}
which, because $n >2$, implies that $\er([t, r]) = 0$ for all $r \in \g$.
\end{proof}
\begin{theorem}\label{gnonnulltheorem}
Suppose $n > 2$ and let $G$ be an $n$-dimensional Lie group with Lie algebra $\g$ and Killing form $B_{\g} \in S^{2}\g^{\ast}$. Suppose $k \in S^{2}\g^{\ast}$ is invariant, meaning that $a\cdot k = 0$ for $a \in \g$ and suppose there is $t \in \g$ such that $k(t, t) \neq 0$. Define $h \in S^{2}\g^{\ast}$ by $h(r, s) = k(t, t)^{-1}k(r, s)$. Define $\Pi \in S^{2}\g^{\ast}\tensor \g$ by
\begin{align}\label{gconedefinedg}
\begin{split}
\Pi(r, s) &=- h(t, r)h(t, s) t + h(t, r)s + h(t, s)r+ \tfrac{1}{2}h(t, r)[t, s] + \tfrac{1}{2}h(t, s)[t, r] \\
& + \tfrac{1}{4(n-2)}\left(B_{\g}(r, s)  +B_{\g}(t, t)h(t, r)h(t, s) - B_{\g}(t, r)h(t, s) - B_{\g}(t, s)h(t, r) - 2h([t, r], [t, s]) \right)t .
\end{split}
\end{align}
The left-invariant connection $\nabla$ determined by $A(\dum, \dum) = \Pi(\dum, \dum) + \tfrac{1}{2}[\dum, \dum] \in \tensor^{2}\g^{\ast}\tensor \g$ constitutes with the left-invariant vector field $E^{t}$ a left-invariant conelike radiant structure with antisymmetric Ricci tensor 
\begin{align}\label{gconericg}
\begin{split}
\lric(r, s) = \tfrac{2n + \ell(t)}{4}h(t, [r, s]),
\end{split}
\end{align}
and satisfying $\er = 0$.
\begin{enumerate}
\item\label{coneautog} For $g \in G$, the left-invariant conelike radiant connection $\nabla$ associated with $\ad_{\g}(g^{-1})t$ is $\Ad(g)^{\ast}(\nabla)$. In particular the stabilizer of $t$ in $\Ad(G)$ acts as automorphisms of $(\nabla, E^{t})$.
\item\label{gnullplanelike} $(\nabla, E^{t})$ admits a complete set of planes.
\item\label{ssnotflat} If $\g$ is semisimple, then $\nabla$ is not Ricci-flat.
\item\label{subalgtotgeodg} A subgroup $H \subset G$ tangent to $t$ is $\nabla$ totally geodesic, and the connection induced on $H$ by $\nabla$ constitutes with the left-invariant vector field $E^{t}$ a left-invariant conelike radiant structure on $H$.
\item\label{abeliansubalgtotgeodg} Suppose $H \subset G$ is an abelian subgroup tangent to $t$, having Lie algebra $\h \subset \g$, and such that $\dim H \geq 3$. The left-invariant connection $\tnabla$ on $H$ associated with the tensor $\tilde{\Pi} \in S^{2}\h^{\ast}\tensor \h$ defined by
\begin{align}\label{hconedefinedg}
\tilde{\Pi}(a, b) =  h(t, a)b + h(t, b)a - h(t, a)h(t, b)t,
\end{align}
constitutes with $E^{t}$ a Ricci-flat left-invariant conelike radiant structure on $H$ having the same planelike surfaces as the connection induced on $H$ by $\nabla$. 
\end{enumerate}
\end{theorem}

\begin{proof}
Define $\bar{\Pi} \in S^{2}\g^{\ast}\tensor \g$ by
\begin{align}\label{pirsgeneralg}
\begin{split}
\bar{\Pi}(r, s) &=  - h(t, r)h(t, s)t +  h(t, r)s + h(t, s)r + \tfrac{1}{2}h(t, r)[t, s] +  \tfrac{1}{2}h(t, s)[t, r]  .
\end{split}
\end{align}
By definition $\bar{\Pi} \in S^{2}\g^{\ast}\tensor \g$. Let $\bnabla$ be the associated left-invariant connection on $G$. Taking $s = t$ yields $\bar{\Pi}(r, t) = r+ \tfrac{1}{2}[t, r]$, so $\bnabla$ forms with $E^{t}$ a left-invariant radiant structure. As $\bar{\Pi}$ is constructed from tensors invariant under the action induced by $\ad(t)$, it is also invariant under the action induced by $\ad(t)$. That is $t \cdot \Pi =0$. By \eqref{lrtrs} this shows the curvature $\bar{\lr}(\dum, \dum)$ of $\bnabla$ satisfies $\bar{\lr}(t, r)s = 0$, and, by Lemma \ref{coneconditionlemma}, this shows that $(\bnabla, E^{t})$ is a conelike left-invariant radiant structure.

In calculating the Ricci curvature of $\bnabla$ it is convenient to use abstract indices and rewrite the definition \eqref{pirsgeneralg} as
\begin{align}\label{pirsgeneral2g}
\bar{\Pi}_{ij}\,^{k} & = - t_{i}t_{j}t^{k} + 2t_{(i}\delta_{j)}\,^{k} + t_{(i}A_{j)}\,^{k},
\end{align}
where $t_{i} = t^{p}h_{ip}$ and $A_{i}\,^{j} = \ad(t)_{i}\,^{j}$. By definition $t_{p}t^{p} = 1$, $t^{p}A_{p}\,^{i} = 0$, $A_{i}\,^{p}\ell_{p} = 0$, and $A_{p}\,^{p} = \ell_{p}t^{p}$. From the invariance of $h$ there follow $A_{i}\,^{p}t_{p} = 0$. In what follows these observations are used without further comment.
From \eqref{pirsgeneral2g} there follow
\begin{align}\label{liricpreg}
\begin{split}
\bar{\Pi}_{ij}\,^{p}t_{p} & = t_{i}t_{j},\qquad
\bar{\Pi}_{ip}\,^{p}   = (n + \tfrac{1}{2}\ell(t))t_{i},\qquad
\bar{\Pi}_{ij}\,^{p}\ell_{p}  = -\ell(t)t_{i}t_{j} + 2t_{(i}\ell_{j)},\\
\bar{\Pi}_{ip}\,^{q}\bar{\Pi}_{jq}\,^{p} & = \left(n + \ell(t) + \tfrac{1}{4}B_{\g}(t, t)\right)t_{i}t_{j} ,\qquad
c_{p(i}\,^{q}\bar{\Pi}_{j)q}\,^{p}  = -\ell_{(i}t_{j)} - \tfrac{1}{2}t^{p}B_{p(i}t_{j)} - \tfrac{1}{2}A_{i}\,^{p}A_{j}\,^{q}h_{pq}.
\end{split}
\end{align}
Substituting \eqref{liricpreg} in \eqref{lricij} yields
\begin{align}\label{lricgeneral2g}
\begin{split}
\bar{\lr}_{(ij)} & =-\tfrac{1}{4}B_{\g}(t, t)t_{i}t_{j} - \tfrac{1}{4}B_{ij} + \tfrac{1}{2}t^{p}B_{p(i}t_{j)} +\tfrac{1}{2}A_{i}\,^{p}A_{j}\,^{q}h_{pq},\\
\bar{\lr}_{[ij]} & = \tfrac{1}{2}\left(n + \tfrac{1}{2}\ell(t)\right)c_{ij}\,^{p}t_{p} =  \tfrac{1}{2}\left(n + \tfrac{1}{2}\ell(t)\right)A_{i}\,^{p}h_{pj}.
\end{split}
\end{align}
This shows that the left-invariant connection $\bnabla$ determined by $\bar{A}(\dum, \dum) = \bar{\Pi}(\dum, \dum) + \tfrac{1}{2}[\dum, \dum] \in \tensor^{2}\g^{\ast}\tensor \g$ constitutes with $E^{t}$ a left-invariant conelike radiant structure with Ricci curvature 
\begin{align}\label{lricgeneralg}
\begin{split}
\bar{\lric}(r, s) & = \tfrac{1}{4}\left(B_{\g}(t, r)h(t, s) + B_{\g}(t, s)h(t, r) -B_{\g}(t, t) h(t, r)h(t, s)  - B_{\g}(r, s) + 2h([t, r], [t, s]) \right) \\
&\qquad + \tfrac{2n + \ell(t)}{2}h(t, [r, s]).
\end{split}
\end{align}
By the proof of Theorem \ref{conenormalizationtheorem}, $\Pi(r, s) + \tfrac{1}{2}[r, s] = A(r, s) = \bar{A}(r, s) -\tfrac{1}{2(n-2)}(\bar{\lric}(r, s) + \bar{\lric}(s, t))t$ yields a left-invariant connection $\nabla$ that constitutes with $E^{t}$ a conelike radiant structure having antisymmetric Ricci tensor equal to the antisymmetric part of the Ricci tensor of $\bnabla$ (Because $\bar{\lr}(t, r)s = 0$, the tensor $T_{ij}$ in the proof of Theorem \ref{conenormalizationtheorem} is identically $0$). The resulting tensor $\Pi$ is as in \eqref{pirsgeneralg} and the Ricci curvature of $\nabla$ is given by \eqref{gconericg}. Because $\lric(t, \dum) = 0$, $\er = 0$.

For $g \in G$, let $\bnabla$ be the connection associated with $\bar{\Pi} \in S^{2}\g^{\ast}\tensor \g$ determined by $\bar{t} = \ad_{\g}(g^{-1})t$ as in \eqref{gconedefinedg}. Because $\ad_{\g}(g)$ is an automorphism of $\g$, $\bar{\Pi}(r, s) = \ad_{\g}(g^{-1})\Pi(\ad_{\g}(g)r, \ad_{\g}(g)s)$, which is the tensor determined by the left-invariant connection $\Ad(g)^{\ast}(\nabla)$. This shows \eqref{coneautog}. By \eqref{coneautog}, $\nabla$ is $E^{t}$-invariant, so, by Corollary \ref{invariantconelikecorollary}, $(\nabla, E^{t})$ admits a complete set of planes.

If $G$ is an $n$-dimensional semisimple real Lie group, then, by the Jacobson-Morozov theorem, its Lie algebra $\g$ contains an $\sll(2, \rea)$-triple $\{t, r, s\}$ such that $[t, r] = 2r$, $[t, s] = -2s$, and $[r, s] = t$. By \eqref{gconeric}, $\lric(r, s) = \tfrac{2n + \ell(t)}{4}h(t, [r, s]) = \tfrac{n}{2}h(t, t)  = \tfrac{n}{2}$. In particular, $\nabla$ is not flat. This shows \eqref{ssnotflat}. 

Suppose $H \subset G$ is a subgroup with Lie algebra $\h \subset \g$ containing $t$. For $r, s \in \h$ it follows from \eqref{gconedefinedg} that $A(r, s) \in \h$, so $H$ is totally geodesic. Consequently, $A$ determines $\bar{A} \in \tensor^{2}\h^{\ast}\tensor \h$ defined by $\bar{A}(a, b) = A(a, b)$, and the left-invariant connection $\bnabla$ on $H$ determined by $\bar{A}$ equals the connection induced on $H$ by $\nabla$. Because $\bar{A}(a, t) = a$, $\bnabla$ is radiant. By \eqref{lrijkl}, its curvature $\bar{\lr}(a, b)c$ equals $\lr(a, b)c$ for $a, b, c \in \h$, and so $\bar{\lr}(t, a)b = \lr(t, a)b = 0$ for $a, b \in \h$, which by Lemma \ref{coneconditionlemma} suffices to show that the left-invariant radiant structure $(\bnabla, \rad)$ on $H$ is conelike. This proves \eqref{subalgtotgeodg}.

Suppose $t\in \h$ and $\h$ normalizes $t$. This means there is $\si \in \h^{\ast}$ such that $[a, t] = \si(a)t$ for $a \in \g$. By the invariance of $g$, $0 = g(a, [t, t]) = g([a, t], t) = -\si(a)g(t, t)$, so $\si = 0$ and $\h$ centralizes $t$.
By \eqref{gconedefined}, for $a, b \in \h$,
\begin{align}\label{habelian}
\Pi(a, b) = Q(a, b)t + S(a)b + S(b)a \in \h,
\end{align}
where $S(a) = h(t, a)$, $S(t) = 1$, $Q(a, t) = -S(a)$, and
\begin{align}
Q(a, b) &= -h(t, a)h(t, b) + \tfrac{1}{4(n-2)}\left(B_{\g}(a, b) + B_{\g}(t, t)h(t, a)h(t, b) - B_{\g}(t, a)h(t, b) - B_{\g}(t, b)h(t, a)\right).
\end{align}
The connection $\bnabla$ induced on $H$ by $\nabla$ is is a left-invariant connection on $H$, and, because $H$ is totally geodesic, its curvature $\bar{\lr}(b, c)a$ equals $\lr(b, c)a$. 
 
Next, suppose that $\h$ is moreover abelian, so that $A(a, b) = \Pi(a, b)$ for $a, b \in \h$.
From \eqref{lrijkl} and \eqref{habelian} it follows that
\begin{align}\label{barlr}
\begin{split}
\bar{\lr}(b, c)a & = \lr(b, c)a = \Pi(b, \Pi(c, a)) - \Pi(c, \Pi(b, a))   \\
&= (Q(a, c) + S(a)S(c))b -(Q(a, b) + S(a)S(b))c + (Q(a, b)S(c) - Q(a, c)S(b))t \in \h,
\end{split}
\end{align}
for all $a, b, c \in \h$. Temporarily using abstract indices to indicate tensors on $\h$, rewrite \eqref{barlr} as
\begin{align}\label{barlr2}
\begin{split}
\bar{\lr}_{ijk}\,^{l} &=  (Q_{jk} + S_{j}S_{k})\delta_{i}\,^{l}  - (Q_{ik} + S_{i}S_{k})\delta_{j}\,^{l} + (S_{j}Q_{ik} - S_{i}Q_{jk})t^{l}\\
& =  (Q_{jk} + S_{j}S_{k})(\delta_{i}\,^{l}- S_{i}t^{l}) - (Q_{ik} + S_{i}S_{k})(\delta_{j}\,^{l} - S_{j}t^{l}).
\end{split}
\end{align}
Let $m = \dim \h$. Tracing \eqref{barlr2} in $il$ yields
\begin{align}\label{barlr3}
\begin{split}
\overline{\lric}_{jk} & = \bar{\lr}_{pjk}\,^{p} = (m - 2)(Q_{jk} + S_{j}S_{k}),
\end{split}
\end{align}
so that
\begin{align}\label{barlr4}
\overline{\lric}(a, b) = \tfrac{m-2}{4(n-2)}\left(B_{\g}(a, b) + B_{\g}(t, t)h(t, a)h(t, b) - B_{\g}(t, a)h(t, b) - B_{\g}(t, b)h(t, a)\right),
\end{align}
for $a, b \in\h$. In particular $\overline{\lric}(t, a) = 0$, so $\bar{\er} =0$ for $\bnabla$. If $m \geq 3$, by the proof of Theorem \ref{conenormalizationtheorem}, the left-invariant connection $\tnabla$ on $H$ associated with the tensor 
\begin{align}
\begin{split}
\tilde{\Pi}(a, b) &= \Pi(a, b) - \tfrac{1}{m-2}\overline{\lric}(a, b)t = S(a)b + S(b)a - S(a)S(b)t = h(t, a)b + h(t, b)a - h(t, a)h(t, b)t,
\end{split}
\end{align}
(which is as in \eqref{hconedefinedg}) constitutes with $E^{t}$ a left-invariant conelike radiant structure on $H$ with Ricci curvature equal to the antisymmetric part of $\overline{\lric}$, so with vanishing Ricci curvature, and having the same planelike surfaces as the connection $\bnabla$ induced on $H$ by $\nabla$. This proves \eqref{abeliansubalgtotgeodg}.
\end{proof}

\begin{remark}
The formula \eqref{gconedefinedg} is simpler than it appears. For $r, s \in \ker k(t, \dum)$ it simplifies to
\begin{align}\label{gconedefinedgker}
&\Pi(t, t) = t,&
&\Pi(r, t)  = \Pi(t, r) = r + \tfrac{1}{2}[t, r],& 
&\Pi(r, s) = \tfrac{1}{4(n-2)}\left(B_{\g}(r, s)  - \tfrac{2}{k(t,t)}k([t, r], [t, s]) \right)t . 
\end{align}
\end{remark}

The utility of the formulation of Theorem \ref{gnonnulltheorem} is that it can be applied even to a Lie group $G$ for which the Killing form is identically zero (such as a nilpotent Lie group), but with some other invariant tensor $k$. A Lie algebra $\g$ equipped with an invariant nondegenerate bilinear form $k$ is said to be a \emph{quadratic Lie algebra}. There are many examples of nilpotent quadratic Lie algebras, and for such the Killing form is identically zero, but Theorem \ref{gnonnulltheorem} still applies. In particular there holds the following.

\begin{corollary}
A Lie group of dimension $n > 2$ whose Lie algebra $\g$ is equipped with a structure of a quadratic Lie algebra admits a left-invariant conelike radiant structure.
\end{corollary}
\begin{proof}
This follows from Theorem \ref{gnonnulltheorem}. If $k$ is an invariant metric on $\g$, its nondegeneracy implies there is $t \in\g$ such that $k(t, t) \neq 0$.
\end{proof}
On the other hand, for $G$ for which the Killing form is not identically zero, the invariant tensor $k$ in Theorem \ref{gnonnulltheorem} can be taken to be a multiple of the Killing form and the statement of the theorem can be simplified and strengthened in certain respects.

\begin{theorem}\label{killingnonnulltheorem}
Suppose $n > 2$ and let $G$ be an $n$-dimensional Lie group with Lie algebra $\g$ and Killing form $B_{\g} \in S^{2}\g^{\ast}$. Suppose there is $t \in \g$ such that $B_{\g}(t, t) \neq 0$. Define $\Pi \in S^{2}\g^{\ast}\tensor \g$ by
\begin{align}\label{gconedefined}
\begin{split}
\Pi(r, s) &= \left(\tfrac{1}{4(n-2)}B_{\g}(r, s) - \tfrac{1}{B_{\g}(t, t)}\left( \left(\tfrac{1}{4(n-2)} + \tfrac{1}{B_{\g}(t, t)}\right)B_{\g}(t, r)B_{\g}(t, s)+\tfrac{1}{2(n-2)} B_{\g}([t, r], [t, s])\right)\right)t \\
&+ \tfrac{1}{B_{\g}(t, t)}\left(B_{\g}(t, r)s + B_{\g}(t, s)r+ \tfrac{1}{2}\left(B_{\g}(t, r)[t, s] + B_{\g}(t, s)[t, r] \right) \right).
\end{split}
\end{align}
The left-invariant connection $\nabla$ determined by $A(\dum, \dum) = \Pi(\dum, \dum) + \tfrac{1}{2}[\dum, \dum] \in \tensor^{2}\g^{\ast}\tensor \g$ constitutes with the left-invariant vector field $E^{t}$ a left-invariant conelike radiant structure with antisymmetric Ricci tensor 
\begin{align}\label{gconeric}
\begin{split}
\lric(r, s) = \tfrac{2n + \ell(t)}{4B_{\g}(t, t)}B_{\g}(t, [r, s]),
\end{split}
\end{align}
and satisfying $\er = 0$.
\begin{enumerate}
\item If $\Phi \in \Aut(G)$ and $\phi = T\Phi(e) \in \eno(\g)$ is the corresponding Lie algebra automorphism, then the left-invariant conelike radiant connection $\nabla$ associated with $\phi^{-1}(t)$ is $\Phi^{\ast}(\nabla)$. In particular the stabilizer of $t$ in $\Aut(G)$ acts as automorphisms of $(\nabla, \rad)$.
\item Claims \eqref{gnullplanelike}, \eqref{ssnotflat}, \eqref{subalgtotgeodg}, and \eqref{abeliansubalgtotgeodg} of Theorem \ref{gnonnulltheorem} hold as stated. 
\item If $\g$ is solvable, then $\nabla$ is Ricci-flat.
\end{enumerate}
\end{theorem}

\begin{proof}
Taking $k = B_{\g} \in S^{2}\g^{\ast}$ (so that $h = B_{\g}(t, t)^{-1}B$) in Theorem \ref{gnonnulltheorem} yields \eqref{gconedefined} and \eqref{gconeric}.
Let $\Phi \in \Aut(G)$ be an automorphism of $G$ and let $\phi = T\Phi(e) \in \eno(\g)$ be the corresponding Lie algebra automorphism. Let $\bnabla$ be the connection associated with $\bar{t} = \phi^{-1}(t)$ by $\bar{\Pi} \in S^{2}\g^{\ast}\tensor \g$ as in \eqref{gconedefined}. Because $\phi$ is an automorphism, $\bar{\Pi}(r, s) = \phi^{-1}(\Pi(\phi(r), \phi(s)))$, which is the tensor determined by the left-invariant connection $\Phi^{\ast}(\nabla)$. 

By the Cartan criterion, over a field of characteristic zero a Lie algebra $\g$ is solvable if and only if $B_{\g}([\g, \g], \g) = \{0\}$, so when $G$ is solvable it follows from \eqref{gconeric} that $\nabla$ has vanishing Ricci tensor.
\end{proof}

\begin{remark}
Theorem \ref{killingnonnulltheorem} is essentially the special case of Theorem \ref{gnonnulltheorem} where $k = B_{\g}$, with the difference that claim \eqref{coneautog} of Theorem \ref{gnonnulltheorem} is improved by replacing the group of inner automorphisms $\Ad(G)$ by the full automorphism group $\Aut(G)$. This is possible because any automorphism of $G$ preserves the Killing form, while an arbitrary invariant bilinear form need only be preserved by inner automorphisms. In the context of nilpotent Lie groups, this improvement is decidedly nontrivial because, by \cite[Theorem $4$]{Jacobson-automorphisms}, a nilpotent Lie algebra admits a derivation which is not inner, so the group of outer automorphisms of a connected nilpotent Lie group has positive dimension. See \cite{Dani} for a survey of automorphism groups of Lie groups.
 
Although claims \eqref{subalgtotgeodg} and \eqref{abeliansubalgtotgeodg} of Theorem \ref{gnonnulltheorem} continue to hold without change, with respect to claim \eqref{subalgtotgeodg} note that the hypothesis $B_{\g}(t, t) \neq 0$ need not imply that $B_{\h}(t, t) \neq 0$, and even if this be the case, in general $\bnabla$ need not be the left-invariant conelike radiant connection associated with $(H, t)$ by \eqref{gconedefined}.
\end{remark}

\begin{remark}
The connection $\nabla$ associated with $t \in \g$ by Theorem \ref{killingnonnulltheorem} depends only on the adjoint orbit of $t$, in the sense that the connection $\bnabla$ associated with $\bar{t} = \Ad(g)t$ equals the pullback of $\nabla$ via $\Ad(g^{-1})$, for if $t$ is used in place of $t$ in \eqref{gconedefined}, the resulting tensor $\bar{\Pi}$ satisfies $\bar{\Pi}(a, b) = \Ad(g)\Pi(\Ad(g^{-1})a, \Ad(g^{-1})b)$. That is the left-invariant conelike radiant structures on $G$ associated by Theorem \ref{killingnonnulltheorem} with elements of a Killing nondegenerate adjoint orbit are isomorphic via inner automorphisms of $G$.
\end{remark}

\begin{example}
The formula \eqref{gconedefined} is simpler than it appears. For $r, s \in \ker B_{\g}(t, \dum)$ it simplifies to
\begin{align}\label{gconedefinedker}
\begin{split}
\Pi(t, t)& = t, \qquad
\Pi(r, t)  = \Pi(t, r) = r + \tfrac{1}{2}[t, r],\\
\Pi(r, s) &= \tfrac{1}{4(n-2)}\left(B_{\g}(r, s)  - \tfrac{2}{B_{\g}(t,t)}B_{\g}([t, r], [t, s]) \right)t = \tfrac{1}{4(n-2)}B_{\g}\left(r, \left(\Id_{\g} + \tfrac{2}{B_{\g}(t, t)}\ad_{\g}(t)^{2}\right)s\right).
\end{split}
\end{align}
The expressions \eqref{gconedefinedker} give rise to the following example. Let $G$ be an $n$-dimensional Lie group with Lie algebra $\g$ that is $|k|$-graded, meaning there is a direct sum of linear subspaces, $\g = \oplus_{i = -k}^{k}\g_{k}$, satisfying $[\g_{i}, \g_{j}] \subset \g_{i+j}$, where $\g_{p} = \{0\}$ if $|p| > k$ and $\g_{k} \neq \{0\}$, $\g_{-k}\neq\{0\}$. If $\g$ is semisimple as a Lie algebra there exists a unique \emph{grading element} $t \in \g$ such that $\ad_{\g}(t)(r) = ir$ for $r \in \g_{i}$ and $t$ is contained in the center of $\g_{0}$ \cite[Proposition $3.1.2$]{Cap-Slovak-book}. If $r \in \g_{i}$ and $s \in \g_{j}$, then $B_{\g}(r, s) = 0$ if $i+j \neq 0$, while the restriction of $B_{\g}$ to $\g_{i}$ yields a linear isomorphism of $\g_{i}$ with $\g_{-i}^{\ast}$ is nondegenerate. There holds $B_{\g}(t, t) = \sum_{i}i^{2}\dim \g_{i} > 0$. Let $\nabla$ be the left-invariant connection on $G$ associated with $t$ by Theorem \ref{gnonnulltheorem}. Specializing \eqref{gconedefinedker} yields, for $r \in \g_{i}$ and $s \in \g_{j}$,
\begin{align}
\begin{aligned}
&\Pi(t, t)  = t,&& A(t, t) = t,\\
&\Pi(t, r)  = (1 + i/2)r, && A(t, r) = (1 + i)r,&\\
&& &A(r, t) = r\\
&\Pi(r, s)  = \begin{cases}
0 & i + j \neq 0,\\
\tfrac{\left(B_{\g}(t, t) +  2i^{2}\right)B_{\g}(r, s)}{4(n-2)B_{\g}(t, t)} t & i + j = 0,
\end{cases}&&A(r, s)  = \begin{cases}
\tfrac{1}{2}[r, s] & i + j \neq 0,\\
\tfrac{1}{2}[r, s] + \tfrac{\left(B_{\g}(t, t) +  2i^{2}\right)B_{\g}(r, s)}{4(n-2)B_{\g}(t, t)}t& i + j = 0.
\end{cases}
\end{aligned}
\end{align}
\end{example}

Over a field $\fie$ of characteristic not equal to $2$, let $(\ste, \om)$ be a $2n$-dimensional $\fie$-vector space equipped with a structure of a symplectic vector space. The Lie algebra $\heisen_{2n+1}(\fie)$ is $\ste \oplus \fie$ equipped with the bracket $[(x, r), (y, s)] = (0, \om(x, y))$. By Lemma \ref{heisenberglemma} any $2$-step nilpotent Lie algebra with one-dimensional center is isomorphic to the Heisenberg Lie algebra.

\begin{lemma}\label{heisenberglemma}
Over a field of characteristic not equal to $2$, there exists a $2$-step nilpotent Lie algebra $\n$ with one-dimensional center $\ctr(\n)$ if and only if $\dim \n$ is odd, and such a Lie algebra is unique up to isomorphism. 
\end{lemma}
\begin{proof}
Choose a linear subspace $\m \subset \n$ such that $\n = \m \oplus \ctr(\n)$. Let $z$ span $\ctr(\n)$ and define an antisymmetric two-form $\omega$ on $\m$ by $[a, b] = \omega(a, b)z$. If $x \in \n$ and $x \notin \ctr(\n)$, then there is $y \in \n$ such that $[x, y] \neq 0$, and evidently $y \notin \ctr(\n)$. this shows that the restriction to $\m$ of $\omega$ is nondegenerate. Since $(\m, \omega)$ is a symplectic vector space, it has even dimension. On $\m \oplus \ctr(\n)$ the bracket $[\,,\,]$ has the form$[(a, rt), (b, st)] = (0, \omega(a, b)t)$. Since any two symplectic vector spaces are linearly isomorphic, the claim follows. 
\end{proof}

The $(2n+1)$-dimensional Heisenberg group $\Heisen_{2n+1}(\fie)$ is the vector space $\ste \oplus \fie$ equipped with the product $(x, r)\cdot (y, s) = (x + y, r + s + \tfrac{1}{2}\om(x, y))$. 
Example \ref{heisenbergexample} shows that are many different structures of a left-invariant conelike radiant structure with antisymmetric Ricci tensor that can be constructed on the Heisenberg group. The simplest such is obtained from Theorem \ref{gnonnulltheorem} with $g$ being the tensor square of the contact one-form on the Heisenberg group.

\begin{example}\label{heisenbergexample}
Let $(\ste, \om)$ be a $2n$-dimensional symplectic real vector space.  Let $\{e_{1}, \dots, e_{2n}\}$ be a basis of $\ste$ and write $E_{i} = (e_{i}, 0) \in \heisen_{2n+1}$. The center of $\heisen_{2n+1}$ is the one-dimensional subspace spanned by $Z = (0, 1)$ and it follows that subgroup generated by exponentiating $z$ is normal and so $\Heisen_{2n+1}$ fibers over its quotient by this subgroup, which is $\ste$ viewed as an additive abelian group. The action of $\reat$ generating the fibers is free, so the fibration $\rho:\Heisen_{2n+1} \to \ste$ given by $\rho(x, r) = x$ is principal. Let $\theta$ be the left invariant one-form such that $\theta(Z) = 1$ and $\{E_{1}, \dots, E_{2n}\}$ span $\ker \theta$. Since $Z$ is central, $R^{\ast}_{\exp(tZ)}d\theta = d\theta$, so $\theta$ is a principal connection on $\rho:\Heisen_{2n+1} \to \ste$, and since $d\theta = -\rho^{\ast}(\om)$, $\theta$ is a contact one-form and is not flat as a principal connection.

The most general left-invariant torsion-free connection, $\nabla$, on $\Heisen_{2n+1}$ such that $Z$ is a radiant vector field satisfies
\begin{align}
&\nabla_{Z}Z = Z, && \nabla_{E_{i}}Z = E_{i},&&\nabla_{Z}E_{i} = E_{i},&
&\nabla_{E_{i}}E_{j} = A_{ij}\,^{k}E_{k} + B_{ij}Z + \tfrac{1}{2}\om_{ij}Z,
\end{align}
for $A_{ij}\,^{k} \in S^{2}\std \tensor \ste$ and $B_{ij} \in S^{2}\std$. The curvature of $\nabla$ satisfies
\begin{align}
\begin{split}
R(\dum, \dum)Z &= 0, \quad R(Z, \dum)\dum = 0, \\
R(E_{i}, E_{j})E_{k} &= (2A_{p[i}\,^{l}A_{j]k}\,^{p} + 2\delta_{[i}\,^{l}B_{j]k})E_{l} + \tfrac{1}{2}\om_{jk}E_{i} + \tfrac{1}{2}\om_{ki}E_{j} - \om_{ij}E_{k}.
\end{split}
\end{align}
Since $R(Z, \dum)\dum = 0$, by Lemma \ref{coneconditionlemma}, $(\nabla, Z)$ is conelike. Its Ricci curvature is
\begin{align}
&\ric(Z, \dum)= 0 =  \ric(\dum, Z), &&
\ric(E_{i}, E_{j}) = A_{ij}\,^{p}A_{qp}\,^{q} - A_{qj}\,^{p}A_{ip}\,^{q} + (2n-1)B_{ij},
\end{align}
so that the Ricci tensor of the conelike radiant structure $(\nabla, Z)$ is antisymmetric if and only if $B_{ij} = \tfrac{1}{2n-1}\left(A_{ip}\,^{q}A_{jq}\,^{p} - A_{ij}\,^{p}A_{qp}\,^{q} \right)$, in which case the Ricci tensor is identically zero. 

Let $\bnabla$ be the connection on $\ste$ such that $\bnabla_{e_{i}}e_{j} = A_{ij}\,^{k}e_{k}$. Then the Ricci tensor $\bar{\ric}$ of $\bnabla$ satisfies $\bar{\ric}(e_{i}, e_{j}) = A_{ij}\,^{p}A_{qp}\,^{q} - A_{ip}\,^{q}A_{jq}\,^{p}$, so that the condition that $\ric$ be antisymmetric is equivalent to $B_{ij} = \tfrac{1}{1-2n}\bar{\ric}_{ij}$. From Lemma \ref{preextendedthomaslemma} it follows that $\nabla$ is the cone connection determined by the extended projective structure $[\bnabla, \theta]$ on $\ste$.  

The connection $\nabla$ on $G$ obtained in the case where $A_{ij}\,^{k} = 0$, so that $\bnabla_{e_{i}}e_{j} = 0$ for all $i$ and $j$, is the special case of Theorem \ref{gnonnulltheorem} where the invariant symmetric tensor $g$ is taken to be $g = \theta \tensor \theta$, so that $\nabla$ is the left-invariant connection associated with the tensor $\Pi + \tfrac{1}{2}[\dum, \dum]$ where $\Pi = \theta \tensor \delta + \delta \tensor \theta - \theta \tensor \theta \tensor Z$.

Let $G = \Heisen_{2n+1}$. Let $\Ga \subset \ste$ be a lattice, let $H = \Ga \oplus \rea \subset G$, and let $K = \Ga \oplus \integer \subset G$. Then $G/H \simeq \ste/\Ga = \torus_{2n}$ is a $2n$-dimensional torus, $H/K \simeq S^{1}$, and the fibration $\rho$ descends to give a fiber bundle $S^{1} = H/K \to G/K \to \torus_{2n} = G/H$. The radiant structure $(\nabla, \theta)$ descends to $G/K$ and is the cone connection of the extended projective structure determined by $\theta$ and the descent of $\bnabla$ to $\torus_{2n}$.
\end{example}

\begin{example}\label{leftinvarianterexample}
This example constructs a left-invariant radiant structure on a nonunimodular solvable Lie group for which $\er$ is neither vanishing nor closed.
The simplest example of a Lie algebra admitting an exact symplectic form is the Lie algebra $\g= \aff(1, \com)$ of affine transformations of the complex affine line. To make computations in the $\g$ there is fixed a basis $\{e_{i}\}$. The dual basis in $\g^{\ast}$ is written $\{e^{i}\}$. Write $[e_{i}, e_{j}] = c_{ij}\,^{k}e_{k}$, so that $de^{k} = -\tfrac{1}{2}c_{ij}\,^{k}e^{i}\wedge e^{j} = -c_{ij}\,^{k}e^{i}\tensor e^{j}$. Here $d$ can be understood either as the Lie algebra cohomology codifferential, or as the exterior differential of the left-invariant one-form on $G$ corresponding to $e^{i}$ (see \cite{Nomizu-nilpotent}). An element $x \in \g$ is written as $\sum_{i}x_{i}e_{i}$.
With respect to the basis $e_{1} = (1, 0)$, $e_{2} = (\j, 0)$, $e_{3} = (0, 1)$, and $e_{4} = (0, \j)$ of $\g= \aff(1, \com)$, the Lie bracket is
\begin{align}\label{aff1rel}
[x, y] & = (x_{1}y_{3} - x_{3}y_{1} - x_{2}y_{4} + x_{4}y_{2})e_{3} + (x_{1}y_{4} - x_{4}y_{1} + x_{2}y_{3} - x_{3}y_{2})e_{4}.
\end{align}
The commutator is $[\g, \g] = \spn\{e_{3}, e_{4}\}$, which is abelian, but stable under the adjoint action, showing that $\g$ is solvable but not nilpotent. Since $\tr \ad_{\g}(\dum) = 2e^{1}$, $\g$ is not unimodular.  
Calculation using \eqref{aff1rel} shows $B(x, y) = 2(x_{1}y_{1} - x_{2}y_{2})$. Using \eqref{gconedefinedker}, it can be checked that the connection associated by Theorem \ref{killingnonnulltheorem} with $t = -e_{1}$ is given by 
\begin{align}
\begin{split}
A(x, y) & = (-x_{1}y_{1} + \tfrac{1}{4}x_{2}y_{2})e_{1}  -(x_{1}y_{2} + x_{2}y_{1})e_{2}\\
&\quad  - (x_{3}y_{1}  + \tfrac{1}{2}(x_{2}y_{4} - x_{4}y_{2}))e_{3} - (x_{4}y_{1} -\tfrac{1}{2}(x_{2}y_{3} - x_{3}y_{2}))e_{4},
\end{split}
\end{align}
and constitutes with $E^{t} = E^{-e_{1}}$ a left-invariant radiant structure on a group $G$ with Lie algebra $\g$. Because $de^{1} = 0$, by \eqref{gconeric} its Ricci curvature vanishes. Consider the connection $\bnabla$ corresponding with the tensor 
\begin{align}
\bar{A} = A + e^{4}\tensor \delta + \delta \tensor e^{4} - (e^{4}\tensor e^{1} + e^{1}\tensor e^{4})\tensor e_{1}. 
\end{align}
Because $\bar{A}(\dum, t) = A(\dum, t)$, $(\nabla, E^{t})$ is again radiant. The skew-symmetric part of its Ricci tensor can be calculated using \eqref{lricij}. Because $\bar{A}_{ip}\,^{p} = 4e^{4} - 2e^{1}$, there results that the skew-symmetric part of $\bar{\lric}$ equals the symplectic form $-2de^{4} = 2(e^{1}\wedge e^{4} + e^{2}\wedge e^{3})$. Consequently, $\er = -de^{4}(t, \dum) = -e^{4}$ is nonvanishing and is not closed. 

Because a Lie group that admits a finite-volume quotient by a discrete subgroup is unimodular \cite[Lemma $6.2$]{Milnor-leftinvariant}, this example cannot yield a compact radiant manifold with $\er$ nonvanishing. Moreover, the construction uses the nonunimodularity in an essential way; a multiple of the radiant field is Killing dual to the trace of the adjoint representation.
\end{example}

\section{Left-invariant conelike radiant structures on three-dimensional unimodular Lie groups}\label{3dsection}
Among the simplest interesting examples of conelike radiant structures are the left-invariant conelike radiant structures on three-dimensional unimodular Lie groups. 
The main result of this section is that the isomorphism classes of left-invariant conelike radiant structures on a three-dimensional unimodular nonnilpotent Lie group with Lie algebra $\g$ are in bijection with the $\Aut(\g)$ orbits of semisimple elements of $\g$.

Three-dimensional unimodular Lie groups were classified by L. Bianchi \cite{Bianchi}. Modern references include \cite{Milnor-leftinvariant} and \cite[Section $2$]{Meeks-Perez}. 

The claims about $3$-dimensional unimodular Lie algebras stated in Lemmas \ref{3dsslemma} and \ref{3dunilemma} are closely related to and could be deduced from the results of \cite{Jacobson-threedimensional, Malcolmson}, which describe simple three-dimensional Lie algebras over arbitrary fields. 

\begin{lemma}\label{3dsslemma}
Let $\g$ be a $3$-dimensional unimodular Lie algebra over a field $\fie$ of characteristic not dividing $6$ and having Killing form $B_{\g} \in S^{2}\g^{\ast}$. For $0 \neq t \in \g$ there holds 
\begin{align}\label{unich}
\ad_{\g}(t)(\ad_{\g}(t)^{2} - \tfrac{1}{2}B_{\g}(t, t)\Id_{\g}) = 0,
\end{align}
$\ad_{\g}(t)$ is semisimple or nilpotent as $B_{\g}(t, t) \neq 0$ or $B_{\g}(t, t) = 0$, and $\ad_{\g}(t)$ is diagonalizable over $\fie$ if and only if $B_{\g}(t, t)/2$ is a square in $\fie$. 
\end{lemma}

\begin{proof}
Let $0 \neq t \in \g$ and write $T = \ad_{\g}(t)$. Because $\chr \fie \notin \{2, 3\}$, there holds $6 \det A = (\tr A)^{3} - 3(\tr A)(\tr A^{2}) + 2\tr A^{3}$ for any $A \in \eno(\g)$ . Since $T(t) =0$ and $\g$ is unimodular, with $T$ in place of $A$ this yields $0 = 6\det T = 2\tr T^{3}$, so $\tr T^{3} = 0$. If $B_{\g}(t, t) =0$, then $\tr T^{k} = 0$ for $1 \leq k \leq 3$, and, because $\chr \fie \notin \{2, 3\}$, this implies $T$ is nilpotent, so $T^{3} = 0$ and there holds \eqref{unich}. Suppose $B_{\g}(t, t) \neq 0$. Since $T(t) = 0$, by the Cayley-Hamilton theorem there are $\al, \be \in \fie$ such that $T^{3} + \al T^{2} + \be T = 0$. Tracing this yields $0 = \tr T^{3} + \al B_{\g}(t, t) + \be \tr T  = \al B_{\g}(t, t)$, so $\al = 0$, since $B_{\g}(t, t, ) \neq 0$. Hence $T(T^{2} + \be \Id) = 0$. If $\be = 0$, then $T^{3} = 0$, so $T$ is nilpotent, $B_{\g}(t, t) =0$, and there holds \eqref{unich}. If $\be \neq 0$, then the characteristic polynomial of $T$ is $x(x^{2} + \be)$ so $T$ is semisimple and there is a $2$-dimensional subspace of $\g$ on which $T^{2} + \be \Id$ restricts to $0$. It follows that $B_{\g}(t, t) = \tr T^{2} = -2\be$, and this shows \eqref{unich} holds for all $t \in \g$. If $B_{\g}(t, t)/2$ is a square in $\fie$, then the characteristic polynomial of $T$ has three distinct factors over $\fie$, so $T$ is diagonalizable. 
\end{proof}

\begin{lemma}\label{3dunilemma}
If $\g$ is a $3$-dimensional unimodular real Lie algebra that is not abelian, then there is a basis $\{t, a, b\}$ of $\g$ satisfying
\begin{align}\label{unibasis}
&[t, a] = 2b,& &[t, b] = \tfrac{1}{4}B_{\g}(t, t)a, &&[a, b] = \ep t,
\end{align}
where $\ep = 0$ if $B_{g}(t, t) = 0$,  $\ep \in \{0, 1, -1\}$ if $B_{\g}(t, t) < 0$, and $\ep \in \{0, -1\}$ if $B_{\g}(t, t) > 0$. There hold $B_{\g}(t, a) = 0$, $B_{\g}(t, b) =0$, $B_{\g}(a, b) = 0$, $B_{\g}(a, a) = -4\ep$, and $B_{\g}(b, b) = \tfrac{1}{2}\ep B_{\g}(t, t)$.
\begin{enumerate}
\item If $\ep =0$ and $B_{\g}(t, t) = 0$, then $\g$ is isomorphic to $\heisen_{3}(\rea)$.
\item If $\ep = 1$ and $B_{\g}(t, t) < 0$, then $\g$ is simple and compact.
\item If $\ep = -1$ and $B_{\g}(t, t) \neq 0$, then $\g$ is simple and noncompact.
\item If $\ep = 0$ and $B_{\g}(t, t) \neq 0$, then $\g$ is solvable but not nilpotent.
\end{enumerate}
\end{lemma}

\begin{proof}
If $\g$ is nilpotent, and its center has dimension at least $2$, then it is abelian. Otherwise, by Lemma \ref{heisenberglemma} it is isomorphic to $\heisen_{3}(\rea)$. In this case the proof of Lemma \ref{heisenberglemma} shows that there is a basis of $\g$ having the form \eqref{unibasis} where $b$ spans $\ctr(\g)$. By Lemma \ref{3dsslemma}, if an element $x$ of a $3$-dimensional unimodular Lie algebra $\g$ is Killing isotropic, then $\ad_{\g}(x)$ is nilpotent, so if every element of $\g$ is Killing isotropic, then every element of $\g$ is $\ad_{\g}$-nilpotent, so $\g$ is nilpotent. Hence if $\g$ is not nilpotent, there is $t \in \g$ such that $B_{\g}(t, t) \neq 0$.
By \eqref{unich} of Lemma \ref{3dsslemma}, if $B_{\g}(t, t) > 0$, then $\ad_{\g}(t)$ is diagonalizable over $\rea$ and there are nonzero $u_{\pm} \in \g$ such that $Tu_{\pm} = \pm \al u_{\pm}$ where $\al = \sqrt{(B_{\g}(t, t)/2)}$, and $a = \al^{-1}(u_{+}+ u_{-})$ and $b = \tfrac{1}{2}(u_{+} - u_{-})$ constitute with $t$ a basis of $\g$ and satisfy $[t, a] = 2b$ and $[t, b] = \tfrac{1}{4}B_{\g}(t, t)b$.
If $B_{\g}(t, t) < 0$, then $\ad_{\g}(t)$ is diagonalizable over $\com$ but not over $\rea$, and there are nonzero $u,v \in \g$ such that $Tu = -\al v$ and $Tv = \al u$ where $\al = \sqrt{-(B_{\g}(t, t)/2)}$, and $a = \al^{-1}(u + v)$ and $b = \tfrac{1}{2}(u - v)$ constitute with $t$ a basis of $\g$ and satisfy $[t, a] = 2b$ and $[t, b] = \tfrac{1}{4}B_{\g}(t, t)b$.
In either case, $[a, b] = pa + qb + rt$ for $p, q, r \in \rea$. Because $\g$ is unimodular, $0 = \tr \ad_{\g}(a) = q$ and $0 = \tr \ad_{\g}(b) = -p$, so $[a, b] = rt$. Replacing $a$ and $b$ by $\bar{a} = ca$ and $\bar{b} = cb$ with $c = \sqrt{|r|}$ yields a basis $\{t, \bar{a}, \bar{b}\}$ of $\g$ satisfying $[t, \bar{a}] = 2\bar{b}$, $[t, \bar{b}] = \tfrac{1}{4}B_{\g}(t, t)\bar{a}$, and $[\bar{a}, \bar{b}] = \ep t$ with $\ep \in \{0, 1, -1\}$. If $\ep = 1$ and $B_{\g}(t, t) > 0$ then taking $\hat{a} = (2/\al)\bar{b}$ and $\hat{b} =(\al/2)a$ the basis $\{t, \hat{a}, \hat{b}\}$ satisfies $[t, \hat{a}] = 2\hat{b}$, $[t, \hat{b}] = \tfrac{1}{4}B_{\g}(t, t)\hat{a}$, and $[\hat{a}, \hat{b}] = -t$, so if $B_{\g}(t, t) > 0$ it can be assumed $\ep \in \{0, -1\}$.
This proves there is a basis as in \eqref{unibasis}.

That the Killing pairings of the elements of the basis as in \eqref{unibasis} are as indicated follows from \eqref{unibasis} by direct computation. If $\ep = 0$ then $\spn\{a, b\} = [\g, \g]$ is a proper ideal but is not nilpotent, so $\g$ is solvable and not nilpotent. If $\ep \neq 0$, then $B$ is nondegenerate, so $\g$ is simple. If $\ep = -1$ then $B_{\g}(b, b)$ and $B_{\g}(t, t)$ have opposite signs, so $\g$ is not compact, whereas if $\ep = 1$, then $B_{\g}(b, b)$ and $B_{\g}(t, t)$ have the same sign while $B_{\g}(a, a) < 0$, so $\g$ is compact if and only if $B_{\g}(t, t) < 0$ too.
\end{proof}

\begin{remark}
In the nonnilpotent case of Lemma \ref{3dunilemma} there are possible four Lie algebras and five cases. The case $\ep = 1$ and $B_{\g}(t, t) < 0$ corresponds to a $2$-dimensional adjoint orbit in $\so(3)$, the cases $\ep = -1$ and $B_{\g}(t, t)$ of either sign corresponds to the two types of $2$-dimensional adjoint orbits in $\sll(2, \rea)$, and the cases $\ep = 0$ correspond to the $2$-dimensional adjoint orbits in the groups of motions of $2$-dimensional Euclidean and $2$-dimensional Minkowski space.
\end{remark}

\begin{lemma}\label{3derlemma}
Let $G$ be a $3$-dimensional Lie group with Lie algebra $\g$ and Killing form $B_{\g} \in S^{2}\g^{\ast}$. Suppose the radiant vector field $\eul = E^{t}$ of a left-invariant conelike radiant structure on $G$ is generated by $t \in \g$.
\begin{enumerate}
\item \label{eruni}
If $\g$ is unimodular and $B_{g}(t, t) \neq 0$, then $\er = 0$.
\item\label{ersimple} 
If $\g$ is simple and $\ad_{\g}(t)$ is semisimple, then $\er = 0$.
\end{enumerate}
\end{lemma}
\begin{proof}
If $\g$ is unimodular and $B_{g}(t, t) \neq 0$, then by Lemma \ref{3dunilemma} there is a basis $\{t, a, b\}$ of $\g$ satisfying \eqref{unibasis}. In particular $\g = \spn\{t\} + \im \ad_{\g}(t)$. By Lemma \ref{ertrlemma} this suffices to show that $\er =0$. This shows \eqref{eruni}.

Suppose $\g$ is simple and $\ad_{\g}(t)$ is semisimple.
Suppose there is $a \in \g$ such that $\er(a) \neq 0$. Consider the Fitting decomposition $\g = \left(\sum_{k \geq 1}\ker \ad_{\g}(t)^{k}\right)\oplus \left( \cap_{k \geq 1}\im \ad_{\g}(t)^{k}\right)$.  By Lemma \ref{ertrlemma}, $\cap_{k \geq 1}\im \ad_{\g}(t)^{k} \subset \ker \er$, so it may be supposed that $a \in \sum_{k \geq 1}\ker \ad_{\g}(t)^{k}$. Since $\ad_{\g}(t)$ is semisimple, $\ker \ad_{\g}(t)^{2} = \ker \ad_{\g}(t)$, so $a \in \ker \ad_{\g}(t)$. Because $\g$ is simple, $\ad_{\g}(t)$ is not the zero endomorphism. Since $\dim \g = 3$, $\dim \ker \ad_{\g}(t) = 3 - \rank \ad_{\g}(t)$ equals $1$ or $2$. Were $\dim \ker \ad_{\g}(t) = 2$, then there would be a basis $\{t, u, v\}$ of $\g$ such that $[t, u] = 0$ and $[t, v] \neq 0$. In this case $[t, v]$ and $[u, v]$ span $[\g, \g]$, but that $[\g, \g]$ be a proper ideal contradicts that $\g$ is simple. Hence $\dim \ker \ad_{\g}(t) = 1$, so $a$ is a multiple of $t$, which contradicts $\er(t) = 0$.
\end{proof}

\begin{remark}
The nonoverlap between claims \eqref{ersimple} and \eqref{eruni} of Lemma \ref{3derlemma} is slight. If $\g$ is simple and noncompact, so is isomorphic to $\sll(2, \rea)$, then there exists $t \in \g$ such that $B_{\g}(t, t) = 0$ and $\ad_{\g}(t)$ is semismple. In this case \eqref{ersimple} applies but \eqref{eruni} does not. On the other hand, if $\g$ is unimodular, and $B_{\g}(t,t) \neq 0$, Lemma \ref{3dsslemma} implies $\ad_{\g}(t)$ is semisimple, and so in this case the conclusion of \eqref{ersimple} follows from \eqref{eruni}.
\end{remark}

\begin{theorem}\label{3dssuniquetheorem}
Let $G$ be a $3$-dimensional unimodular Lie group with Lie algebra $\g$ and Killing form $B_{\g} \in S^{2}\g^{\ast}$. Given $t \in \g$ such that $B_{\g}(t, t) \neq 0$, there is a unique left-invariant affine connection on $G$ having antisymmetric Ricci tensor and constituting with the vector field $\rad = E^{t}$ a left-invariant conelike radiant structure, and it is the connection associated with $t$ as in \eqref{gconedefined} of Theorem \ref{killingnonnulltheorem}.
\end{theorem}

\begin{proof}
Let $t \in \g$ satsify $B_{\g}(t, t) \neq 0$ and let $\{t, a, b\}$ be the basis of $\g$ given by \eqref{unibasis} of Lemma \ref{3dunilemma}. Suppose $\nabla$ is a left-invariant affine connection that with $E^{t}$ constitutes a left-invariant conelike radiant structure. The associated tensor $\Pi$ satisfies $\Pi(t, t) = t$ and
\begin{align}
&\Pi(t, a) = \Pi(a, t) = a + \tfrac{1}{2}[t, a] = a + b,&
&\Pi(t, b)= \Pi(b, t) = b + \tfrac{1}{2}[t, b] = b + \tfrac{1}{8}B_{\g}(t, t)a.
\end{align}
By \eqref{unich}, $2\ad_{\g}(t)^{3} = B_{\g}(t, t)\ad_{\g}(t)$. Combining this with \eqref{gconedefined} or the equivalent expression \eqref{gconedefinedker} shows that if $\nabla$ is the connection of Theorem \ref{killingnonnulltheorem}, then
\begin{align}\label{3dssexplicit}
\begin{aligned}
\Pi(a, a) & = \tfrac{1}{4}B_{\g}\left(a, \left(\Id_{\g} + \tfrac{2}{B_{\g}(t, t)}\ad_{\g}(t)^{2}\right)a\right)t = \tfrac{1}{2}B_{\g}(a, a)t = - 2\ep t,\\
\Pi(b, a) & = \tfrac{1}{4}B_{\g}\left(b, \left(\Id_{\g} + \tfrac{2}{B_{\g}(t, t)}\ad_{\g}(t)^{2}\right)a\right)t =\tfrac{1}{2}B_{\g}(b, a)t = 0,\\
\Pi(b, b) & = \tfrac{1}{4}B_{\g}\left(b, \left(\Id_{\g} + \tfrac{2}{B_{\g}(t, t)}\ad_{\g}(t)^{2}\right)b\right)t =\tfrac{1}{2}B_{\g}(b, b)t = \tfrac{1}{4}B_{\g}(t, t)\ep t,
\end{aligned}
\end{align}
and the Ricci curvature of the associated left-invariant conelike radiant structure satisfies $\lric(a, b) = \tfrac{3}{2}\ep$, $\lric(a,a) =0$, $\lric(b, b) = 0$, and $\lric(t, \dum) = 0$.

Now suppose $\nabla$ is any left-invariant affine connection that with $E^{t}$ constitutes a left-invariant conelike radiant structure. It will be shown that the associated tensor $\Pi$ is as in \eqref{3dssexplicit}. By Lemma \ref{3derlemma}, $\er = 0$.  Write $T = \ad_{\g}(t) \in \eno(\g)$. By Lemma \ref{coneconditionlemma}, \eqref{tpirs}, and \eqref{qrs}, there is $Q \in S^{2}\g^{\ast}$ such that $Q(t, \dum) = 0$, and
\begin{align}
\label{qabt}&Q(u, v) = Q(Tu, v) + Q(u, Tv), & &T\Pi(u, v) = \Pi(Tu, v) + \Pi(u, Tv) + Q(u, v)t,
\end{align}
for $u, v \in \g$. Define $n \in \rea$ by $-8n = B_{\g}(t, t)$ so that $[t, b] = -2na$. Substituting $a$ and $b$ in \eqref{qabt} and using \eqref{unibasis} yields
\begin{align}
\label{qaaqbb}&Q(a, a) = 4Q(a, b),&&Q(b, b) = -4nQ(a, b),\\
\label{tpaatpbb}&T\Pi(a, a) = 4\Pi(a, b) + Q(a, a)t,&&T\Pi(b, b) = -4n\Pi(a, b) + Q(b, b)t,\\
\label{tpab}&T\Pi(a, b) = 2\Pi(b, b) - 2n\Pi(a, a) + Q(a, b)t,
\end{align}
Combining \eqref{qaaqbb} shows $Q(a, b) = 2Q(b, b) - 2nQ(a, a) = -16nQ(a, b)$. Hence, if $1 \neq -16n$ then $Q(a, b) =0$, so, by \eqref{qaaqbb}, $Q$ vanishes identically. 

The following consequence of \eqref{unibasis} is used several times in what follows: if $Tx \in \spn\{t\}$ then $x \in \spn\{t\}$ and $Tx = 0$. 
Applying $T$ to \eqref{tpab} and using $T\Pi(b, b) = -nT\Pi(a, a)$ and \eqref{tpaatpbb} yields
\begin{align}\label{ttpab}
T^{2}\Pi(a, b) & = 2T\Pi(b, b) - 2nT\Pi(a, a) = 4T\Pi(b, b) = -16n\Pi(a, b) + 4Q(b, b)t .
\end{align}
Applying $T$ to \eqref{ttpab} and using \eqref{unich} yields $-4nT\Pi(a, b) = T^{3}\Pi(a, b) = -16nT\Pi(a, b)$, so that $T\Pi(a, b) = 0$ and $\Pi(a, b) \in \spn\{t\}$. In \eqref{tpab} this implies $2n\Pi(a, a) - 2\Pi(b, b)=  Q(a, b)t$. Because $n\Pi(a, a) + \Pi(b, b) \in \spn\{t\}$, this yields $\Pi(a, a), \Pi(b,b) \in \spn\{t\}$, so there are $\al, \be \in \rea$ such that $\Pi(a, a) = \al t$ and $\Pi(b, b) = \be t$ and, hence, $Q(a, b) = 2n\al - 2\be$. With \eqref{tpaatpbb}, this yields 
\begin{align}\label{piab2}
& \Pi(a, b) = -\tfrac{1}{4}Q(a,a)t = \tfrac{1}{4n}Q(b, b)t = -Q(a, b)t = (2\be - 2n \al)t.
\end{align}
There follow $\lr(a, b)a = (-\tfrac{3}{2}\ep + 2\be - 2n\al)a - (\al + 2\ep)b$, and $\lr(a, b)b = (\be + 2n\ep)a + (- \tfrac{3}{2}\ep - 2\be + 2n\al)b$, from which it follows that
\begin{align}
&\lric(a, a) = \al + 2\ep, & & \lric(b, b) = \be + 2n\ep, & &\lric(a, b) = \tfrac{3}{2}\ep + 2\be - 2n\al,&& \lric(b, a) = -\tfrac{3}{2}\ep + 2\be - 2n\al.
\end{align}
For $\nabla$ to have antisymmetric Ricci tensor it must be $\al = -2\ep$, $\be = -2n\ep$, and $\be = n \al$. Hence $Q(a, b) = 2n\al - 2\be = 0$, and, by \eqref{qaaqbb}, this implies that $Q$ is identically zero.
This shows that $Q$ is identically zero in all cases. By \eqref{piab2}, $\Pi(a, b) = 0$, and by the preceding, $\Pi(b, b) = \be t = -2n\ep t$ and $\Pi(a, a) = \al t = -2\ep t$, so that $\Pi$ is determined completely and uniquely in terms of $n$ and $\ep$ alone. Comparing with \eqref{3dssexplicit} shows that the unique connection $\Pi$ found here is the connection associated with $t$ as in \eqref{gconedefined} of Theorem \ref{killingnonnulltheorem}.
\end{proof}

\begin{corollary}\label{3dsscorollary}
Let $G$ be a $3$-dimensional unimodular Lie group with Lie algebra $\g$ and Killing form $B_{\g}$. Let $t \in \g$ be such that $B_{\g}(t, t) \neq 0$ and let $\{t, a, b\}$ be a basis of $\g$ satisfying $[a, b] = \ep t$, $[t, a] = 2b$, $[t, b] = \tfrac{1}{4}B_{\g}(t, t)a$ where $\ep \in \{0, 1, -1\}$ if $B_{\g}(t, t) < 0$ and $\ep \in \{0, -1\}$ if $B_{\g}(t, t) > 0$. Let $A = E^{a}$, $B= E^{b}$, and $T  = E^{r} = \rad$ be the corresponding left-invariant frame on $G$ and let $\{\al, \be, \tau\}$ be the left-invariant coframe dual to $\{A, B, T\}$.

The unique left-invariant affine connection $\nabla$ on $G$ having antisymmetric Ricci tensor and constituting with the vector field $\rad = E^{t}$ a left-invariant conelike radiant structure given by Theorem \ref{3dssuniquetheorem} has the form
\begin{align}\label{3dssconnection}
\begin{aligned}
&\nabla_{A}T = A,& &\nabla_{B}T = B,& &\nabla_{T}T = T,\\
&\nabla_{A}A = -2\ep T,&&\nabla_{B}A = -\tfrac{1}{2}\ep T,&&\nabla_{T}A = A +2B,\\
&\nabla_{A}B = \tfrac{1}{2}\ep T,& &\nabla_{B}B = \tfrac{1}{4}\ep B_{\g}(t, t) T,& &\nabla_{T}B = B + \tfrac{1}{4} B_{\g}(t, t) A.
\end{aligned}
\end{align}
Equivalently:
\begin{align}\label{3dssconnectiondual}
\begin{split}
\nabla \al  -\tfrac{1}{2}d\al& =  -(\al \tensor \tau + \tau \tensor \al) - \tfrac{1}{4}B_{\g}(t, t)(\be \tensor \tau + \tau \tensor \be),\\
\nabla \be  - \tfrac{1}{2}d\be& =  -(\al \tensor \tau + \tau \tensor \al) - (\be \tensor \tau + \tau \tensor \be),\\
\nabla \tau  - \tfrac{1}{2}d\tau& =  -2\ep \al \tensor \al - \tfrac{1}{4}\ep B_{\g}(t, t)\be \tensor \be - \tau \tensor \tau.
\end{split}
\end{align}
The curvature and Ricci curvature of $\nabla$ are
\begin{align}\label{3dsscurvature}
&R(\dum, \dum) = \tfrac{3}{2}\ep \al \wedge \be \tensor (\al \tensor A + \be \tensor B), && \ric(\dum, \dum) = \tfrac{3}{2}\ep \al \wedge \be = -\tfrac{3}{2}d\tau.
\end{align} 
The connection $\nabla$ is not projectively flat.
\end{corollary}

\begin{proof}
That there is a basis of $\g$ of the form claimed follow from Lemma \ref{3dunilemma}. That the connection $\nabla$ has the form \eqref{3dssconnection} follows from the proof of Theorem \ref{3dssuniquetheorem}. The identities \eqref{3dssconnectiondual} follow from \eqref{3dssconnection} and $d\al = \tfrac{1}{4}B_{\g}(t, t)\be \wedge \tau$, $d\be = -2\tau \wedge \al$, and $d\tau = -\ep \al\wedge \be$.
The identities 
\begin{align}
&\lr(a, b)a = -\tfrac{3}{2}\ep , & &\lr(a, b)b =  - \tfrac{3}{2}\ep,
\end{align}
and $\lric(a, b) = \tfrac{3}{2}\ep$ express all nontrivial parts of the curvature tensor, and \eqref{3dsscurvature} follows.
The connection $\nabla$ is not projectively flat. Its projective Schouten tensor is the antisymmetric tensor $P = -\tfrac{1}{4}\ric = \tfrac{3}{8}\ep \al \wedge \be = -\tfrac{3}{8}d\tau$. The projective Weyl tensor $W(\dum, \dum)$ does not vanish, for $W(T, A)B = \tfrac{3}{8}T$.
\end{proof}

By Lemma \ref{3dsslemma} and Theorem \ref{3dssuniquetheorem}, with a semisimple element $t \in \g$ of the Lie algebra $\g$ of a three-dimensional unimodular Lie group $G$ there is associated a unique left-invariant conelike radiant connection having antisymmetric Ricci tensor, and the connections associated with elements of the same semisimple adjoint orbit in $\g$ are isomorphic. 
Corollary \ref{3dsscorollary} has the following slightly stronger consequence: if there is an automorphism of $\g$ mapping one semisimple adjoint orbit to another then the associated connections are isomorphic. This occurs for $\g = \sll(2, \rea)$, for which $\Aut(\sll(2, \rea)) = PGL(2, \rea)$ and there is a nontrivial outer automorphism interchanging the two adjoint orbits with the same negative value for $B_{g}(t, t)$.

To describe the resulting connections more explicitly, it is necessary to describe the three-dimensional unimodular Lie groups and their adjoint orbits in a more explicit manner. To do so, it is convenient to give a uniform description of the nonabelian unimodular Lie groups based on a slight generalization of real quaternion algebras. Let $\fiet$ be the group of invertible elements of a field $\fie$ with $\chr \fie \neq 2$. 
For $\al_{1}, \al_{2} \in \fie$ define $\qaf[\al_{1}, \al_{2}]\fie$ to be the $4$-dimensional algebra generated by $1$, $e_{1}$, $e_{2}$, and $e_{3}$ such that $1$ is a multiplicative identity element, and
\begin{align}\label{qadefined}
&e_{1}^{2} = \al_{1}1, &&e_{2}^{2} = \al_{2}1, &&e_{1}e_{2} = - e_{2}e_{1} = e_{3}.
\end{align}
The relations \eqref{qadefined} imply
\begin{align}\label{qaderived}
&e_{3}^{2} = -\al_{1}\al_{2},& &e_{3}e_{1} = -e_{1}e_{3} = -\al_{1}e_{2}, && e_{2}e_{3} = -e_{3}e_{2} = -\al_{2}e_{1}.
\end{align}
Although ugly, the notation $\qaf[\al_{1}, \al_{2}]\fie$ is standard.
When $\al_{1}, \al_{2} \in \fiet$, then $\qaf[\al_{1}, \al_{2}]\fie$ is a \emph{quaternion algebra}. For background on quaternion algebras see \cite[Chapter $2$]{Maclachlan-Reid} or \cite[Chapter $3$]{Lam-quadraticforms}. Here it is convenient to allow $\al_{1}, \al_{2}$ to take the value $0$. 
The algebra $\qaf[\al_{1}, \al_{2}]\fie$ is isomorphic to the Clifford algebra $\cliff(\fie^{2}, Q)$ of the quadratic form $Q(x_{1}, x_{2}) = \al_{1}x_{1}^{2} + \al_{2}x_{2}^{2}$ on $\fie^{2}$. This identification gives $\qaf[\al_{1}, \al_{2}]\fie$ the additional structure of a $\zmodtwo$-graded algebra, with $\qaf[\al_{1}, \al_{2}]{\fie}_{+} = \spn\{1, e_{3}\}$ and $\qaf[\al_{1}, \al_{2}]{\fie}_{-} = \spn\{e_{1}, e_{2}\}$. As a $\zmodtwo$-graded algebra, $\qaf[\al_{1}, \al_{2}]\fie \simeq \cliff(\fie, \al_{1}|x|^{2}) \stensor \cliff(\fie, \al_{2}|x|^{2})$ where $\stensor$ is the $\zmodtwo$-graded tensor product of $\fie$-superalgebras. This is the most convenient way to think about $\qaf[\al_{1}, \al_{2}]\fie$ in the degenerate cases where $\al_{1}\al_{2} = 0$. The \emph{dual numbers} (over $\fie$) are $\cliff(\fie, 0)$, the paracomplex numbers are $\cliff(\rea, x^{2})$, and the complex numbers are $\cliff(\rea, -x^{2})$. These are all $\zmodtwo$-graded algebras, and their $\zmodtwo$-graded tensor products with $\cliff(\rea, 0)$ yield the Clifford algebras of possibly degenerate quadratic forms on $\fie^{2}$. 

It follows straightforwardly from \eqref{qaderived} that the permutations of $\{1, 2, 3\}$ generate isomorphisms of $\qaf[\al_{1}, \al_{2}]\fie$ with the algebras $\qaf[\al_{2}, -\al_{1}\al_{2}]\fie$, $\qaf[-\al_{1}\al_{2}, \al_{1}]\fie$, and $\qaf[\al_{2}, \al_{1}]\fie$. The linear map sending $1$ to $1$ and rescaling $e_{i}$ by $t_{i} \in \fiet$, $i \in \{1, 2, 3\}$, is an isomorphism $\qaf[\al_{1}, \al_{2}]\fie \simeq \qaf[t_{1}^{2}\al_{1}, t_{2}^{2}\al_{2}]\fie$.

The case of most interest here is $\fie = \rea$. In this case the preceding remarks show that $\qaf[\al_{1}, \al_{2}]\rea$ is isomorphic to one of $\qaf[-1, -1]\rea$, $\qaf[1, 1]\rea$, $\qaf[1, 0]\rea$, $\qaf[-1,0]\rea$, and $\qaf[0, 0]\rea$.

For $q \in \qaf[\al_{1}, \al_{2}]\fie$ write $q = q^{0}1 + q^{1}e_{1} + q^{2}e_{2} + q^{3}e_{3}$. The matrices with respect to the basis $\{e_{0} = 1, e_{1}, e_{2}, e_{3}\}$ of the multiplication endomorphisms $L,R:\qaf[\al_{1}, \al_{2}]\fie \to \eno\qaf[\al_{1}, \al_{2}]\fie$ defined by $L(p)q = pq = R(q)p$ are
\begin{align}\label{qalr}
&L(q) = \begin{pmatrix}
q^{0} & \al_{1} q^{1} & \al_{2} q^{2} & -\al_{1}\al_{2}q^{3}\\
q^{1} & q^{0} & \al_{2}q^{3} & -\al_{2}q^{2}\\
q^{2} & -\al_{1}q^{3} & q^{0} & \al_{1}q^{1}\\
q^{3} & -q^{2} & q^{1} & q^{0}
\end{pmatrix},&&R(q) = 
\begin{pmatrix}
q^{0} & \al_{1} q^{1} & \al_{2} q^{2} & -\al_{1}\al_{2}q^{3}\\
q^{1} & q^{0} & -\al_{2}q^{3} & \al_{2}q^{2}\\
q^{2} & \al_{1}q^{3} & q^{0} & -\al_{1}q^{1}\\
q^{3} & q^{2} & -q^{1} & q^{0}
\end{pmatrix}.
\end{align}
Let $\re \qaf[\al_{1}, \al_{2}]\fie = \spn\{1\}$ and $\im \qaf[\al_{1}, \al_{2}]\fie = \spn\{e_{1}, e_{2}, e_{3}\}$. 
For $q \in \qaf[\al_{1}, \al_{2}]\fie$ define the \emph{conjugate} $\bar{q} \in \qaf[\al_{1}, \al_{2}]\fie$ to be the image of $q$ under the $\fie$-linear reflection of $\qa$ along $\re \qaf[\al_{1}, \al_{2}]\fie$ and through $\im \qaf[\al_{1}, \al_{2}]\fie$, so $\bar{q} = q^{0}1 - q^{1}e_{1} - q^{2}e_{2} - q^{3}e_{3}$. Alternatively, the conjugation is simply the canonical antiautomorphism of the Clifford algebra $\cliff(\fie^{2}, Q)$.
Define the \emph{(reduced) trace} $\tr q = q + \bar{q} = 2q^{0}$ and \emph{(reduced) norm}
$n(q) = q\bar{q} = (q^{0})^{2} - \al_{1}(q^{1})^{2} - \al_{2}(q^{2})^{2} + \al_{1}\al_{2}(q^{3})^{2}$.
On $\mat{2}{\fie}$ these recover the usual trace and determinant and the conjugate is the adjugate.
Note that $n$ is nondegenerate if and only if $\al_{1}, \al_{2} \in \fiet$. Using \eqref{qalr}, it can be checked that $\tr L(q) = 2\tr q = \tr R(q)$ and $\det L(q) = n(q)^{2} = \det R(q)$.

The algebra $\qaf[1, 1]\fie$ is the algebra $\mat{2}{\fie}$ of $2 \times 2$ matrices over $\fie$. A quaternion algebra isomorphic to $\mat{2}{\fie}$ is said to be \emph{split}.
Since every element of an algebraically closed field is a square, the isomorphism $\qaf[\al_{1}, \al_{2}]\fie \simeq \qaf[t_{1}^{2}\al_{1}, t_{2}^{2}\al_{2}]\fie$ implies that every quaternion algebra over an algebraically closed field $\fie$ is isomorphic to $\mat{2}{\fie}$.
In general there holds $\qaf[\al_{1}, \al_{2}]\fie \tensor_{\fie}{\bar{\fie}} = \qaf[\al_{1}, \al_{2}]{\bar{\fie}} \simeq \mat{2}{\bar{\fie}}$ where $\bar{\fie}$ is an algebraic closure of $\fie$, and since $\mat{2}{\bar{\fie}}$ is central simple, it follows that $\qaf[\al_{1}, \al_{2}]\fie$ is central simple when $\al_{1}, \al_{2} \in \fiet$. It follows from the Artin-Wedderburn theorem that an algebra over a field $\fie$ of characteristic not equal to $2$ is a $4$-dimensional central simple algebra if and only if it is a quaternion algebra; moreover, any such algebra is either split or is a division ring. 

One or both of these claims fail when $\al_{2} = 0$. The center $\ctr(\qaf[\al_{1}, \al_{2}]\fie)$ of $\qaf[\al_{1}, \al_{2}]\fie$ is $\ker (L - R)$, while a $2$-sided ideal of $\qaf[\al_{1}, \al_{2}]\fie$ is a subspace invariant with respect to $L(q)$ and $R(q)$ for all $q \in \qaf[\al_{1}, \al_{2}]\fie$. It is standard that the Jacobson radical $\jrad \cliff(\fie^{2}, Q)$ of $\cliff(\fie^{2}, Q)$ is the ideal generated by the radical of the bilinear form associated with $Q$ via polarization. Let $g \in S^{2}\qaf[\al_{1}, \al_{2}]\fie^{\ast}$ be the $\fie$-bilinear form obtained by polarizing $n$, so that $g(q, q) = n(q)$ and $2g(p, q) = p\bar{q} + q\bar{p} = \tr(p\bar{q})$. 
From \eqref{qalr} it follows that if $\al_{1} \in \fiet$, then the center $\ctr(\qaf[\al_{1}, 0]\fie)$ is still $\spn\{1\}$ but the Jacobson radical $\jrad \cliff(\fie^{2}, Q) = \spn\{e_{1}, e_{2}\} = \brad g$ is nontrivial. Similarly, $\ctr(\qaf[0, 0]\fie) = \spn\{1, e_{3}\}$ and $\jrad \cliff(\fie^{2}, Q) = \spn\{e_{1}, e_{2}, e_{3}\} = \brad g$ is nontrivial.

The first part of Lemma \ref{qalielemma} is an observation of P. Meyer \cite{Meyer-sl2}. Although it makes sense over general fields, here it is needed only over $\rea$.
\begin{lemma}\label{qalielemma}
For a $3$-dimensional unimodular nonabelian real Lie algebra $\g$, $\hat{\g} = \g \oplus \rea$ equipped with the multiplication $\mlt$ defined for $(x, r), (y, s) \in \hat{\g}$ by
\begin{align}
(x, r) \mlt (y, s) = \left(\tfrac{1}{2}[x, y] + ry + sx, rs + \tfrac{1}{8}B_{\g}(x, y)\right),
\end{align}
is isomorphic to $\qaf[\al_{1}, \al_{2}]\rea$ with $(\al_{1}, \al_{2}) \in \{(1, 1), (-1, -1), (1, 0), (-1, 0), (0, 0)\}$. The canonical antiautomorphism of $\qaf[\al_{1}, \al_{2}]\rea$ is given on $\hat{\g}$ by $(x, r) \to (-x, r)$ and its $-1$ eigenspace $\g\oplus \{0\} = \{(x, 0)\in \hat{\g}: x \in \g\}$ equipped with the product $[(x, 0), (y, 0)] = (x, 0)\mlt(y, 0) - (y, 0)\mlt (x, 0)$ is a Lie algebra isomorphic to $\g$.

More precisely, if $\g$ is nilpotent then $(\hat{\g}, \mlt)$ is isomorphic to $\qaf[0, 0]\rea$, while otherwise each $t \in \g$ such that $B_{\g}(t, t) \neq 0$ determines an isomorphism $\Phi_{t}:\hat{\g} \to \qaf[\al_{1}, \al_{2}]\rea$ where $\al_{1}$ is $0$ or $\sign B_{\g}(t, t)$ as $B_{\g}(t, t)$ is or is not $0$ and $\al_{2}$ is $-1$, $1$, or $0$ as the restriction of $B_{\g}(t, t)$ to $\im \ad_{\g}(t)$ is negative definite, split, or degenerate.

If $\Psi:\g \to \bar{\g}$ is an isomorphism of $3$-dimensional unimodular nonabelian real Lie algebras then $\hat{\Psi}:\hat{\g} \to \hat{\bar{g}}$ defined by $\hat{\Psi}(x, r) = (\Psi(x), r)$ is an isomorphism of the associated algebras, and every such isomorphism arises in this way. As a consequence a $3$-dimensional unimodular nonabelian real Lie algebra is determined up to isomorphism by the associated quadratic space $(\g, B_{\g})$. 
\end{lemma}

\begin{proof}
Let $\{t, a, b\}$ be a basis as in \eqref{unibasis} of Lemma \ref{3dunilemma} and let $\ka = \sqrt{|B_{\g}(t, t)|/8}$. The $\mlt$ products of the elements
$e_{0} = (0, 1)$, $e_{1} = (\ka t, 0)$, $e_{2} = (\sqrt{2}a, 0)$, and $e_{3} = (\sqrt{2}\ka b, 0)$ of $\hat{\g}$ satisfy the relations \eqref{qadefined} with $\al_{1} = \sign B_{\g}(t, t)$ and $\al_{2} = -\ep$. The remaining claims follow via straightforward computations. 
\end{proof}

Henceforth let $\fie = \rea$ and write $\qa = \qaf[\al_{1}, \al_{2}]\rea$. 
It follows from the definitions that $\overline{pq} = \bar{q}\bar{p}$ for $p, q \in \qa$, so $n(pq) = n(p)n(q)$. Hence the set of units $\qa^{\times}$ in $\alg$ is a group. From $\det L(q) = n(q)^{2}$ it follows that $\qa^{\times} = \{q \in \alg: n(q) \neq 0\}$. Factoring by the dilation action of the positive real numbers yields the three-dimensional Lie subgroup $\sphere(\qa) = \{q \in \alg: n(q) = 1\} \subset \qa^{\times}$ that is connected if $\al_{1}\al_{2} \neq 0$ or $\al_{1} < 0$ and has two connected components if $\al_{2} = 0$ and $\al_{1} \geq 0$. The image $\proj\sphere(\qa)$ of $\sphere(\qa)$ in the projectivization $\proj(\qa)$ is a connected Lie group.

For $\ka \in \rea$, let $\cn_{\ka}(t)= \sum_{j \geq 0}\tfrac{(-\ka)^{j}}{(2j)!}t^{2j}$ and $\sn_{\ka}(t)= \sum_{j \geq 0}\tfrac{(-\ka)^{j}}{(2j+1)!}t^{2j+1}$ be the solutions of $\ddot{x} + \ka x = 0$ with the initial conditions $x(0) = 1$, $\dot{x}(0) =0$, and $x(0) = 0$, $\dot{x}(0) = 1$, respectively.
Note that $\cn_{\ka}^{2} + \ka\sn_{\ka}^{2} = 1$. For $r \in \rea$, $\cn_{c^{2}\ka}(r) = \cn_{\ka}(cr)$ and $\sn_{c^{2}\ka}(r) = c^{-1}\sn_{\ka}(cr)$, from which it follows that $2\cn_{\ka}(r)\sn_{\ka}(r) = \sn_{\ka}(2r)$ and $\cn_{\ka}^{2}(r) - \ka \sn_{\ka}^{2}(r) = \cn_{\ka}(2r)$. 
For a three-dimensional real Lie algebra, $\g$, define a quadratic form $n$ on $\g$ by $-8n(u) = B_{\g}(u, u)$ for $u \in \g$.
From \eqref{unich} there follows
\begin{align}\label{3drodrigues}
\Ad(\exp_{\g}(ru)) = e^{r\ad_{\g}(u)} = \Id_{\g} + \tfrac{1}{2}\sn_{n(u)}(2r)\ad_{\g}(u) + \tfrac{1}{4}\tfrac{1 - \cn_{n(u)}(2r)}{n(u)}\ad_{\g}(u)^{2},
\end{align}
for $u \in \g$ and $r \in \rea$. Formula \eqref{3drodrigues} generalizes a version of the usual Euler-Rodrigues formula for $SO(3)$.

The Lie algebra of $\sphere(\qa) \subset \qa^{\times}$ is $\im \qa \subset \qa$ with the bracket $[r, s] = rs - sr$ and exponential map $\exp:\qa \to \qa^{\times}$ given by the usual exponential, $e^{r} = \sum_{k \geq 0}\tfrac{1}{k!}r^{k}$. 
For any $w \in \im \qa$ and $t \in \rea$, there holds the generalized Euler formula
\begin{align}\label{eulerformula}
e^{tw} = \cn_{n(w)}(t) + \sn_{n(w)}(t)w.
\end{align}
For $r \in \im \qa$, the vector field $E^{r}_{q} = \tfrac{d}{dt}\big|_{t = 0}R(\exp(tr))q = R(r)q$ is left-invariant with respect to the action of $\sphere(\qa)$. 
In  particular, $E_{i} = E^{e_{i}}$, $1 \leq i \leq 3$, satisfy
\begin{align}\label{3dbrackets}
&[E_{1}, E_{2}] = 2E_{3}, &&[E_{2}, E_{3}] = -2\al_{2}E_{1}, && [E_{3}, E_{1}] = -2\al_{1}E_{2}.
\end{align} 
The adjoint representation is the homomomorphism $\Ad:\qa^{\times} \to \Aut(\qa)$ defined by $\Ad(q)(x) = qxq^{-1}$ for $q \in \qa^{\times}$ and $x \in \qa$. Note that $\Ad(\qa^{\times})$ preserves the Lie bracket $[\dum, \dum]$, the norm $n$, and the vector space direct sum $\qa = \re\qa\oplus \im \qa$. The restriction of $\Ad$ to $\sphere(\qa)$ is the usual adjoint representation of $\sphere(\qa)$. 
By \eqref{3dbrackets}, for $w \in \im \qa$,
\begin{align}\label{adimqa}
\ad_{\im \qa}(w) = \begin{pmatrix}
0 & 2\al_{2}w^{3} &- 2\al_{2}w^{2}\\
-2\al_{1}w^{3} & 0 & 2\al_{1}w^{1}\\
-2w^{2} & 2w^{1} & 0
\end{pmatrix}.
\end{align}
It follows from \eqref{adimqa} or \eqref{unich} that
\begin{align}\label{prerodrigues}
\ad_{\im \qa}(w)^{3} = -4n(w)\ad_{\im \qa}(w).
\end{align}
The matrix of the restriction of $\Ad(q) = n(q)^{-1}L(q)R(\bar{q})$ to $\im \qa$ with respect to the basis $\{e_{1}, e_{2}, e_{3}\}$ can be computed from \eqref{qalr}, and, using \eqref{adimqa} and \eqref{prerodrigues}, there results
\begin{align}\label{rodriguesvariant}
\begin{split}
&\resizebox{.9\textwidth}{!}{$
\begin{pmatrix}
 (q^{0})^{2} - \al_{1}(q^{1})^{2} + \al_{2}(q^{2})^{2} - \al_{1}\al_{2}(q^{3})^{2} & 2\al_{2}(q^{0}q^{3} - q^{1}q^{2})& -2\al_{2}(q^{0}q^{2} - \al_{1}q^{1}q^{3})\\
-2\al_{1}(q^{0}q^{3} + q^{1}q^{2}) &  (q^{0})^{2} + \al_{1}(q^{1})^{2} - \al_{2}(q^{2})^{2} - \al_{1}\al_{2}(q^{3})^{2} & 2\al_{1}(q^{0}q^{2} + \al_{2}q^{2}q^{3})\\
 -2(q^{0}q^{2} + \al_{1}q^{1}q^{3}) & 2(q^{0}q^{1} - \al_{2}q^{2}q^{3})&  (q^{0})^{2} + \al_{1}(q^{1})^{2} + \al_{2}(q^{2})^{2} + \al_{1}\al_{2}(q^{3})^{2}\\
\end{pmatrix}$}\\
& = n(q)\Id_{\im \qa} + (\re q) \ad_{\im\qa}(\im q) + \tfrac{1}{2}\ad_{\im \qa}(\im q)^{2},
\end{split}
\end{align}
where $q = \re q + \im q$ with $\re q= q^{0}$ and $\im q = q^{1}e_{1} + q^{2}e_{2} + q^{3}e_{3} \in \im \qa$. 
From \eqref{prerodrigues}, \eqref{rodriguesvariant} (or \eqref{3drodrigues}) and \eqref{eulerformula} there follows
\begin{align}\label{rodrigues}
\Ad(\exp_{\im \qa}(tw)) = e^{t\ad_{\im \qa}(w)} = \Id_{\im \qa} + \tfrac{1}{2}\sn_{n(w)}(2t)\ad_{\im \qa}(w) + \tfrac{1}{4}\tfrac{1 - \cn_{n(w)}(2t)}{n(w)}\ad_{\im\qa}(w)^{2} = \Ad(e^{tw}),
\end{align}
for $w \in \im \qa$ and $t \in \rea$.
Formula \eqref{rodrigues} generalizes a version of the usual Euler-Rodrigues formula for $SO(3)$.

If $\Ad(q) = \Id_{\im \qa}$ for $q \in \sphere(\qa)$, then $qr = rq$ for all $r \in \im \qa$, and so $q \in \sphere(\qa) \cap \ctr(\qa)$. If $\al_{1} \neq 0$ or $\al_{2} \neq 0$, then $\sphere(\qa) \cap \ctr(\qa) = \{\pm 1\}$, and the map $\Ad:\sphere(\qa) \to \eno(\im \qa)$ is a twofold covering of its image, so descends to a group isomorphism from $\proj\sphere(\qa)$ onto its image. If $\al_{1} = 0 = \al_{2}$, then $\sphere(\qa) \cap \ctr(\qa) = \{\pm 1 + t e_{3}: t \in \rea\}$, and $\theta:\proj\sphere(\qa) \to \eno(\im \qa)$ is a line bundle over its image, which is the group of translations of $\spn\{e_{1}, e_{2}\}$. 

The differentials of covectors of the left-invariant coframe $\{\ep^{1}, \ep^{2}, \ep^{3}\}$ on $\sphere(\qa)$ dual to the left-invariant frame $\{E_{1}, E_{2}, E_{3}\}$ of \eqref{3dbrackets} satisfy
\begin{align}\label{3ddual}
& d\ep^{1} = 2\al_{2}\ep^{2}\wedge \ep^{3},& & d\ep^{2} = 2\al_{1}\ep^{3}\wedge \ep^{1},& &d\ep_{3} = -2\ep^{1}\wedge \ep^{2}.
\end{align}
From \eqref{3ddual} there follow
\begin{itemize}
\item The volume form $\ep^{1}\wedge \ep^{2} \wedge \ep^{3}$ is biinvariant, reflecting that $\sphere(\alg)$ is a unimodular Lie group.
\item There is a biinvariant one-form $\om$ on $\sphere(\qa)$ if and only if $\al_{1}\al_{2} = 0$, in which case $\om = \om_{2}\ep^{2} + \om_{3}\ep^{3}$ with $\al_{1}\om_{2} = 0$ and $\al_{2}\om_{3} = 0$.
\end{itemize}
The restriction to $\im \qa$ of the  $\fie$-bilinear form $g$ obtained by polarizing $n$ can be viewed as defining a left-invariant tensor on $\sphere(\qa)$, also denoted $g$ and satisfying $g(E_{q}^{r}, E_{q}^{s}) = g(r, s)$ for $r, s \in \im \qa$. Moreover, because $g([r, s], t) = g(r, [s, t])$ for $r, s, t \in \im \qa$, the restriction to $\im \qa$ of $g$ determines a biinvariant tensor on $\sphere(\alg)$.

\begin{lemma}\label{imqakillinglemma}
If $\al_{1}$ and $\al_{2}$ are not both zero, a biinvariant symmetric bilinear form $\be$ on $(\im \qa, [\dum, \dum])$ is unique up to homothety, equal to a multiple of the Killing form of $(\im \qa, [\dum, \dum])$, which is $-8n(q)$. Moreover, $\Ad(\sphere(\qa))$ preserves $\be$, so that $\Ad(\sphere(\qa)) \subset SO(\im \qa, g)$ if $\al_{1}\al_{2} \neq 0$. If $\al_{1} = 0 = \al_{2}$, any symmetric bilinear form on $\spn\{e_{1}, e_{2}\}$ is biinvariant.
\end{lemma}

\begin{proof}
It follows from \eqref{adimqa} that the Killing form of $\im \qa$ is $B_{\im\qa}(w, w) = \tr \ad_{\im \qa}(w)^{2} = -8n(w)$.

A symmetric bilinear form $\be \in S^{2}(\im \qa)^{\ast}$ determines a left-invariant tensor on $\sphere_{0}(\qa)$. This tensor is biinvariant if and only if $\be([x, y], z) = \be(x, [y, z])$ for all $x, y, z \in \im \qa$. In particular $\be([x, y], y) = \be(x, [y, y]) = 0$ for all $x, y \in \im \qa$. If $i \neq j$, then $\be([e_{i}, e_{j}], e_{j}) = 0$, which yields the equations
\begin{align}
\label{biim0}
&\al_{1}\be(e_{1}, e_{2}) = 0, && \al_{2}\be(e_{1}, e_{2}) = 0, & \be(e_{1}, e_{3}) = 0,& &\be(e_{2}, e_{3}) =0,
\end{align}
and $\be([e_{1}, e_{2}], e_{3}) = \be([e_{2}, e_{3}], e_{1}) = \be([e_{3}, e_{1}], e_{2})$, which, by \eqref{3dbrackets}, imply that $\be$ has the form
\begin{align}\label{biim1}
\be = b_{1}\ep^{1}\tensor \ep^{1} + b_{2}\ep^{2}\tensor \ep^{2} + b_{3}\ep^{3}\tensor \ep^{3} + c(\ep^{1}\tensor e^{2} + \ep^{2}\tensor \ep^{1})
\end{align}
where $b_{3} = - \al_{2}b_{1} = -\al_{1}b_{2}$, and $al_{1}c = 0 = \al_{2}c$. If $\al_{1}$ and $\al_{2}$ are not both $0$, these equations have a unique solution up to homothety, which must be multiple of $B_{\im \qa}$. 
Since $\Ad(\sphere(\qa))$ acts as automorphisms of the Lie algebra $(\im \qa, [\dum, \dum])$, the uniqueness implies that it preserves $\be$ up to multiplication by a nonzero constant. Since $\det \Ad(q) = \det L(q)\det R(\bar{q}) = n(q)^{2}n(\bar{q})^{2} = 1$ and $\dim \im \qa$ is odd, the constant must be $1$, so $\Ad(\sphere(\qa))$ preserves $\be$. In particular, if $\al_{1}\al_{2} \neq 0$, $\Ad(\sphere(\qa)) \subset SO(\im \qa, g)$. 

Finally, if $\al_{1} = 0 = \al_{2}$, then $b_{1}$, $b_{2}$, and $c$ can be chosen arbitrarily.
\end{proof}

Up to isomorphism there are five nonabelian three-dimensional unimodular real Lie algebras \cite{Bianchi, Milnor-leftinvariant}: $\so(3) \simeq \su(2) \simeq \sp(1)$,  $\so(1, 2) \simeq \sll(2, \rea) \simeq \sp(1, \rea)$, $\euc(2)$, $\euc(1, 1)$, and the Heisenberg Lie algebra $\heisen_{3}$. They are represented by \eqref{3dbrackets} for appropriate choices of parameters $(\al_{1}, \al_{2})$: $(-1, -1)$ for $\so(3) \simeq \su(2) \simeq \sp(1)$, $(1, 1)$ for $\so(1, 2) \simeq \sll(2, \rea) \simeq \sp(1, \rea)$, $(0, -1)$, for $\euc(2)$, $(0, 1)$ for $\euc(1, 1)$, and $(0, 0)$ for $\heisen_{3}$. Here $\Euc(2)$  and $\Euc(1, 1)$ denote the groups of motions of a $2$-dimensional Euclidean vector space and a $2$-dimensional Lorenztian vector space, and $\euc(2)$ and $\euc(1, 1)$ are their Lie algebras.

Although the Lie algebra $(\im \qa, [\dum, \dum])$ has been identified, there remains to identify the Lie group $\sphere(\qa)$. In the cases where $\al_{2} = 0$ and $\al_{1} \geq 0$, $\sphere(\qa)$ has two connected components and $\sphere_{0}(\qa)$ denotes the connected component of the identity. 

When $(\al_{1}, \al_{2}) =(-1, -1)$, $\sphere(\qa)$ comprises the unit norm quaternions, so is $\sphere^{3}$, while when $(\al_{1}, \al_{2}) = (1, 1)$, $\sphere(\qa)$ comprises the unit determinant elements of $\mat{2}{\rea}$, so is $SL(2, \rea)$. In either case, $\Ad$ is a nontrivial homomorphism from the connected Lie group $\sphere(\qa)$ to $SO(\im \qa, g)$ which is a twofold covering of its image. Since the image is connected it is a connected component of $SO(\im \qa, g)$. In the case $(\al_{1}, \al_{2}) = (1, 1)$, this means the image is the connected component of the identity, $SO_{0}(\im \qa, g)$. In summary, when $(\al_{1}, \al_{2}) = (-1, -1)$, $\Ad$ is the twofold cover $Sp(1) \to SO(3)$, while when $(\al_{1}, \al_{2}) = (1, 1)$, $\Ad$ is the twofold cover $SL(2, \rea) \to SO_{0}(1, 2)$. 

It is claimed that in the cases $(-1, 0)$ and $(1, 0)$, $\sphere(\alg)$ is isomorphic via $\Ad$ to the connected component of the identity of the special Euclidean group acting on the algebraically trivial ideal $\jrad \qa$ preserving the metric $(q^{2})^{2} - \al_{1}(q^{3})^{2}$. Suppose $\be \in S^{2}(\jrad\im\qa)^{\ast}$ is $(\im \qa, [\dum, \dum])$ invariant. The invariance implies $2\be(e_{3}, e_{3}) = \be([e_{1}, e_{2}], e_{3}) = - \be(e_{2}, [e_{1}, e_{3}]) = -2\al_{1}\be(e_{2}, e_{2})$ and $2\be(e_{2}, e_{3}) = \be(e_{2}, [e_{1}, e_{2}]) = \be(e_{1}, [e_{2}, e_{2}]) = 0$, so that $\be$ has the form $B((q^{2})^{2} - \al_{1}(q^{3})^{2})$ for some $B \in \rea$. Since $\Ad(\sphere(\qa))$ acts on $\im \qa$ by algebra automorphisms, it preserves $\be = (q^{2})^{2} - \al_{1}(q^{3})^{2}$ up to a constant factor. The constant factor is $1$. For $x \in \jrad\im \qa$ and $q \in \sphere(\qa)$ calculating using \eqref{prerodrigues} and \eqref{rodriguesvariant} yields $\be(\Ad(q)x, \Ad(q)x) = n(q)^{2}\be(x, x) = \be(x, x)$.
It is claimed that $\Ad(\sphere(\qa))$ is the connected component of the identity of the Euclidean group of motions of $(\jrad \qa, \be)$, that is $\SEuc(2)$ or $\SEuc_{0}(1, 1)$, as $\al_{1} = -1$ or $\al_{1} = 1$. 
Define $\phi:\rea \times \rea^{2} \to \sphere(\alg)$ by 
\begin{align}\label{sphereparam}
\begin{split}
\phi(&t,w_{1},w_{2}) = (1 - \tfrac{w_{2}}{2}e_{2} - \tfrac{w_{1}}{2\al_{1}}e_{3})e^{te_{1}} =  (1 - \tfrac{w_{2}}{2}e_{2} - \tfrac{w_{1}}{2\al_{1}}e_{3})(\cn_{-\al_{1}}(t) + \sn_{-\al_{1}}(t)e_{1})\\
& = \cn_{-\al_{1}}(t) + \sn_{-\al_{1}}(t) e_{1}+ \tfrac{1}{2}\left(w_{1}\sn_{-\al_{1}}(t)  - w_{2}\cn_{-\al_{1}}(t)\right)e_{2}  + \tfrac{1}{2}\left( w_{2}\sn_{-\al_{1}}(t) - \tfrac{w_{1}}{\al_{1}}\cn_{-\al_{1}}(t)\right)e_{3}.
\end{split}
\end{align}
(Some version of the parametrization \eqref{sphereparam} is well known in applied kinematics; see e.g. \cite[Chapter $9$]{Selig}.) It is apparent that $\phi$ maps $\im \qa$ onto $\sphere_{0}(\qa)$. If $\phi(t,w_{1},w_{2}) = \phi(s, x_{1}, x_{2})$, then $e^{(s-t)e_{1}} = (1 + \tfrac{x_{2} - w_{2}}{2}e_{2} + \tfrac{x_{1} - w_{1}}{2\al_{1}}e_{3})$, which yields $x_{1} = w_{1}$, $x_{2} = w_{2}$, and $e^{(s-t)e_{1}} = 0$, so, when $\al_{1} = 1$, $\phi$ is a diffeomorphism onto $\sphere_{0}(\qa)$, while when $\al_{1} = -1$, $\phi$ descends to a diffeomorphism $\rea/2\pi\integer \times \rea^{2} \to \sphere(\alg)$. A straightforward calculation combining \eqref{rodriguesvariant} and \eqref{sphereparam} shows that, for $q = \phi(t, w_{1}, w_{2})$, the matrix of $\Ad(\phi(t, w_{1}, w_{2}))$ with respect to the ordered basis $\{e_{1}, e_{2}, e_{3}\}$ of $\im \qa$ is
\begin{align}
\begin{split}
&\begin{pmatrix}
1 & 0 & 0\\
-2\al_{1}(q^{0}q^{3} + q^{1}q^{2}) &  (q^{0})^{2} + \al_{1}(q^{1})^{2} & 2\al_{1}q^{0}q^{2}\\
 -2(q^{0}q^{2} + \al_{1}q^{1}q^{3}) & 2q^{0}q^{1}&  (q^{0})^{2} + \al_{1}(q^{1})^{2}  +\\
\end{pmatrix} =
\begin{pmatrix} 
1 & 0 & 0 \\
w_{1} &  \cn_{-\al_{1}}(2t) & \al_{1}\sn_{-\al_{1}}(2t)\\
w_{2} & \sn_{-\al_{1}}(2t)&  \cn_{-\al_{1}}(2t) \\
\end{pmatrix},
\end{split}
\end{align}
so that the action of $\Ad(\phi(t, w_{1}, w_{2}))$ on $\jrad \qa = \spn\{e_{2}, e_{3}\}$ is given by the affine map
\begin{align}
\begin{pmatrix} 
x_{1} \\
x_{2}  \\
\end{pmatrix}\to \begin{pmatrix} 
 \cn_{-\al_{1}}(2t) & \al_{1}\sn_{-\al_{1}}(2t)\\
\sn_{-\al_{1}}(2t)&  \cn_{-\al_{1}}(2t) \\
\end{pmatrix}\begin{pmatrix} 
x_{1} \\
x_{2}  \\
\end{pmatrix} + \begin{pmatrix} 
w_{1} \\
w_{2}  \\
\end{pmatrix}.
\end{align}
This shows that $\Ad(\sphere(\alg))$ is the connected component of the identity of the special Euclidean group acting on $\jrad \qa$ preserving the metric $(q^{2})^{2} - \al_{1}(q^{3})^{2}$.

In the case $(\al_{1}, \al_{2}) = (0, 0)$, $\sphere(\qa) = \{\pm 1 + q^{1}e_{1} + q^{2}e_{2} + q^{3}e_{3}\}$ has two connected components and it is simpler to consider $\sphere_{0}(\qa)= \{1 + q^{1}e_{1} + q^{2}e_{2} + q^{3}e_{3}\}$. Because the Lie algebra $(\im \qa, [\dum, \dum])$ is $2$-step nilpotent with one-dimensional center $\spn\{e_{3}\}$, by Lemma \ref{heisenberglemma} it is isomorphic to $\heisen_{3}(\rea)$.
The Lie algebra exponential map $\exp:\im \qa \to \sphere_{0}(\qa)$ is the diffeomorphism given by $\exp(q) = \ga(1) = 1 + q$, where $\ga(t) = e^{tq}$ is the one-parameter subgroup such that $\ga(0) = 1$ and $\dot{\ga}(0) = q \in \im \qa$. 

For $c \in \rea$, define $\jsphere_{c}(\qa)  = \{q \in \im\qa: n(q) = c\} = \{q \in \qa: \bar{q} = -q, q^{2} = -c\}$.
Note that $\jsphere(\qa)$ can have at most two connected components. 
Each $r \in \im \qa$ determines a map $\Psi_{r}:\sphere(\qa) \to \jsphere_{n(r)}(\qa)$ defined by $\Psi_{r}(q) = \Ad(q)r$. The image $\Psi_{r}(\sphere(\qa))$ is the adjoint orbit of $r \in \im \qa$. 
By definition $\Psi_{r}(qe^{tr}) = \Psi_{r}(q)$ for all $t \in \rea$, so the one-parameter subgroup $qe^{tr}$ of $\sphere(\qa)$ is contained in the fiber $\Psi_{r}^{-1}(\Psi_{r}(q))$. 
For $w \in T_{q}\sphere(\qa)$ the differential $T\Psi_{r}(q)w = [w\bar{q}, \Psi_{r}(q)] = \Ad(q)([\bar{q}w, r])$, where it is used that $w\bar{q} = -q\bar{w}$ is the condition defining $T_{q}\sphere(\qa)$. Equivalently $T_{q}\sphere(\qa) = L(q)\im \qa$, so every element to $T_{q}\sphere(\qa)$ has the form $w = qv$ for some $v \in \im \qa$. It follows that $T\Psi_{r}(q)w = 0$ if and only if $0 = [\bar{q}w, r] = [v, r]$. Hence $\ker T\Psi_{r} = L(q)\ker \ad_{\im\qa}(r)$. 
In particular, $qr \in \ker T\Psi_{r}(q)$.
 By \eqref{prerodrigues}, the characteristic polynomial of $\ad_{\im \qa}(r)$ is $t(t^{2} + 4n(r))$, so $\ad_{\im \qa}(r)$ has rank $2$ if and only if $n(r) \neq 0$. In this case $\ker T\Psi_{r} = \spn\{qr\}$ and the adjoint orbit $\Psi_{r}(\im \qa)$ is a two-dimensional smooth submanifold equal to the connected component of $\jsphere_{n(r)}(\qa)$ containing $r$. In general the dimension of the adjoint orbit $\Psi_{r}(\im\qa)$ is $\rank \ad_{\im \qa}(r)$

For $n(r) \neq 0$ the map $\Psi_{r} \to \Psi_{r}(\im \qa)$ is a version of the Hopf fibration. 

When $(\al_{1}, \al_{2}) = (-1, -1)$ the adjoint orbits with positive $n(r)$ are the spheres $\{s \in \im \qa: n(s) = n(r)\}$ centered on the origin in $\im \qa$ and the only adjoint orbit with $n(r) =0$ is the origin.
When $(\al_{1}, \al_{2}) = (1, 1)$ there are three kinds of adjoint orbits. Those with $n(r) < 0$ are the $1$-sheeted hyperboloids $\{s \in \im \qa: n(s) = n(r)\}$ centered on the axis spanned by $e_{3}$. Those with $n(r) > 0$ are the two connected components of the $2$-sheeted hyperboloids $\{s \in \im \qa: n(s) = n(r)\}$. There are three adjoint orbits with $n(r) = 0$, the two connected components of the complement of the origin in $\{s \in \im \qa: n(s) = 0\}$ and the origin itself.

When $(\al_{1}, \al_{2}) = (-1, 0)$, the adjoint orbits with $n(r) > 0$ are the two planes $\{s \in \im\qa: s^{1} = \pm \sqrt{n(r)}\}$, while the adjoint orbits with $n(r) = 0$ are the origin and the circles centered on the origin in the plane $\{s \in \im \qa: s^{1} = 0\}$.
When $(\al_{1}, \al_{2}) = (1, 0)$, the adjoint orbits with $n(r) < 0$ are the two planes $\{s \in \im\qa: s^{1} = \pm \sqrt{-n(r)}\}$, while the adjoint orbits with $n(r) = 0$ are the origin and the connected components of the level sets of $(q^{2})^{2} - (q^{3})^{2}$ in the plane $\{s \in \im \qa: s^{1} = 0\}$.
When $(\al_{1}, \al_{2}) = (0, 0)$, the adjoint orbits planes $\{s \in \im\qa: s^{1} = c\}$ for $c \neq 0$ and the individual points of the plane $\{s \in \im \qa: s^{1} = 0\}$.

Two pairs $(\g, t)$ and $(\bar{\g}, \bar{t})$ where $t \in \g$ and $\bar{r} \in \bar{\g}$ are isomorphic if there is a Lie algebra isomorphism $\g \to \bar{\g}$ that sends $t \to \bar{t}$. The connections associated by Theorem \ref{3dssuniquetheorem} with two adjoint orbits in $\g$ are isomorphic if there is an automorphism of $\g$ mapping one adjoint orbit to the other. Consequently, Theorem \ref{3dssuniquetheorem} yields a bijection between nontrivial $\Aut(\g)$ orbits in $\g$ and the isomorphism classes of left-invariant conelike radiant structures on a three-dimensional unimodular Lie group with Lie algebra $\g$.

\begin{lemma}\label{gqalemma}
Let $G$ be a unimodular three-dimensional Lie group with Lie algebra $\g$ and suppose $t \in \g$ is $\ad_{\g}(t)$-semisimple, so $B_{\g}(t, t) \neq 0$. Up to an inner isomorphism of $\g$, $(\g, t)$ has the form $(\g = \im \qaf[\al_{1}, \al_{2}]\rea, t = \ka^{-1}e_{1})$ where $\ka \in (0, \infty)$ and the parameters $\al_{1}$ and $\al_{2}$ are as in the following table:
\begin{align}\label{paramtable}
\begin{tabular}{|c|c|l|l|}
\hline
$\al_{1} = \sign B_{\g}(t, t)$ & $\al_{2} = -\ep$ &  $B_{\g}(t, t) = -8n(t)$ & $\g = \im \qaf[\al_{1}, \al_{2}]\rea$\\
\hline
$-1$ & $-1$ &  $-8\ka^{-2} < 0$ & $\so(3)$\\
\hline
$1$ & $1$ &  $8\ka^{-2} > 0$ & $\sll(2, \rea)$\\
\hline
$-1$ & $1$ & $-8\ka^{-2} < 0$ & $\sll(2, \rea)$\\
\hline
$-1$ & $0$ &  $-8\ka^{-2} < 0$ & $\euc(2)$\\
\hline
$1$ & $0$ & $8\ka^{-2} > 0$ & $\euc(1,1)$\\
\hline
\end{tabular}
\end{align}
\end{lemma}

\begin{proof}
Given $0 \neq \ka \in \rea$, the basis 
\begin{align}\label{qaunibasis}
& t = \ka^{-1}e_{1}, && a = \tfrac{1}{\sqrt{2}}e_{2}, &&b = \tfrac{1}{\sqrt{2}\ka}e_{3} 
\end{align}
of $\im \qaf[\al_{1}, \al_{2}]\rea$ has brackets 
\begin{align}\label{qaunibrackets}
& [t, a] = 2b, & & [t, b] = \tfrac{1}{4}B_{\g}(t, t)a, && [a, b] = -\al_{2}t,
\end{align}
where there has been used $-8n(t) = B_{\g}(t, t) = 8\ka^{-2}\al_{1}$. The brackets \eqref{qaunibrackets} are as in \eqref{unibasis} provided $\ep = -\al_{2}$. By Lemma \ref{3dunilemma} any unimodular three-dimensional Lie algebra containing an $\ad$-semisimple element is isomorphic to one with a basis as in \eqref{unibasis}, so this shows that $\g$ is isomorphic to $\im \qaf[\al_{1}, \al_{2}]\rea$ for some $(\al_{1}, \al_{2} = -\ep)$ as in Table \ref{paramtable}, where $\ep \in \{0, 1, -1\}$ and $B_{\g}(t, t) < 0$ if $\ep = 1$. In every case there are at most two distinct adjoint orbits corresponding with a given value of $B_{\g}(t, t)$, and by choosing $\al_{1}$ suitably, it can always be supposed that $t$ is a multiple of $e_{1}$.

If $-\al_{2} = \ep = 1$ and $B_{\g}(t, t) < 0$, then $\g$ is compact and simple. In this case $\al_{1} = -1$ and the adjoint orbits of $\sphere(\qaf[-1, -1]\rea) = S^{3}$ are the spheres comprising vectors of constant Killing norm, so after an inner automorphism of $\g$ it can be supposed that $t =\ka^{-1}e_{1}$ for some $\ka \in (0, \infty)$.

If $\g$ is simple but noncompact, then it can be supposed that $-\al_{2} = \ep = -1$ and $\al_{1}$ is either $1$ or $-1$. The adjoint orbits of $\sphere(\qaf[1, 1]\rea)$ with $B_{\g}(t, t) < 0$ are $1$-sheeted hyperboloids, so in this case it can be supposed that $\al_{1} = 1$ and $t = \ka^{-1}e_{1}$ with $\ka \in (0, \infty)$. The adjoint orbits with $B_{\g}(t, t) > 0$ are connected components of $2$-sheeted hyperboloids and $\sphere(\qaf[-1, 1]\rea)$ acts transitively on each of these connected components, so it can be supposed $\al_{1} = -1$ and $t = \ka^{-1}e_{1}$ with $\ka \neq 0$.

If $\g$ is not simple then it is solvable because it contains an $\ad$-semisimple element, so $\ep =0$ but $\al_{1}\al_{2} \neq 0$. It can be supposed that $-\al_{2} = \ep = 0$ and $\g$ is isomorphic to $\euc(2) \simeq \im\qaf[-1, 0]\rea$ or $\euc(1, 1)\simeq \im\qaf[1, 0]\rea$. In either case, each of the two connected components of a nonzero level set of $B_{\g}(t, t)$ is an adjoint orbit of $\sphere(\qaf[\pm 1, 0]\rea)$, and so it can be supposed $t = \ka^{-1}e_{1}$ with $\ka \neq 0$.

The preceding shows that up to inner automorphism $(\g, t)$ has the form $(\g = \im \qaf[\al_{1}, \al_{2}]\rea, t = \ka^{-1}e_{1})$ where the parameters $\al_{1}$, $\al_{2}$ and $\ka$ are as in Table \ref{paramtable2}.
\begin{align}\label{paramtable2}
\begin{tabular}{|c|c|l|l|l|}
\hline
$\al_{1}$ & $\al_{2} = -\ep$ & range for $\ka$ & $B_{\g}(t, t) = -8n(t)$ & $\g = \im \qaf[\al_{1}, \al_{2}]\rea$\\
\hline
$-1$ & $-1$ & $(0, \infty)$ & $-8\ka^{-2} < 0$ & $\so(3)$\\
\hline
$1$ & $1$ & $(0, \infty)$ & $8\ka^{-2} > 0$ & $\sll(2, \rea)$\\
\hline
$-1$ & $1$ & $(-\infty, 0) \cup(0, \infty)$ & $-8\ka^{-2} < 0$ & $\sll(2, \rea)$\\
\hline
$-1$ & $0$ & $(-\infty, 0) \cup (0, \infty)$ & $-8\ka^{-2} < 0$ & $\euc(2)$\\
\hline
$1$ & $0$ & $(-\infty, 0) \cup (0, \infty)$ & $8\ka^{-2} > 0$ & $\euc(1,1)$\\
\hline
\end{tabular}
\end{align}
To complete the proof, it suffices to observe that in each case where $\ka < 0$ is possible there is an outer automorphism of $\g$ sending $t$ to $-t$.
\end{proof}

\begin{theorem}\label{qaconetheorem}
Let $G$ be a unimodular three-dimensional Lie group with Lie algebra $\g$ and suppose $t \in \g$ is $\ad_{\g}(t)$-semisimple, so $B_{\g}(t, t) \neq 0$. The Lie algebra $\g$ is isomorphic to $\im \qaf[\al_{1}, \al_{2}]\rea$ for some $\al_{1}, \al_{2} \in \{0, \pm 1\}$ and there is a unique left-invariant torsion-free affine connection $\nabla$ on $G$ constituting with $\rad = E^{t}$, where $t = \ka^{-1}e_{1}$, a left-invariant conelike radiant structure having antisymmetric Ricci tensor, and it has the form
\begin{align}\label{qacone}
\begin{aligned}
&\nabla_{E_{1}}E_{1} = \ka E_{1}, & &\nabla_{E_{2}}E_{1} = \ka E_{2},& &\nabla_{E_{3}}E_{1} = \ka E_{3},\\
&\nabla_{E_{1}}E_{2} = \ka E_{2} + 2E_{3},& &\nabla_{E_{2}}E_{2} = 4\al_{2}\ka^{-1}E_{1},& &\nabla_{E_{3}}E_{2} = \al_{2}E_{1},\\
&\nabla_{E_{1}}E_{3} = \ka E_{3} + 2\al_{1}E_{2},& &\nabla_{E_{2}}E_{3} = -\al_{2}E_{1},&&\nabla_{E_{3}}E_{3} = -4\al_{1}\al_{2}\ka^{-1}E_{1},
\end{aligned}
\end{align}
where $\{E_{1}, E_{2}, E_{3}\}$ is the left-invariant frame satisfying \eqref{3dbrackets}. Equivalently, with respect to the dual left-invariant coframe $\{\ep^{1}, \ep^{2}, \ep^{3}\}$, 
\begin{align}\label{qaconedual}
\begin{split}
\nabla \ep^{1} - \tfrac{1}{2}d\ep^{1} & = -\ka \ep^{1}\tensor \ep^{1} -4\al_{2}\ka^{-1}\ep^{2}\tensor \ep^{2} + 4\al_{1}\al_{2}\ka^{-1}\ep^{3}\tensor \ep^{3},\\
\nabla \ep^{2} - \tfrac{1}{2}d\ep^{2} & = -\ka(\ep^{1}\tensor \ep^{2} + \ep^{2} \tensor \ep^{1}) - \al_{1}(\ep^{1}\tensor \ep^{3} + \ep^{3}\tensor \ep^{1}),\\
\nabla \ep^{3} - \tfrac{1}{2}d\ep^{3} & = -(\ep^{1}\tensor \ep^{2} + \ep^{2} \tensor \ep^{1}) - \ka(\ep^{1}\tensor \ep^{3} + \ep^{3}\tensor \ep^{1}).
\end{split}
\end{align}
\begin{enumerate}
\item The curvature and Ricci curvature of $\nabla$ have the form
\begin{align}\label{qacurvature}
\begin{aligned}
&R(\dum, \dum) =-3\al_{2}\ka (\ep^{2}\wedge \ep^{3}) \tensor \left(\ep^{2}\tensor E_{2} + \ep^{3}\tensor E_{3}\right)
= -\tfrac{3}{2}\ka d\ep^{1} \tensor \left(\ep^{2}\tensor E_{2} + \ep^{3}\tensor E_{3}\right),\\
&\ric(\dum, \dum) = -3\al_{2}\ka \ep^{2}\wedge \ep^{3} = -\tfrac{3}{2}\ka d\ep^{1}.
\end{aligned}
\end{align}
\item The connection $\nabla$ is not projectively flat if $\al_{2} \neq 0$ (it is flat if $\al_{2} = 0$).
\item For $\be = \ka \ep^{1}$, the symmetric tensor
\begin{align}\label{qagdefined}
G = \nabla \be - \tfrac{1}{2}d\be + 2\be \tensor \be & = \ka^{2} \ep^{1}\tensor \ep^{1} -4\al_{2} \ep^{2}\tensor \ep^{2} + 4\al_{1}\al_{2} \ep^{3}\tensor \ep^{3},
\end{align}
satisfies $G(\rad, \dum) = \be$ and
\begin{align}\label{qag}
\nabla_{i}G_{jk} &= - 2\be_{i}G_{jk} + d\beta_{i(j}\be_{k)}= - 2\be_{i}G_{jk} + \tfrac{1}{3}\nabla_{(j}d\be_{k)i} ,\\
\label{qagskew}
\nabla_{[i}G_{j]k} &= -2\be_{[i}G_{j]k} + \tfrac{1}{2}\be_{k}d\be_{ij} - \tfrac{1}{2}\be_{[i}d\be_{j]k}= -2\be_{[i}G_{j]k}  - \tfrac{1}{4}\nabla_{k}d\be_{ij}.
\end{align}
(Abstract indices are used because they facilitate indication of the symmetries).
\end{enumerate}
\end{theorem}

\begin{proof}
That $\g$ is isomorphic to $\im \qaf[\al_{1}, \al_{2}]\rea$ for some $\al_{1}, \al_{2} \in \{0, \pm 1\}$ follows from Lemma \ref{gqalemma}.
That $\nabla$ has the form \eqref{qacone} follows from Corollary \ref{3dsscorollary} in conjunction with \eqref{qaunibasis} and \eqref{qaunibrackets}. The identities \eqref{qaconedual} follow from \eqref{qacone} and \eqref{3ddual}. That its curvature has the form \eqref{qacurvature} follows from \eqref{3dsscurvature}.
The projective Weyl tensor $B$ of $\nabla$ satisfies $B(t, \dum)\dum = \tfrac{3}{4}\al_{2}\ka \ep^{2}\wedge \ep^{3}$, so $\nabla$ is not projectively flat if $\al_{2} \neq 0$. For $G$ as in \eqref{qagdefined}, straightforward calculations using \eqref{qacone} or \eqref{qaconedual} show \eqref{qag}. Straightforward calculations using \eqref{qacone} or \eqref{qaconedual} show 
\begin{align}\label{qadbe}
\nabla_{i}d\be_{jk} = - 2\be_{i}d\be_{jk} + 2\be_{[j}d\be_{k]i},
\end{align}
and this yields the second equality of \eqref{qag}. Antisymmetrizing \eqref{qag} yields \eqref{qagskew}.
\end{proof}

\begin{example}\label{nosolutionexample}
This example completes Remark \ref{conenormalizationremark}.
Regard $S^{3}$ as the group of unit quaternions.
By Theorem \ref{qaconetheorem}, for $\ka \neq 0$, the left-invariant connection defined by \eqref{qacone} with $\al_{1} = -1 = \al_{2}$ 
determines with $\rad = \ka^{-1}E_{1}$ a conelike nonsingular radiant structure on $S^{3}$ having Ricci curvature $\ric(\dum, \dum) =  -\tfrac{3}{2}\ka d\ep^{1}$.
However, there is no one-form $\si$ on $S^{3}$ satisfying simultaneously $\si(E_{1}) = 0$ and $d\si = d\ep^{1}$. Were there, then
\begin{align}
\begin{split}
d(\si \wedge \ep^{1}) & = d\si \wedge \ep^{1} - \si \wedge d\ep^{1} = d\ep^{1} \wedge \ep^{1} + 2 \si \wedge \ep^{2}\wedge \ep^{3}\\
& = - 2\ep^{1}\wedge \ep^{2}\wedge \ep^{3} + 2\si(E_{1})\ep^{1}\wedge \ep^{2}\wedge \ep^{3}= - 2\ep^{1}\wedge \ep^{2}\wedge \ep^{3}
\end{split}
\end{align}
which by Stokes' Theorem contradicts that $\ep^{1}\wedge \ep^{2}\wedge \ep^{3}$ is a volume form on $S^{3}$. 
\end{example}

\begin{remark}
Suppose $\al_{1}\al_{2} \neq 0$. In this case, the tensor $G$ of \eqref{qagdefined} is a pseudo-Riemannian metric. It follows from \eqref{qag} that the Levi-Civita connection $D$ of $G$ is
\begin{align}
D = \nabla - 2\be_{(i}\delta_{j)}\,^{k} + G_{ij}\rad^{k} + \be_{(i}d\be_{j)}\,^{k}G^{pk},
\end{align}
where here and in what follows indices are raised and lowered using $G_{ij}$ and the inverse symetric bivector $G^{ij}$. In particular, for $\qaf[-1, -1]\rea$, $G^{\ka}$ (the superscript indicates the dependence on $\ka$) is a one-parameter family of Riemannian metrics, equal, up to normalizations, to the Berger metrics on the $3$-sphere \cite[Example $3.35$]{Cheeger-Ebin}. 

As in \eqref{hdsdefined} of Lemma \ref{liftedmetriclemma}, consider the connection 
\begin{align}\label{ewdiffdefined}
\Ds = \nabla + \Om_{ij}\,^{k} = \nabla + \Pi_{ij}\,^{k} + (s - 1)(2\be_{(i}\delta_{j)}\,^{k} - G_{ij}\rad^{k}) = D + s(2\be_{(i}\delta_{j)}\,^{k} - G_{ij}\rad^{k}),
\end{align}
where $\Pi_{ij}\,^{k} = \be_{(i}d\be_{j)p}G^{pk}$ and $\Om_{ij}\,^{k} = \Pi_{ij}\,^{k} + (s - 1)(2\be_{(i}\delta_{j)}\,^{k} - G_{ij}\rad^{k})$. From \eqref{qag} it follows that $\Ds_{i}G_{jk}  = -2s\be_{i}G_{jk}$, so that $(\Ds, G)$ is a Weyl structure for every $s \in \rea$. For $s = 0$, $\Ds = D$ is the Levi-Civita connection of $G_{ij}$. 

Because $\rad^{p}G_{ip} = \be_{i}$ and $\rad^{p}d\be_{ip} = 0$, $\Pi_{ip}\,^{p} = 0$. 
By \eqref{qadbe}, $G^{pq}\nabla_{p}d\be_{iq} = 0$, and, by \eqref{qag}, $G^{pq}\nabla_{i}G_{pq} = -6\be_{i}$ and $G^{pq}\nabla_{p}G_{qi} = -2\be_{i}$. Together with these observations, the same calculations showing \eqref{richds} of Lemma \ref{liftedmetriclemma} show $\Om_{ip}\,^{p} = 3(s-1)\be_{i}$ and
\begin{align}\label{lierichds2a}
\begin{aligned}
&\Om_{ij}\,^{p}\Om_{pq}\,^{q} = 3(s-1)^{2}\left(2\be_{i}\be_{j} - G_{ij}\right),& 
&\Om_{ip}\,^{q}\Om_{jq}\,^{p} = \tfrac{1}{4}\be_{i}\be_{j}d\be_{p}\,^{q}d\be_{q}\,^{p} + (s-1)^{2}\left(5\be_{i}\be_{j} - 2G_{ij}\right),&
\end{aligned}
\end{align}
\begin{align}\label{lierichds2b}
\begin{aligned}
&\nabla_{i}\Om_{jp}\,^{p} = 3(s-1)(G_{ij} + \tfrac{1}{2}d\be_{ij} - 2\be_{i}\be_{j}),&
&\nabla_{p}\Om_{ij}\,^{p} = \tfrac{1}{2}d\be_{i}\,^{p}d\be_{pj} + (s-1)\left(G_{ij} - 4\be_{i}\be_{j}\right).&
\end{aligned}
\end{align}
Substituting \eqref{lierichds2a} and \eqref{lierichds2b} into \eqref{ewdiffdefined} yields that the Ricci curvature $\ric(\Ds)_{ij}$ of $\Ds$ is related to the Ricci curvature $R_{ij} = -\tfrac{3}{2}d\be_{ij}$ of $\nabla$ by 
\begin{align}\label{ew1}
\begin{aligned}
\ric(\Ds)_{ij} &= R_{ij} + \nabla_{p}\Om_{ij}\,^{p} - \nabla_{i}\Om_{pj}\,^{p} + \Om_{pq}\,^{p}\Om_{ij}\,^{q} - \Om_{ip}\,^{q}\Om_{jq}\,^{p}\\
& =  -\tfrac{3}{2}sd\be_{ij}  + \tfrac{1}{2}d\be_{i}\,^{p}d\be_{pj} + (1 - s^{2})G_{ij} + \left(s^{2} - 1 - \tfrac{1}{4}d\be_{p}\,^{q}d\be_{q}\,^{p}\right)\be_{i}\be_{j}.
\end{aligned}
\end{align}
The tensor corresponding with $d\be_{i}\,^{j} = G^{jp}d\be_{ip}$ is $\tfrac{\ka}{2}\ep^{3} \tensor E_{2} + \tfrac{\ka}{2\al_{1}}\ep^{2}\tensor E_{3}$, and it follows that the tensor corresponding with $d\be_{ip}d\be_{qj}G^{pq}$ is
\begin{align}\label{dbecontracted}
&-\tfrac{\ka^{2}\al_{2}}{\al_{1}}\ep^{2}\tensor \ep^{2} + \ka^{2}\al_{2}\ep^{3} \tensor \ep^{3} = \tfrac{\ka^{2}}{4\al_{1}}(G - \be \tensor \be),
\end{align}
so that
\begin{align}\label{ewp1}
& \tfrac{1}{2}d\be_{ip}d\be^{p}\,_{j} = \tfrac{\ka^{2}}{8\al_{1}}(G_{ij} - \be_{i}\be_{j}), && -\tfrac{1}{4}d\be_{p}\,^{q}d\be_{q}\,^{p}\be_{i}\be_{j} = -\tfrac{\ka^{2}}{8\al_{1}}\be_{i}\be_{j}.
\end{align}
Substituting \eqref{ewp1} in \eqref{ew1} yields
\begin{align}
\begin{aligned}
\ric(\Ds)_{ij}=& -\tfrac{3}{2}sd\be_{ij} + \left(\tfrac{\ka^{2}}{8\al_{1}} + 1 - s^{2}\right)G_{ij} + \left(s^{2} - 1 - \tfrac{\ka^{2}}{4\al_{1}}\right) \be_{i}\be_{j}.
\end{aligned}
\end{align}
If $s =\pm \sqrt{1 +  \tfrac{\ka^{2}}{4\al_{1}}}$ (when $\al_{1} = -1$, this necessitates $\ka^{2} \leq 4$), then
\begin{align}
\begin{aligned}
\ric(\Ds)_{ij}=& -\tfrac{3}{2}sd\be_{ij} + \left(\tfrac{\ka^{2}}{8\al_{1}} + 1 - s^{2}\right)G_{ij}= -\tfrac{3}{2}sd\be_{ij} -\tfrac{\ka^{2}}{8\al_{1}}G_{ij},
\end{aligned}
\end{align}
so $(\Ds, G)$ is an Einstein-Weyl structure. In the case $\al_{1} = \al_{2} = -1$, these are the well-known Einstein-Weyl structures on the Berger spheres \cite[Equation $5.10$]{Jones-Tod-minitwistor} (the case $s = 0$ being the standard round metric of scalar curvature $3/2$). See also  \cite[p. $387$]{Pedersen-Swann-submersions} and \cite[Section $6$]{Calderbank-Pedersen}. The resulting Einstein-Weyl structures in the Lorentz signature case $\al_{1} = -1$ and $\al_{2} = 1$ are described in \cite{Bengtsson-Sandin}. 
\end{remark}


\begin{lemma}
A biinvariant section of $S^{2}T\sphere(\qa)$ is a constant multiple of $\Ga = -\al_{2}E_{1}\tensor E_{1} - \al_{1}E_{2}\tensor E_{2} + E_{3}\tensor E_{3}$. For any left-invariant $\theta  = \sum_{i = 1}^{3}\theta_{i}\ep^{i} \in \Ga(T^{\ast}\sphere(\qa))$, 
\begin{align}\label{ombiimdual}
\theta \wedge d\theta = 2(\al_{2}\theta_{1}^{2} + \al_{1}\theta_{2}^{2} - \theta_{3}^{2})\ep^{1}\wedge \ep^{2}\wedge \ep^{3} = -2\Ga(\theta, \theta) \ep^{1}\wedge \ep^{2}\wedge \ep^{3} .
\end{align}
Consequently:
\begin{enumerate}
\item $\theta$ is a contact one-form if and only if $\Ga(\theta, \theta) \neq 0$.
\item $\Ga(\theta, \theta)$ is constant on a coadjoint orbit of $\sphere(\alg)$.
\item $\ker \theta$ generates an integrable left-invariant subbundle of $T\sphere(\qa)$ if and only if $\Ga(\theta, \theta) = 0$. This can occur for $\theta \neq 0$ if and only if at least one of $\al_{1}$ and $\al_{2}$ is positive.
\item If $\al_{1} \leq 0$ and $\al_{2} \leq 0$ then $\im \qa$ contains no two-dimensional subalgebras.
\item If $\al_{1} = \al_{2} = 1$, then, for $\theta^{\pm} = \ep^{2}\pm \ep^{3}$, $\ker \theta^{\pm} = \spn\{E_{1}, E_{2} \mp E_{3}\}$ are integrable, $\nabla$-totally geodesic subbundles of $T\sphere(\qa)$ where $\nabla$ is the connection of Theorem \ref{qaconetheorem}.
\end{enumerate}
\end{lemma}
\begin{proof}
Identify $\Ga \in S^{2}T\sphere(\qa)$ with $\ga = \sum_{i = 1}^{3}a_{i}e_{i}\tensor e_{i} + \sum_{1 \leq i < j \leq 3}a_{ij}(e_{i}\tensor e_{j} + e_{j}\tensor e_{i}) \in S^{2}\im \qa$. That $\Ga$ be biinvariant is equivalent to $r \cdot \ga = 0$ for all $r \in \im \qa$. This yields
\begin{align}
\begin{split} 
0 & =e_{1}\cdot \ga = 2(a_{2} + \al_{1}a_{3})(e_{2}\tensor e_{3} + e_{3}\tensor e_{2}) + 2a_{12}(e_{1} \tensor e_{3} + e_{3}\tensor e_{1}) \\
&\quad+ 2\al_{1}a_{13}(e_{1}\tensor e_{2} + e_{2}\tensor e_{1}) + 4a_{23}(\al_{1}e_{2}\tensor e_{2} + e_{3}\tensor e_{3}),
\end{split}
\end{align}
so that $a_{12} = 0$, $a_{13} = 0$, $a_{23} = 0$, and $a_{2} = -\al_{1}a_{3}$. Hence $\ga = a_{1}e_{1}\tensor e_{1} - \al_{1}a_{3}e_{2}\tensor e_{2} + a_{3}e_{3}\tensor e_{3}$. Consequently,
\begin{align}
\begin{split} 
0 & =e_{2}\cdot \ga = -2(a_{1} + \al_{2} a_{3})(e_{1}\tensor e_{3} + e_{3}\tensor e_{1}),
\end{split}
\end{align}
so that $a_{1} = -\al_{2}a_{3}$. Thus $\ga = a_{3}(-\al_{2}e_{1}\tensor e_{1} -\al_{1}e_{2}\tensor e_{2} + e_{3}\tensor e_{3})$, and it is straightforward to check that $e_{3}\cdot \ga = 0$. Straightforward calculations show \eqref{ombiimdual}.

By \eqref{ombiimdual}, $\ker \theta$ generates an integrable left-invariant distribution if and only if $\Ga(\theta, \theta) = 0$, and this has nontrivial solutions if and only if $\theta_{3}^{2} = \al_{2}\theta_{1}^{2} + \al_{1}\theta_{2}^{2}$, which occurs if and only if at least one of $\al_{1}$ and $\al_{2}$ is positive. If $\im \qa$ has a two-dimensional subalgebra, then this subalgebra generates an integrable left-invariant subbundle of $T\sphere(\qa)$, and a one-form $\theta$ annihilating the subalgebra must satisfy $\Ga(\theta, \theta) = 0$, so if $\al_{1} \leq 0$ and $\al_{2} \leq 0$ then $\im \qa$ has no two-dimensional subalgebras.

If $\al_{1} = \al_{2} = 1$, then $d\theta^{\pm} = \pm 2 \theta^{\pm}\wedge \ep^{1}$ and, by \eqref{qaconedual}, $\nabla \theta^{\pm} = \mp \theta^{\pm}\wedge \ep^{1} - (\ka \pm 1)(\ep^{1}\tensor \theta^{\pm} + \theta^{\pm}\tensor \ep^{1})$, which shows that $\ker \theta^{\pm}$ are integrable and totally geodesic subbundles.
\end{proof}

\bibliographystyle{amsplain}
\def\polhk#1{\setbox0=\hbox{#1}{\ooalign{\hidewidth
  \lower1.5ex\hbox{`}\hidewidth\crcr\unhbox0}}} \def\cprime{$'$}
  \def\cprime{$'$} \def\cprime{$'$}
  \def\polhk#1{\setbox0=\hbox{#1}{\ooalign{\hidewidth
  \lower1.5ex\hbox{`}\hidewidth\crcr\unhbox0}}} \def\cprime{$'$}
  \def\cprime{$'$} \def\cprime{$'$} \def\cprime{$'$}
  \def\polhk#1{\setbox0=\hbox{#1}{\ooalign{\hidewidth
  \lower1.5ex\hbox{`}\hidewidth\crcr\unhbox0}}} \def\cprime{$'$}
  \def\Dbar{\leavevmode\lower.6ex\hbox to 0pt{\hskip-.23ex \accent"16\hss}D}
  \def\cprime{$'$} \def\cprime{$'$} \def\cprime{$'$} \def\cprime{$'$}
  \def\cprime{$'$} \def\cprime{$'$} \def\cprime{$'$} \def\cprime{$'$}
  \def\cprime{$'$} \def\cprime{$'$} \def\cprime{$'$} \def\dbar{\leavevmode\hbox
  to 0pt{\hskip.2ex \accent"16\hss}d} \def\cprime{$'$} \def\cprime{$'$}
  \def\cprime{$'$} \def\cprime{$'$} \def\cprime{$'$} \def\cprime{$'$}
  \def\cprime{$'$} \def\cprime{$'$} \def\cprime{$'$} \def\cprime{$'$}
  \def\cprime{$'$} \def\cprime{$'$} \def\cprime{$'$} \def\cprime{$'$}
  \def\cprime{$'$} \def\cprime{$'$} \def\cprime{$'$} \def\cprime{$'$}
  \def\cprime{$'$} \def\cprime{$'$} \def\cprime{$'$} \def\cprime{$'$}
  \def\cprime{$'$} \def\cprime{$'$} \def\cprime{$'$} \def\cprime{$'$}
  \def\cprime{$'$} \def\cprime{$'$} \def\cprime{$'$} \def\cprime{$'$}
  \def\cprime{$'$} \def\cprime{$'$} \def\cprime{$'$}
\providecommand{\bysame}{\leavevmode\hbox to3em{\hrulefill}\thinspace}
\providecommand{\MR}{\relax\ifhmode\unskip\space\fi MR }
\providecommand{\MRhref}[2]{%
  \href{http://www.ams.org/mathscinet-getitem?mr=#1}{#2}
}
\providecommand{\href}[2]{#2}

\end{document}